\documentclass[letterpaper]{article} 
\usepackage[preprint]{aaai2027}
\usepackage[hyphens]{url}  
\usepackage{graphicx} 
\def\UrlFont{\rm}  
\usepackage{natbib}  
\usepackage{caption} 
\usepackage{amsmath}
\usepackage{amssymb}
\usepackage{amsthm}
\usepackage{mathtools}
\usepackage{bm}
\usepackage{booktabs}
\usepackage{algorithm}
\usepackage{algorithmic}
\usepackage{subcaption}
\usepackage{multirow}
\usepackage{microtype}
\usepackage{nicefrac}
\usepackage{booktabs}

\newtheorem{theorem}{Theorem}
\newtheorem{proposition}{Proposition}
\newtheorem{lemma}{Lemma}
\newtheorem{corollary}{Corollary}
\newtheorem{assumption}{Assumption}
\newtheorem{definition}{Definition}
\newtheorem{remark}{Remark}

\newtheorem{openproblem}{Open Problem}

\newcommand{\R}{\mathbb{R}}
\newcommand{\E}{\mathbb{E}}
\newcommand{\Prob}{\mathbb{P}}
\newcommand{\norm}[1]{\left\lVert #1\right\rVert}
\newcommand{\inner}[2]{\left\langle #1,\, #2\right\rangle}
\newcommand{\grad}{\nabla}
\newcommand{\eps}{\varepsilon}
\newcommand{\Lz}{L_0}
\newcommand{\Lo}{L_1}
\newcommand{\Rb}{\bar{R}}
\newcommand{\kbar}{\bar{\kappa}}
\newcommand{\Deltabar}{\bar{\Delta}}
\newcommand{\Gbar}{\bar{G}}
\newcommand{\gest}{\hat{g}}
\newcommand{\Gest}{\hat{G}}
\newcommand{\xhat}{\hat{x}}
\newcommand{\dist}{\mathrm{dist}}
\newcommand{\proj}{\Pi}
\newcommand{\ball}[2]{B\!\left(#1,\,#2\right)}
\newcommand{\ARCSG}{\textsf{ARC-SG}}
\newcommand{\RSTM}{\textsf{R-STM}}
\newcommand{\STM}{\textsf{STM}}
\newcommand{\GradEst}{\textsf{GradEst}}
\newcommand{\PhaseI}{\textsf{Phase~I}}
\newcommand{\PhaseII}{\textsf{Phase~II}}
\newcommand{\logp}{\log_{+}}
\newcommand{\tO}{\widetilde{O}}
\newcommand{\Lball}{L_{\mathrm{ball}}}
\newcommand{\rsafe}{r_{\mathrm{safe}}}
\DeclareMathOperator*{\argmin}{arg\,min}
\DeclareMathOperator*{\argmax}{arg\,max}

\newcommand{\epsg}{\varepsilon_{\mathrm g}}

\newcommand{\ip}[2]{\left\langle #1,#2\right\rangle}

\newcommand{\B}{\mathbb{B}}

\newcommand{\Dinit}{\Delta_{\mathrm{init}}}
\newcommand{\Ddef}{\Delta_{\mathrm{def}}}
\newcommand{\Dcert}{\Delta_{\mathrm{cert}}}
\newcommand{\Duser}{\Delta_{\mathrm{user}}}

\usepackage{tabularx}
\title{Localize, Restart, Accelerate: Stochastic Optimization under Generalized Smoothness}

\author{
Darina Dvinskikh\textsuperscript{\rm 1},
Alexander Gasnikov\textsuperscript{\rm 2},
Aleksandr Lobanov\textsuperscript{\rm 3}, Ilgam Latypov\textsuperscript{\rm 3}
}

\affiliations{
\textsuperscript{\rm 1}HSE University, Moscow, Russia\\
\textsuperscript{\rm 2}Innopolis University, Innopolis, Republic of Tatarstan, Russia\\
\textsuperscript{\rm 3}MSU AI Center, Moscow, Russia\\
dmdvinskikh@hse.ru, gasnikov@yandex.ru, lobbsasha@mail.ru,
i.latypov@iai.msu.ru
}
\begin{document}

\maketitle

\begin{abstract}

We study stochastic convex optimization under asymmetric
\((L_0,L_1)\)-generalized smoothness, a model motivated by
machine-learning objectives whose local curvature may grow with
the gradient norm. We assume an unbiased first-order oracle with
additive norm-sub-Gaussian noise. Acceleration is difficult in
this setting because momentum may enter regions of much larger
curvature, while stochastic gradients cannot reliably certify
an unrestricted trajectory.
We propose \textsf{ARC-SG}, a two-phase accelerated method:
Phase~I reduces excessively large gradients using a
generalized-smoothness-aware stochastic step, then Phase~II solves
strongly convex proximal subproblems by a restarted accelerated
solver confined to certified smoothness balls. Exact proximal
points do not increase the gradient norm, allowing these
certificates to propagate through the outer loop. The contribution
is a query-by-query certified-localization construction with
explicit generalized-smoothness factors and a strongly convex
restart extension. \textsf{ARC-SG} achieves, with high probability,
an accelerated optimization contribution and smooth-subclass-optimal
statistical dependence on accuracy, up to logarithmic and
generalized-smoothness factors. Its convex accuracy exponents agree
with a contemporaneous public stochastic-acceleration result under a
broader smoothness and affine-variance model; our distinction is the
certified geometry, explicit parameter accounting, and strongly
convex guarantee. The results recover classical accelerated
stochastic rates when \(L_1=0\). Experiments on
objectives with unbounded gradients illustrate the two-phase
mechanism and its finite-budget advantage.


\end{abstract}

\section{Introduction}
\label{sec:intro}

We consider the stochastic convex optimization problem
\begin{equation}
\label{eq:problem}
    \min_{x\in\R^d}
    f(x)
    :=
    \E_{\xi\sim\mathcal D}\bigl[f_\xi(x)\bigr],
\end{equation}
where \(f:\R^d\to\R\) is convex and differentiable and is
accessed through an unbiased stochastic first-order oracle.
Classical first-order theory typically assumes that \(f\) is
globally \(L\)-smooth, that is, that \(\grad f\) is globally
Lipschitz continuous. A single global smoothness constant,
however, can be either invalid or highly uninformative when the
curvature varies substantially across the domain.

Such behaviour has been observed in modern machine-learning
objectives. In particular, \citet{zhang2020gradient} reported
that the local smoothness encountered during neural-network
training can be strongly correlated with the gradient norm.
This motivates the condition
   $ \norm{\grad^2 f(x)}
    \le
    L_0+L_1\norm{\grad f(x)},$
     for all  $x\in\R^d,$
known as \((L_0,L_1)\)- smoothness. The classical
smooth setting is recovered when \(L_1=0\), whereas \(L_1>0\)
allows the curvature to grow without a finite global Lipschitz
constant. The model also helps explain why gradient clipping and
normalization can be effective in regions where both gradients
and curvature are large. For twice-differentiable objectives,
the condition above implies, up to a universal rescaling of
\(L_0\) and \(L_1\), the asymmetric finite-difference condition
used in our analysis.

A substantial literature has developed around generalized
smoothness, first for nonconvex optimization
\citep{zhang2020improved,crawshaw2022robustness,
chen2023generalized,huebler2024parameter,
wang2023convergence} and subsequently for convex problems
\citep{koloskova2023revisiting,li2023convex,
gorbunov2024methods,vankov2024optimizing,
lobanov2024linear}. For convex optimization, the existing
literature now includes non-accelerated stochastic methods,
deterministic accelerated methods, and recent stochastic
acceleration.

On the stochastic side, \citet{sgd2025generalized} analyzed
non-accelerated SGD with constant, clipped, and
gradient-adaptive stepsizes, obtaining guarantees whose
deterministic optimization contribution decays as \(O(1/T)\).
\citet{gaash2025clipped} established high-probability guarantees
for clipped SGD that match classical smooth SGD up to
logarithmic and lower-order additive terms. More general
\(\ell\)-smooth objectives were studied by
\citet{li2023convex}, while related stochastic proximal-point
guarantees were obtained by \citet{tovmasyan2025proximal}.

On the deterministic side, \citet{gorbunov2024methods} and
\citet{vankov2024optimizing} developed accelerated methods whose
optimization error decreases as \(O(1/T^2)\) after an
\(L_1\)-dependent transient phase. 
\citet{tyurin2026near} obtained near-optimal deterministic rates
whose leading high-accuracy term has no nonconstant
multiplicative dependence on \(L_1R\), where
\(R\geq \norm{x_0-x^\star}\), $x_0$ is a starting point, $x_\star$ is a solution of  \eqref{eq:problem}.

On stochastic acceleration, \citet{yu2025stochastic} establish
high-probability guarantees for RSAG, SGD, and AdaGrad-Norm under
generalized smoothness and a relaxed affine-variance noise model.
Their general-noise convex function-gap rate is
\(\widetilde O(T^{-1/2})\), while an accelerated optimization
dependence is recovered under an additional low-noise condition.
The oracle models are related rather than disjoint: the additive
norm-sub-Gaussian assumption used here is the specialization
\(A=B=0\), \(C=\sigma^2\) of their affine-variance exponential-moment
condition (with fresh samples giving the corresponding conditional
form).

A still closer public result is the ICLR 2026 submission
\citep{anonymous2026nesterov}. Its stochastic Nesterov method has a
high-probability convex rate of the form
\(\widetilde O(T^{-2}+\sqrt{(A+B+C)/T})\) under a broader
\((L_0,L_1,L_2)\)-type generalized-smoothness condition and the same
affine-variance family. Setting their gradient exponent to \(p=1\),
\(L_2=0\), and \(A=B=0,C=\sigma^2\) specializes the displayed
assumptions to the gradient-dependent smoothness and additive-noise
regime considered here, and yields the same convex
\(\eps^{-1/2}\) optimization and \(\eps^{-2}\) statistical
exponents. As public prior art, its hidden parameter factors and
algorithmic certificates differ from ours. The present paper additionally
states a parameter-explicit strongly convex restart guarantee; the
comparison here is between the convex theorems.

Accordingly, we do not claim the first convex stochastic accelerated
accuracy exponent under generalized smoothness. Our more specific
contribution is a parameter-explicit construction in which every
stochastic query is confined to a certified ordinary-smoothness
region, together with a strongly convex restart extension. For
convex objectives, the relevant statistical benchmark is
   $ \Theta\left({\sigma^2R^2}/{\eps^2}\right),$
whereas for \(\mu\)-strongly convex objectives it is
   $ \Theta\left({\sigma^2}/{(\mu\eps)}\right)$
\citep{nemirovsky1983problem,nemirovski2009robust,
lan2012optimal,ghadimi2012optimal}.

This gap is not merely an artifact of existing analyses.
Accelerated methods are not descent methods: momentum may
temporarily move the iterates into regions with a substantially
larger gradient norm and hence a larger local smoothness scale.
Existing deterministic accelerated methods control this effect
through carefully coupled clipping rules
\citep{gorbunov2024methods} or monotone
\((L_0,L_1)\)-adaptive steps \citep{vankov2024optimizing}.
Both approaches use exact gradients to certify that the
trajectory remains in a region of controlled curvature. With
stochastic gradients, a single fluctuation may invalidate such
a trajectory-wide certificate, while repeatedly reconstructing
it would require prohibitively accurate gradient estimation.

We address this difficulty with \ARCSG{}
(\textbf{A}ccelerated, \textbf{R}estarted,
\textbf{C}ertified-ball method for
\textbf{S}tochastic optimization under
\textbf{G}eneralized smoothness). Rather than certifying an
unrestricted accelerated trajectory, ARC-SG builds the
certified region into the algorithm.

Phase~I is a stabilization stage that uses a
generalized-smoothness-aware stochastic step to reduce
excessively large gradients to the critical scale \(L_0/L_1\).
Phase~II approximately solves a sequence of strongly convex
proximal subproblems using a restarted accelerated stochastic
method confined to explicitly certified balls of radius
\(O(1/L_1)\). On each such ball, generalized smoothness reduces
to ordinary smoothness.
The proximal regularization is floored at the current certified
gradient scale. This floor keeps the exact proximal point inside
the certified ball and controls the condition number of the
subproblem. The analysis then exploits the fact that exact
ball-constrained proximal points do not increase the gradient
norm, allowing the local smoothness certificates to propagate
through the outer loop. A travel--halving argument controls the
number of productive proximal levels, while a capped schedule
and a certified finishing procedure ensure unconditional
termination.

The scope of our novelty is therefore the certified-query geometry,
the explicit \((L_0,L_1)\)-parameter accounting, and the strongly
convex restart extension, rather than priority for the convex
accuracy exponents. The guarantees cover both convex and strongly
convex objectives. In the strongly convex high-accuracy regime, the
statistical term has the optimal \(\sigma^2/(\mu\eps)\) scaling, and
when \(L_1=0\) our results recover classical accelerated stochastic
approximation.

The inner solver builds on accelerated stochastic approximation
and stochastic similar-triangles methods
\citep{nesterov1983method,nesterov2018lectures,
lan2012optimal,ghadimi2012optimal,ghadimi2013optimal,
gorbunov2020stochastic,sadiev2023high}. The proximal-level
architecture is related to recursive regularization for making
stochastic gradients small
\citep{nesterov2012make,allenzhu2018make,
foster2019complexity}. In ARC-SG, however, regularization serves
an additional geometric purpose: it certifies the local
smoothness region required by the accelerated stochastic inner
method.

Localized subproblems also have direct precedents. The ball
optimization oracle of \citet{carmon2020balloracle} is accelerated by
an outer method, and \citet{carmon2023resque} combine ball-oracle
acceleration with stochastic first-order estimation. The Broximal
method of \citet{gruntkowska2025broximal} studies minimization over
successive balls. A concurrent 2026 preprint---subsequent to those
earlier ball frameworks---by \citet{li2026trustregion} studies the
same formal ball-constrained
proximal objective
\(\min_{z\in B_t(x)}\{f(z)+\lambda\norm{z-x}^2/2\}\) in a
deterministic trust-region framework. Its parameter regimes are
designed to make the trust-region constraint active and obtain
linear decrease outside a target neighborhood. By contrast, our
gradient-scale regularization floor keeps the exact proximal point
strictly inside the certified ball, while the constraint prevents
the approximate noisy inner trajectory from leaving the region on
which ordinary smoothness has been certified. Thus the ball-proximal
form itself is not a novelty claim; the claimed contribution is its
use as a propagated stochastic smoothness certificate.

Table~\ref{tab:comparison} highlights the main contribution of this work.
Throughout the table, \(R\) denotes the initial-distance
quantity.  The target
\(\eps\) denotes function-gap accuracy, whereas
\(\varepsilon_{\mathrm g}\) in the recursive-regularization row
denotes gradient accuracy,
\(\norm{\grad f(x)}\le\varepsilon_{\mathrm g}\). For rows that depend on an initial function gap,
    $F_0:=f(x_0)-f^\star$
denotes the initial suboptimality in the corresponding result (\(F_0\) is interpreted separately in each row).
This generic notation is distinct from the ARC-SG inputs.
In the stochastic-NAG row, \(A,B,C\) are the affine-variance
parameters of the public ICLR 2026 submission
\citep{anonymous2026nesterov}, not parameters of ARC-SG.
For ARC-SG, \(R_0\) is a known distance certificate satisfying
    $R_0\ge\norm{x^0-x^\star},$
and we define
 $   \overline R:=3R_0,$
    and
    $\overline\kappa:=1+L_1\overline R.$
The ARC-SG initial-gap certificate is denoted by
\(\Delta_{\mathrm{cert}}\), not by \(F_0\).
 For AGMsDR,
\(\nu\) bounds the number of additional one-dimensional
line-search oracle calls per iteration. All entries suppress
absolute constants, polylogarithmic factors, and
accuracy-independent lower-order terms.
The two panels correspond to different oracle models and should
not be compared directly. Within the stochastic panel, the
guarantee type and parameterization differ across rows. The central
comparison is no longer a separation at the level of convex accuracy
exponents: the public ICLR 2026 submission
\citep{anonymous2026nesterov} attains the
same exponents in a broader model. ARC-SG instead contributes
query-by-query certified localization, explicit
generalized-smoothness factors, and a strongly convex restart result
under additive norm-sub-Gaussian noise.

\begin{table*}[t]
\centering
\footnotesize
\setlength{\tabcolsep}{3pt}
\renewcommand{\arraystretch}{1.12}
\resizebox{\textwidth}{!}{%
\begin{tabular}{@{}llllc@{}}
\toprule
\textbf{Method}
&
\textbf{Setting}
&
\textbf{Smoothness}
&
\textbf{Noise}
&
\textbf{Oracle complexity (leading terms)}
\\
\midrule

\multicolumn{5}{@{}l}{
    \emph{(A) Deterministic oracle (exact gradients)}
}
\\
\midrule

AGMsDR, line search \citep{vankov2024optimizing}
&
convex
&
\((\Lz,\Lo)\)
&
---
&
\(\displaystyle
(\nu+1)\left(
\sqrt{\frac{\Lz R^2}{\eps}}
+
(\Lo R)^{2/3}\log\frac{2F_0}{\eps}
\right)
\)
\\[2pt]

Clipped AGD \citep{gorbunov2024methods}
&
convex
&
\((\Lz,\Lo)\)
&
---
&
\(\displaystyle
\sqrt{
\frac{
\Lz\bigl(1+\Lo R e^{\Lo R}\bigr)R^2
}{\eps}
}
\)
\\[2pt]

Near-optimal AGM \citep{tyurin2026near}
&
convex
&
\((\Lz,\Lo)\)
&
---
&
\(\displaystyle
\sqrt{\frac{\Lz R^2}{\eps}}
+
\text{additive \(\Lo\)-terms}
\)
\\[2pt]

\textbf{\ARCSG{} specialization (\(\sigma=0\))}
&
convex
&
\((\Lz,\Lo)\)
&
---
&
\(\displaystyle
1+\Lo R_0
+
\bar{\kappa}^{1/2}
\sqrt{\frac{\Lz \bar R^2}{\eps}}
\)
\\[3pt]

\midrule
\multicolumn{5}{@{}l}{
    \emph{(B) Stochastic first-order oracle}
}
\\
\midrule

SGD \citep{nemirovski2009robust}
&
convex, in exp.
&
\(L\)
&
additive
&
\(\displaystyle
\frac{LR^2}{\eps}
+
\frac{\sigma^2R^2}{\eps^2}
\)
\\[2pt]

AC-SA \citep{lan2012optimal,ghadimi2012optimal}
&
convex, in exp.
&
\(L\)
&
additive
&
\(\displaystyle
\sqrt{\frac{LR^2}{\eps}}
+
\frac{\sigma^2R^2}{\eps^2}
\)
\\[2pt]

ClipSSTM \citep{gorbunov2020stochastic}
&
convex, high prob.
&
\(L\)
&
heavy-tailed
&
\(\displaystyle
\sqrt{\frac{LR^2}{\eps}}
+
\frac{\sigma^2R^2}{\eps^2}
\)
\\[2pt]

Recursive reg. \citep{foster2019complexity}
&
convex,
\(\norm{\grad f}\le\varepsilon_{\mathrm g}\)
&
\(L\)
&
additive
&
\(\displaystyle
\frac{\sqrt{L F_0}}{\varepsilon_{\mathrm g}}
+
\frac{\sigma^2}{\varepsilon_{\mathrm g}^2}
\)
\\[2pt]

SGD, \((\Lz,\Lo)\)-steps
\citep{sgd2025generalized}
&
convex, in exp.
&
\((\Lz,\Lo)\)
&
additive
&
\(\displaystyle
\frac{\Lz R^2}{\eps}
+
\frac{\sigma^2R^2}{\eps^2}
+
\text{burn-in}(\Lo)
\)
\\[2pt]

Clipped SGD \citep{gaash2025clipped}
&
convex, high prob.
&
\((\Lz,\Lo)\)
&
sub-Gaussian
&
\(\displaystyle
\frac{\Lz R^2}{\eps}
+
\frac{\sigma^2R^2}{\eps^2}
+
\text{additive terms}
\)
\\[2pt]

Acc.\ (RSAG) \citep{yu2025stochastic}
&
convex, high prob.
&
generalized
&
affine variance
&
\(\displaystyle
\eps^{-2}
\ \text{statistical (general noise)};
~
\eps^{-1/2}
\ \text{optimization (low noise)}
\)
\\[2pt]

Stoch.\ NAG \citep{anonymous2026nesterov}
&
convex, high prob.
&
\((L_0,L_1,L_2)\)-type
&
affine variance
&
\(\displaystyle
\eps^{-1/2}
+
\frac{A+B+C}{\eps^2}
\quad\text{(parameter factors suppressed)}
\)
\\[2pt]

\textbf{\ARCSG{} (Theorem~\ref{thm:convex})} & convex, high prob. & $(\Lz,\Lo)$ & norm-sub-Gauss. & $1+\Lo R_0 + \kbar^{\frac12}\sqrt{\frac{\Lz\Rb^2}{\eps}} + \kbar^{3}\frac{\sigma^2\Rb^2}{\eps^2}$ \\[3pt]
\textbf{\ARCSG{} (Theorem~\ref{thm:strongly})} & $\mu$-str.\ cvx, high prob. & $(\Lz,\Lo)$ & norm-sub-Gauss. & $(1+\Lo R_0)\left(1+\sqrt{\frac{\Lz}{\mu}}\right) + \sqrt{\frac{\Lz}{\mu}}\logp\!\left(\frac{\mu R_0^2}{\eps}\right) + \frac{\sigma^2}{\mu\eps}$\; $\left(\eps\le\frac{\mu}{36\Lo^2}\right)$ \\

\bottomrule
\end{tabular}%
}
\caption{Complexity comparison up to absolute constants and polylogarithmic factors. The stochastic-NAG row reports the public ICLR 2026 submission \citep{anonymous2026nesterov} for prior-art transparency. Its displayed assumptions specialize to the present convex model at $p=1$, $L_2=0$, and $A=B=0,C=\sigma^2$, so its accuracy exponents coincide with ours; the parameter factors and certification mechanisms differ. The \ARCSG{} rows omit the $\eps$-independent burn-ins shown in Theorems~\ref{thm:convex}--\ref{thm:strongly} and treat $\Dcert$ as an externally supplied certificate. Under the default certificate, with $q:=\Lo R_0$, the Phase-I deterministic entry is $\tO((1+q)^2)$ and its statistical burn-in is $\tO((1+q)^4\sigma^2\Lo^2/\Lz^2)$; see Remark~\ref{rem:convex-default} and the supplementary material.}
\label{tab:comparison}
\end{table*}

\paragraph{Contributions.} The main contributions  are the following:

\begin{itemize}
    \item We introduce \ARCSG{}, a two-phase algorithm combining
generalized-smoothness-aware gradient stabilization with restarted
accelerated solves on certified proximal balls, with every stochastic
query confined to a region carrying an explicit ordinary-smoothness
certificate.
\item We prove high-probability guarantees for convex and
strongly convex objectives that combine accelerated optimization
with smooth-subclass-optimal statistical accuracy dependence;
when \(L_1=0\), the bounds recover standard accelerated
stochastic rates.
\item We combine gradient--gap localization, monotonicity of the
gradient norm at ball-constrained proximal points, certified
smoothness on local balls, and a travel--halving argument to control
the number of proximal levels. The claim is about this certified
stochastic construction, not priority for each ingredient in
isolation.
\item Numerical experiments on objectives
with unbounded gradients isolate the roles of stabilization,
localization, acceleration, and late-stage averaging.
\end{itemize}

\section{Preliminaries}
\label{sec:setup}

Let 
$X^\star:=\arg\min\limits_{x\in\R^d}f(x)\neq\varnothing,$
    and
    $f^\star:=\min\limits_{x\in\R^d}f(x),$
and fix \(x^\star\in X^\star\). We assume that a distance
certificate
   $ R_0\ge\norm{x^0-x^\star}$
is available. For every \(x\in\R^d\), define
\[
    \Delta_x:=f(x)-f^\star,
    \qquad
    G_x:=\norm{\grad f(x)}.
\]
In particular,
    $\Dinit:=\Delta_{x^0}=f(x^0)-f^\star.$
We use
\(1/L_1:=+\infty\) when \(L_1=0\), and for every \(a\ge0\) set
\(\log_+ a:=\log(\max\{e,a\})\). Thus \(\log_+a=\max\{1,\log a\}\)
when \(a>0\), while the convention also covers zero-valued certificates.


\begin{assumption}[Convexity]
\label{ass:convex}
$f(x)$ is convex and differentiable.
\end{assumption}

\begin{assumption}[$(\Lz,\Lo)$-generalized smoothness]
\label{ass:gs}
There exist $\Lz>0$, $\Lo\ge 0$ such that for all $x,y\in\R^d$ satisfying $\norm{y-x}\le 1/\Lo$ we have
\begin{equation*}
    \norm{\grad f(y)-\grad f(x)} \le \big(\Lz+\Lo\norm{\grad f(x)}\big)\norm{y-x}.
\end{equation*}
\end{assumption}

Assumption~\ref{ass:gs} states that the gradient is locally
Lipschitz continuous.
The condition is asymmetric because the right-hand side depends
on \(\norm{\grad f(x)}\), rather than symmetrically on both
\(x\) and \(y\). This is the local formulation used by
\citet{vankov2024optimizing,gorbunov2024methods}.
When \(\Lo=0\), the restriction on \(\norm{y-x}\) disappears
and Assumption~\ref{ass:gs} reduces to ordinary global
\(\Lz\)-smoothness. Moreover, if \(f\) is twice differentiable
and satisfies
\[
    \norm{\grad^2 f(x)}
    \le
    \Lz+\Lo\norm{\grad f(x)}
    \qquad
    \text{for all }x\in\R^d,
\]
then Assumption~\ref{ass:gs} follows after a universal rescaling
of the constants \(\Lz\) and \(\Lo\).

\begin{assumption}[Sub-Gaussian stochastic oracle]
\label{ass:noise}
For every query point \(x\in\R^d\), the oracle returns a random
vector \(g(x,\xi)\) satisfying
   $ \E\!\left[g(x,\xi)\mid x\right]
    =
    \grad f(x).$
The oracle noise
\[
    \zeta(x,\xi)
    :=
    g(x,\xi)-\grad f(x)
\]
is conditionally norm-sub-Gaussian with parameter \(\sigma\ge0\).
For \(\sigma>0\), this means
\[
    \E\!\left[
        \exp\left(
            {\norm{\zeta(x,\xi)}^2}/{\sigma^2}
        \right)
        \,\middle|\, x
    \right]
    \le e.
\]
For \(\sigma=0\), we use the deterministic-oracle convention
\(\zeta(x,\xi)=0\) almost surely (conditionally on the past), so no
division by zero is intended.

More generally, let \(\mathcal F\) denote the sigma-algebra
generated by all randomness observed before an oracle call, and
let \(x\) be an \(\mathcal F\)-measurable query point. When
\(\sigma>0\), conditional on \(\mathcal F\), every newly drawn oracle
sample satisfies
\begin{align*}
 \E\!\left[g(x,\xi)\mid\mathcal F\right]
    =
    \grad f(x), ~ ~ \E\!\left[
        \exp\left(
           \frac{\norm{\zeta(x,\xi)}^2}{\sigma^2}
        \right)
        \,\middle|\,\mathcal F
    \right]
    \le e.
\end{align*}
Samples within each mini-batch are conditionally independent
given \(\mathcal F\). When \(\sigma=0\), the same statement is read
using the deterministic-oracle convention above.
\end{assumption}

Assumption~\ref{ass:noise} covers almost surely bounded noise and implies $\E\norm{\zeta}^2\le\sigma^2$.
Moreover, let
\begin{equation}
    \label{eq:grad_est}
      \gest
    :=
    \frac{1}{B}
    \sum_{i=1}^{B} g(x,\xi_i)
\end{equation}
be the average of \(B\)  oracle calls at the same
\(\mathcal F\)-measurable point \(x\). By the vector Hoeffding inequality for norm-sub-Gaussian random
vectors \citep{jin2019short}, there exists an absolute constant
\(c_{\mathrm g}>0\) such that, for every \(\rho\in(0,1)\),
   $ \norm{\gest-\grad f(x)}
    \le
    c_{\mathrm g}\sigma
    \sqrt{{\log(2/\rho)}/{B}}$
with conditional probability at least \(1-\rho\).

The next  statements (proved in the supplementary material) are the geometric core of the method. Lemmas~\ref{lem:descent}--\ref{lem:localization}
justify the stabilization mechanism of Phase~I, while
Lemma~\ref{lem:prox-grad} and Lemma~\ref{lem:ball} provide containment and uniform local smoothness for the
accelerated proximal solves of Phase~II. 

\paragraph{Local descent under generalized smoothness.}
 The next result  is
the generalized-smoothness analogue of the classical descent
lemma and is the main tool used to analyze Phase~I and to derive
the gradient--gap relation below.

\begin{lemma}[Generalized descent]
\label{lem:descent}
Under Assumption~\ref{ass:gs}, if $\norm{y-x}\le1/L_1$, then
\[
 f(y)\le f(x)+\ip{\grad f(x)}{y-x}
 +(L_0+L_1G_x)\norm{y-x}^2/2.
\]
\end{lemma}


\paragraph{Gradient localization on sublevel sets.}
The convergence guarantees of ARC-SG are stated in terms of the
function gap, whereas the algorithm controls the gradient norm
and the corresponding local smoothness scale
\(L_0+L_1G_x\). The generalized descent inequality allows us to
connect these quantities. The following lemma shows that a
bound on the function gap automatically yields bounds on both
the gradient norm and the local curvature. It therefore
identifies the critical regime in which the large-gradient
stabilization phase can terminate and accelerated local solves
become possible.

\begin{lemma}[Gradient--gap relation]
\label{lem:localization}
Under Assumptions~\ref{ass:convex}--\ref{ass:gs}, for every $x \in \R^d$,
\begin{equation*}
G_x^2\le 2(L_0+L_1G_x)\Delta_x,
\mbox{ hence }
G_x\le 2L_1\Delta_x+\sqrt{2L_0\Delta_x}.
\end{equation*}
Consequently, on the sublevel set $\{ x \in \R^d : \Delta_x\le\Delta\}$,
\[
L_0+L_1G_x\le 2L_0+3L_1^2\Delta.
\]
When \(L_1>0\), define the critical gradient and
gap scales by
   $ \overline G:={L_0}/{L_1},$
  and
   $ \overline\Delta:={L_0}/{L_1^2}.$
If \(\Delta_x\le\overline\Delta\), then
\[
    G_x
    \le
    (2+\sqrt 2)\,\overline G,
    \qquad
    L_0+L_1G_x
    \le
    (3+\sqrt 2)L_0
    <
    5L_0.
\]
\end{lemma}


\paragraph{Initial-gap certificates.}
The algorithm does not require the exact value of
the  actual initial function gap \( \Dinit:=\Delta_{x^0}=f(x^0)-f(x^\star)\). It only uses a certified upper bound
on it: $\Dcert$ such that $\Dcert\ge\Dinit.$
 A parameter-free choice for this upper bound is
\begin{equation}
\Ddef:=
\begin{cases}
\displaystyle {L_0}\bigl(e^{L_1R_0}-1-L_1R_0\bigr)/{L_1^2},&L_1>0,\\[1ex]
\displaystyle {L_0R_0^2}/{2},&L_1=0,
\end{cases}
\label{eq:default-gap}
\end{equation}
which always satisfies $\Dinit\le\Ddef$. If another certificate $\Delta_{\rm user}$ is available, we use $\Dcert:=\min\{\Delta_{\rm user},\Ddef\}$.
Both \(\Delta_{\rm user}\) and \(\Ddef\) are upper bounds on \(\Dinit\);
hence their minimum is also valid:
    $\Dinit\le\Dcert.$


\paragraph{Exact proximal points.}
For the analysis of Phase~II, we introduce the exact solution of
a ball-constrained proximal subproblem. Given a center
\(c\in\R^d\), a regularization parameter \(\lambda>0\), and a
radius \(r\in(0,+\infty]\), define
\begin{equation}
\label{eq:prox}
\widehat x(c;\lambda,r)
:=
\arg\min_{x\in\B(c,r)}
\left\{
f(x)+{\lambda}\norm{x-c}^2/2
\right\},
\end{equation}
where
   $ \B(c,r):=
    \{x\in\R^d:\norm{x-c}\le r\}.$
The algorithm does not compute this point exactly: R-STM
returns an approximate solution of the corresponding proximal
subproblem. The exact point \(\widehat x(c;\lambda,r)\) serves
as an analytical reference for establishing containment,
gradient monotonicity, and progress across proximal levels. Its
key monotonicity properties are summarized in the following
lemma.
\begin{lemma}[Proximal gradient and distance monotonicity]
\label{lem:prox-grad}
Let Assumption~\ref{ass:convex} hold and let
\(\widehat x:=\widehat x(c;\lambda,r)\). Then
   $ \norm{\grad f(\widehat x)}
    \le
    \norm{\grad f(c)},$
and
  $  \norm{\widehat x-c}
    \le
    {\norm{\grad f(c)}}/{\lambda}.$

Moreover, for every \(u\in X^\star\),
   $ \norm{\widehat x-u}^2
    \le
    \norm{c-u}^2-\norm{\widehat x-c}^2.$
If \(r=+\infty\), then
    $\norm{\widehat x-c}
    =
   {\norm{\grad f(\widehat x)}}/{\lambda}$.

\end{lemma}
\paragraph{From gradient caps to certified smoothness balls.}
The generalized-smoothness condition provides a local curvature
bound that depends on the gradient norm at the current point.
This is not sufficient by itself to apply R-STM, whose analysis
requires a single ordinary smoothness constant on the entire
region visited by the inner solver. The following lemma converts
a certified gradient cap at the center of a sufficiently small
ball into a uniform smoothness bound on that ball.

\begin{lemma}[Certified smoothness ball]
\label{lem:ball}
Assume $L_1>0$, let $\widehat G\ge\norm{\grad f(c)}$, and choose $0<r\le1/(2L_1)$. For all $x,y\in\B(c,2r)$,
\[
G_x\le e\left(\widehat G+\frac{L_0}{L_1}\right), ~
\norm{\grad f(x)-\grad f(y)}\le L_{\rm ball}\norm{x-y}.
\]
where $L_{\rm ball}:=4(L_0+L_1\widehat G)$. Hence $f$ is $L_{\rm ball}$-smooth along every segment in $\B(c,2r)$. For $L_1=0$, the conclusion holds globally with $L_{\rm ball}=L_0$.
\end{lemma}

\section{The ARC-SG Method}
\label{sec:method}
We now describe the method for solving the stochastic convex optimization problem \eqref{eq:problem}. It
has two phases with different objectives: \PhaseI{} is a short stabilization stage: it reduces a potentially very large gradient to the critical scale $\overline G=L_0/L_1$, and \PhaseII{} then accelerates a sequence of localized proximal problems. We state the phases separately and then combine them in a wrapper, which we refer to as
\ARCSG{}  (\textbf{A}ccelerated, \textbf{R}estarted, \textbf{C}ertified-ball method for \textbf{S}tochastic optimization under \textbf{G}eneralized smoothness). The method is described in Algorithm~\ref{alg:arc}, see also  Figure~\ref{fig:schematic}  for greater clarity. 

\begin{figure}[ht]
\centering
\includegraphics[width=0.85\columnwidth]{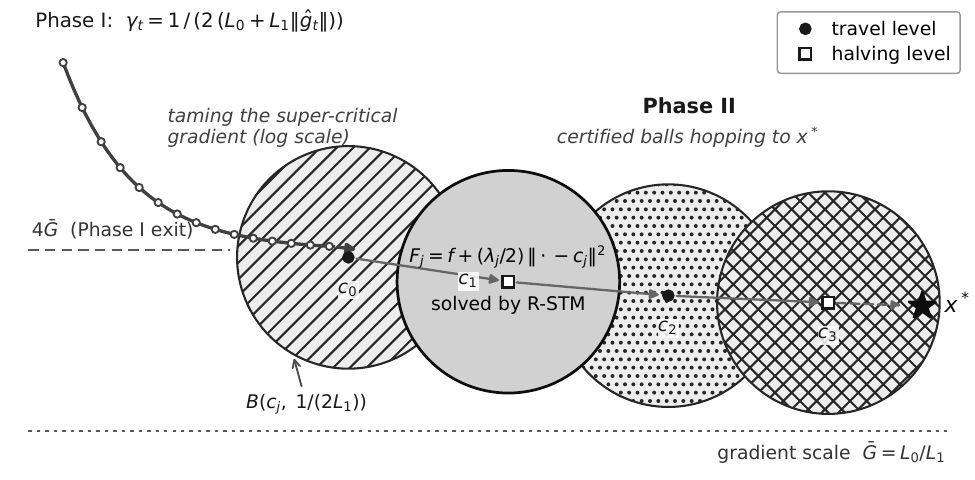}
\caption{Anatomy of \ARCSG{} on a generalized-smooth objective. \PhaseI{} (red) contracts the gap while the gradient is super-critical; \PhaseII{} (blue) alternates travel and halving proximal levels, each solved by \RSTM{} inside a ball of radius $O(1/\Lo)$ (shaded), until the gradient certificate, the regularization floor, or the certified finisher certifies $\eps$-optimality.}
\label{fig:schematic}
\end{figure}


\paragraph{The accelerated inner solver: \RSTM{}}
The inner routine  is an epoch-restarted constrained stochastic similar-triangles method (\RSTM{}). It is run only on a ball where Lemma~\ref{lem:ball} supplies an ordinary smoothness constant. Strong convexity shrinks the certified epoch radius together with the current gap, preventing the statistical term from scaling quadratically in the target subproblem accuracy.

\begin{proposition}[R-STM]
\label{prop:rstm}
Let $F$ be $\lambda$-strongly convex and $L$-smooth on $\B(c,2r)$, with minimizer $x_F^\star\in\B(c,r)$. Suppose $y^{(0)}\in\B(c,r)$ and $F(y^{(0)})-F(x_F^\star)\le H_0$. For subproblem accuracy $\delta>0$ and a failure
probability \(\beta\in(0,1)\), R-STM returns a point \(y\) such
that, with probability at least \(1-\beta\),
\[
F(y)-F(x_F^\star)\le\delta,\qquad
\norm{y-x_F^\star}\le\sqrt{{2\delta}/{\lambda}},
\]
using
$\tO\!\left(
\sqrt{\frac{L}{\lambda}}\logp\!\left(\frac{H_0}{\delta}\right)
+\frac{\sigma^2}{\lambda\delta}
\right)$
stochastic gradient calls. Every query generated by R-STM lies in \(\B(c,2r)\).
\end{proposition}

\begin{algorithm}[ht]
\caption{\PhaseI{}: tame the gradient}
\label{alg:phase1}
\footnotesize
\begin{algorithmic}[1]
\REQUIRE $x^0$, $L_0,L_1,\sigma,R_0,\Dcert,\alpha$
\STATE Set $\overline G=L_0/L_1$ and choose the batch size and budget from the high-probability schedule.
\FOR{$t=0,1,\ldots,T_1^{\max}$}
    \STATE $\widehat g_t\gets \mathrm{GradEst}(x_t,B_1)$ from eq. \eqref{eq:grad_est}
    \IF{$\norm{\widehat g_t}\le 3\overline G$}
        \RETURN $x^{\mathrm I}\gets x_t$
    \ENDIF
    \STATE $\gamma_t \gets \frac{1}{2}(L_0+L_1\norm{\widehat g_t})^{-1}$
    \STATE $x_{t+1}\gets x_t- \gamma_t \widehat g_t$
\ENDFOR
\RETURN the last iterate
\end{algorithmic}
\end{algorithm}

\paragraph{Phase I: taming the gradient.}
Phase I uses the stochastic counterpart of the generalized-smoothness stepsize. The step length is at most $1/(2L_1)$, so Lemma~\ref{lem:descent} always applies; Lemma~\ref{lem:localization} then controls both descent and the distance to $x^\star$.
 \ARCSG{} uses \(\Dcert\) only to set the Phase~I budget and
the associated confidence allocation. It does not use the
unknown value \(\Dinit\).

\begin{proposition}[Phase I guarantee]
\label{prop:phase1}
Assume $L_1>0$. With probability at least $1-\alpha/4$, Algorithm~\ref{alg:phase1} returns $x^{\mathrm I}$ after at most
$1+7L_1R_0\log_+\!\left({2L_1^2\Dcert}/{L_0}\right)$
iterations such that
\[
\norm{\grad f(x^{\mathrm I})}\le {7}\overline G/2,
~ f(x^{\mathrm I})\le f(x^0),
~ \norm{x_t-x^\star}\le {3}R_0/2.
\]
Its oracle cost is
$\tO\!\left(1+L_1R_0+(1+L_1R_0){\sigma^2L_1^2}/{L_0^2}\right).$
The additive terms are required by the floors $T_1\ge1$ and
$B_1\ge1$, even when $L_1R_0$ is small.

\end{proposition}

\begin{algorithm}[ht]
\caption{Phase II: accelerated proximal levels}
\label{alg:phase2}
\footnotesize
\begin{algorithmic}[1]
\REQUIRE $c_1=x^{\mathrm I}$, $\widehat G_1=4L_0/L_1$, $\Lambda_1=32L_0$, target $\eps$
\FOR{$j=1,2,\ldots,J_{\rm sch}$}
    \IF{$\widehat G_j\le \eps/(2\overline R)$}
        \RETURN $c_j$
    \ENDIF
    \IF{the drift threshold or schedule cap is reached}
        \RETURN $\mathrm{CertifiedFinisher}(c_j,\widehat G_j)$
    \ENDIF
    \STATE $\lambda_j\gets\max\{4L_1\widehat G_j,\Lambda_j\}$
    \STATE $r_j\gets 2\widehat G_j/\lambda_j$; \quad $F_j(x)\gets f(x)+\frac{\lambda_j}{2}\norm{x-c_j}^2$
    \STATE $c_{j+1}\gets\mathrm{R\text{-}STM}(F_j,\B(c_j,r_j),\delta_j,\beta_j)$
    \STATE $L_j\gets4(L_0+L_1\widehat G_j)$, $s_j\gets\sqrt{2\delta_j/\lambda_j}$
    \STATE $\widehat G'_{j+1}\gets\widehat G_j+L_js_j$
    \STATE $\widehat g_j\gets\mathrm{GradEst}(c_{j+1},B_j^{\rm est})$
    \STATE $\widehat G_{j+1}\gets\min\{\widehat G'_{j+1},\norm{\widehat g_j}+\tau_j\}$
    \IF{$\lambda_j=\Lambda_j$}
        \IF{$\Lambda_j\le\eps/(2\overline R^2)$}
            \RETURN $c_{j+1}$
        \ENDIF
        \STATE $\Lambda_{j+1}\gets\Lambda_j/2$
    \ELSE
        \STATE $\Lambda_{j+1}\gets\Lambda_j$
    \ENDIF
\ENDFOR
\end{algorithmic}
\end{algorithm}
Here \(\overline R:=3R_0\). The full schedules for
\(\delta_j,\beta_j,\tau_j,B_j^{\rm est}\), as well as the
certified finisher, are specified in the supplementary
material.


\paragraph{\PhaseII{}: accelerated proximal levels.}
\PhaseII{}
 maintains a center $c_j$, a certified \emph{gradient cap} $\Gest_j\ge G_{c_j}$, and a geometrically decreasing \emph{regularization floor} $\Lambda_j$. Level $j$ approximately solves $F_j=f+\tfrac{\lambda_j}{2}\norm{\cdot-c_j}^2$ with the floored regularization \eqref{eq:lambda} below,
\begin{equation}
\label{eq:lambda}
\lambda_j=\max\{4\Lo\Gest_j,\ \Lambda_j\},
\end{equation}
by running \RSTM{} on the ball $\ball{c_j}{r_j}$, $r_j=2\Gest_j/\lambda_j\le\tfrac{1}{2\Lo}$. The design closes four loops at once. For each proximal level, let
 $   \widehat x_j
    :=
    \widehat x(c_j;\lambda_j,r_j)$
denote the exact ball-constrained proximal point introduced in
\eqref{eq:prox}.
\emph{Containment:} by Lemma~\ref{lem:prox-grad}, $\norm{\xhat_j-c_j}\le G_{c_j}/\lambda_j\le r_j/2$, so the exact prox lies strictly inside the ball while the \RSTM{} trajectory is confined to it.
\emph{Smoothness:} by Lemma~\ref{lem:ball}, $f$ is $L_j$-smooth on $\ball{c_j}{2r_j}$ with $L_j=4(\Lz+\Lo\Gest_j)$; hence $F_j$ is $\lambda_j$-strongly convex and $(L_j+\lambda_j)$-smooth there, with condition number $\kappa_j\le 2+\Lz/(\Lo\Gest_j)$ on gradient-driven levels.
\emph{Gradient bookkeeping:} $G_{c_{j+1}}\le\Gest_j+L_js_j$ by Lemma~\ref{lem:prox-grad} and the \RSTM{} accuracy $s_j=\sqrt{2\delta_j/\lambda_j}$, so the cap update is valid \emph{without any measurement}; the relative-accuracy measurement re-anchors the cap to the true gradient so that caps track gradients up to a constant factor.
\emph{Travel--halving progress.}
The analysis classifies each proximal level through the
displacement of its exact proximal point \(\widehat x_j\).
On a \emph{halving level},
   $ \norm{\widehat x_j-c_j}
    \le
   {1}/{(8L_1)},$
and the gradient scale decreases by a constant factor on
gradient-driven levels. On a \emph{travel level}, the proximal
point moves a macroscopic distance, and the proximal inequality
forces a definite decrease of the original objective.
Consequently, only \(O(L_1\Rb)\) travel levels can occur while
the gradient cap remains in a fixed dyadic band, whereas the
number of halving levels in that band is logarithmic. The
schedule cap and the certified finisher convert this
per-band progress argument into unconditional termination.
\emph{Termination certificates.}
Phase~II has three termination mechanisms.
First, if
   $ \Gest_j\le\frac{\eps}{2\Rb},$
then convexity and the distance invariant
\(\norm{c_j-x^\star}\le\Rb\) imply
  $  f(c_j)-f^\star
    \le
    G_{c_j}\norm{c_j-x^\star}
    \le
    \Gest_j\Rb
    \le
    \frac{\eps}{2}.$
Second, suppose the regularization floor binds,
\(\lambda_j=\Lambda_j\), and
    $\Lambda_j\le\frac{\eps}{2\Rb^2}.$
The level schedule also ensures
    $\delta_j\le\frac{\eps}{2}.$
By containment, the constrained proximal point lies in the
interior of the ball and therefore coincides with the
unconstrained proximal point. Hence
   $ f(c_{j+1})-f^\star
    \le
    \delta_j
    +
    \frac{\lambda_j}{2}\norm{c_j-x^\star}^2
    \le
    \frac{3\eps}{4}
    \le
    \eps.$
Third, the certified finisher handles the corner case in which
the gradient cap falls below the resolution of the cap
bookkeeping or the deterministic schedule limit is reached.
Its construction and complexity are given in the supplementary
material.
\emph{Containment}. A third, rarely triggered mechanism---the \emph{finisher}---handles the corner case in which the cap falls below the drift resolution of the bookkeeping, and completes the run by $\tO(\Lo\Rb)$ ball-constrained proximal steps with $\lambda=\Lo\eps/\Rb$ (supplementary material).
\begin{proposition}[\PhaseII{} guarantee]
\label{prop:phase2}
Suppose \PhaseII{} is initialized at a point satisfying
\[
\norm{\grad f(c_1)}\le 4\overline G,
\qquad
\norm{c_1-x^\star}\le 3R_0/2.
\]
With the schedules specified in the supplement, the joint good event of all inner solves and gradient measurements has probability at least $1-\alpha/2$ (one quarter for each family); the finisher adds no new random event. On this event, Algorithm~\ref{alg:phase2}, together with its certified finisher, terminates and returns $\widehat x$ with $f(\widehat x)-f^\star\le\eps$. Its oracle cost is
\[
\tO\!\left(
\begin{aligned}
  L_1R_0
  &+\overline\kappa^{1/2}\sqrt{{L_0\overline R^2}/{\eps}}
    +\overline\kappa^3{\sigma^2\overline R^2}/{\eps^2}\\
  &+\overline\kappa^3{\sigma^2L_1^2}/{L_0^2}
\end{aligned}
\right),
\]
where $\overline\kappa:=1+L_1\overline R$.
\end{proposition}

\begin{algorithm}[ht]
\caption{ARC-SG}
\label{alg:arc}
\footnotesize
\begin{algorithmic}[1]
\REQUIRE $x^0,R_0,L_0,L_1,\sigma,\eps,\alpha$ and optionally $\Delta_{\rm user}$
\STATE Set $\Dcert$ by \eqref{eq:default-gap}; if supplied, replace it by $\min\{\Delta_{\rm user},\Dcert\}$.
\IF{$\Dcert\le\eps$}
    \RETURN $x^0$
\ENDIF
\STATE Use $L_{1,\rm eff}=\max\{L_1,1/(4\overline R)\}$ in the schedules.
\STATE $x^{\mathrm I}\gets\mathrm{Phase\ I}(x^0)$
\RETURN $\mathrm{Phase\ II}(x^{\mathrm I},\eps,\alpha)$
\end{algorithmic}
\end{algorithm}

\section{Main Results}
\label{sec:main}

\begin{theorem}[Convex case]
\label{thm:convex}
Let Assumptions~\ref{ass:convex}--\ref{ass:noise} hold.
Fix \(x^\star\in X^\star\), a distance certificate
    $R_0\ge\norm{x^0-x^\star},$
and a valid initial-gap certificate
   $ \Dcert\ge f(x^0)-f^\star.$
Let \(\eps>0\), \(\alpha\in(0,1)\), and define
    $\Rb:=3R_0,
    \kbar:=1+\Lo\Rb.$
Then for the output $\xhat$ of the \ARCSG{}, with probability at least $1-\alpha$, $\alpha\in(0,1)$, 
we have
$f(\xhat)-f^\star\le\eps$
after at most
\begin{align*}
N \;=\; \tO\Big(&\underbrace{1+\Lo R_0}_{\textnormal{\PhaseI{} + levels}}
\;+\;\underbrace{\kbar^{1/2}\sqrt{\tfrac{\Lz\Rb^2}{\eps}}}_{\textnormal{optimization}}
\;+\;\underbrace{\kbar^{3}\,\tfrac{\sigma^2\Rb^2}{\eps^2}}_{\textnormal{statistical}}\\[-2pt]
&\;+\;\underbrace{(1+\Lo R_0)\tfrac{\sigma^2\Lo^2}{\Lz^2}+\kbar^{3}\tfrac{\sigma^2\Lo^2}{\Lz^2}}_{\textnormal{burn-in}}\Big)
\end{align*}
stochastic gradient evaluations. 
\end{theorem}

\begin{remark}[Default certificate]
\label{rem:convex-default}
The default certificate is always valid: $\Delta_{\mathrm{init}}\le\Delta_{\mathrm{def}}$ (one line, proved in the supplementary material). If \ARCSG{} runs on $\Dcert=\Delta_{\mathrm{def}}$, then with $q:=\Lo R_0$ the \PhaseI{} entry terms scale as $(1+q)^2$ and the statistical burn-in as $(1+q)^4\sigma^2\Lo^2/\Lz^2$:
$N=\tO\big((1+q)^2+\kbar^{1/2}\sqrt{\Lz\Rb^2/\eps}+\kbar^{3}\sigma^2\Rb^2/\eps^2+(1+q)^4\sigma^2\Lo^2/\Lz^2+\kbar^{3}\sigma^2\Lo^2/\Lz^2\big)$,
with $\tO$ hiding only absolute constants and polylogarithmic factors in $(q,\Lz\Rb^2/\eps,1/\alpha)$. With a user-supplied certificate the statement of Theorem~\ref{thm:convex} stands as printed.
\end{remark}

\begin{theorem}[Strongly convex case]
\label{thm:strongly}
Under the assumptions and initialization of
Theorem~\ref{thm:convex}, suppose additionally that \(f\) is
\(\mu\)-strongly convex. Then for the output $\xhat$ of  \ARCSG{} with distance-halving outer
restarts (described in the supplementary material), 
we have
$f(\xhat)-f^\star\le\eps$  with probability at least $1-\alpha$, $\alpha\in(0,1)$,  using
\begin{align*}
N=\tO\Big(&(1+\Lo R_0)\Big(1+\sqrt{\tfrac{\Lz}{\mu}}\Big)+\sqrt{\tfrac{\Lz}{\mu}}\,\logp\!\left(\tfrac{\mu R_0^2}{\eps}\right)\\
&\quad+\Big(1+\Lo\sqrt{\tfrac{\eps}{\mu}}\Big)^{3}\tfrac{\sigma^2}{\mu\eps}
+\tfrac{\sigma^2\Lo^3R_0}{\mu^2}+\kbar^{3}\tfrac{\sigma^2\Lo^2}{\Lz^2}\Big)
\end{align*}
oracle calls. In particular, for $\eps\le\mu/(36\Lo^2)$ the statistical term is $\tO(\sigma^2/(\mu\eps))$, whose $\eps$-exponent is optimal on the smooth subclass.
\end{theorem}

For $\Lo=0$ (via the $L_{1,\mathrm{eff}}$ substitution) both theorems reduce to $\tO\big(1+\sqrt{\Lz R_0^2/\eps}+\sigma^2R_0^2/\eps^2\big)$ and $\tO\big(\sqrt{\Lz/\mu}\,\logp(\mu R_0^2/\eps)+\sigma^2/(\mu\eps)\big)$---the complexities of accelerated stochastic approximation \citep{lan2012optimal,ghadimi2013optimal} up to logarithms---and for every $\Lo\ge0$ the $\eps$-exponents of both the optimization and the statistical terms are unimprovable, since $(\Lz,\Lo)$-smooth functions include the $\Lz$-smooth ones and the classical lower bounds \citep{nemirovsky1983problem} apply on that subclass. Whether the multiplicative $\kbar$-dependence is necessary remains open.

\paragraph{Proof ideas.} The complete proofs are deferred to the supplementary material; the architecture of the argument is as follows. \emph{Certified balls:} the floor $\lambda_j\ge4\Lo\Gest_j$ forces $r_j\le\tfrac{1}{2\Lo}$, so Lemma~\ref{lem:ball} certifies $L_j$-smoothness of $f$ on $\ball{c_j}{2r_j}$ and Lemma~\ref{lem:prox-grad} keeps the exact prox strictly inside $\B(c_j, r_j)$: every stochastic query of the accelerated inner solver stays in a region of provably controlled smoothness, which is exactly what unconstrained accelerated trajectories cannot guarantee. \emph{Prox monotonicity and approximate transfer:} the same lemma makes the gradient norm at the exact proximal reference point non-increasing. The implemented center is only an approximate prox, so the certified cap at the next center may rise; the transfer bound limits the rise to $L_js_j$, and the drift and travel--halving potentials control the accumulated increase. \emph{Travel/halving potential:} function decrease is charged against the running gradient scale, bounding productive levels per dyadic band by $O(\Lo\Rb)$; the schedule cap $J^{\mathrm{sch}}=\tO(\Lo\Rb)$ and the certified finisher turn the per-excursion count into an unconditional bound. \emph{Confidence budget:} one quarter of $\alpha$ is allocated to each of \PhaseI{}, all \RSTM{} runs, and all gradient measurements, with each family split by a union bound over a deterministic maximum number of uses. The fourth quarter is reserved for finisher-only checks; these are deterministic on the shared good event, so the analysis actually spends at most $3\alpha/4$. The strongly convex theorem follows by restarting \ARCSG{} on distance halving, the per-round targets summing geometrically.

\section{Numerical Experiments}
\label{sec:experiments}
We evaluate practical \textsf{ARC-SG}, a computationally
streamlined implementation designed to illustrate the
two-phase mechanism of the method. The implementation follows
the architecture of \ARCSG{}, but uses less conservative
constants and stopping rules; all deviations from the analyzed
schedule are listed in the supplementary material.

\paragraph{Experimental setup.} We consider two controlled  problem families, $d=40$. In both
cases, the objective has unbounded gradients, its minimizer and
optimal value are available in closed form, and the
\((\Lz,\Lo)\)-generalized-smoothness condition of
Assumption~\ref{ass:gs} holds globally with the certified pair  $(\Lz,\Lo)=(2.1258,\,1/\ln 2\approx1.4427).$
The proof of this certificate is given in the supplementary
material. These instances make the distinction between methods
that exploit generalized smoothness and methods that rely on a
single global smoothness estimate particularly pronounced.
\begin{itemize}
    \item \emph{Convex instance.}
The first family contains \(\cosh\)-type coordinates together
with ten quartic coordinates. The Hessian vanishes along the
quartic directions at the minimizer, so the objective is not
strongly convex, even locally. The initial point is placed in a
quartic valley where the gradient norm is far above the critical
scale \(\overline G=\Lz/\Lo\).
\item \emph{Strongly convex instance.}
The second family consists entirely of \(\cosh\)-type
coordinates and is globally \(\mu\)-strongly convex with the
exact modulus
    $\mu=0.245.$
\end{itemize}
In both experiments, the stochastic oracle adds Gaussian noise
calibrated to satisfy Assumption~\ref{ass:noise} with a certified
norm-sub-Gaussian parameter. Each curve reports the geometric
mean of the objective gap over \(10\) evaluation seeds \(0,\ldots,9\); each
seed independently regenerates the orthogonal rotation, root, and oracle
noise, while all methods at a given seed use the same regenerated instance. The
shaded region shows one multiplicative standard deviation across
runs, rather than uncertainty in the estimated mean.

\paragraph{Baselines.}
We compare against the following methods:
\textsf{SGD-GS}, batched SGD with the
\((\Lz,\Lo)\)-adaptive stepsize of
\citet{sgd2025generalized};
\textsf{Clip-SGD}, clipped SGD with a tuned stepsize and clipping
threshold
\citep{koloskova2023revisiting,gaash2025clipped};
\textsf{Acc-blind}, a Nesterov-type stochastic accelerated
method using the fixed curvature estimate
\(\Lz+\Lo G_{x^0}\);
\textsf{ClipSSTM}, an accelerated clipped similar-triangles
method without restarts
\citep{gorbunov2020stochastic};
\textsf{PJ-SGD}, obtained by applying Polyak--Juditsky tail
averaging to \textsf{SGD-GS}; and
\textsf{DS-SGD}, SGD with a decreasing stepsize.
All free parameters of the baselines are selected on separate
validation runs using the full evaluation budget.

\begin{figure}[h]
\centering
\begin{subfigure}[t]{0.5\columnwidth}
    \includegraphics[width=\textwidth]{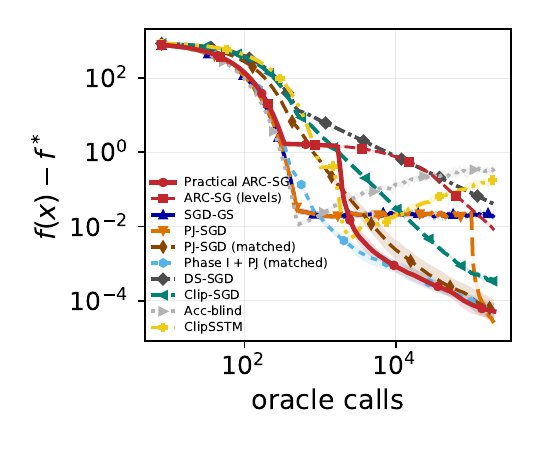}
    \caption{Convex, quartic flat directions}
\end{subfigure}\hfill
\begin{subfigure}[t]{0.49\columnwidth}
    \includegraphics[width=\textwidth]{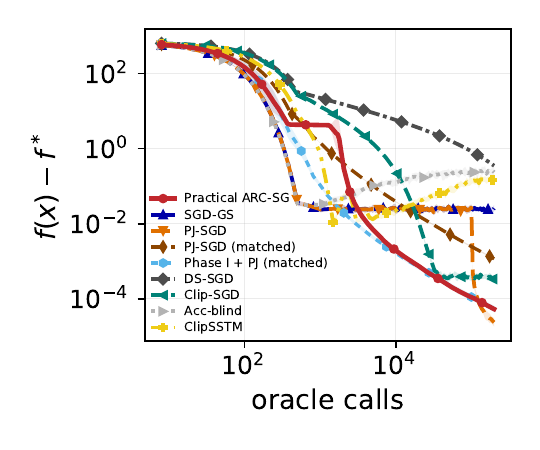}
    \caption{\(\mu\)-strongly convex problem }
\end{subfigure}
\caption{
Objective gap versus the number of stochastic oracle calls.
Both axes are logarithmic. Curves show the geometric mean over
\(10\) independently regenerated instance--noise runs; shaded bands indicate one
multiplicative standard deviation across runs.
}
\label{fig:experiments}
\end{figure}

\paragraph{Results.}
Figure~\ref{fig:experiments} is intended to illustrate the
mechanism of \ARCSG{}, rather than to validate the worst-case
constants in Theorems~\ref{thm:convex}--\ref{thm:strongly}.
Across both problem families, practical \textsf{ ARC-SG} attains
substantially smaller gaps over most of the computational budget
than the non-accelerated generalized-smoothness-aware methods
and the accelerated methods that ignore the changing local
curvature.

The behaviour of the baselines is qualitatively consistent
across the two instances. The conservative fixed curvature
estimate makes \textsf{Acc-blind} uniformly slow.
\textsf{SGD-GS} improves rapidly in the initial
large-gradient regime but subsequently appears to plateau.
\textsf{ClipSSTM} is competitive at moderate accuracy, but
becomes unstable at smaller gaps when used without restarts.
By contrast, practical \textsf{ ARC-SG} continues to decrease the
objective gap (final gaps $5.2\times10^{-5}$ and $5.3\times10^{-5}$ in (a) and (b)).
The matched ablations decompose this behaviour. On the strongly convex instance \PhaseI{} is important: the raw PJ-averaged driver run from $x^0$ stalls, while \PhaseI{} followed by the same driver recovers the full accuracy. On both instances the accelerated proximal levels account for the initial drop, and the remainder of the late-stage gain is identified by the ablations as coming from the interior Polyak--Juditsky averaging stage---which the theorems do not cover. Consistently, the \textsf{PJ-SGD} baseline, which isolates that stage, trails \textsf{Practical ARC-SG} by factors of about $100\times$ (at $25\%$ of the budget) and $300\times$ (at $50\%$) but attains a better endpoint: $2.6\times10^{-5}$ and $2.4\times10^{-5}$ versus $5.2\times10^{-5}$ and $5.3\times10^{-5}$. Against this baseline \textsf{Practical ARC-SG}'s advantage is speed, not endpoint accuracy.

\paragraph{Theory versus implementation.}
The analyzed schedule's worst-case constants are prohibitive: evaluated analytically from the paper's own schedule on the convex instance at the run's accuracy, the analyzed algorithms require $\approx3.2\times10^{18}$ oracle calls, so the printed constants are not what runs, and the curves above use the streamlined practical schedule. Two complementary checks close the gap between the theory and the implementation: a literal run of the analyzed algorithms on a small certified one-dimensional instance (reported in the supplementary material) satisfies every certified invariant---gradient caps, distances to the solution, prox containment---level by level, and targeted small instances exercise the finisher and the strong-growth hop stages that the main instances do not trigger.

\section{Conclusion}
\label{sec:conclusion}

We introduced \ARCSG{}, a high-probability accelerated stochastic method for convex and strongly convex optimization under asymmetric $(\Lz,\Lo)$-generalized smoothness with an additive sub-Gaussian oracle. Its optimization term is accelerated and its statistical terms have the smooth-subclass-optimal $\eps$-dependence (in the strongly convex case, unconditionally so for $\eps\le\mu/(36\Lo^2)$). A contemporaneous public ICLR 2026 submission \citep{anonymous2026nesterov} already obtains the same convex accuracy exponents in a broader generalized-smoothness/affine-variance model. We therefore claim novelty for the parameter-explicit, query-by-query certified-ball construction and the strongly convex restart extension, not for the convex exponents themselves. The method rests on a certified-ball mechanism---regularization floored at the gradient scale confines every accelerated stochastic run to a ball of provably controlled smoothness---and on the structural fact that exact proximal reference points do not increase the gradient norm; the drift of the approximate centers is controlled separately by the transfer estimates. Our bounds carry an explicit multiplicative $\kbar$-dependence, assume an exact projection oracle with projection calls counted separately, and require knowledge of $(\Lz,\Lo,\sigma,R_0)$ and, in the strongly convex case, $\mu$. 


\FloatBarrier

\def\ARCIncludedSupplement{1}
\ifdefined\ARCIncludedSupplement
\else
\documentclass[letterpaper]{article} 
\usepackage[submission]{aaai2027}  
\usepackage[hyphens]{url}  
\usepackage{graphicx} 
\urlstyle{rm} 
\def\UrlFont{\rm}  
\usepackage{natbib}  
\usepackage{caption} 
\frenchspacing  
\setlength{\pdfpagewidth}{8.5in} 
\setlength{\pdfpageheight}{11in} 

\usepackage{amsmath}
\usepackage{amssymb}
\usepackage{amsthm}
\usepackage{mathtools}
\usepackage{bm}
\usepackage{booktabs}
\usepackage{algorithm}
\usepackage{algorithmic}
\usepackage{subcaption}
\usepackage{multirow}
\usepackage{tabularx}
\usepackage{microtype}

\pdfinfo{
/TemplateVersion (2027.1)
}

\newtheorem{theorem}{Theorem}
\newtheorem{proposition}{Proposition}
\newtheorem{lemma}{Lemma}
\newtheorem{corollary}{Corollary}
\newtheorem{assumption}{Assumption}
\newtheorem{definition}{Definition}
\newtheorem{remark}{Remark}
\newtheorem{example}{Example}
\newtheorem{openproblem}{Open Problem}

\newcommand{\R}{\mathbb{R}}
\newcommand{\E}{\mathbb{E}}
\newcommand{\Prob}{\mathbb{P}}
\newcommand{\norm}[1]{\left\lVert #1\right\rVert}
\newcommand{\inner}[2]{\left\langle #1,\, #2\right\rangle}
\newcommand{\grad}{\nabla}
\newcommand{\eps}{\varepsilon}
\newcommand{\epsg}{\varepsilon_{\mathrm{g}}}
\newcommand{\Lz}{L_0}
\newcommand{\Lo}{L_1}
\newcommand{\Rb}{\bar{R}}
\newcommand{\kbar}{\bar{\kappa}}
\newcommand{\Deltabar}{\bar{\Delta}}
\newcommand{\Gbar}{\bar{G}}
\newcommand{\Dinit}{\Delta_{\mathrm{init}}}
\newcommand{\Ddef}{\Delta_{\mathrm{def}}}
\newcommand{\Dcert}{\Delta_{\mathrm{cert}}}
\newcommand{\Duser}{\Delta_{\mathrm{user}}}
\newcommand{\gest}{\hat{g}}
\newcommand{\Gest}{\hat{G}}
\newcommand{\xhat}{\hat{x}}
\newcommand{\dist}{\mathrm{dist}}
\newcommand{\proj}{\Pi}
\newcommand{\ball}[2]{B\!\left(#1,\,#2\right)}
\newcommand{\ARCSG}{\textsf{ARC-SG}}
\newcommand{\RSTM}{\textsf{R-STM}}
\newcommand{\STM}{\textsf{STM}}
\newcommand{\GradEst}{\textsf{GradEst}}
\newcommand{\PhaseI}{\textsf{Phase~I}}
\newcommand{\PhaseII}{\textsf{Phase~II}}
\newcommand{\logp}{\log_{+}}
\newcommand{\tO}{\widetilde{O}}
\newcommand{\Lball}{L_{\mathrm{ball}}}
\newcommand{\rsafe}{r_{\mathrm{safe}}}
\DeclareMathOperator*{\argmin}{arg\,min}
\DeclareMathOperator*{\argmax}{arg\,max}

\setcounter{secnumdepth}{2}

\title{Supplementary Material for: Localize, Restart, Accelerate: Stochastic Optimization under Generalized Smoothness}

\author{
    Anonymous submission
}
\affiliations{
    Anonymous submission
}

\begin{document}
\onecolumn

\begin{center}
{\LARGE\bfseries Supplementary Material\par}
\vspace{0.35em}
{\Large\bfseries Localize, Restart, Accelerate:\\
Stochastic Optimization under Generalized Smoothness\par}
\vspace{0.6em}
{\large Anonymous submission\par}
\end{center}
\vspace{0.5em}
\fi

\ifdefined\ARCIncludedSupplement
\clearpage
\onecolumn
\fi
\appendix

\section*{Supplement overview}
This document contains the complete proofs, parameter schedules,
auxiliary algorithms, and experimental details supporting the main
submission. It is self-contained, but its global notation is aligned
with the main paper. In particular, $R_0\ge\norm{x^0-x^\star}$ is the
initial-distance certificate, $\Dinit=f(x^0)-f^\star$ is the unknown
actual initial gap, and $\Dcert\ge\Dinit$ is the certificate consumed
by the algorithm. The target accuracy of the original problem is
$\eps$, whereas $\delta$ and $\delta_j$ denote accuracies of the
strongly convex proximal subproblems. The total failure probability is
$\alpha$; local stochastic events use symbols such as $\beta$.

The convex and strongly convex main theorems are restated and
proved below. The convex proof is
assembled from the complete Phase~I analysis, the certified-ball
invariants, the travel--halving counting argument, the hard schedule
cap, and the certified finisher. The strongly convex result follows by
applying the convex routine inside distance-halving restarts. Additional
sections record the strong-growth/interpolation extension and the full
experimental protocol; these sections are supplementary extensions and
are not needed for the two principal theorems of the main paper.

The confidence budget allocates $\alpha/4$ to Phase~I,
$\alpha/4$ to all \RSTM{} calls, and $\alpha/4$ to gradient-cap
measurements; the remaining quarter is reserved for finisher-only
checks. In the present construction those checks are deterministic on
the shared good event. Each stochastic budget is divided by a union
bound over a deterministically bounded number of calls.

\section*{Contents}
\small
\begin{tabularx}{\textwidth}{@{}cXr@{}}
\textbf{\ref{app:notation}} & Notation and abbreviations & \pageref{app:notation} \\
\textbf{\ref{sec:prelim}} & Preliminaries and structural lemmas & \pageref{sec:prelim} \\
\textbf{\ref{app:conc}} & Concentration toolbox & \pageref{app:conc} \\
\textbf{\ref{app:stm}} & Inner stochastic solver: constrained batched STM with epoch restarts & \pageref{app:stm} \\
\textbf{\ref{app:phase1}} & Phase I analysis & \pageref{app:phase1} \\
\textbf{\ref{app:phase2-section}} & Complete ARC-SG schedule and Phase II analysis & \pageref{app:phase2-section} \\
\textbf{\ref{app:strongly}} & Strongly convex restarts & \pageref{app:strongly} \\
\textbf{\ref{app:sec:main}} & Main-results remarks and comparisons & \pageref{app:sec:main} \\
\textbf{\ref{app:interp}} & Strong-growth and interpolation extensions & \pageref{app:interp} \\
\textbf{\ref{app:experiments}} & Complete experimental details & \pageref{app:experiments} \\
\textbf{\ref{app:additional}} & Additional experiments and limitations & \pageref{app:additional}
\end{tabularx}
\normalsize
\vspace{0.75em}

\section{Notation and abbreviations}
\label{app:notation}

Table~\ref{tab:abbreviations} collects every abbreviation used in this paper with its canonical expansion; entries marked \emph{(baseline)} are labels of the experimental baselines of Section~\ref{app:sec:experiments} and Appendix~\ref{app:experiments} rather than literature abbreviations. Table~\ref{tab:notation} fixes the main global symbols; the few symbols used with two meanings are listed with their scopes (see also the notation paragraph of Appendix~\ref{sec:prelim}).

\begin{table*}[t]
\centering
\footnotesize
\setlength{\tabcolsep}{4pt}
\begin{tabular}{@{}ll@{}}
\toprule
Abbreviation & Canonical expansion \\
\midrule
\ARCSG & \textbf{A}ccelerated, \textbf{R}estarted, \textbf{C}ertified-ball method for \textbf{S}tochastic optimization under \textbf{G}eneralized smoothness (this paper) \\
\ARCSG-sc & \ARCSG{} with distance-halving restarts, strongly convex case (Algorithm~\ref{alg:sc}; this paper) \\
\ARCSG-interp & \ARCSG, hop-and-solve variant for the interpolation/strong-growth regime (Algorithm~\ref{alg:interp}; this paper) \\
\RSTM & restarted stochastic similar-triangles method (this paper's inner solver, Proposition~\ref{app:prop:rstm}) \\
\STM & constrained stochastic similar-triangles method (single run, \eqref{eq:stm}, Proposition~\ref{prop:stm}) \\
SGD & stochastic gradient descent \\
GD & gradient descent \\
AGD & accelerated gradient descent \citep{gorbunov2024methods} (Table~\ref{app:tab:comparison}) \\
AGM & accelerated gradient method \citep{tyurin2026near} (Table~\ref{app:tab:comparison}) \\
SGC & strong growth condition \citep{schmidt2013fast,vaswani2019fast} \\
SGC-GS & strong growth condition under generalized smoothness (Assumption~\ref{ass:sgc}) \\
RSAG & Randomized Stochastic Accelerated Gradient \citep{yu2025stochastic} \\
AC-SA & Accelerated Stochastic Approximation \citep{lan2012optimal,ghadimi2012optimal} \\
AGMsDR & Accelerated Gradient Method with Small-Dimensional Relaxation \citep{vankov2024optimizing} \\
AdGD & Adaptive Gradient Descent without Descent \citep{malitsky2020adaptive} \\
SSTM & stochastic similar-triangles method (as in ClipSSTM \citep{gorbunov2020stochastic,sadiev2023high}) \\
ClipSSTM & clipped accelerated similar-triangles template \citep{gorbunov2020stochastic} \emph{(baseline)} \\
GRAAL & golden-ratio algorithm \citep{malitsky2020golden}; only the form ``GRAAL-type'' is used in the text \\
KKT & Karush--Kuhn--Tucker \\
MDS & martingale difference sequence \\
MoM & median-of-means \\
PJ & Polyak--Juditsky (averaging) \\
SGD-GS & generalized-smoothness-aware SGD \emph{(baseline)} \\
Clip-SGD & clipped SGD \emph{(baseline)} \\
PJ-SGD & Polyak--Juditsky-averaged SGD \emph{(baseline)} \\
DS-SGD & decreasing-stepsize SGD \emph{(baseline)} \\
Acc-blind & accelerated method blind to generalized smoothness, with the heuristic constant $L_{\mathrm{blind}}$ \emph{(baseline)} \\
M-SGD-SGC & momentum SGD under the strong growth condition \citep{vaswani2019fast} \emph{(baseline)} \\
a.s. & almost surely \\
\bottomrule
\end{tabular}
\caption{Abbreviation glossary; expansions are the canonical ones used in this paper. Entries marked \emph{(baseline)} are experiment labels of Section~\ref{app:sec:experiments} and Appendix~\ref{app:experiments}, not literature abbreviations.}
\label{tab:abbreviations}
\end{table*}

\begin{table*}[t]
\centering
\footnotesize
\setlength{\tabcolsep}{4pt}
\begin{tabularx}{\textwidth}{@{}lX@{}}
\toprule
Symbol & Meaning (and scope, where restricted) \\
\midrule
$R_0$ & certified initial distance, $R_0\ge\norm{x^0-x^\star}$ \\
$\Rb$ & $:=3R_0$, uniform outer-loop distance bound \\
$\kbar$ & $:=1+\Lo\Rb$ \\
$\Dinit$ & $:=f(x^0)-f^\star$, the unknown actual initial gap \\
$\Duser$ & optional user-supplied upper certificate, $\Duser\ge\Dinit$ \\
$\Ddef$ & canonical computable certificate from $R_0$ (Lemma~\ref{lem:gap0}) \\
$\Dcert$ & $:=\min\{\Duser,\Ddef\}$ if $\Duser$ is supplied, and $:=\Ddef$ otherwise \\
$F_0$ & generic initial function gap used only in comparison tables; interpreted separately for each cited result \\
$\Gbar$, $\Deltabar$ & $:=\Lz/\Lo$ and $:=\Lz/\Lo^2$: critical gradient and critical gap (Corollary~\ref{cor:sublevel}) \\
$G_x$ & $:=\norm{\grad f(x)}$; at level $j$, $G_{c_j}=\norm{\grad f(c_j)}$ \\
$\Gest_j$ & certified gradient cap at level $j$, $\Gest_j\ge G_{c_j}$ \\
$L_{1,\mathrm{eff}}$ & $:=\max\{\Lo,\,1/(4\Rb)\}$, the effective parameter the algorithm and the analysis run with (Algorithm~\ref{app:alg:arc}, \eqref{eq:wlog}) \\
$q$ & dimensionless: $q:=\Lo R_0$ in Corollary~\ref{cor:phase1-default}, Remarks~\ref{app:rem:convex-default}--\ref{rem:strongly-default} and Proposition~\ref{prop:phase1-interp}; $q:=\Lo\Rb$ in the discussion of Section~\ref{sec:discussion-kbar} \\
$\eps$, $\alpha$ & target accuracy of the original problem; total failure probability \\
$H_0$, $H_{0,j}$ & certified initial gap of one proximal subproblem; level-$j$ initial subproblem-gap certificate \\
$\delta$, $\delta_j$ & target accuracy of one proximal subproblem; level-$j$ subproblem accuracy \\
$\beta$, $\beta_j$ & local failure probability; level-$j$ inner-solver failure probability \\
$\sigma$ & sub-Gaussian noise scale (Assumption~\ref{app:ass:noise}) \\
$\mu$ & strong-convexity modulus (where stated) \\
$\rho$ & local failure probability in the batch-concentration display of Appendix~\ref{sec:prelim}; strong-growth constant in Appendix~\ref{app:interp} (the two uses are explicitly scoped) \\
$T_1$, $B_1$ & \PhaseI{} iteration count and batch size \eqref{eq:phase1-params} \\
$J^{\mathrm{sch}}$, $K_{\mathrm{cert}}$, $J_{\mathrm{tot}}$ & schedule cap on levels; finisher-level cap; their sum $J^{\mathrm{sch}}+K_{\mathrm{cert}}$ \eqref{eq:schedule-global} \\
$K_j$, $N_{\mathrm{ep}}$, $b_k$ & \RSTM{} epochs at level $j$; iterations per epoch; per-iteration batch in epoch $k$ (Appendix~\ref{app:rstm}) \\
$m$ & number of independent Stage-3 runs (Algorithm~\ref{alg:interp}) \\
$\Lambda_{\lg}$ & \PhaseI{} logarithmic factor \eqref{eq:phase1-params} \\
$\ell$, $\ell_\star$ & umbrella logarithms of Appendix~\ref{app:cost} and of \eqref{eq:nu-ellstar} (Appendix~\ref{app:interp}), respectively \\
$A_t$ & \STM{} accumulator, \eqref{eq:stm} \\
$A$ & $:=\Lz\Rb^2/\eps$ (Section~\ref{sec:discussion-kbar} only); the distinct $A:=1+\Lz\Rb^2/\eps$ (Theorem~\ref{thm:interp-convex} and its proof only) \\
$Q$ & \STM{} feasible set; in Appendix~\ref{app:experiments}, the orthogonal rotation of the experimental instances \\
$x^\star$, $x^{\mathrm{int}}$ & selected minimizer of $f$; common interpolating minimizer (Appendix~\ref{app:interp} only) \\
\bottomrule
\end{tabularx}
\caption{Main global symbols. Local proof notation is defined where it appears; symbols carrying two meanings are scoped in the table.}
\label{tab:notation}
\end{table*}

\section{Preliminaries and structural lemmas}
\label{sec:prelim}

For self-containedness we restate the problem, the assumptions, and the global notation of the main paper. We consider the stochastic convex optimization problem
\begin{equation}
\label{app:eq:problem}
    \min_{x\in\R^d}\; f(x) \;=\; \E_{\xi\sim\mathcal{D}}\left[f_\xi(x)\right],
\end{equation}
accessed through an unbiased stochastic first-order oracle that, at a query point $x$, returns $g(x,\xi)$ with $\E[g(x,\xi)]=\grad f(x)$.

\paragraph{Global notation and initial-gap certificates.}
The norm $\norm{\cdot}$ is Euclidean,
$\ball{c}{r}:=\{x:\norm{x-c}\le r\}$, and $\proj_Q$ denotes
Euclidean projection onto a closed convex set $Q$. For every $a\ge0$ we write
$\logp a:=\log(\max\{e,a\})$; for $a>0$ this is
$\max\{1,\log a\}$, and the convention remains defined at $a=0$. Let
$X^\star:=\argmin f\neq\varnothing$, fix $x^\star\in X^\star$, and
write $f^\star:=f(x^\star)$. The known distance certificate satisfies
$R_0\ge\norm{x^0-x^\star}$, and
\[
    \Dinit:=f(x^0)-f^\star
\]
is the unknown actual initial gap. The algorithm only consumes a valid
upper certificate $\Dcert\ge\Dinit$. The computable default is
$\Ddef$ from Lemma~\ref{lem:gap0}; if a user certificate
$\Duser\ge\Dinit$ is supplied, then
$\Dcert:=\min\{\Duser,\Ddef\}$, and otherwise $\Dcert:=\Ddef$.
All schedule quantities that depend on an initial-gap bound are monotone in $\Dcert$.

Absolute constants are denoted by $c,c_1,c_2,\dots$.
The notation $\tO(\cdot)$ hides absolute constants and polylogarithmic
factors in the displayed problem parameters, including
$\Lo R_0$, $\Lo^2\Dcert/\Lz$, $\Lz\Rb^2/\eps$, and $1/\alpha$.
When $\Lo<1/(4\Rb)$, the schedules use
$L_{1,\mathrm{eff}}:=\max\{\Lo,1/(4\Rb)\}$; this is legitimate because
Assumption~\ref{app:ass:gs} is monotone in $\Lo$. The symbols $A$ and $Q$
are reused only in explicitly stated local scopes. In
Appendix~\ref{app:interp}, $x^{\mathrm{int}}$ denotes the common
interpolating minimizer, whereas $x^\star$ remains the selected
minimizer of the population objective.

\begin{assumption}[Convexity]
\label{app:ass:convex}
The function $f:\R^d\to\R$ is convex and differentiable.
\end{assumption}

\begin{assumption}[$(\Lz,\Lo)$-generalized smoothness]
\label{app:ass:gs}
There exist $\Lz>0$, $\Lo\ge 0$ such that for all $x,y\in\R^d$ with $\norm{y-x}\le 1/\Lo$,
\begin{equation}
\label{eq:gs}
    \norm{\grad f(y)-\grad f(x)} \le \big(\Lz+\Lo\norm{\grad f(x)}\big)\norm{y-x}.
\end{equation}
\end{assumption}

Condition \eqref{eq:gs} is the asymmetric local form used by \citet{vankov2024optimizing} and \citet{gorbunov2024methods}; for twice-differentiable $f$ it is implied by $\norm{\grad^2 f(x)}\le \Lz+\Lo\norm{\grad f(x)}$ up to an adjustment of constants (Appendix~\ref{app:lemmas}), and for $\Lo=0$ it is ordinary $\Lz$-smoothness.

\begin{assumption}[Sub-Gaussian oracle]
\label{app:ass:noise}
The oracle returns $g(x,\xi)$ with $\E[g(x,\xi)\mid x]=\grad f(x)$ and has a parameter $\sigma\ge0$. For $\sigma>0$, the noise $\zeta=g(x,\xi)-\grad f(x)$ satisfies $\E[\exp(\norm{\zeta}^2/\sigma^2)\mid x]\le e$. Whenever an algorithm draws a batch at a (possibly random, history-measurable) query point $x$, the samples are fresh: conditionally on the filtration $\mathcal F$ generated by the history up to that call, the batch is independent of $\mathcal F$ given $x$, and the conditional unbiasedness $\E[g(x,\xi)\mid\mathcal F]=\grad f(x)$ and the conditional exponential bound hold. For $\sigma=0$, we instead require $\zeta=0$ almost surely conditionally on $\mathcal F$; this is the deterministic-oracle convention, and no quotient by $\sigma^2$ is formed.
\end{assumption}

Assumption~\ref{app:ass:noise} covers almost surely bounded noise and implies $\E\norm{\zeta}^2\le\sigma^2$. We use the vector Hoeffding inequality for norm-sub-Gaussian vectors \citep{jin2019short}: there is an absolute constant $c_g\ge 1$ such that the batch mean $\gest=\GradEst(x,B)$ of $B$ oracle calls satisfies, with probability at least $1-\beta$,
$\norm{\gest-\grad f(x)}\le c_g\sigma\sqrt{\log(2/\beta)/B}$, for every $\beta\in(0,1)$.

The following facts, proved in Appendix~\ref{app:lemmas} with explicit constants, are used throughout. Write $G_x:=\norm{\grad f(x)}$ and $\Delta_x:=f(x)-f^\star$.

\begin{lemma}[Gradient growth]
\label{lem:growth}
For all $x,y$ with $r=\norm{y-x}$,
$G_y \le e^{\Lo r}G_x + \tfrac{\Lz}{\Lo}(e^{\Lo r}-1)\le e^{\Lo r}(G_x+\Lz r)$.
At $\Lo=0$ the middle term is read by continuity as $\Lz r$ (since $(e^{\Lo r}-1)/\Lo\to r$).
\end{lemma}

\begin{lemma}[Descent inequality]
\label{app:lem:descent}
For all $x,y$ with $\norm{y-x}\le 1/\Lo$,
$f(y)\le f(x)+\inner{\grad f(x)}{y-x}+\tfrac{\Lz+\Lo G_x}{2}\norm{y-x}^2$.
\end{lemma}

\begin{lemma}[Localization]
\label{app:lem:localization}
Under Assumptions~\ref{app:ass:convex} and \ref{app:ass:gs}, for every $x$,
\begin{equation}
\label{eq:localization}
G_x^2 \le 2\big(\Lz+\Lo G_x\big)\Delta_x,
\ \ \text{hence}\ \
G_x \le 2\Lo\Delta_x + \sqrt{2\Lz\Delta_x}.
\end{equation}
\end{lemma}

\begin{corollary}[Effective smoothness on sublevel sets]
\label{cor:sublevel}
Let $G(\Delta):=2\Lo\Delta+\sqrt{2\Lz\Delta}$. On the sublevel set $\{x: \Delta_x\le\Delta\}$ one has $G_x\le G(\Delta)$ and $\Lz+\Lo G_x \le 2\Lz + 3\Lo^2\Delta$. In particular, for the \emph{critical gap} $\Deltabar:=\Lz/\Lo^2$ and the \emph{critical gradient} $\Gbar:=\Lz/\Lo$: on $\{\Delta_x\le\Deltabar\}$ the local smoothness modulus is at most $5\Lz$. At $\Lo=0$ one sets $\Deltabar=\Gbar=+\infty$; the last claim then asserts modulus at most $5\Lz$ on $\{\Delta_x<\infty\}$, where it is $\Lz$.
\end{corollary}

\begin{remark}[$\mu\le\Lz$ in the strongly convex case]
\label{rem:mu-le-lz}
If $f$ is moreover $\mu$-strongly convex, then $\mu\le\Lz$: strong convexity with $\grad f(x^\star)=0$ gives $f(y)\ge f^\star+\tfrac\mu2\norm{y-x^\star}^2$ for \emph{all} $y$, while Lemma~\ref{app:lem:descent} at the base $x^\star$ (where $G_{x^\star}=0$) gives $f(y)\le f^\star+\tfrac\Lz2\norm{y-x^\star}^2$ for $\norm{y-x^\star}\le1/\Lo$ (for every $y$ if $\Lo=0$); applying both at the single point $y=x^\star+tv$ with $\norm{v}=1$ and $t:=\min\{1,1/\Lo\}>0$ yields $\tfrac\mu2t^2\le\tfrac\Lz2t^2$.
\end{remark}

\begin{lemma}[Initial gap]
\label{lem:gap0}
$\Dinit \le \Ddef$, where $\Ddef:=\tfrac{\Lz}{\Lo^2}\,\psi(\Lo R_0)$ with $\psi(t)=e^t-1-t\le\tfrac{t^2}{2}e^t$ for $\Lo>0$, and $\Ddef:=\tfrac{\Lz R_0^2}{2}$ at $\Lo=0$ (the continuous extension, since $\psi(t)/t^2\to\tfrac12$ as $t\to0$).
\end{lemma}

The next two facts about proximal steps are the structural backbone of \PhaseII. For $\lambda>0$, $r\in(0,\infty]$ and a center $c$, define the (ball-constrained) proximal point
\begin{equation}
\label{app:eq:prox}
\xhat(c;\lambda,r):=\argmin_{x\in\ball{c}{r}}\Big\{f(x)+\tfrac{\lambda}{2}\norm{x-c}^2\Big\}.
\end{equation}
(Throughout, $r$ and its decorated variants denote ball radii; the symbol $\rho$ is reserved for the strong growth constant of Appendix~\ref{app:interp}.)

\begin{lemma}[Proximal gradient and distance monotonicity]
\label{app:lem:prox-grad}
Under Assumption~\ref{app:ass:convex}, let $\lambda>0$,
$r\in(0,\infty]$, and $\xhat:=\xhat(c;\lambda,r)$. Then
\[
    \norm{\grad f(\xhat)}\le\norm{\grad f(c)},
    \qquad
    \norm{\xhat-c}\le\frac{\norm{\grad f(c)}}{\lambda}.
\]
Moreover, for every $u\in X^\star$,
\begin{equation}
\label{eq:fejer}
    \norm{\xhat-u}^2
    \le
    \norm{c-u}^2-\norm{\xhat-c}^2.
\end{equation}
If $r=\infty$, then
\[
    \norm{\xhat-c}=\frac{\norm{\grad f(\xhat)}}{\lambda}.
\]
\end{lemma}

\begin{corollary}[Fej\'er monotonicity of the prox]
\label{lem:fejer}
Under the assumptions of Lemma~\ref{app:lem:prox-grad}, inequality
\eqref{eq:fejer} holds for every $u\in X^\star$.
\end{corollary}

Lemma~\ref{app:lem:prox-grad} follows from monotonicity of $\grad f$ combined with the Karush--Kuhn--Tucker (KKT) stationarity condition $\grad f(\xhat)=-(\lambda+\eta)(\xhat-c)$, $\eta\ge0$: the gradient at the (possibly ball-constrained) exact prox is \emph{anti-parallel} to the displacement, which forces its norm below $\norm{\grad f(c)}$. Thus the gradient norm at the exact proximal reference point cannot increase. The algorithmic center is only an approximate prox, however, so its certified cap need not be monotone: the transfer estimate gives $\Gest^{+}\le\Gest+L_cs$, and Lemma~\ref{lem:drift} together with the travel--halving potential controls the accumulated rise.

Finally, we record the smoothness certificate for balls anchored at points with controlled gradients; it follows from Lemma~\ref{lem:growth} and \eqref{eq:gs}.

\begin{lemma}[Ball certificate]
\label{app:lem:ball}
Let $\Gest\ge G_c$ and $0<r\le\tfrac{1}{2\Lo}$. Then for all $x,y\in\ball{c}{2r}$:
$G_x\le e\,(\Gest+\Lz/\Lo)$ and
$\norm{\grad f(x)-\grad f(y)}\le \Lball\norm{x-y}$ with
$\Lball:=4(\Lz+\Lo\Gest)$.
Consequently $f$ is $\Lball$-smooth on $\ball{c}{2r}$ in the two-point sense, and the descent inequality with modulus $\Lball$ holds for all segments inside $\ball{c}{2r}$. At $\Lo=0$ the radius restriction is vacuous and the claims read $G_x\le\Gest+2\Lz r$ (via the continuous reading of Lemma~\ref{lem:growth}) and the global smoothness constant is $\Lball=\Lz$.
\end{lemma}

\subsection{Proofs of the structural lemmas}
\label{app:lemmas}

Throughout this section Assumptions~\ref{app:ass:convex} and \ref{app:ass:gs} are in force. Recall $G_x=\norm{\grad f(x)}$, $\Delta_x=f(x)-f^\star$.

\subsubsection{Proof of Lemma~\ref{lem:growth}}
If $\Lo=0$ the claim is the definition of $\Lz$-smoothness, so let $\Lo>0$. Let $r=\norm{y-x}$, $m:=\max\{1,\lceil \Lo r\rceil\}$ and $u_i:=x+\tfrac{i}{m}(y-x)$, $i=0,\dots,m$, so that $\norm{u_{i+1}-u_i}=r/m\le 1/\Lo$. Applying \eqref{eq:gs} on each segment with base point $u_i$,
\begin{equation*}
G_{u_{i+1}}\le G_{u_i}+\big(\Lz+\Lo G_{u_i}\big)\tfrac{r}{m}=\Big(1+\tfrac{\Lo r}{m}\Big)G_{u_i}+\tfrac{\Lz r}{m}.
\end{equation*}
Unrolling with $q:=1+\Lo r/m$,
\begin{equation*}
G_y\le q^mG_x+\tfrac{\Lz r}{m}\cdot\tfrac{q^m-1}{q-1}=q^mG_x+\tfrac{\Lz}{\Lo}(q^m-1)\le e^{\Lo r}G_x+\tfrac{\Lz}{\Lo}\big(e^{\Lo r}-1\big),
\end{equation*}
using $q^m=(1+\Lo r/m)^m\le e^{\Lo r}$. The second form follows from $e^t-1\le te^t$. 

\hfill$\square$

\subsubsection{Proof of Lemma~\ref{app:lem:descent}}
For $t\in[0,1]$ the point $x_t:=x+t(y-x)$ satisfies $\norm{x_t-x}\le\norm{y-x}\le1/\Lo$, so \eqref{eq:gs} with base $x$ gives $\norm{\grad f(x_t)-\grad f(x)}\le(\Lz+\Lo G_x)t\norm{y-x}$. Hence
\begin{equation*}
f(y)-f(x)-\inner{\grad f(x)}{y-x}=\int_0^1\inner{\grad f(x_t)-\grad f(x)}{y-x}\,dt\le\big(\Lz+\Lo G_x\big)\norm{y-x}^2\int_0^1 t\,dt,
\end{equation*}
which is the claim. 

\hfill$\square$

\subsubsection{Proof of Lemma~\ref{app:lem:localization} and Corollary~\ref{cor:sublevel}}
Let $s:=\tfrac{1}{\Lz+\Lo G_x}$ and $y:=x-s\grad f(x)$; then $\norm{y-x}=sG_x=\tfrac{G_x}{\Lz+\Lo G_x}\le\tfrac1\Lo$, so Lemma~\ref{app:lem:descent} applies:
\begin{equation*}
f^\star\le f(y)\le f(x)-sG_x^2+\tfrac{\Lz+\Lo G_x}{2}s^2G_x^2=f(x)-\tfrac{s}{2}G_x^2 ,
\end{equation*}
i.e., $G_x^2\le2(\Lz+\Lo G_x)\Delta_x$, the first inequality of \eqref{eq:localization}. Viewing it as a quadratic inequality in $G_x$,
$G_x\le \Lo\Delta_x+\sqrt{\Lo^2\Delta_x^2+2\Lz\Delta_x}\le2\Lo\Delta_x+\sqrt{2\Lz\Delta_x}$.
For the corollary, monotonicity of $G(\cdot)$ gives $G_x\le G(\Delta)$ on the sublevel set, and by Young's inequality $\Lo\sqrt{2\Lz\Delta}=\sqrt{2\Lz\cdot\Lo^2\Delta}\le\Lz+\tfrac{\Lo^2\Delta}{2}$, so
$\Lz+\Lo G(\Delta)=\Lz+2\Lo^2\Delta+\Lo\sqrt{2\Lz\Delta}\le2\Lz+\tfrac52\Lo^2\Delta$.
At $\Delta=\Deltabar=\Lz/\Lo^2$ this is $\tfrac92\Lz\le5\Lz$, and $G(\Deltabar)=2\Gbar+\sqrt2\,\Gbar\le\tfrac72\Gbar$. 

\hfill$\square$

\subsubsection{Proof of Lemma~\ref{lem:gap0}}
The claim is convexity-free: differentiability, Assumption~\ref{app:ass:gs} and $X^\star\neq\emptyset$ suffice. Since $\grad f(x^\star)=0$ by Fermat's condition, the fundamental theorem of calculus along $x_t^\star:=x^\star+t(x^0-x^\star)$ gives the identity $\Dinit=\int_0^1\inner{\grad f(x_t^\star)}{x^0-x^\star}\,dt$ (valid for the actual gap, not for an arbitrary certificate), and Lemma~\ref{lem:growth} with base $x^\star$ (where $G_{x^\star}=0$) yields $G_{x_t^\star}\le\tfrac{\Lz}{\Lo}(e^{\Lo tR_0}-1)$, hence
\begin{equation*}
\Dinit=\int_0^1\inner{\grad f(x_t^\star)}{x^0-x^\star}\,dt\le R_0\int_0^1\tfrac{\Lz}{\Lo}\big(e^{\Lo tR_0}-1\big)dt=\tfrac{\Lz}{\Lo^2}\big(e^{\Lo R_0}-1-\Lo R_0\big)=\tfrac{\Lz}{\Lo^2}\,\psi(\Lo R_0).
\end{equation*}
Finally $\psi(t)=\sum_{k\ge2}\tfrac{t^k}{k!}=\tfrac{t^2}{2}\sum_{k\ge0}\tfrac{2t^k}{(k+2)!}\le\tfrac{t^2}{2}e^t$ since $2/(k+2)!\le1/k!$. Moreover $\psi(t)/t^2=\sum_{k\ge0}\tfrac{t^k}{(k+2)!}$ is increasing in $t\ge0$ with limit $\tfrac12$ at $t=0$; hence $\Ddef$ extends continuously to $\Ddef=\tfrac{\Lz R_0^2}{2}$ at $\Lo=0$ and is monotone in $\Lo$, so the $L_{1,\mathrm{eff}}$-substitution of Algorithm~\ref{app:alg:arc} only enlarges it.

\hfill$\square$

\subsubsection{Proof of Lemma~\ref{app:lem:prox-grad}}
The problem \eqref{app:eq:prox} is convex with strongly feasible constraint ($c$ is interior when $r<\infty$; when $r=\infty$ there is no constraint), so KKT conditions hold at $\xhat$: there is $\eta\ge0$ (with $\eta=0$ if $\norm{\xhat-c}<r$) such that
\begin{equation}
\label{eq:kkt}
\grad f(\xhat)=-(\lambda+\eta)(\xhat-c).
\end{equation}
If $\xhat=c$ then $\grad f(c)=0$ by \eqref{eq:kkt} and the claim is trivial. Otherwise, monotonicity of $\grad f$ gives
$0\le\inner{\grad f(\xhat)-\grad f(c)}{\xhat-c}=-(\lambda+\eta)\norm{\xhat-c}^2-\inner{\grad f(c)}{\xhat-c}$,
hence $(\lambda+\eta)\norm{\xhat-c}^2\le\norm{\grad f(c)}\norm{\xhat-c}$ and therefore
$\norm{\grad f(\xhat)}=(\lambda+\eta)\norm{\xhat-c}\le\norm{\grad f(c)}$.
The same inequality gives
$\norm{\xhat-c}\le\norm{\grad f(c)}/(\lambda+\eta)\le\norm{\grad f(c)}/\lambda$.
For $r=\infty$ one has $\eta=0$, and \eqref{eq:kkt} yields
$\norm{\xhat-c}=\norm{\grad f(\xhat)}/\lambda$.
The Fej\'er inequality is proved in the next paragraph. 

\hfill$\square$

\subsubsection{Proof of the Fej\'er part of Lemma~\ref{app:lem:prox-grad}}
Let $\lambda>0$, let $\xhat$ be the minimizer of $f+\tfrac\lambda2\norm{\cdot-c}^2$ over $\ball{c}{r}$ with $r\in(0,\infty]$, and let $u\in X^\star$, so that $\grad f(u)=0$. By the optimality conditions \eqref{eq:kkt} there is $\eta\ge0$ with $\grad f(\xhat)=-(\lambda+\eta)(\xhat-c)$, and monotonicity of $\grad f$ gives
\begin{equation}
\label{eq:fejer-mono}
0\;\le\;\inner{\grad f(\xhat)-\grad f(u)}{\xhat-u}\;=\;-(\lambda+\eta)\inner{\xhat-c}{\xhat-u} .
\end{equation}
Since $\lambda+\eta\ge\lambda>0$, \eqref{eq:fejer-mono} yields $\inner{c-\xhat}{\xhat-u}\ge0$, whence
\begin{equation}
\label{eq:fejer-expand}
\norm{c-u}^2=\norm{c-\xhat}^2+2\inner{c-\xhat}{\xhat-u}+\norm{\xhat-u}^2\;\ge\;\norm{c-\xhat}^2+\norm{\xhat-u}^2 .
\end{equation}
This proves the lemma for $r=\infty$ and, with the same argument, the ball-constrained version used for the finisher and the strong-growth hops.

The restriction $\lambda>0$ is necessary and is satisfied at every call made by our algorithms, since $\lambda_j\ge\Lambda_j>0$ at every level (Appendix~\ref{app:params}) and $\lambda_{\mathrm{cert}}>0$ at the finisher. It cannot be dropped from the distance and Fej\'er conclusions: if $\lambda=0$ and the ball constraint is inactive at some minimizer (e.g.\ $f$ constant on a neighbourhood of $c$), then $\lambda+\eta=0$, the direction of $\xhat-c$ is unconstrained by \eqref{eq:fejer-mono}, the constrained minimizer need not be unique, and an arbitrary selection $\xhat\in\ball{c}{r}$ can violate \eqref{eq:fejer-expand}. Only the gradient-norm conclusion $\norm{\grad f(\xhat)}\le\norm{\grad f(c)}$ extends to $\lambda=0$, by the same KKT and monotonicity argument; the displacement bound (which divides by $\lambda$), the Fej\'er inequality, and the unconstrained distance identity do not.

\hfill$\square$

\subsubsection{Proof of Lemma~\ref{app:lem:ball}}
Let $x,y\in\ball{c}{2r}$ with $r\le\tfrac{1}{2\Lo}$. By Lemma~\ref{lem:growth}, $G_x\le e^{2\Lo r}(G_c+2\Lz r)\le e\,(\Gest+\Lz/\Lo)$, the first claim. For the two-point bound, the segment $[x,y]$ lies in $\ball{c}{2r}$ and has length $\norm{x-y}\le4r\le\tfrac2\Lo$; split it into $m\le2$ pieces of length at most $\tfrac1\Lo$ with endpoints $u_0=x,\dots,u_m=y$. On each piece, \eqref{eq:gs} with base $u_i$ and the first claim give
$\norm{\grad f(u_{i+1})-\grad f(u_i)}\le\big(\Lz+\Lo\, e(\Gest+\Lz/\Lo)\big)\norm{u_{i+1}-u_i}=\big((1+e)\Lz+e\Lo\Gest\big)\norm{u_{i+1}-u_i}$.
Summing over the pieces and using $1+e\le4$, $e\le4$ yields
$\norm{\grad f(x)-\grad f(y)}\le4(\Lz+\Lo\Gest)\norm{x-y}=\Lball\norm{x-y}$.
The descent inequality along any segment $[x,y]\subset\ball{c}{2r}$ follows by the integration argument of Lemma~\ref{app:lem:descent}, using the two-point bound at the pairs $(x,x+t(y-x))$. 

\hfill$\square$

\subsubsection{The Hessian form of generalized smoothness}
\begin{remark}
\label{rem:hessian}
If $f$ is twice differentiable with $\norm{\grad^2f(z)}\le\Lz+\Lo\norm{\grad f(z)}$ for all $z$, then Assumption~\ref{app:ass:gs} holds with the rescaled constants $(\Lz',\Lo'):=(e\Lz,e\Lo)$. We must verify \eqref{eq:gs} for the pair $(\Lz',\Lo')$, i.e.\ on \emph{its} admissible range $\norm{y-x}\le1/\Lo'=1/(e\Lo)$. We do so by proving the required bound on the \emph{larger} range $\norm{y-x}\le1/\Lo$, which contains $[0,1/(e\Lo)]$. Fix $x,y$ with $r:=\norm{y-x}\le1/\Lo$ and set $\varphi(t):=\norm{\grad f(x+t(y-x))}$; then $\varphi'(t)\le r\,(\Lz+\Lo\varphi(t))$ for a.e.\ $t\in[0,1]$, so Gr\"onwall's inequality gives $\Lz+\Lo\varphi(t)\le(\Lz+\Lo G_x)\,e^{\Lo rt}$, whence
$\norm{\grad f(y)-\grad f(x)}\le\int_0^1\norm{\grad^2f(x+t(y-x))}\,r\,dt\le r(\Lz+\Lo G_x)\tfrac{e^{\Lo r}-1}{\Lo r}\le e\,(\Lz+\Lo G_x)\,r=(\Lz'+\Lo' G_x)\,r$,
using $\tfrac{e^{s}-1}{s}\le e$ for $s=\Lo r\le1$. (For $\Lo=0$ the Gr\"onwall step divides by $\Lo r$; the conclusion then follows directly, $\varphi'\le\Lz r$ giving $\norm{\grad f(y)-\grad f(x)}\le\Lz r$, or by continuity in $\Lo$.) Since this bound holds for every $r\le1/\Lo$, it holds in particular for every $r\le1/(e\Lo)$, which is exactly the admissible range of the pair $(\Lz',\Lo')$; hence Assumption~\ref{app:ass:gs} holds with $(\Lz',\Lo')$. (The two ranges must be kept separate: the bound is \emph{proved} on $r\le1/\Lo$ but only \emph{needed}, and used, on the smaller rescaled range $r\le1/(e\Lo)$.) The factor $e$ in the rescaling can be improved to $1/\ln2$: what the argument needs is only $\tfrac{e^{s}-1}{s}\le c$ for all $s\le\tfrac1c$, and since $\tfrac{e^{s}-1}{s}$ is increasing the constraint binds at $s=\tfrac1c$, giving $e^{1/c}-1\le1$; the optimal fixed point is $e^{1/c}=2$, i.e.\ $c=\tfrac1{\ln2}$. Hence the Hessian form $\norm{\grad^2f(z)}\le\Lz+\Lo\norm{\grad f(z)}$ implies Assumption~\ref{app:ass:gs} already with the constants $(\Lz/\ln2,\Lo/\ln2)$, the optimal factor for this argument.
\end{remark}

\section{Concentration toolbox}
\label{app:conc}

We record the probabilistic facts used throughout; all constants are absolute. A random vector $\zeta\in\R^d$ is \emph{$\sigma$-norm-sub-Gaussian} if $\sigma>0$ and $\E\exp(\norm{\zeta}^2/\sigma^2)\le e$; at $\sigma=0$ this terminology means $\zeta=0$ almost surely, as stipulated in Assumption~\ref{app:ass:noise}. For $\sigma>0$, Jensen's inequality implies $\E\norm{\zeta}^2\le\sigma^2$, and Markov's inequality gives, for any $s>0$,
\begin{equation}
\label{eq:subexp-tail}
\Prob\big[\norm{\zeta}^2\ge\sigma^2(1+s)\big]\le e\cdot e^{-(1+s)}=e^{-s}.
\end{equation}

\begin{lemma}[Batch estimates; \citealp{jin2019short}]
\label{lem:batch}
There are absolute constants $c_g\ge1$, $c_b\ge1$ such that if $\zeta_1,\dots,\zeta_B$ are independent, zero-mean and $\sigma$-norm-sub-Gaussian conditionally on the past, then (i) for any $\beta\in(0,1)$, with probability at least $1-\beta$, $\norm{\tfrac1B\sum_i\zeta_i}\le c_g\sigma\sqrt{\log(2/\beta)/B}$; and (ii) $\tfrac1B\sum_i\zeta_i$ is $(c_b\sigma/\sqrt{B})$-norm-sub-Gaussian.
\end{lemma}

\begin{lemma}[Azuma for norm-sub-Gaussian martingale difference sequence (MDS); \citealp{jin2019short}]
\label{lem:azuma}
There is an absolute constant $c_{\mathrm{Az}}\ge1$ such that if $(\zeta_t)_{t\le N}$ are martingale differences with $\zeta_t$ $\varsigma$-norm-sub-Gaussian conditionally on $\mathcal F_{t-1}$, and $(w_t)_{t\le N}$ are $\mathcal F_{t-1}$-measurable vectors, then with probability at least $1-\beta$,
$\big|\sum_{t\le N}\inner{\zeta_t}{w_t}\big|\le c_{\mathrm{Az}}\,\varsigma\sqrt{\sum_{t\le N}\norm{w_t}^2\,\log(2/\beta)}$
whenever $\norm{w_t}\le W$ deterministically for all $t$ (in which case $\sum\norm{w_t}^2\le NW^2$ may be used).
\end{lemma}

\begin{lemma}[Uniform bound on squared noise]
\label{lem:maxnoise}
If $(\zeta_t)_{t\le N}$ are each $\varsigma$-norm-sub-Gaussian conditionally on the past, then with probability at least $1-\beta$, $\max_{t\le N}\norm{\zeta_t}^2\le\varsigma^2\big(1+\log(N/\beta)\big)$, and consequently for any deterministic weights $a_t$, $\sum_t a_t^2\norm{\zeta_t}^2\le\varsigma^2(1+\log(N/\beta))\sum_t a_t^2$.
\end{lemma}
\begin{proof}
Apply \eqref{eq:subexp-tail} with $s=\log(N/\beta)$ to each $t$ and take a union bound.
\end{proof}
We write $\GradEst(x,B)$ for the average of $B$ oracle answers at $x$; under Assumption~\ref{app:ass:noise} and Lemma~\ref{lem:batch}, $\GradEst(x,B)$ deviates from $\grad f(x)$ by at most $c_g\sigma\sqrt{\log(2/\beta)/B}$ with probability $1-\beta$, and is $(c_b\sigma/\sqrt B)$-norm-sub-Gaussian around it.

\begin{algorithm}[t]
\caption{$\GradEst(x,B)$: batched gradient estimate}
\label{alg:gradest}
\begin{algorithmic}[1]
\REQUIRE query point $x$ (possibly random, history-measurable); batch size $B\ge1$
\STATE draw $B$ fresh oracle samples $g(x,\xi_1),\dots,g(x,\xi_B)$ at $x$, conditionally independent given the history (Assumption~\ref{app:ass:noise})
\RETURN the batch mean $\gest\leftarrow\tfrac1B\sum_{i=1}^{B}g(x,\xi_i)$
\STATE \hfill (accuracy guarantee: Lemma~\ref{lem:batch}; the caller fixes $B$ and the confidence)
\end{algorithmic}
\end{algorithm}

\section{The inner stochastic solver: constrained batched \STM{} with epoch restarts}
\label{app:stm}

Let $Q\subseteq\R^d$ be a closed convex set, $F$ convex and $L$-smooth on $Q$ with minimizer $x_Q\in Q$, and let the oracle for $F$ have variance at most $\varsigma^2$ per call (in \ARCSG{}, $F=F_j$ is a proximal subproblem, $Q$ a ball, and batching reduces $\varsigma^2=\sigma^2/b$). The \emph{constrained stochastic similar-triangles method} (\STM) iterates, from $z_0=y_0=x_0\in Q$, with $A_t=\sum_{s\le t}a_s$,
\begin{align}
x_{t+1}&=\tfrac{A_ty_t+a_{t+1}z_t}{A_{t+1}}, \notag\\
z_{t+1}&=\proj_Q\big[z_t-a_{t+1}\gest(x_{t+1})\big],\notag\\
y_{t+1}&=\tfrac{A_ty_t+a_{t+1}z_{t+1}}{A_{t+1}},
\label{eq:stm}
\end{align}
with $a_t=\tfrac{t}{4L}$. All queries lie in $Q$, so only smoothness \emph{on} $Q$ is ever used. Our first building block is an explicit high-probability guarantee whose martingale terms are controlled by the diameter of $Q$; no clipping and no amplification are needed.

\begin{proposition}[Constrained batched \STM, high probability]
\label{prop:stm}
Let $D\ge\max_{x\in Q}\norm{x-x_Q}$ and let the per-query noise be $(\varsigma)$-norm-sub-Gaussian. Then for any $\beta\in(0,1)$, with probability at least $1-\beta$,
\begin{equation*}
\scalebox{0.83}{$\displaystyle
F(y_N)-F(x_Q)\le \frac{8LD^2}{N^2} + \frac{c_1\varsigma D\sqrt{\log(4/\beta)}}{\sqrt{N}} + \frac{c_2\varsigma^2(N{+}1)\log\tfrac{4N}{\beta}}{L}
$}
\end{equation*}
with absolute constants $c_1,c_2$ specified in Appendix~\ref{app:stm}.
\end{proposition}

\emph{Projection model.} Throughout the paper, the theory assumes an \emph{exact} Euclidean projection oracle: \STM{} performs one projection onto $Q$ per iteration, and \RSTM{} projects onto the two-ball intersections $Q_k=\ball{y^{(k)}}{r_k}\cap Q$. Projection calls are counted separately from the stochastic-gradient oracle calls that our complexity bounds report---one projection per \STM{} iteration, hence $\lceil24\sqrt{L/\lambda}\,\rceil$ per \RSTM{} epoch---and no projection-error analysis is provided. The two-ball projection does not reduce to a single-multiplier problem in general but is obtained from a two-dimensional dual system, solvable to machine precision (Remark~\ref{rem:proj}); the experiments of Section~\ref{app:sec:experiments} use single-ball projections only (Appendix~\ref{app:experiments}).

For $\lambda$-strongly convex $F$ (as every $F_j$ is), \RSTM{} wraps \STM{} into epochs: epoch $k$ runs \STM{} for $N_{\mathrm{ep}}=\lceil 24\sqrt{(L+\lambda)/\lambda}\,\rceil$ iterations on the ball $Q_k=\ball{y^{(k)}}{r_k}\cap Q$ with the \emph{shrinking certified radius} $r_k=2\sqrt{2\bar\Delta_k/\lambda}$, where $\bar\Delta_k$ is the running high-probability bound on the $F$-gap, halved each epoch by choosing the per-iteration batch $b_k$ so that the noise terms of Proposition~\ref{prop:stm} are at most $\bar\Delta_k/4$. Strong convexity certifies $x_Q\in Q_k$ from $\bar\Delta_k$; shrinking $r_k$ is what makes the statistical cost of epoch $k$ scale as $\tO(\sigma^2/(\lambda\bar\Delta_k))$---the rate of the smooth strongly convex case---rather than as $\sigma^2 D^2/\bar\Delta_k^2$.

\begin{proposition}[\RSTM]
\label{app:prop:rstm}
Let $F$ be $\lambda$-strongly convex and $L$-smooth on
$\ball{c}{2r}$, and define
\[
    x_F^\star:=\argmin_{x\in\ball{c}{r}}F(x).
\]
Suppose $y^{(0)}\in\ball{c}{r}$ and
$F(y^{(0)})-F(x_F^\star)\le H_0$. For a target subproblem
accuracy $\delta>0$ and a failure probability $\beta\in(0,1)$,
\RSTM{} returns a point $y$ such that, with probability at least
$1-\beta$,
\[
    F(y)-F(x_F^\star)\le\delta,
    \qquad
    \norm{y-x_F^\star}\le\sqrt{\frac{2\delta}{\lambda}}.
\]
The number of stochastic-gradient evaluations is
\[
\tO\left(
    \sqrt{\frac{L}{\lambda}}\logp\!\left(\frac{H_0}{\delta}\right)
    +\frac{\sigma^2}{\lambda\delta}
\right),
\]
and every query generated by \RSTM{} lies in $\ball{c}{2r}$.
\end{proposition}

\subsection{Constrained batched \STM: proof of Proposition~\ref{prop:stm}}

\begin{algorithm}[t]
\caption{\STM$(F,Q,L,N)$: constrained batched similar-triangles method}
\label{alg:stm}
\begin{algorithmic}[1]
\REQUIRE closed convex $Q$; convex $F$, $L$-smooth along segments of $Q$; iteration count $N$; per-iteration batch $b$ of fresh samples (Assumption~\ref{app:ass:noise}), giving per-query noise $\varsigma=c_b\sigma/\sqrt b$
\STATE $x_0\leftarrow y_0\leftarrow z_0\leftarrow$ an arbitrary point of $Q$;\quad $A_0\leftarrow0$
\FOR{$t=0,1,\dots,N-1$}
    \STATE $a_{t+1}\leftarrow\tfrac{t+1}{4L}$;\quad $A_{t+1}\leftarrow A_t+a_{t+1}$
    \STATE $x_{t+1}\leftarrow\big(A_ty_t+a_{t+1}z_t\big)/A_{t+1}$
    \STATE $z_{t+1}\leftarrow\proj_Q\big[z_t-a_{t+1}\gest(x_{t+1})\big]$ \hfill (one exact projection and one batched query per iteration)
    \STATE $y_{t+1}\leftarrow\big(A_ty_t+a_{t+1}z_{t+1}\big)/A_{t+1}$
\ENDFOR
\RETURN $y_N$ \hfill (high-probability guarantee: Proposition~\ref{prop:stm})
\end{algorithmic}
\end{algorithm}

Recall the setting: $Q$ closed convex, $F$ convex on $Q$ and $L$-smooth along segments of $Q$ (two-point Lipschitz gradients plus the descent inequality, as provided by Lemma~\ref{app:lem:ball}), $x_Q\in\argmin_QF$, $D\ge\max_{x\in Q}\norm{x-x_Q}$, and the oracle returns $\gest(x)=\grad F(x)+\zeta$ with $\zeta$ conditionally $\varsigma$-norm-sub-Gaussian. \STM{} runs \eqref{eq:stm} with $a_t=\tfrac{t}{4L}$, $A_t=\sum_{s\le t}a_s=\tfrac{t(t+1)}{8L}$, which satisfy
\begin{equation}
\label{eq:stm-step}
A_{t+1}=A_t+a_{t+1},\qquad La_{t+1}^2=\tfrac{(t+1)^2}{16L}\le\tfrac{(t+1)(t+2)}{16L}=\tfrac{A_{t+1}}{2}.
\end{equation}
All iterates lie in $Q$ (convex combinations and projections of points of $Q$). Write $\gest_{t+1}:=\gest(x_{t+1})=\grad F(x_{t+1})+\zeta_{t+1}$.

\emph{One-step inequality.} Since $y_{t+1}-x_{t+1}=\tfrac{a_{t+1}}{A_{t+1}}(z_{t+1}-z_t)$ and $[x_{t+1},y_{t+1}]\subset Q$, the descent inequality gives
\begin{equation}
\label{eq:stm1}
A_{t+1}F(y_{t+1})\le A_{t+1}F(x_{t+1})+a_{t+1}\inner{\grad F(x_{t+1})}{z_{t+1}-z_t}+\tfrac{La_{t+1}^2}{2A_{t+1}}\norm{z_{t+1}-z_t}^2 .
\end{equation}
The projection step means $z_{t+1}=\argmin_{z\in Q}\{a_{t+1}\inner{\gest_{t+1}}{z}+\tfrac12\norm{z-z_t}^2\}$, whose optimality condition against any $u\in Q$ reads
\begin{equation}
\label{eq:stm2}
a_{t+1}\inner{\gest_{t+1}}{z_{t+1}-u}\le\tfrac12\norm{u-z_t}^2-\tfrac12\norm{u-z_{t+1}}^2-\tfrac12\norm{z_{t+1}-z_t}^2 .
\end{equation}
From the definition of $x_{t+1}$, $a_{t+1}(x_{t+1}-z_t)=A_t(y_t-x_{t+1})$, so by convexity
\begin{equation}
\label{eq:stm3}
a_{t+1}\inner{\grad F(x_{t+1})}{x_{t+1}-z_t}=A_t\inner{\grad F(x_{t+1})}{y_t-x_{t+1}}\le A_t\big(F(y_t)-F(x_{t+1})\big),
\end{equation}
and $a_{t+1}\inner{\grad F(x_{t+1})}{u-x_{t+1}}\le a_{t+1}(F(u)-F(x_{t+1}))$. Decompose
$\inner{\grad F(x_{t+1})}{z_{t+1}-z_t}=\inner{\grad F(x_{t+1})}{z_{t+1}-u}+\inner{\grad F(x_{t+1})}{u-x_{t+1}}+\inner{\grad F(x_{t+1})}{x_{t+1}-z_t}$
and $\inner{\grad F(x_{t+1})}{z_{t+1}-u}=\inner{\gest_{t+1}}{z_{t+1}-u}-\inner{\zeta_{t+1}}{z_{t+1}-u}$; for the noise term use Young's inequality:
\begin{align}
\label{eq:stm4}
-a_{t+1}\inner{\zeta_{t+1}}{z_{t+1}-u}&=-a_{t+1}\inner{\zeta_{t+1}}{z_t-u}-a_{t+1}\inner{\zeta_{t+1}}{z_{t+1}-z_t} \notag\\
&\le-a_{t+1}\inner{\zeta_{t+1}}{z_t-u}+a_{t+1}^2\norm{\zeta_{t+1}}^2+\tfrac14\norm{z_{t+1}-z_t}^2 .
\end{align}
Combining \eqref{eq:stm1}--\eqref{eq:stm4} and $\tfrac{La_{t+1}^2}{2A_{t+1}}\le\tfrac14$ from \eqref{eq:stm-step}, the $\norm{z_{t+1}-z_t}^2$ terms cancel and
\begin{equation*}
A_{t+1}F(y_{t+1})\le A_tF(y_t)+a_{t+1}F(u)+\tfrac12\norm{u-z_t}^2-\tfrac12\norm{u-z_{t+1}}^2+a_{t+1}^2\norm{\zeta_{t+1}}^2-a_{t+1}\inner{\zeta_{t+1}}{z_t-u}.
\end{equation*}
Telescoping over $t=0,\dots,N-1$ with $u=x_Q$ and $A_0=0$:
\begin{equation}
\label{eq:stm-master}
A_N\big(F(y_N)-F(x_Q)\big)\le\tfrac12\norm{x_Q-z_0}^2+\sum_{t=1}^{N}a_t^2\norm{\zeta_t}^2+\Big|\sum_{t=1}^{N}a_t\inner{\zeta_t}{x_Q-z_{t-1}}\Big| .
\end{equation}
\emph{Concentration.} The weights $w_t:=a_t(x_Q-z_{t-1})$ are predictable with $\norm{w_t}\le a_tD$, and $\sum_ta_t^2=\tfrac{1}{16L^2}\sum_{t\le N}t^2\le\tfrac{N(N+1)^2}{48L^2}$. By Lemma~\ref{lem:azuma} (with failure probability $\beta/2$) and Lemma~\ref{lem:maxnoise} (with $\beta/2$),
\begin{equation*}
\Big|\sum_ta_t\inner{\zeta_t}{x_Q-z_{t-1}}\Big|\le c_{\mathrm{Az}}\varsigma D\tfrac{(N+1)\sqrt N}{\sqrt{48}\,L}\sqrt{\log\tfrac4{\beta}},\qquad
\sum_ta_t^2\norm{\zeta_t}^2\le\varsigma^2\big(1+\log\tfrac{2N}{\beta}\big)\tfrac{N(N+1)^2}{48L^2},
\end{equation*}
simultaneously with probability at least $1-\beta$. Dividing \eqref{eq:stm-master} by $A_N=\tfrac{N(N+1)}{8L}$ and using $\norm{x_Q-z_0}\le D$, $1+\log(2N/\beta)\le2\log(4N/\beta)$:
\begin{equation*}
F(y_N)-F(x_Q)\le\frac{4LD^2}{N(N+1)}+\frac{8c_{\mathrm{Az}}}{\sqrt{48}}\cdot\frac{\varsigma D\sqrt{\log(4/\beta)}}{\sqrt N}+\frac{(N+1)\,\varsigma^2\log(4N/\beta)}{3L}.
\end{equation*}
Proposition~\ref{prop:stm} follows with $c_1:=2c_{\mathrm{Az}}\ge8c_{\mathrm{Az}}/\sqrt{48}$ and $c_2:=\tfrac13$.

\hfill$\square$

\subsection{\RSTM: proof of Proposition~\ref{app:prop:rstm}}
\label{app:rstm}

\begin{algorithm}[t]
\caption{\RSTM$(F,\ball{c}{r},\lambda,L_F,H_0,\delta,\beta)$: restarted \STM{} on shrinking certified balls}
\label{alg:rstm}
\begin{algorithmic}[1]
\REQUIRE $F$ $\lambda$-strongly convex and $L_F$-smooth on $\ball{c}{2r}$; $x_F^\star:=\argmin_{x\in\ball{c}{r}}F(x)$; entry point $y^{(0)}\in\ball{c}{r}$ with certified gap $F(y^{(0)})-F(x_F^\star)\le H_0$; level target $\delta$; confidence $\beta$
\STATE $K\leftarrow\big\lceil\log_2\big(H_0/\delta\big)\big\rceil_+$ epochs;\quad $N_{\mathrm{ep}}\leftarrow\big\lceil24\sqrt{L_F/\lambda}\,\big\rceil$ iterations per epoch
\FOR{$k=0,1,\dots,K-1$}
    \STATE $\bar\Delta_k\leftarrow H_0\,2^{-k}$;\quad $r_k\leftarrow2\sqrt{2\bar\Delta_k/\lambda}$;\quad $Q_k\leftarrow\ball{y^{(k)}}{r_k}\cap\ball{c}{r}$ \hfill (projection set; Remark~\ref{rem:proj})
    \STATE $b_k\leftarrow\max\Big\{1,\ \Big\lceil C_b\,\dfrac{\sigma^2\log\big(8KN_{\mathrm{ep}}/\beta\big)}{\sqrt{\lambda L_F}\;\bar\Delta_k}\Big\rceil\Big\}$ \hfill (per-iteration batch, \eqref{eq:bk})
    \STATE $y^{(k+1)}\leftarrow\STM\big(F,Q_k,L_F,N_{\mathrm{ep}}\big)$ with batch $b_k$ of fresh samples per iteration and per-epoch confidence $\beta/K$
\ENDFOR
\RETURN $y^{(K)}$ \hfill ($F(y^{(K)})-F(x_F^\star)\le\delta$ and $\norm{y^{(K)}-x_F^\star}\le\sqrt{2\delta/\lambda}$ with probability at least $1-\beta$; Proposition~\ref{app:prop:rstm})
\end{algorithmic}
\end{algorithm}

Let $F$ be $\lambda$-strongly convex on $\ball{c}{2r}$ and $L_F$-smooth there, and define $x_F^\star:=\argmin_{x\in\ball{c}{r}}F(x)$ (in \PhaseII{} this is the unconstrained prox by containment; in the finisher it may lie on the boundary --- both are covered). Given $y^{(0)}\in\ball{c}{r}$ with $F(y^{(0)})-F(x_F^\star)\le H_0$, set $K:=\lceil\log_2(H_0/\delta)\rceil_+$, $\bar\Delta_k:=H_0\,2^{-k}$, and for $k=0,\dots,K-1$ run \STM{} for
\begin{equation*}
N_{\mathrm{ep}}:=\big\lceil24\sqrt{L_F/\lambda}\big\rceil\quad\text{iterations on}\quad Q_k:=\ball{y^{(k)}}{r_k}\cap\ball{c}{r},\qquad r_k:=2\sqrt{2\bar\Delta_k/\lambda},
\end{equation*}
with per-iteration batch size
\begin{equation}
\label{eq:bk}
b_k:=\max\Big\{1,\ \Big\lceil C_b\,\frac{\sigma^2\log\big(8KN_{\mathrm{ep}}/\beta\big)}{\sqrt{\lambda L_F}\;\bar\Delta_k}\Big\rceil\Big\},\qquad C_b:=48c_1^2c_b^2+70c_b^2+1,
\end{equation}
and let $y^{(k+1)}$ be the \STM{} output. We verify the invariant $F(y^{(k)})-F(x_F^\star)\le\bar\Delta_k$ by induction; it holds at $k=0$. Given the invariant, strong convexity yields $\norm{y^{(k)}-x_F^\star}\le\sqrt{2\bar\Delta_k/\lambda}=r_k/2$, so $x_F^\star\in Q_k$; also $Q_k\subseteq\ball{c}{r}$, so $F$ is $L_F$-smooth along segments of $Q_k$ and $D_k:=\max_{x\in Q_k}\norm{x-x_F^\star}\le r_k+r_k/2=\tfrac32r_k=3\sqrt{2\bar\Delta_k/\lambda}$. Applying Proposition~\ref{prop:stm} on $Q_k$ with confidence $\beta/K$ and per-query sub-Gaussian parameter $\varsigma_k=c_b\sigma/\sqrt{b_k}$ (Lemma~\ref{lem:batch}(ii)):
\begin{equation*}
F(y^{(k+1)})-F(x_F^\star)\le\underbrace{\frac{8L_F\cdot18\bar\Delta_k/\lambda}{N_{\mathrm{ep}}^2}}_{\le\bar\Delta_k/4\ \text{by }N_{\mathrm{ep}}^2\ge576L_F/\lambda}+\underbrace{\frac{c_1\varsigma_kD_k\sqrt{\log(4K/\beta)}}{\sqrt{N_{\mathrm{ep}}}}}_{\le\bar\Delta_k/8\ \text{by \eqref{eq:bk}}}+\underbrace{\frac{c_2\varsigma_k^2(N_{\mathrm{ep}}+1)\log(4KN_{\mathrm{ep}}/\beta)}{L_F}}_{\le\bar\Delta_k/8\ \text{by \eqref{eq:bk}}}\le\frac{\bar\Delta_k}{2},
\end{equation*}
where the two noise bounds follow by substituting $\varsigma_k^2=c_b^2\sigma^2/b_k$, $D_k^2=18\bar\Delta_k/\lambda$, $N_{\mathrm{ep}}\ge24\sqrt{L_F/\lambda}$ and the two summands of $C_b$ respectively. A union bound over the $K$ epochs completes the induction with probability $1-\beta$, giving $F(y^{(K)})-F(x_F^\star)\le\delta$ and, by strong convexity, $\norm{y^{(K)}-x_F^\star}\le\sqrt{2\delta/\lambda}$. The total number of oracle calls is
\begin{equation*}
\sum_{k=0}^{K-1}N_{\mathrm{ep}}\,b_k\le K\,N_{\mathrm{ep}}+25\sqrt{\tfrac{L_F}{\lambda}}\cdot C_b\,\frac{\sigma^2\log(8KN_{\mathrm{ep}}/\beta)}{\sqrt{\lambda L_F}}\sum_{k}\frac{1}{\bar\Delta_k}
\le25K\sqrt{\tfrac{L_F}{\lambda}}+\frac{100\,C_b\,\sigma^2\log(8KN_{\mathrm{ep}}/\beta)}{\lambda\,\delta},
\end{equation*}
using $\lceil24\sqrt{L_F/\lambda}\rceil\le25\sqrt{L_F/\lambda}$ (since $L_F\ge\lambda$) and $\sum_k\bar\Delta_k^{-1}\le2^{K+1}/H_0\le4/\delta$. All queries lie in $Q_k\subseteq\ball{c}{r}\subseteq\ball{c}{2r}$. \hfill$\square$

\begin{remark}[Projections and initial gaps]
\label{rem:proj}
Each $Q_k$ is the intersection of two Euclidean balls with nonempty interior. Its exact Euclidean projection does \emph{not} reduce to a single-multiplier problem in general: both spherical boundaries can be active simultaneously, and the KKT system then carries two multipliers $(\eta_1,\eta_2)\ge0$,
\begin{equation}
\label{eq:two-ball-kkt}
z-\Pi_{Q_k}(z)=\eta_1\big(\Pi_{Q_k}(z)-y^{(k)}\big)+\eta_2\big(\Pi_{Q_k}(z)-c\big),
\end{equation}
so that the projection is obtained from a two-dimensional dual problem (which is smooth, concave and solvable to machine precision by any standard method) or, alternatively, by Dykstra's alternating algorithm, which converges for this pair of convex sets and does so linearly under standard regularity conditions (satisfied here, since the interiors intersect). Only when one of the two constraints is known to be inactive does the one-dimensional reduction apply. This distinction is purely computational: throughout the paper, ``complexity'' means the number of \emph{stochastic gradient evaluations}, and the arithmetic cost of projections (as well as of any line searches in the methods we compare against) is not counted in it; we make this convention explicit here because the two costs are conflated easily.

In \PhaseII{} the initial gap of level $j$ is bounded by convexity as
\begin{equation}
\label{eq:init-gap}
F_j(c_j)-F_j(\xhat_j)\le\inner{\grad F_j(c_j)}{c_j-\xhat_j}=\inner{\grad f(c_j)}{c_j-\xhat_j}\le\Gest_j\cdot\frac{\Gest_j}{\lambda_j},
\end{equation}
so $H_{0,j}:=\Gest_j^2/\lambda_j$ is a valid certified input; in the finisher, $H_0:=\Gest\cdot\rsafe=\Gest/(2\Lo)$ works by the same argument.
\end{remark}

\section{\PhaseI{} analysis}
\label{app:phase1}

\begin{algorithm}[t]
\caption{\PhaseI{} (stochastic $(\Lz,\Lo)$-stepsize descent)}
\label{app:alg:phase1}
\begin{algorithmic}[1]
\REQUIRE $x_0=x^0$; threshold $G_{\mathrm{stop}}=\tfrac{4\Lz}{\Lo}$; budget $T_1^{\max}$, batch $B_1\ge1$, accuracy $\eps_1$
\REQUIRE parameters computed from $(\Lz,\Lo,R_0,\sigma,\alpha)$ and the gap certificate $\Dcert$ (the default $\Ddef$ of Lemma~\ref{lem:gap0}, or $\min\{\Duser,\Ddef\}$ with a user-provided $\Duser\ge f(x^0)-f^\star$; Appendix~\ref{app:params})
\FOR{$t=0,1,\dots,T_1^{\max}$}
    \STATE $\gest_t\leftarrow\GradEst(x_t,B_1)$
    \IF{$\norm{\gest_t}\le\tfrac34 G_{\mathrm{stop}}$} \RETURN $x^{\mathrm{I}}\leftarrow x_t$ \ENDIF
    \STATE $x_{t+1}\leftarrow x_t-\dfrac{\gest_t}{2\big(\Lz+\Lo\norm{\gest_t}\big)}$
\ENDFOR
\STATE \textbf{return} $x^{\mathrm{I}}\leftarrow x_{T_1^{\max}+1}$ \hfill (budget exhausted; reached only off the good event $\mathcal E_1$)
\end{algorithmic}
\end{algorithm}
\PhaseI{} (Algorithm~\ref{app:alg:phase1}) is batched SGD with the smoothed-clipping stepsize $\gamma_t=\tfrac12(\Lz+\Lo\norm{\gest_t})^{-1}$, the stochastic counterpart of the $(\Lz,\Lo)$-stepsizes of \citet{sgd2025generalized,vankov2024optimizing}. Its purpose is not to solve \eqref{app:eq:problem} but to reach a point with a \emph{certified} gradient bound of critical order $\Gbar=\Lz/\Lo$ while provably staying near $x^\star$. Three properties conspire (Appendix~\ref{app:phase1}): the step length never exceeds $\tfrac{1}{2\Lo}$, so Lemma~\ref{app:lem:descent} applies with the local modulus; the localization Lemma~\ref{app:lem:localization} makes the term $\gamma_t^2\norm{\gest_t}^2$ in the distance recursion comparable to $\gamma_t\Delta_t$, so the distance to $x^\star$ is non-increasing up to controlled statistical slack; and while the gradient exceeds $\Gbar$ the objective gap contracts \emph{linearly} at rate $1-\Theta(1/(\Lo R_0))$, the mechanism behind the ``linear convergence in the convex setup'' phenomenon of \citet{lobanov2024linear}.

\begin{proposition}[\PhaseI]
\label{app:prop:phase1}
Let Assumptions~\ref{app:ass:convex}--\ref{app:ass:noise} hold and $\Lo>0$, and let $\Dcert$ be the gap certificate (the default $\Ddef$ of Lemma~\ref{lem:gap0}, or $\min\{\Duser,\Ddef\}$ with a user-supplied $\Duser\ge f(x^0)-f^\star$). With the parameters of Appendix~\ref{app:params}, with probability at least $1-\alpha/4$ \PhaseI{} stops after
$T_1 \le 1+7\,\Lo R_0\,\logp\!\big(\tfrac{2\Lo^2\Dcert}{\Lz}\big)$
iterations and returns $x^{\mathrm{I}}$ with
(i) $\norm{\grad f(x^{\mathrm{I}})}\le\tfrac{7}{8}G_{\mathrm{stop}}$;
(ii) $f(x^{\mathrm{I}})\le f(x^0)$;
(iii) $\norm{x_t-x^\star}\le \tfrac{3}{2}R_0$ for all iterates.
The total number of oracle calls is
$T_1B_1=\tO\big(1+\Lo R_0+(1+\Lo R_0)\sigma^2\Lo^2/\Lz^2\big)$,
with $\tO$ hiding absolute constants and factors polylogarithmic in
$(\Lo R_0,\Lo^2\Dcert/\Lz,1/\alpha)$, i.e.\ with the certificate treated
as an independent input (the default certificate is expanded in
Corollary~\ref{cor:phase1-default}). The additive $1$ and
\(\sigma^2\Lo^2/\Lz^2\) contributions record the floors $T_1\ge1$
and $B_1\ge1$; they cannot in general be suppressed when
\(\Lo R_0\) is small.
\end{proposition}

\begin{corollary}[Default certificate]
\label{cor:phase1-default}
Let $q:=\Lo R_0$. This corollary is a statement about the default certificate: when $\Dcert=\Ddef=\tfrac{\Lz}{\Lo^2}\psi(q)$ of Lemma~\ref{lem:gap0} one has $\logp\tfrac{2\Lo^2\Ddef}{\Lz}=\logp\!\big(2\psi(q)\big)\le 1+q+\ln 2$ (since $\psi(q)\le e^q$), hence the logarithmic factor $\Lambda_{\lg}$ of \eqref{eq:phase1-params} is $\Theta(1+q)$, and the bounds of Proposition~\ref{app:prop:phase1} become
\begin{align*}
T_1&\le 1+7\,q\,(1+q+\ln2),\\
T_1B_1&=\tO\Big((1+q)^2+(1+q)^4\,\tfrac{\sigma^2\Lo^2}{\Lz^2}\Big),
\end{align*}
i.e., in the regime $q\gtrsim1$, $T_1$ is bounded by a $\Theta(q^2)$ quantity and $T_1B_1=\tO\big(q^2+q^4\sigma^2\Lo^2/\Lz^2\big)$, with $\tO$ hiding only absolute constants and polylogarithmic factors in $(q,1/\alpha)$. Thus, under the default gap certificate, the \PhaseI{} entry price is quadratic rather than linear in $q$; with an externally supplied certificate the corresponding form is $\tO\big(1+q+(1+q)\sigma^2\Lo^2/\Lz^2\big)$. Whether the quadratic dependence is sharp is open: on some example classes the actual iteration count is $\Theta(q)$, and sharpening the analysis (or proving a matching lower bound) is an open problem we do not address.
\end{corollary}

\subsection{Proof of Proposition~\ref{app:prop:phase1}}

Throughout $\Lo>0$ and the parameters are
\begin{align}
\label{eq:phase1-params}
&\Lambda_{\lg}:=\big\lceil2+2\logp\tfrac{2\Lo^2\Dcert}{\Lz}\big\rceil,\qquad
\eps_1:=\frac{\Gbar}{8\Lambda_{\lg}},\qquad
B_1:=\max\Big\{1,\ \Big\lceil\frac{c_g^2\sigma^2\log\tfrac{8(T_1^{\max}+1)}{\alpha}}{\eps_1^2}\Big\rceil\Big\},\notag\\
&T_1^{\max}:=\max\Big\{1,\;\big\lceil6\Lambda_{\lg}\Lo R_0\big\rceil,\;1+\big\lceil7\Lo R_0\logp\tfrac{2\Lo^2\Dcert}{\Lz}\big\rceil\Big\}.
\end{align}
By default $\Dcert:=\tfrac{\Lz}{\Lo^2}\psi(\Lo R_0)$, the certified bound $\Ddef$ of Lemma~\ref{lem:gap0}, which makes the algorithm executable from $(x^0,R_0,\Lz,\Lo,\sigma,\alpha,\eps)$ alone; any user-provided upper bound $\Duser\ge f(x^0)-f^\star$ may be substituted, with the algorithm consuming $\min\{\Duser,\Ddef\}$, which can only tighten the logarithms.
Define the good event $\mathcal E_1:=\{\norm{\gest_t-\grad f(x_t)}\le\eps_1\ \forall t\le T_1^{\max}\}$; by Lemma~\ref{lem:batch}(i) and a union bound, $\Prob[\mathcal E_1]\ge1-\alpha/4$. We work on $\mathcal E_1$. Note $\Lambda_{\lg}\ge4$, so $\eps_1\le\Gbar/32$.

\emph{Step executed $\Rightarrow$ super-critical gradient.} If iteration $t$ executes a step then $\norm{\gest_t}>\tfrac34G_{\mathrm{stop}}=3\Gbar$, hence
\begin{equation}
\label{eq:supercrit}
G_t\ge\norm{\gest_t}-\eps_1\ge3\Gbar-\tfrac{\Gbar}{32}\ge\tfrac{23}{8}\Gbar,\qquad
\frac{\eps_1}{\norm{\gest_t}}\le\frac{1}{24\Lambda_{\lg}}\le\frac1{96},\qquad
\frac{\eps_1}{G_t}\le\frac{1}{23\Lambda_{\lg}}\le\frac1{92},
\end{equation}
and consequently $\tfrac{22}{23}G_t\le\norm{\gest_t}\le\tfrac{24}{23}G_t$. The step length is $\gamma_t\norm{\gest_t}=\tfrac{\norm{\gest_t}}{2(\Lz+\Lo\norm{\gest_t})}\le\tfrac{1}{2\Lo}$, so Lemma~\ref{app:lem:descent} applies at base $x_t$.

\emph{(i) Per-step decrease.} By Lemma~\ref{app:lem:descent} with $y=x_{t+1}=x_t-\gamma_t\gest_t$,
\begin{equation*}
f(x_{t+1})\le f(x_t)-\gamma_t\inner{\grad f(x_t)}{\gest_t}+\tfrac{\Lz+\Lo G_t}{2}\gamma_t^2\norm{\gest_t}^2 .
\end{equation*}
Now $\inner{\grad f(x_t)}{\gest_t}\ge\norm{\gest_t}^2-\eps_1\norm{\gest_t}\ge\tfrac{23}{24}\norm{\gest_t}^2$ by \eqref{eq:supercrit}, and
$\tfrac{(\Lz+\Lo G_t)\gamma_t}{2}=\tfrac{\Lz+\Lo G_t}{4(\Lz+\Lo\norm{\gest_t})}\le\tfrac{23}{88}$ since $\norm{\gest_t}\ge\tfrac{22}{23}G_t$. Hence the bracket is at least $\tfrac{23}{24}-\tfrac{23}{88}=\tfrac{23}{33}$, and using $\norm{\gest_t}^2\ge(\tfrac{22}{23})^2G_t^2$ and $\gamma_t\ge\tfrac{23}{24}\cdot\tfrac{1}{2(\Lz+\Lo G_t)}$,
\begin{equation}
\label{eq:phase1-decrease}
f(x_t)-f(x_{t+1})\;\ge\;\frac{23}{33}\Big(\frac{22}{23}\Big)^2\frac{23}{24}\cdot\frac{G_t^2}{2(\Lz+\Lo G_t)}\;\ge\;\frac{3}{10}\cdot\frac{G_t^2}{\Lz+\Lo G_t}\;>\;0 .
\end{equation}
In particular $f$ is non-increasing along the trajectory. In the super-critical regime \eqref{eq:supercrit} we have $\Lz\le\tfrac{8}{23}\Lo G_t$, so $\Lz+\Lo G_t\le\tfrac{31}{23}\Lo G_t$ and
\begin{equation}
\label{eq:phase1-lin}
f(x_t)-f(x_{t+1})\ge\frac{3}{10}\cdot\frac{23}{31}\cdot\frac{G_t}{\Lo}\ge\frac{2G_t}{9\Lo}.
\end{equation}

\emph{(iii) Distance control.} Let $r_t:=\norm{x_t-x^\star}$. Expanding and using convexity ($\inner{\grad f(x_t)}{x_t-x^\star}\ge\Delta_t$),
\begin{equation*}
r_{t+1}^2=r_t^2-2\gamma_t\inner{\gest_t}{x_t-x^\star}+\gamma_t^2\norm{\gest_t}^2\le r_t^2-2\gamma_t\Delta_t+2\gamma_t\eps_1r_t+\gamma_t\cdot\frac{\norm{\gest_t}^2}{2(\Lz+\Lo\norm{\gest_t})}.
\end{equation*}
The map $u\mapsto\tfrac{u^2}{2(\Lz+\Lo u)}$ is increasing, so with $\norm{\gest_t}\le\tfrac{24}{23}G_t$ and Lemma~\ref{app:lem:localization},
$\tfrac{\norm{\gest_t}^2}{2(\Lz+\Lo\norm{\gest_t})}\le(\tfrac{24}{23})^2\tfrac{G_t^2}{2(\Lz+\Lo G_t)}\le\tfrac{12}{11}\Delta_t$, whence
\begin{equation}
\label{eq:phase1-dist}
\begin{split}
r_{t+1}^2&\le r_t^2-\tfrac{10}{11}\gamma_t\Delta_t+2\gamma_t\eps_1r_t\le r_t^2+2\gamma_t\eps_1r_t,\\
2\gamma_t\eps_1&=\frac{\eps_1}{\Lz+\Lo\norm{\gest_t}}\le\frac{\eps_1}{3\Lo\Gbar}=\frac{1}{24\Lambda_{\lg}\Lo}.
\end{split}
\end{equation}
\emph{Case $\Lo R_0\le\tfrac1{24}$ (degenerate).} One has $\Dcert\le\Ddef=\tfrac{\Lz}{\Lo^2}\psi(\Lo R_0)\le\tfrac{\Lz}{\Lo^2}\psi(\tfrac1{24})\le0.00089\,\tfrac{\Lz}{\Lo^2}$, since $\Lo R_0\le\tfrac1{24}$ in this branch and $\psi$ is increasing; hence, by Lemma~\ref{app:lem:localization} applied with $\Dinit\le\Dcert$, $G_0\le2\Lo\Dcert+\sqrt{2\Lz\Dcert}\le(0.0018+0.0422)\Gbar\le0.05\,\Gbar$, so $\norm{\gest_0}\le G_0+\eps_1\le0.09\,\Gbar<3\Gbar$: the algorithm returns $x^{\mathrm I}=x^0$ at $t=0$, and all claims hold trivially (with $\norm{\grad f(x^{\mathrm I})}\le0.05\Gbar$).

\emph{Case $\Lo R_0>\tfrac1{24}$.} We prove $r_t\le\tfrac32R_0$ for all $t\le T_1^{\max}$ by induction; the budget-exhausted iterate $x_{T_1^{\max}+1}$ is returned only off the good event $\mathcal E_1$ (Algorithm~\ref{app:alg:phase1}) and therefore needs no distance guarantee. Assume it up to $t$. Summing \eqref{eq:phase1-dist},
$r_{t+1}^2\le R_0^2+T_1^{\max}\cdot\tfrac{1}{24\Lambda_{\lg}\Lo}\cdot\tfrac32R_0=R_0^2+\tfrac{T_1^{\max}R_0}{16\Lambda_{\lg}\Lo}$.
We bound the last term by $\tfrac{5}{4}R_0^2$ in each branch of the maximum in \eqref{eq:phase1-params}, using $\Lambda_{\lg}\Lo R_0\ge\tfrac{\Lambda_{\lg}}{24}\ge\tfrac16$: if $T_1^{\max}=\lceil6\Lambda_{\lg}\Lo R_0\rceil\le6\Lambda_{\lg}\Lo R_0+1$ then
$\tfrac{T_1^{\max}R_0}{16\Lambda_{\lg}\Lo}\le\tfrac38R_0^2+\tfrac{R_0}{16\Lambda_{\lg}\Lo}\le\tfrac38R_0^2+\tfrac{6}{16}R_0^2\le\tfrac34R_0^2$;
if $T_1^{\max}=1+\lceil7\Lo R_0\logp(\cdot)\rceil\le2+7\Lo R_0\logp(\cdot)$ then, since $7\logp(\cdot)\le\tfrac72\Lambda_{\lg}$,
$\tfrac{T_1^{\max}R_0}{16\Lambda_{\lg}\Lo}\le\tfrac{7\logp(\cdot)}{16\Lambda_{\lg}}R_0^2+\tfrac{2R_0}{16\Lambda_{\lg}\Lo}\le\tfrac{7}{32}R_0^2+\tfrac34R_0^2\le\tfrac{31}{32}R_0^2$;
and $T_1^{\max}=1$ is trivial. In all cases $r_{t+1}^2\le R_0^2+\tfrac54R_0^2=\tfrac94R_0^2$, closing the induction: $r_t\le\tfrac32R_0$.

\emph{(iv) Iteration count and output.} Let $\bar\Delta_{\mathrm{stop}}:=\tfrac{\Lz}{2\Lo^2}$. If $\Delta_t\le\bar\Delta_{\mathrm{stop}}$ then by Lemma~\ref{app:lem:localization} $G_t\le2\Lo\bar\Delta_{\mathrm{stop}}+\sqrt{2\Lz\bar\Delta_{\mathrm{stop}}}=2\Gbar<\tfrac{23}8\Gbar$, contradicting \eqref{eq:supercrit}; hence every executed step has $\Delta_t>\bar\Delta_{\mathrm{stop}}$. Combining \eqref{eq:phase1-lin} with $G_t\ge\Delta_t/r_t\ge\tfrac{2\Delta_t}{3R_0}$ (convexity),
$\Delta_{t+1}\le\big(1-\tfrac{4}{27\Lo R_0}\big)\Delta_t$
(if $\tfrac{4}{27\Lo R_0}\ge1$ a single step already gives $\Delta_{t+1}\le0$). Therefore after at most $\tfrac{27}{4}\Lo R_0\ln\tfrac{\Dcert}{\bar\Delta_{\mathrm{stop}}}\le7\Lo R_0\logp\tfrac{2\Lo^2\Dcert}{\Lz}$ executed steps the gap is below $\bar\Delta_{\mathrm{stop}}$, and the next check sees $\norm{\gest_t}\le G_t+\eps_1\le2\Gbar+\tfrac{\Gbar}{32}<3\Gbar$ and returns. Thus $T_1\le1+7\Lo R_0\logp\tfrac{2\Lo^2\Dcert}{\Lz}\le T_1^{\max}$, and the returned point satisfies $\norm{\grad f(x^{\mathrm I})}\le\norm{\gest_{T_1}}+\eps_1\le3\Gbar+\tfrac{\Gbar}{32}\le\tfrac{25}{8}\Gbar\le\tfrac78G_{\mathrm{stop}}$, with $f(x^{\mathrm I})\le f(x^0)$ by \eqref{eq:phase1-decrease} and $\norm{x^{\mathrm I}-x^\star}\le\tfrac32R_0$. The oracle-call count is
\begin{equation}
\label{eq:phase1-cost}
\begin{split}
T_1B_1&\le\Big(2+7\Lo R_0\logp\tfrac{2\Lo^2\Dcert}{\Lz}\Big)\Big(1+64c_g^2\Lambda_{\lg}^2\tfrac{\sigma^2\Lo^2}{\Lz^2}\log\tfrac{8(T_1^{\max}+1)}{\alpha}\Big)
=\tO\Big(1+\Lo R_0+(1+\Lo R_0)\tfrac{\sigma^2\Lo^2}{\Lz^2}\Big);
\end{split}
\end{equation}
the additive floors remain because the deterministic iteration count is
at least one and $B_1\ge1$.

\hfill$\square$

\section{The complete \ARCSG{} schedule and \PhaseII{} analysis}
\label{app:phase2-section}

\begin{algorithm}[t]
\caption{\PhaseII{} of \ARCSG{} (proximal levels)}
\label{app:alg:phase2}
\begin{algorithmic}[1]
\REQUIRE $c_1=x^{\mathrm{I}}$, cap $\Gest_1=\tfrac{4\Lz}{\Lo}$, floor $\Lambda_1=32\Lz$, distance bound $\Rb=3R_0$
\REQUIRE target $\eps$, schedule $J^{\mathrm{sch}}$, drift resolution $\iota_{\mathrm{abs}}$, accuracies $\{\delta_j,\tau_j,s_j\}$ and inner failure budgets $\{\beta_j\}$ (Appendix~\ref{app:params})
\FOR{$j=1,2,\dots,J^{\mathrm{sch}}$}
    \IF{$\Gest_j\le\tfrac{\eps}{2\Rb}$} \RETURN $c_j$ \ENDIF
    \IF{$\Gest_j< 32\,\iota_{\mathrm{abs}}$ \textbf{or} $j=J^{\mathrm{sch}}$} \RETURN \textsf{Finisher}$(c_j,\Gest_j)$ \hfill (Alg.~\ref{alg:finisher}) \ENDIF
    \STATE $\lambda_j\leftarrow\max\{4\Lo\Gest_j,\;\Lambda_j\}$;\quad $r_j\leftarrow \tfrac{2\Gest_j}{\lambda_j}$
    \STATE $F_j(\cdot)\leftarrow f(\cdot)+\tfrac{\lambda_j}{2}\norm{\cdot-c_j}^2$
    \STATE $c_{j+1}\leftarrow\RSTM\big(F_j,\ \ball{c_j}{r_j},\ \text{target }\delta_j,\ \text{failure }\beta_j\big)$
    \STATE $\Gest_{j+1}'\leftarrow \Gest_j + L_j s_j$ \hfill (cap bookkeeping)
    \STATE $\gest_j\leftarrow\GradEst(c_{j+1},B_j^{\mathrm{est}})$;\; $\Gest_{j+1}\leftarrow\min\{\Gest_{j+1}',\,\norm{\gest_j}+\tau_j\}$
    \IF{$\lambda_j=\Lambda_j$}
        \IF{$\Lambda_j\le\tfrac{\eps}{2\Rb^2}$} \RETURN $c_{j+1}$ \ENDIF
        \STATE $\Lambda_{j+1}\leftarrow\Lambda_j/2$
    \ELSE
        \STATE $\Lambda_{j+1}\leftarrow\Lambda_j$
    \ENDIF
\ENDFOR
\end{algorithmic}
\end{algorithm}

\begin{algorithm}[t]
\caption{\ARCSG{}}
\label{app:alg:arc}
\begin{algorithmic}[1]
\REQUIRE $x^0$, $R_0$, $(\Lz,\Lo)$, $\sigma$, target $\eps$, confidence $\alpha$
\REQUIRE optionally a user certificate $\Duser\ge f(x^0)-f^\star$; set $\Dcert:=\min\{\Duser,\Ddef\}$ if supplied, else $\Dcert:=\Ddef$ of Lemma~\ref{lem:gap0}
\STATE \textbf{if} $\Dcert\le\eps$ \textbf{then return} $x^0$ \hfill (trivial regime)
\STATE set the effective parameter $L_{1,\mathrm{eff}}:=\max\big\{\Lo,\tfrac{1}{4\Rb}\big\}$, used in place of $\Lo$ throughout \hfill (Appendix~\ref{app:params}; covers $\Lo=0$)
\STATE $x^{\mathrm{I}}\leftarrow$ \PhaseI{}$(x^0)$
\STATE \textbf{return} \PhaseII{}$(x^{\mathrm{I}},\eps,\alpha)$
\end{algorithmic}
\end{algorithm}
\PhaseII{} (Algorithm~\ref{app:alg:phase2}) maintains a center $c_j$, a certified \emph{gradient cap} $\Gest_j\ge\norm{\grad f(c_j)}$, and a geometrically decreasing \emph{regularization floor} $\Lambda_j$. Level $j$ approximately solves the proximal subproblem $F_j=f+\tfrac{\lambda_j}{2}\norm{\cdot-c_j}^2$ with
\begin{equation}
\label{app:eq:lambda}
\lambda_j=\max\{4\Lo\Gest_j,\ \Lambda_j\},
\end{equation}
by running \RSTM{} on the ball $\ball{c_j}{r_j}$, $r_j=2\Gest_j/\lambda_j\le\tfrac{1}{2\Lo}$. The design closes four loops at once.
\emph{Containment:} by Lemma~\ref{app:lem:prox-grad}, $\norm{\xhat_j-c_j}\le\norm{\grad f(c_j)}/\lambda_j\le r_j/2$, so the exact prox lies in $\ball{c_j}{r_j/2}$ and the ball constraint is inactive at the minimizer, while every \RSTM{} query remains in the certified region $\ball{c_j}{2r_j}$.
\emph{Smoothness:} by Lemma~\ref{app:lem:ball}, on $\ball{c_j}{2r_j}$ the function $f$ is $L_j$-smooth with $L_j=4(\Lz+\Lo\Gest_j)$; hence $F_j$ is $\lambda_j$-strongly convex and $(L_j+\lambda_j)$-smooth there, with condition number $\kappa_j=1+L_j/\lambda_j\le 2+\Lz/(\Lo\Gest_j)$ on gradient-driven levels and $\kappa_j\le 2+4\Lz/\Lambda_j+ \Lo\Gest_j/\Lambda_j$ in general.
\emph{Gradient bookkeeping:} $\norm{\grad f(c_{j+1})}\le\norm{\grad f(\xhat_j)}+L_j\norm{c_{j+1}-\xhat_j}\le \Gest_j+L_js_j$ by Lemma~\ref{app:lem:prox-grad} and Proposition~\ref{app:prop:rstm} ($s_j=\sqrt{2\delta_j/\lambda_j}$), so the cap update of Algorithm~\ref{app:alg:phase2} is valid \emph{without any measurement}; the relative-accuracy measurement made there (precision $\tau_j=\Gest'_{j+1}/16$, batch $B_j^{\mathrm{est}}=\tO(\sigma^2/\tau_j^2)$) re-anchors the cap to the true gradient so that caps track gradients up to a constant factor.
\emph{Termination:} if $\Gest_j\le\eps/(2\Rb)$ then $f(c_j)-f^\star\le G_{c_j}\norm{c_j-x^\star}\le\eps$ by convexity; alternatively, once the floor reaches $\Lambda_j\le\eps/(2\Rb^2)$ and binds, the classical regularization argument $f(c_{j+1})-f^\star\le\delta_j+\tfrac{\lambda_j}{2}\norm{c_j-x^\star}^2\le\eps$ applies because the prox is unconstrained by \emph{Containment}. A third, rarely triggered mechanism (the \emph{finisher}, Algorithm~\ref{alg:finisher} in Appendix~\ref{app:phase2}) handles the corner case in which the gradient cap falls below the drift resolution $32\iota_{\mathrm{abs}}=\Theta(\Lz/(\Lo J^{\mathrm{sch}}))$ of the cap bookkeeping: there the gap is already $O(\Lz\Rb/(\Lo J^{\mathrm{sch}}))$, and $\tO(\Lo\Rb)$ ball-constrained proximal steps with $\lambda=\Lo\eps/\Rb$ finish by a geometric contraction argument at the same $\kbar^{1/2}$-price. Distances $\norm{c_j-x^\star}\le\Rb$ are maintained by Lemma~\ref{lem:fejer} plus the accuracy schedule.

\paragraph{Why it is fast: travel and halving.}
The analysis classifies each level by the behaviour of the exact prox. If $\norm{\xhat_j-c_j}\le\tfrac{1}{8\Lo}$ (a \emph{halving} level), then $\norm{\grad f(\xhat_j)}=\lambda_j\norm{\xhat_j-c_j}\le\Gest_j/2$ on gradient-driven levels: the gradient scale halves. Otherwise (a \emph{travel} level) the prox moved a macroscopic distance $\ge\tfrac{1}{8\Lo}$, and the proximal inequality forces $f(c_{j+1})\le f(c_j)-\Gest_j/(48\Lo)$: the method pays for distance with function decrease. Because $f(c_j)-f^\star\le G_{c_j}\Rb\le\Gest_j\Rb$, the total decrease available while the gradient scale is of order $g$ is at most $2g\Rb$, so at most $O(\Lo\Rb)$ travel levels can occur \emph{per excursion of the gradient scale through a dyadic band}, and halving levels are logarithmically few per band (Lemma~\ref{lem:level-count}). The potential argument does not by itself bound the number of such excursions; the unconditional guarantee is obtained by a hard schedule cap of $J^{\mathrm{sch}}=\tO(\Lo\Rb)$ levels, after which control passes to a certified finishing routine (the \emph{finisher}, Algorithm~\ref{alg:finisher}) whose $\tO(\Lo\Rb)$ ball-constrained proximal steps complete the run at the same $\kbar^{1/2}$-price (Lemma~\ref{lem:cert}); the cap-drift Lemma~\ref{lem:drift} limits upward movements of the certified gradient scale to a total of $\Lz/(6\Lo)$, which is why revisits of a band from below are heavily constrained in practice, although we do not use this to sharpen the worst-case count. Each level costs $\tO(\sqrt{\kappa_j})$ deterministic-equivalent iterations plus its statistical batches, and summing yields Theorem~\ref{app:thm:convex}. Figure~\ref{app:fig:schematic} visualizes the two phases.

\begin{figure}[t]
\centering
\includegraphics[width=0.9\columnwidth]{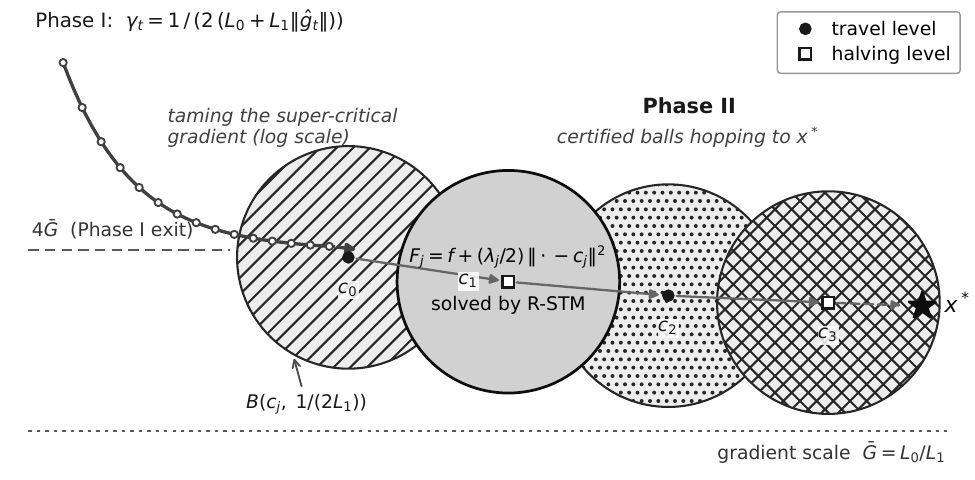}
\caption{Anatomy of \ARCSG{} on a generalized-smooth objective. \PhaseI{} (red) contracts the gap linearly while the gradient is super-critical; \PhaseII{} (blue) alternates travel and halving proximal levels, each solved by the constrained accelerated solver \RSTM{} inside a ball of radius $O(1/\Lo)$ (shaded), until the gradient certificate, the regularization floor, or the certified finisher certifies $\eps$-optimality.}
\label{app:fig:schematic}
\end{figure}

\subsection{The full parameter schedule of \PhaseII{} and the finisher}
\label{app:params}

\paragraph{Standing convention.} Throughout the \PhaseII{} analysis, $\Lo$ stands for the effective parameter $L_{1,\mathrm{eff}}:=\max\big\{\Lo,\,1/(4\Rb)\big\}$ introduced in Algorithm~\ref{app:alg:arc}; equivalently, we assume without loss of generality that
\begin{equation}
\label{eq:wlog}
\Lo\ \ge\ \tfrac{1}{4\Rb}\qquad\text{(equivalently }\Lo\Rb\ge\tfrac14,\ \kbar\ge\tfrac54\text{)}.
\end{equation}
This is indeed no loss: Assumption~\ref{app:ass:gs} is monotone in $\Lo$ (if it holds with $(\Lz,\Lo)$ it holds with $(\Lz,\Lo')$ for any $\Lo'\ge\Lo$, on the smaller range $\norm{y-x}\le1/\Lo'$), so when the problem's $\Lo$ is smaller than $1/(4\Rb)$ the algorithm and the analysis may be run with the enlarged effective parameter $L_{1,\mathrm{eff}}$ in place of $\Lo$---a change of the algorithm's input, not a redefinition of the problem; every $\kbar$-factor changes by a factor at most $\tfrac54$, which is absorbed in $\tO$. The convention also covers $\Lo=0$ (the classical smooth case); under the substitution the effective threshold of \PhaseI{} is $3\Gbar=12\Lz\Rb$, and since $G_{x^0}\le2\Lo\Dcert+\sqrt{2\Lz\Dcert}$ (Lemma~\ref{app:lem:localization}, applicable since $\Dinit\le\Dcert$), the condition $\Dcert\le\Lz\Rb^2$ gives $G_{x^0}\le2\cdot\tfrac{1}{4\Rb}\cdot\Lz\Rb^2+\sqrt{2\Lz^2\Rb^2}=(\tfrac12+\sqrt2)\Lz\Rb<2\Lz\Rb=\tfrac{\Gbar}{2}$, hence $\norm{\gest_0}\le G_{x^0}+\eps_1<\tfrac{\Gbar}{2}+\tfrac{\Gbar}{32}<3\Gbar$: \PhaseI{} exits at $t=0$ whenever $\Dcert\le\Lz\Rb^2$---in the floor regime $\Lo\le1/(4\Rb)$ unconditionally, since then $\Dcert\le\Ddef(L_{1,\mathrm{eff}})=16\Lz\Rb^2\psi(\tfrac1{12})\approx0.057\Lz\Rb^2<\Lz\Rb^2$ by the monotonicity of $\Ddef$ in the smoothness parameter (Lemma~\ref{lem:gap0})---so Theorem~\ref{app:thm:convex} degenerates to the smooth-case complexity as claimed in Remark~\ref{rem:smooth}; when $\Dcert>\Lz\Rb^2$ the (cheap, $\sigma$-scaled) \PhaseI{} iterations simply run as usual and their count is dominated by the stated bounds. From Algorithm~\ref{app:alg:arc} on, every $\Lo$ appearing in parameters, algorithms and analysis denotes the effective parameter $L_{1,\mathrm{eff}}$ (so in particular the finisher's $\lambda_{\mathrm{cert}}=\Lo\eps/\Rb$ of \eqref{eq:cert-params} is positive even at $\Lo=0$, as Lemma~\ref{lem:fejer} requires).

\paragraph{Trivial regime.} If the certified initial gap already meets the target, i.e.\ $\Dcert\le\eps$, the algorithm returns $x^0$ and there is nothing to prove; \emph{all statements below therefore assume $\eps<\Dcert$}, and this test is the first line of Algorithm~\ref{app:alg:arc}. Since $\Dcert$ is computed explicitly---as the default $\Ddef$ of Lemma~\ref{lem:gap0}, or as $\min\{\Duser,\Ddef\}$ when a user certificate is supplied---this is a finite, checkable condition. The counters below are nevertheless defined with positive parts, so that every schedule quantity is well defined and positive for \emph{all} $\eps>0$.

\paragraph{Global quantities.} With $\Rb=3R_0$, $\kbar=1+\Lo\Rb$, $\Gbar=\Lz/\Lo$, and $\logp a=\max\{1,\log a\}$:
\begin{align}
&S_G:=\max\Big\{1,\Big\lceil\log_2^{+}\frac{10\Lz\Rb}{\Lo\eps}\Big\rceil\Big\},\qquad
S_\Lambda:=\max\Big\{1,\Big\lceil\log_2^{+}\frac{128\Lz\Rb^2}{\eps}\Big\rceil\Big\},\notag\\
&J^{\mathrm{sch}}:=60\lceil\kbar\rceil S_G(S_\Lambda+4)+16,\notag\\
&K_{\mathrm{cert}}:=\big\lceil2(1+\Lo\Rb)(S_G+4)\big\rceil+1,\notag\\
&J_{\mathrm{tot}}:=J^{\mathrm{sch}}+K_{\mathrm{cert}},\notag\\
&\iota_{\mathrm{abs}}:=\frac{\Lz}{8\Lo J_{\mathrm{tot}}},\qquad
\iota_{\mathrm{flr}}:=\frac{\Lz}{32\Lo S_\Lambda},
\label{eq:schedule-global}
\end{align}
where $\log_2^{+}a:=\max\{0,\log_2a\}$. Both counters are thus at least $1$, so $J^{\mathrm{sch}},K_{\mathrm{cert}},J_{\mathrm{tot}}\ge1$ and the denominators in $\iota_{\mathrm{abs}},\iota_{\mathrm{flr}}$ never vanish. In the nontrivial regime $\eps<\Dcert$ the positive part of $S_\Lambda$ is inactive whenever $\eps\le\Lz\Rb^2$, and that of $S_G$ whenever $\eps\le10\Gbar\Rb$; since $10\Gbar\Rb\ge\Lz\Rb^2$ iff $\Lo\Rb\le10$, the single condition $\eps\le\Lz\Rb^2$ (a fortiori whenever $\Dcert\le\Lz\Rb^2$, e.g.\ in the floor regime $\Lo\le1/(4\Rb)$) makes both inactive in the moderate regime $\Lo\Rb\le10$, and then the bounds proved below coincide with the ones obtained from the unclipped definitions. When a positive part \emph{is} active the corresponding counter equals $1$ and every bound in this appendix only improves; for $S_G$ this case means $\eps>10\Gbar\Rb$, hence $\Gest_1=4\Gbar<\tfrac{\eps}{2\Rb}$ and termination (a) fires at level $1$, so the clipped schedule is never actually executed.

\paragraph{Level quantities.} At level $j$ with state $(c_j,\Gest_j,\Lambda_j)$ (initialized $\Gest_1=4\Gbar$, $\Lambda_1=32\Lz$):
\begin{align}
&L_j:=4(\Lz+\Lo\Gest_j),\qquad
\lambda_j:=\max\{4\Lo\Gest_j,\ \Lambda_j\},\qquad
r_j:=\frac{2\Gest_j}{\lambda_j}\ \Big(\le\frac{1}{2\Lo}\Big),\qquad
L_{F,j}:=L_j+\lambda_j;\notag\\
&\text{\emph{gradient-driven level} if }4\Lo\Gest_j>\Lambda_j,\quad\text{\emph{floor level} if }4\Lo\Gest_j\le\Lambda_j;\notag\\
&s_j:=\begin{cases}\min\big\{\tfrac{R_0}{4J_{\mathrm{tot}}},\ \tfrac{\iota_{\mathrm{abs}}}{L_j}\big\}&\text{(gradient-driven and finisher levels)}\\[2pt]
\min\big\{\tfrac{R_0}{8S_\Lambda},\ \tfrac{\iota_{\mathrm{flr}}}{L_j}\big\}&\text{(floor levels)}\end{cases}\notag\\
&\delta_j:=\min\Big\{\frac{\lambda_js_j^2}{2},\ \frac{\eps}{8}\Big\};\notag\\
&\Gest'_{j+1}:=\Gest_j+L_js_j,\qquad \tau_j:=\frac{\Gest'_{j+1}}{16},\qquad
B^{\mathrm{est}}_j:=\max\Big\{1,\ \Big\lceil\frac{c_g^2\sigma^2\log(16J_{\mathrm{tot}}/\alpha)}{\tau_j^2}\Big\rceil\Big\}.
\label{eq:schedule-level}
\end{align}
For every normal or finisher level, set
$\beta_j:=\alpha/(4J_{\mathrm{tot}})$. Each level runs
\RSTM{} on $(F_j,\ball{c_j}{r_j})$ with strong convexity
$\lambda_j$, smoothness $L_{F,j}$, certified initial gap
$H_{0,j}:=\Gest_j^2/\lambda_j$ (Remark~\ref{rem:proj}), target
$\delta_j$, and failure probability $\beta_j$; the estimate
$\gest_j:=\GradEst(c_{j+1},B^{\mathrm{est}}_j)$ has accuracy
$\tau_j$ with failure probability $\alpha/(4J_{\mathrm{tot}})$
(Lemma~\ref{lem:batch}); the new cap is $\Gest_{j+1}:=\min\{\Gest'_{j+1},\,\norm{\gest_j}+\tau_j\}$, and the level is a \emph{halving} level if $\Gest_{j+1}\le\tfrac34\Gest_j$ (equivalently $\norm{\gest_j}\le\tfrac34\Gest_j-\tau_j$). At floor levels, if $\Lambda_j\le\eps/(2\Rb^2)$ the algorithm returns $c_{j+1}$ (termination (b)); otherwise $\Lambda_{j+1}:=\Lambda_j/2$. The gradient termination (a) returns $c_j$ when $\Gest_j\le\eps/(2\Rb)$, checked at the head of every level (including finisher levels).

\paragraph{The finisher.} Entered when $\Gest_j<32\iota_{\mathrm{abs}}$ or when $j=J^{\mathrm{sch}}$ (Algorithm~\ref{alg:finisher}); runs at most $K_{\mathrm{cert}}$ \emph{finisher levels}, each a ball-constrained proximal step with
\begin{equation}
\label{eq:cert-params}
\begin{split}
\lambda_{\mathrm{cert}}&:=\frac{\Lo\eps}{\Rb},\qquad
\rsafe:=\frac{1}{2\Lo},\qquad
\theta:=\frac{1}{2\Lo\Rb},\\
\delta_j^{\mathrm{cert}}&:=\min\Big\{\frac{\eps}{64\Lo\Rb},\ \frac{\lambda_{\mathrm{cert}}s_j^2}{2}\Big\},\qquad
H_{0,\mathrm{cert}}:=\frac{\Gest_j}{2\Lo},
\end{split}
\end{equation}
solved by \RSTM{} for the ball-constrained minimizer (Appendix~\ref{app:rstm} covers this case verbatim), with the same cap bookkeeping \eqref{eq:schedule-level} and the same $s_j$-formula as gradient-driven levels.

\paragraph{Confidence budget (single accounting).} To avoid any double reading, Table~\ref{tab:events} lists every stochastic event family of the analysis, its budget, and the number of instances it is split over, together with the unused reserved quarter. Normal and finisher levels draw from the \emph{same} counters: the finisher's \RSTM{} runs and measurements are among the $J_{\mathrm{tot}}=J^{\mathrm{sch}}+K_{\mathrm{cert}}$ instances already counted. Its remaining certification checks are deterministic on the three stochastic good events and therefore spend none of the reserved fourth quarter. Thus the \emph{good event} $\mathcal E$ has failure probability at most $3\alpha/4$ and, in particular, $\Prob[\mathcal E]\ge1-\alpha$.

\begin{table*}[t]
\centering
\footnotesize
\setlength{\tabcolsep}{4pt}
\begin{tabular}{@{}llll@{}}
\toprule
Event family & Budget & Instances & Per instance \\
\midrule
\PhaseI{} gradient estimates & $\alpha/4$ & $T_1^{\max}+1$ & $\alpha/(4(T_1^{\max}{+}1))$ \\
\RSTM{} runs (all levels, incl.\ finisher) & $\alpha/4$ & $J_{\mathrm{tot}}$ & $\alpha/(4J_{\mathrm{tot}})$ \\
Cap measurements (all levels, incl.\ finisher) & $\alpha/4$ & $J_{\mathrm{tot}}$ & $\alpha/(4J_{\mathrm{tot}})$ \\
Finisher-only certification events & $\alpha/4$ reserved & deterministic on the shared good event & not used \\
\bottomrule
\end{tabular}
\caption{Confidence budget. Each stochastic row is a separately budgeted event family; no independence or disjointness between rows is needed, since the budgets are combined by a union bound. Within each stochastic row the budget is split uniformly over a deterministically bounded number of instances. The fourth quarter is reserved for finisher-only checks. In the present construction those checks are deterministic on the shared good event, so the random events actually consume at most $3\alpha/4$, leaving slack.}
\label{tab:events}
\end{table*}

\subsection{The convex main theorem (restated)}
For convenience we restate the convex main theorem of the main paper; it is proved in full below (Theorem~\ref{thm:convex-full} and the proof following it).

\begin{theorem}[Convex case]
\label{app:thm:convex}
Let Assumptions~\ref{app:ass:convex}--\ref{app:ass:noise} hold. Fix
$x^\star\in X^\star$, a distance certificate
$R_0\ge\norm{x^0-x^\star}$, and a valid initial-gap certificate
$\Dcert\ge f(x^0)-f^\star$. Let $\eps>0$, $\alpha\in(0,1)$,
and define $\Rb:=3R_0$ and $\kbar:=1+\Lo\Rb$. With the parameter
schedule of Appendix~\ref{app:params}, \ARCSG{} returns $\xhat$
such that
\[
    \Prob\bigl(f(\xhat)-f^\star\le\eps\bigr)\ge1-\alpha
\]
using at most
\begin{align*}
N \;=\; \tO\Big(&\underbrace{1+\Lo R_0}_{\textnormal{\PhaseI{} + levels}}
\;+\;\underbrace{\kbar^{1/2}\sqrt{\tfrac{\Lz\Rb^2}{\eps}}}_{\textnormal{optimization}}
\;+\;\underbrace{\kbar^{3}\,\tfrac{\sigma^2\Rb^2}{\eps^2}}_{\textnormal{statistical}}+\;\underbrace{(1+\Lo R_0)\tfrac{\sigma^2\Lo^2}{\Lz^2}+\kbar^{3}\tfrac{\sigma^2\Lo^2}{\Lz^2}}_{\textnormal{burn-in}}\Big)
\end{align*}
stochastic-gradient evaluations. Here $\Dcert$ is treated as an
externally supplied certificate, and $\tO$ hides absolute constants
and factors polylogarithmic in
$\bigl(\Lo R_0,\Lo^2\Dcert/\Lz,\Lz\Rb^2/\eps,1/\alpha\bigr)$.
Under the default certificate of Lemma~\ref{lem:gap0}, the expanded
form of Remark~\ref{app:rem:convex-default} applies. All logarithmic
factors are explicit in Theorem~\ref{thm:convex-full}.
\end{theorem}

\subsection{Analysis of \PhaseII: invariants, counting, the finisher, and Theorem~\ref{app:thm:convex}}
\label{app:phase2}

Throughout this section the good event $\mathcal E$ of Appendix~\ref{app:params} holds, the convention \eqref{eq:wlog} is in force, and ``level'' means a normal (\PhaseII) or finisher level; there are at most $J_{\mathrm{tot}}$ of them by construction. We abbreviate $\iota:=\iota_{\mathrm{abs}}$.

\paragraph{Logical structure of the proof.} The argument has one entry point and two possible exits, and each arrow below is a separate statement proved in this appendix:
\begin{equation}
\label{eq:proof-map}
\begin{split}
&\underbrace{\PhaseI}_{\text{Prop.~\ref{app:prop:phase1}}}
\ \longrightarrow\
\underbrace{\text{invariants (I1)--(I5)}}_{\text{Lem.~\ref{lem:inv}, \ref{lem:drift}}}
\ \longrightarrow
\begin{cases}
\ \underbrace{\text{normal levels}}_{\text{Lem.~\ref{lem:level-count}: \emph{per excursion}}}\ \longrightarrow\ \underbrace{\text{gradient or floor exit}}_{\text{Lem.~\ref{lem:exit}}}\\[2mm]
\ \underbrace{\text{finisher}}_{\text{Lem.~\ref{lem:cert}: \emph{unconditional}}}\ \longrightarrow\ \underbrace{\eps\text{-solution}}_{\text{Thm.~\ref{thm:convex-full}}}
\end{cases}
\end{split}
\end{equation}
The distinction matters and we flag it wherever it arises: Lemma~\ref{lem:level-count} bounds the number of \emph{productive} levels \emph{within one gradient-scale excursion} and does not by itself bound the number of excursions; unconditional termination comes from the hard schedule cap $J^{\mathrm{sch}}$ together with the finisher (Lemma~\ref{lem:cert}), which is the branch that makes Theorem~\ref{app:thm:convex} an unconditional statement.

To make this guarantee checkable in one place---rather than distributed across Lemmas~\ref{lem:inv}--\ref{lem:cert}---we record it as a single statement; the constituent lemmas below supply its ingredients, and Theorem~\ref{thm:convex-full} then only adds the cost accounting.

\begin{proposition}[Termination and exit correctness of \PhaseII]
\label{prop:phase2-terminates}
On $\mathcal E$, Algorithm~\ref{app:alg:phase2}---together with the finisher (Algorithm~\ref{alg:finisher})---halts after executing at most $J_{\mathrm{tot}}=J^{\mathrm{sch}}+K_{\mathrm{cert}}$ levels, and, whichever of its three exits fires, the returned point $\xhat$ satisfies $f(\xhat)-f^\star\le\eps$.
\end{proposition}
\begin{proof}
\emph{Termination} is by construction: the outer loop runs $j=1,\dots,J^{\mathrm{sch}}$ and, at the latest when $j=J^{\mathrm{sch}}$, hands control to the finisher (line~3 of Algorithm~\ref{app:alg:phase2}), which itself executes at most $K_{\mathrm{cert}}$ ball-constrained levels and then returns (Algorithm~\ref{alg:finisher}); hence at most $J_{\mathrm{tot}}$ levels execute in all. \emph{Exit correctness:} the invariants that make the three exits applicable---the certified caps $\norm{\grad f(c_j)}\le\Gest_j$ and the distances $\norm{c_j-x^\star}\le\Rb$---hold at every executed level by Lemma~\ref{lem:inv}. Given them, the gradient exit~(a) certifies $f(\xhat)-f^\star\le\tfrac\eps2$ and the floor exit~(b) certifies $f(\xhat)-f^\star\le\tfrac58\eps$ (Lemma~\ref{lem:exit}), while the finisher returns $\xhat$ with $f(\xhat)-f^\star\le\tfrac{7}{16}\eps$ (Lemma~\ref{lem:cert}); in every case $f(\xhat)-f^\star\le\eps$.

\end{proof}

\subsubsection{Invariants}
\begin{lemma}[Floor-level count]
\label{lem:floor-count}
On $\mathcal E$, $\Lambda_j$ is non-increasing, halves at every non-terminal floor level, and never drops below $\tfrac{\eps}{4\Rb^2}$; hence at most $S_\Lambda$ floor levels execute in the entire run. Moreover, after a floor level $i$ one has $\Lambda\ge2\Lo\Gest_i$, so the cap cannot halve between consecutive floor levels without the floor binding again.
\end{lemma}
\begin{proof}
$\Lambda$ changes only at floor levels, where it halves unless the level terminates; a floor level with $\Lambda_j\le\eps/(2\Rb^2)$ returns, so all halvings start from $\Lambda_j>\eps/(2\Rb^2)$ and $\Lambda\ge\eps/(4\Rb^2)$ throughout. The number of halvings from $\Lambda_1=32\Lz$ down to $\eps/(4\Rb^2)$ is at most $\lceil\log_2\tfrac{128\Lz\Rb^2}{\eps}\rceil=S_\Lambda$. At a binding floor level $i$, $\Lambda_i\ge4\Lo\Gest_i$, so after halving $\Lambda\ge2\Lo\Gest_i$; if later $\Gest_j\le\Gest_i/2$ then $4\Lo\Gest_j\le2\Lo\Gest_i\le\Lambda_j$ and the floor binds again.

\end{proof}

\begin{lemma}[Invariants]
\label{lem:inv}
On $\mathcal E$, for every executed level $j$:
\begin{enumerate}
\item[\textup{(I1)}] \emph{Certified caps:} $\norm{\grad f(c_j)}\le\Gest_j$.
\item[\textup{(I2)}] \emph{Distances:} $\norm{c_j-x^\star}\le\tfrac{15}{8}R_0\le\Rb$.
\item[\textup{(I3)}] \emph{Local regularity:} $F_j$ is $\lambda_j$-strongly convex and $L_{F,j}$-smooth along segments of $\ball{c_j}{2r_j}$, where $r_j\le\tfrac{1}{2\Lo}$ at all levels, and $L_j\ge\Lz+\Lo G_x$ pointwise on the ball (Lemma~\ref{app:lem:ball} with \textup{(I1)}).
\item[\textup{(I4)}] \emph{Containment (normal levels):} $\norm{\xhat_j-c_j}\le G_{c_j}/\lambda_j\le r_j/2$, so $\xhat_j$ is the unconstrained minimizer of $F_j$ and $H_{0,j}=\Gest_j^2/\lambda_j$ certifies the initial gap (Remark~\ref{rem:proj}). \RSTM{} succeeds: $F_j(c_{j+1})-F_j(\xhat_j)\le\delta_j$, $\norm{c_{j+1}-\xhat_j}\le\sqrt{2\delta_j/\lambda_j}\le s_j$, and $c_{j+1}\in\ball{c_j}{r_j}$. At finisher levels the same holds with $\xhat_j$ replaced by the ball-constrained minimizer $\tilde x_j$.
\item[\textup{(I5)}] \emph{Cap-update validity:} $G_{c_{j+1}}\le\min\{\Gest_j+L_js_j,\ \norm{\gest_j}+\tau_j\}=\Gest_{j+1}$, and each level raises the cap by at most $L_js_j$ ($\le\iota$ at gradient-driven and finisher levels, $\le\iota_{\mathrm{flr}}$ at floor levels).
\end{enumerate}
\end{lemma}
\begin{proof}
Induction on $j$. \emph{Base:} $\norm{\grad f(c_1)}\le\tfrac{25}{8}\Gbar\le4\Gbar=\Gest_1$ and $\norm{c_1-x^\star}\le\tfrac32R_0$ by Proposition~\ref{app:prop:phase1}. \emph{Step:} given (I1), Lemma~\ref{app:lem:ball} with center $c_j$, cap $\Gest_j$ and radius $r_j$ yields (I3); here $r_j=\tfrac{1}{2\Lo}$ exactly at gradient-driven and finisher levels, and $r_j=2\Gest_j/\Lambda_j\le\tfrac{1}{2\Lo}$ at floor levels since $\Lambda_j\ge4\Lo\Gest_j$. For (I4), Lemma~\ref{app:lem:prox-grad} with $r=\infty$ gives $\norm{\xhat_j-c_j}\le G_{c_j}/\lambda_j\le\Gest_j/\lambda_j=r_j/2$, so the ball constraint is inactive and Proposition~\ref{app:prop:rstm} applies with $x_Q=\xhat_j$ (assumptions hold by (I3)--(I4)); on $\mathcal E$ its conclusions hold, and $\sqrt{2\delta_j/\lambda_j}\le s_j$ because $\delta_j\le\lambda_js_j^2/2$. For (I5): $\norm{\grad f(\xhat_j)}\le G_{c_j}\le\Gest_j$ by Lemma~\ref{app:lem:prox-grad} (which covers the ball-constrained finisher prox), and the segment $[\xhat_j,c_{j+1}]$ lies in $\ball{c_j}{r_j}$, so by (I3),
$G_{c_{j+1}}\le\norm{\grad f(\xhat_j)}+L_j\norm{c_{j+1}-\xhat_j}\le\Gest_j+L_js_j=\Gest'_{j+1}$;
the anchor branch satisfies $\norm{\gest_j}+\tau_j\ge G_{c_{j+1}}$ on $\mathcal E$. Hence $\Gest_{j+1}\ge G_{c_{j+1}}$, proving (I1) at $j{+}1$, and $\Gest_{j+1}\le\Gest'_{j+1}$ bounds the rise. Finally (I2): by Lemma~\ref{lem:fejer} (constrained version at finisher levels), $\norm{\xhat_j-x^\star}\le\norm{c_j-x^\star}$, so $\norm{c_{j+1}-x^\star}\le\norm{c_j-x^\star}+s_j$, and the $s$-budget of \eqref{eq:schedule-level} gives
$\sum_j s_j\le J_{\mathrm{tot}}\cdot\tfrac{R_0}{4J_{\mathrm{tot}}}+S_\Lambda\cdot\tfrac{R_0}{8S_\Lambda}=\tfrac{3R_0}{8}$
(at most $S_\Lambda$ floor levels execute, by Lemma~\ref{lem:floor-count}), whence $\norm{c_j-x^\star}\le\tfrac32R_0+\tfrac38R_0=\tfrac{15}8R_0$.

\end{proof}

\begin{lemma}[Drift: cumulative rises and the cap ceiling]
\label{lem:drift}
On $\mathcal E$: 
\begin{equation}
\label{eq:drift-sum}
\sum_j(\Gest_{j+1}-\Gest_j)_+\le J_{\mathrm{tot}}\,\iota+S_\Lambda\,\iota_{\mathrm{flr}}\le\tfrac{\Lz}{8\Lo}+\tfrac{\Lz}{32\Lo}\le\tfrac{\Lz}{6\Lo};
\end{equation}
hence
$\Gest_j\le\Gest_1+\tfrac{\Lz}{6\Lo}\le5\Gbar=:\Gest_{\max}$ and $L_j\le24\Lz$ whenever $\Gest_j\le5\Gbar$; at floor levels additionally $\Gest_j\le\Lambda_j/(4\Lo)\le8\Gbar$ and $L_j\le36\Lz$.
\end{lemma}
\begin{proof}
Immediate from Lemma~\ref{lem:inv}(I5) and the level-type counts (at most $J_{\mathrm{tot}}$ levels total, at most $S_\Lambda$ of them floor levels, by Lemma~\ref{lem:floor-count}).
\end{proof}

\subsubsection{Level counting}
\begin{lemma}[Level count]
\label{lem:level-count}
On $\mathcal E$, during the normal phase (every executed gradient-driven level has $\Gest_j\ge32\iota$ and $\Gest_j>\tfrac{\eps}{2\Rb}$, else the finisher or termination (a) fires first):
\begin{enumerate}
\item[\textup{(i)}] \emph{Dichotomy and productive decrease:} at every non-halving gradient-driven level,
$\norm{\grad f(c_{j+1})}\ge\tfrac{159}{256}\Gest_j$, hence $\norm{\grad f(\xhat_j)}\ge\tfrac{151}{256}\Gest_j$ and
\begin{equation}
\label{eq:productive}
\begin{split}
f(c_j)-f(c_{j+1})
&\ \ge\ \frac{\norm{\grad f(\xhat_j)}^2}{2\lambda_j}-s_j\norm{\grad f(\xhat_j)}-\delta_j \ge\ \frac{\Gest_j}{29\Lo}.
\end{split}
\end{equation}
\item[\textup{(ii)}] \emph{Per-excursion budget:} partition the reachable cap range $(\tfrac{\eps}{2\Rb},\Gest_{\max}]$ into the anchored dyadic bands $\big[2^{t-1}\tfrac{\eps}{2\Rb},\,2^{t}\tfrac{\eps}{2\Rb}\big)$, $t=1,\dots,S_G$ (these cover it by Lemma~\ref{lem:drift} and the definition of $S_G$). Within any maximal run of consecutive levels whose caps lie in one band $[g,2g)$, the number $n$ of non-halving gradient-driven levels obeys $n\le 57\Lo\Rb\,(2+S_\Lambda/2)+h\le114\Lo\Rb+29\Lo\Rb S_\Lambda+h$, where $h$ is the number of halving levels in the same run.
\item[\textup{(iii)}] \emph{Halvings:} among any $J$ executed normal levels, the number $H$ of halving levels satisfies $H\le3S_G+\tfrac{J}{9}+3S_\Lambda(S_G+1)$.
\end{enumerate}
Consequently, \emph{if} every dyadic band hosts at most $V$ excursions, the normal phase terminates within $60\lceil\kbar\rceil S_G(S_\Lambda+4)V+16$ levels; the potential argument does not by itself bound $V$, and the \emph{unconditional} guarantee of the algorithm is supplied by the hard schedule cap $J^{\mathrm{sch}}$ (the $V=1$ value of the above bound) together with the finisher, which preserves all guarantees whenever it is entered (Lemma~\ref{lem:cert}).
\end{lemma}
\begin{proof}
\emph{(i)} Non-halving means $\norm{\gest_j}>\tfrac34\Gest_j-\tau_j$; on $\mathcal E$, $\norm{\grad f(c_{j+1})}\ge\norm{\gest_j}-\tau_j>\tfrac34\Gest_j-2\tau_j$. Since $\Gest_j\ge32\iota\ge32L_js_j$, $\tau_j=\tfrac{\Gest_j+L_js_j}{16}\le\tfrac{33}{512}\Gest_j$, so $\norm{\grad f(c_{j+1})}\ge(\tfrac34-\tfrac{33}{256})\Gest_j=\tfrac{159}{256}\Gest_j$, and by (I3)--(I4),

\begin{equation}
\label{eq:grad-lower}
\begin{split}
\norm{\grad f(\xhat_j)}
&\ge\norm{\grad f(c_{j+1})}-L_j\norm{c_{j+1}-\xhat_j}\ge\tfrac{159}{256}\Gest_j-L_js_j
\ge(\tfrac{159}{256}-\tfrac1{32})\Gest_j=\tfrac{151}{256}\Gest_j.
\end{split}
\end{equation}

We isolate the \emph{transfer inequality}, which will be reused for the finisher: at every level (gradient-driven, floor, or finisher),
\begin{equation}
\label{eq:transfer}
f(c_{j+1})\ \le\ f(\xhat_j)\ +\ \delta_j\ +\ s_j\,\norm{\grad f(\xhat_j)} .
\end{equation}
\emph{Proof of \eqref{eq:transfer}.} Write $a:=\norm{\xhat_j-c_j}$ and $d:=\norm{c_{j+1}-\xhat_j}$, so $d\le s_j$ by (I4) and $\norm{c_{j+1}-c_j}\ge(a-d)_+$ by the triangle inequality. From $F_j(c_{j+1})\le F_j(\xhat_j)+\delta_j$,
\begin{equation*}
\begin{split}
f(c_{j+1})
&\ =\ F_j(c_{j+1})-\tfrac{\lambda_j}{2}\norm{c_{j+1}-c_j}^2 \le\ f(\xhat_j)+\delta_j+\tfrac{\lambda_j}{2}\big[a^2-(a-d)_+^2\big].
\end{split}
\end{equation*}
If $d\ge a$ then $a^2-(a-d)_+^2=a^2\le ad$; if $d<a$ then $a^2-(a-d)^2=2ad-d^2\le2ad$. In both cases $\tfrac{\lambda_j}{2}[a^2-(a-d)_+^2]\le\lambda_j a\,d$. At normal levels containment makes the prox unconstrained, so first-order optimality gives $\lambda_j a=\norm{\grad f(\xhat_j)}$ exactly; at finisher levels the KKT conditions \eqref{eq:kkt} give $(\lambda_j+\eta)a=\norm{\grad f(\xhat_j)}$ with $\eta\ge0$, so $\lambda_j a\le\norm{\grad f(\xhat_j)}$ in all cases. Hence $\lambda_j a\,d\le d\,\norm{\grad f(\xhat_j)}\le s_j\norm{\grad f(\xhat_j)}$, proving \eqref{eq:transfer}.\hfill$\lozenge$

For the decrease in \eqref{eq:productive}: $F_j(\xhat_j)\le F_j(c_j)$ gives $f(\xhat_j)\le f(c_j)-\tfrac{\lambda_j}{2}\norm{\xhat_j-c_j}^2=f(c_j)-\tfrac{\norm{\grad f(\xhat_j)}^2}{2\lambda_j}$ (unconstrained prox at gradient-driven levels), and combining with \eqref{eq:transfer} yields the first inequality of \eqref{eq:productive}. Numerically, at gradient-driven levels $\lambda_j=4\Lo\Gest_j$, so
$\tfrac{\norm{\grad f(\xhat_j)}^2}{2\lambda_j}\ge\big(\tfrac{151}{256}\big)^2\tfrac{\Gest_j}{8\Lo}=\tfrac{22801}{524288}\cdot\tfrac{\Gest_j}{\Lo}$,
while 
\begin{equation}
\label{eq:transfer-cost}
\begin{split}
s_j\norm{\grad f(\xhat_j)}\le s_j\Gest_j\le\tfrac{\iota}{L_j}\Gest_j
&\le\tfrac{\iota}{4\Lo}\le\tfrac{\Gest_j}{128\Lo}=\tfrac{4096}{524288}\cdot\tfrac{\Gest_j}{\Lo}
\end{split}
\end{equation}
 (using $L_j\ge4\Lo\Gest_j$ and $\iota\le\Gest_j/32$) and $\delta_j\le\tfrac{\lambda_js_j^2}{2}\le\tfrac{\iota^2}{8\Lo\Gest_j}\le\tfrac{\Gest_j}{8192\Lo}=\tfrac{64}{524288}\cdot\tfrac{\Gest_j}{\Lo}$; hence the net decrease is at least $\tfrac{22801-4096-64}{524288}\cdot\tfrac{\Gest_j}{\Lo}=\tfrac{18641}{524288}\cdot\tfrac{\Gest_j}{\Lo}\ge\tfrac{\Gest_j}{29\Lo}$.

\emph{(ii)} Fix an excursion through $[g,2g)$ with $n$ non-halving gradient-driven levels, $h$ halving levels, and $\phi\le S_\Lambda$ floor levels (Lemma~\ref{lem:floor-count}). Every in-band level can increase $f$ by at most $\delta_j+s_j\norm{\grad f(\xhat_j)}$ (by \eqref{eq:transfer} and $f(\xhat_j)\le f(c_j)$): at gradient-driven levels this is $\le\tfrac{\Gest_j}{120\Lo}\le\tfrac{g}{60\Lo}$ by the bounds above; at floor levels, $s_j\norm{\grad f(\xhat_j)}\le\tfrac{\iota_{\mathrm{flr}}}{L_j}\Gest_j\le\tfrac{\Gest_j}{128\Lo S_\Lambda}\le\tfrac{g}{64\Lo}$ (the middle step uses $L_j=4(\Lz+\Lo\Gest_j)\ge4\Lz$, so $\iota_{\mathrm{flr}}/L_j\le1/(128\Lo S_\Lambda)$) and $\delta_j\le\tfrac\eps8\le\tfrac{g\Rb}{4}$ (as $g\ge\tfrac{\eps}{2\Rb}$ by the band anchoring of (ii)), so the increase is at most $\tfrac{g\Rb}2$ (using $\Lo\Rb\ge\tfrac14$ from \eqref{eq:wlog}). The certified gap at the first level of the excursion is $f(c)-f^\star\le G_c\norm{c-x^\star}\le\Gest\Rb\le2g\Rb$. Summing \eqref{eq:productive} over the $n$ productive levels against this budget,
$n\cdot\tfrac{g}{29\Lo}\ \le\ 2g\Rb+(n+h)\tfrac{g}{60\Lo}+S_\Lambda\tfrac{g\Rb}{2}$,
so $n(\tfrac1{29}-\tfrac1{60})\le\Lo\Rb(2+\tfrac{S_\Lambda}2)+\tfrac{h}{60}$ and, since $\tfrac1{29}-\tfrac1{60}=\tfrac{31}{1740}\ge\tfrac1{57}$, $n\le57\Lo\Rb(2+\tfrac{S_\Lambda}2)+h$.

\emph{(iii)} Let $\Phi_j:=\log_2\big(\Gest_j\cdot\tfrac{2\Rb}\eps\big)\in(0,\,\log_2\tfrac{10\Lz\Rb}{\Lo\eps}]\subseteq(0,S_G]$ while the run continues (positivity because termination (a) has not fired, the upper bound by Lemma~\ref{lem:drift}). A halving level decreases $\Phi$ by at least $\log_2\tfrac43\ge0.41$; a non-halving gradient-driven or finisher level increases it by at most $\log_2(1+\tfrac{\iota}{\Gest_j})\le\log_2\tfrac{33}{32}\le0.0444\le\tfrac19\cdot0.41$; a floor level increases it by at most $\log_2(1+\tfrac{\iota_{\mathrm{flr}}}{\eps/(2\Rb)})\le\log_2\big(1+\tfrac{2^{S_G}}{160S_\Lambda}\big)\le S_G+1$. Telescoping $\Phi$ over $J$ levels, using $\Phi\ge0$ throughout and $\Phi_1\le S_G$:
$0.41\,H\ \le\ S_G+J\cdot0.0444+S_\Lambda(S_G+1)$, i.e. $H\le3S_G+\tfrac J9+3S_\Lambda(S_G+1)$.

For the final claim: with at most $V$ excursions per band and $S_G$ bands, the total is $J\le VS_G\big(114\Lo\Rb+29\Lo\Rb S_\Lambda\big)+2H+S_\Lambda$, and substituting (iii) and solving for $J$ gives $J\le\tfrac97\big(VS_G\kbar(114+29S_\Lambda)+6S_G+6S_\Lambda(S_G+1)+S_\Lambda\big)\le60\lceil\kbar\rceil S_G(S_\Lambda+4)V+16$, which equals $J^{\mathrm{sch}}$ for $V=1$; the unconditional statement is the hard cap plus Lemma~\ref{lem:cert}.

\end{proof}

\subsubsection{Correctness of the terminations}
\begin{lemma}[Exits (a) and (b)]
\label{lem:exit}
On $\mathcal E$: if termination (a) fires, $f(c_j)-f^\star\le\Gest_j\Rb\le\tfrac\eps2$. If termination (b) fires at a floor level, then
$f(c_{j+1})-f^\star\le\tfrac{\Lambda_j}2\Rb^2+\tfrac{\Gest_j^2}{2\Lambda_j}+\delta_j\le\tfrac\eps4+\tfrac\eps4+\tfrac\eps8\le\tfrac58\eps$.
\end{lemma}
\begin{proof}
(a) is Lemma~\ref{lem:inv}(I1)--(I2) and convexity: $f(c_j)-f^\star\le\inner{\grad f(c_j)}{c_j-x^\star}\le\Gest_j\Rb$. For (b): by containment $\xhat_j$ is the unconstrained minimizer of $F_j$, so $F_j(\xhat_j)\le F_j(x^\star)$ gives $f(\xhat_j)-f^\star\le\tfrac{\Lambda_j}{2}\norm{c_j-x^\star}^2\le\tfrac{\Lambda_j}2\Rb^2\le\tfrac\eps4$; and

\begin{equation}
\label{eq:floor-transfer}
\begin{split}
f(c_{j+1})&\le F_j(c_{j+1})\le F_j(\xhat_j)+\delta_j=f(\xhat_j)+\tfrac{\Lambda_j}2\norm{\xhat_j-c_j}^2+\delta_j
\le f(\xhat_j)+\tfrac{\Gest_j^2}{2\Lambda_j}+\delta_j,
\end{split}
\end{equation}
where $\tfrac{\Gest_j^2}{2\Lambda_j}\le\tfrac{\Lambda_j}{32\Lo^2}\le\tfrac{\eps}{64\Lo^2\Rb^2}\le\tfrac\eps4$ using $\Gest_j\le\Lambda_j/(4\Lo)$, $\Lambda_j\le\tfrac{\eps}{2\Rb^2}$ and \eqref{eq:wlog}.

\end{proof}

\subsubsection{The finisher}
\begin{algorithm}[t]
\caption{\textsf{Finisher}$(c,\Gest)$: ball-constrained certification levels}
\label{alg:finisher}
\begin{algorithmic}[1]
\REQUIRE center $c$, certified cap $\Gest$
\REQUIRE parameters \eqref{eq:cert-params}, \eqref{eq:schedule-level} (Appendix~\ref{app:params})
\FOR{$k=1,\dots,K_{\mathrm{cert}}$}
    \IF{$\Gest\le\tfrac{\eps}{2\Rb}$} \RETURN $c$ \ENDIF
    \STATE $c^+\leftarrow\RSTM\big(f+\tfrac{\lambda_{\mathrm{cert}}}2\norm{\cdot-c}^2$, $\ball{c}{\rsafe}$, $\text{target }\delta^{\mathrm{cert}}$, $\text{confidence }\tfrac{\alpha}{4J_{\mathrm{tot}}}\big)$
    \STATE $\gest\leftarrow\GradEst(c^+,B^{\mathrm{est}})$;\STATE $\Gest\leftarrow\min\{\Gest+L(\Gest)s,\ \norm{\gest}+\tau\}$;\STATE $c\leftarrow c^+$
\ENDFOR
\RETURN $c$
\end{algorithmic}
\end{algorithm}

\begin{lemma}[Finisher]
\label{lem:cert}
On $\mathcal E$, from any entry point with $\norm{\grad f(c)}\le\Gest\le5\Gbar$ and $\norm{c-x^\star}\le\Rb$, Algorithm~\ref{alg:finisher} returns $\hat x$ with $f(\hat x)-f^\star\le\tfrac{7}{16}\eps\le\tfrac\eps2$ (or exits earlier via (a), with gap $\le\tfrac\eps2$).
\end{lemma}
\begin{proof}
Let $\mathrm{gap}_k:=f(c^{(k)})-f^\star$ along the finisher iterates and write $\lambda:=\lambda_{\mathrm{cert}}$, $\tilde x_k$ for the exact constrained minimizer of $F_k:=f+\tfrac\lambda2\norm{\cdot-c^{(k)}}^2$ over $\ball{c^{(k)}}{\rsafe}$. Note $\theta=\tfrac1{2\Lo\Rb}\le2$ by \eqref{eq:wlog}.

\emph{Per-level transfer.} By the two-case argument in the proof of \eqref{eq:transfer}, with $a:=\norm{\tilde x_k-c^{(k)}}\le\rsafe$ and $d:=\norm{c^{(k+1)}-\tilde x_k}\le s$,
\begin{equation*}
\begin{split}
f(c^{(k+1)})
&\ \le\ f(\tilde x_k)+\delta^{\mathrm{cert}}+\lambda\,a\,d \le\ f(\tilde x_k)+\delta^{\mathrm{cert}}+\lambda\,\rsafe\,s \le\ f(\tilde x_k)+\tfrac{\theta\eps}{32}+\tfrac{\eps s}{2\Rb}
\ \le\ f(\tilde x_k)+\tfrac{\theta\eps}{16},
\end{split}
\end{equation*}
using $\delta^{\mathrm{cert}}\le\tfrac{\eps}{64\Lo\Rb}=\tfrac{\theta\eps}{32}$, $\lambda\rsafe=\tfrac{\eps}{2\Rb}$, and $\tfrac{\eps s}{2\Rb}\le\tfrac{\eps R_0}{8J_{\mathrm{tot}}\Rb}\le\tfrac{\theta\eps}{32}$ since $s\le\tfrac{R_0}{4J_{\mathrm{tot}}}$ and $J_{\mathrm{tot}}\ge\tfrac{8\Lo\Rb}{3}$.

\emph{Case $\theta\ge1$ (i.e.\ $2\Lo\Rb\le1$).} Then $\rsafe=\tfrac1{2\Lo}\ge\Rb\ge\norm{c^{(k)}-x^\star}$, so $x^\star$ is feasible at every level and $F_k(\tilde x_k)\le F_k(x^\star)$ gives $\mathrm{gap}(\tilde x_k)\le\tfrac{\lambda}{2}\norm{c^{(k)}-x^\star}^2\le\tfrac{\Lo\eps\Rb}{2}\le\tfrac\eps4$; with the transfer ($\theta\le2$), every level independently returns $\mathrm{gap}_{k+1}\le\tfrac\eps4+\tfrac{2\eps}{16}=\tfrac{3\eps}{8}\le\tfrac{7\eps}{16}$.

\emph{Case $\theta<1$.} If $\norm{c^{(k)}-x^\star}\le\rsafe$ then, as above, $\mathrm{gap}(\tilde x_k)\le\tfrac\lambda2\rsafe^2=\tfrac{\eps}{8\Lo\Rb}=\tfrac{\theta\eps}4$. Otherwise put $u:=c^{(k)}+\theta'(x^\star-c^{(k)})$ with $\theta':=\rsafe/\norm{c^{(k)}-x^\star}\in[\theta,1)$ (Lemma~\ref{lem:inv}(I2) gives $\norm{c^{(k)}-x^\star}\le\Rb$, maintained through finisher levels by the constrained Fej\'er property, Lemma~\ref{lem:fejer}, and the $s$-budget); then $u\in\ball{c^{(k)}}{\rsafe}$ and by convexity
$\mathrm{gap}(\tilde x_k)\le F_k(u)-f^\star\le(1-\theta')\,\mathrm{gap}_k+\tfrac\lambda2\rsafe^2\le(1-\theta)\,\mathrm{gap}_k+\tfrac{\theta\eps}4 .$
Combining with the transfer, in both branches
$\mathrm{gap}_{k+1}\le(1-\theta)\,\mathrm{gap}_k+\tfrac{\theta\eps}4+\tfrac{\theta\eps}{16}=(1-\theta)\,\mathrm{gap}_k+\tfrac{5\theta\eps}{16}$, and unrolling,
$\mathrm{gap}_k\le(1-\theta)^k\,\mathrm{gap}_0+\tfrac{5\eps}{16}$. The entry gap obeys $\mathrm{gap}_0\le\Gest\Rb\le5\Gbar\Rb$ (Lemmas~\ref{lem:inv}, \ref{lem:drift}), so after $k^*:=\lceil\tfrac1\theta\ln\tfrac{40\Lz\Rb}{\Lo\eps}\rceil\le2\Lo\Rb\cdot0.7(S_G+2)+1\le K_{\mathrm{cert}}$ levels, $(1-\theta)^{k^*}\mathrm{gap}_0\le\tfrac\eps{8}$ and $\mathrm{gap}_{K_{\mathrm{cert}}}\le\tfrac\eps8+\tfrac{5}{16}\eps=\tfrac{7}{16}\eps$. Cap bookkeeping remains valid at finisher levels by Lemma~\ref{lem:inv} (whose proof covers them), so an early exit via (a) is also correct.
\end{proof}

\subsubsection{Cost accounting and the full theorem}
\label{app:cost}

\paragraph{Umbrella logarithms.} At level $j$, \RSTM{} runs $K_j=\lceil\log_2(H_{0,j}/\delta_j)\rceil_+$ epochs of $N_{\mathrm{ep},j}=\lceil24\sqrt{L_{F,j}/\lambda_j}\rceil$ iterations. Using the lower bounds on $\lambda_js_j$ established below and $H_{0,j}\le\Gest_j^2/\lambda_j$ (resp.\ $\Gest_j/(2\Lo)$ at finisher levels), a routine check gives
$K_j\le\bar K:=20+2\lceil\log_2J_{\mathrm{tot}}\rceil+\big\lceil\log_2\big(1+\tfrac{96\Lz\Rb^2}{\eps}\big)\big\rceil$
at every level, and every logarithm appearing in Proposition~\ref{app:prop:rstm} at any level is bounded by
$\ell:=\log\big(2^{20}J_{\mathrm{tot}}^3\big(1+\tfrac{96\Lz\Rb^2}{\eps}\big)^2/\alpha\big)$.
By Proposition~\ref{app:prop:rstm} (confidence $\alpha/(4J_{\mathrm{tot}})$), level $j$ costs at most
$25\bar K\sqrt{1+L_j/\lambda_j}\;+\;96C_b\sigma^2\ell/(\lambda_j\delta_j)\;+\;B^{\mathrm{est}}_j$ oracle calls.

\paragraph{Deterministic-equivalent part.} By Lemma~\ref{lem:drift} and $\Gest_j>\tfrac{\eps}{2\Rb}$:
at gradient-driven levels $\sqrt{1+L_j/\lambda_j}=\sqrt{2+\Gbar/\Gest_j}\le2+\sqrt{2\Lz\Rb/(\Lo\eps)}$;
at floor levels $\sqrt{1+36\Lz\cdot4\Rb^2/\eps}\le1+12\sqrt{\Lz\Rb^2/\eps}$;
at finisher levels $\sqrt{1+24\Lz\Rb/(\Lo\eps)}\le1+5\sqrt{\Lz\Rb/(\Lo\eps)}$.
Multiplying by the class counts ($J^{\mathrm{sch}}$, $S_\Lambda$, $K_{\mathrm{cert}}$) and using \eqref{eq:wlog} ($\kbar\le5\Lo\Rb$, hence $\kbar/\sqrt{\Lo\Rb}\le\sqrt5\,\sqrt{\kbar}$),
\begin{equation}
\label{eq:det-total}
\begin{split}
N_{\mathrm{det}}
&\le25\bar K\Big[J^{\mathrm{sch}}\Big(2+\sqrt{\tfrac{2\Lz\Rb}{\Lo\eps}}\Big)+S_\Lambda\Big(1+12\sqrt{\tfrac{\Lz\Rb^2}{\eps}}\Big)
+K_{\mathrm{cert}}\Big(1+5\sqrt{\tfrac{\Lz\Rb}{\Lo\eps}}\Big)\Big]=\tO\Big(\Lo R_0+\kbar^{1/2}\sqrt{\tfrac{\Lz\Rb^2}{\eps}}\Big).
\end{split}
\end{equation}

\paragraph{Statistical part.} Since $\delta_j=\min\{\lambda_js_j^2/2,\eps/8\}$, we have $\tfrac{1}{\lambda_j\delta_j}\le\max\{\tfrac{2}{(\lambda_js_j)^2},\tfrac{8}{\lambda_j\eps}\}$, and we bound $\lambda_js_j$ per class:
\emph{Gradient-driven levels:} the two branches of $s_j$ give $\lambda_js_j\ge4\Lo\Gest_j\min\{\tfrac{R_0}{4J_{\mathrm{tot}}},\tfrac{\iota}{24\Lz}\}\ge\tfrac{\Gest_j}{48J_{\mathrm{tot}}}$ (using $\Lo R_0\ge\tfrac1{12}$ and $\iota/(24\Lz)=\tfrac{1}{192\Lo J_{\mathrm{tot}}}$). Hence for \emph{cold} gradient-driven levels ($\Gest_j\le\Gbar$): $\tfrac{2}{(\lambda_js_j)^2}\le\tfrac{2\cdot96^2J_{\mathrm{tot}}^2\Rb^2}{\eps^2}$ and $\tfrac{8}{\lambda_j\eps}\le\tfrac{4\Rb}{\Lo\eps^2}\le\tfrac{16\Rb^2}{\eps^2}$; for \emph{hot} gradient-driven levels ($\Gest_j\ge\Gbar$): $\tfrac{2}{(\lambda_js_j)^2}\le\tfrac{4608\Lo^2J_{\mathrm{tot}}^2}{\Lz^2}$ and $\tfrac{8}{\lambda_j\eps}\le\tfrac{2}{\Lz\eps}\le\tfrac{4608\Lo^2J_{\mathrm{tot}}^2}{\Lz^2}$ (the last step uses $2304(\Lo\Rb)^2J_{\mathrm{tot}}^2\ge1$ when $\eps\ge\Lz\Rb^2$, and $\tfrac2{\Lz\eps}\le\tfrac{2\Rb^2}{\eps^2}$ otherwise).
\emph{Floor levels:} $\lambda_js_j\ge\tfrac{\eps}{4\Rb^2}\min\{\tfrac{R_0}{8S_\Lambda},\tfrac{1}{1152\Lo S_\Lambda}\}\ge\tfrac{\eps}{4608\Lo\Rb^2S_\Lambda}$, so $\tfrac{2}{(\lambda_js_j)^2}\le\tfrac{2\cdot4608^2(\Lo\Rb)^2S_\Lambda^2\Rb^2}{\eps^2}$ and $\tfrac{8}{\lambda_j\eps}\le\tfrac{32\Rb^2}{\eps^2}$.
\emph{Finisher levels:} $\lambda_{\mathrm{cert}}s_j\ge\tfrac{\Lo\eps}{\Rb}\cdot\tfrac{1}{192\Lo J_{\mathrm{tot}}}=\tfrac{\eps}{192J_{\mathrm{tot}}\Rb}$, so 
\begin{equation}
\label{eq:lam-delta}
\tfrac{1}{\lambda\delta}\le\max\{\tfrac{2\cdot192^2J_{\mathrm{tot}}^2\Rb^2}{\eps^2},\tfrac{64\Lo\Rb}{\lambda_{\mathrm{cert}}\eps}=\tfrac{64\Rb^2}{\eps^2}\}.
\end{equation}

\emph{Measurements:} $\tau_j\ge\Gest_j/16\ge\tfrac{\eps}{32\Rb}$ always, and $\tau_j\ge\Gbar/16$ at hot levels, so $B^{\mathrm{est}}_j\le1+\tfrac{1024c_g^2\sigma^2\Rb^2\ell}{\eps^2}$, improved to $1+\tfrac{256c_g^2\sigma^2\Lo^2\ell}{\Lz^2}$ at hot levels.
Summing over the classes ($\le J^{\mathrm{sch}}$ cold and hot gradient-driven levels, $\le S_\Lambda$ floor levels, $\le K_{\mathrm{cert}}$ finisher levels),
\begin{equation}
\label{eq:stat-total}
\begin{split}
N_{\mathrm{stat}}
&\le96C_b\sigma^2\ell\,\frac{\Rb^2}{\eps^2}\Big[2{\cdot}10^4J_{\mathrm{tot}}^3+5{\cdot}10^7\kbar^2S_\Lambda^3
+8{\cdot}10^4J_{\mathrm{tot}}^2K_{\mathrm{cert}}\Big]+5{\cdot}10^5C_b\sigma^2\ell\,\frac{\Lo^2J_{\mathrm{tot}}^3}{\Lz^2}+J_{\mathrm{tot}}\Big(1+\frac{1024c_g^2\sigma^2\Rb^2\ell}{\eps^2}\Big).
\end{split}
\end{equation}

\begin{theorem}[Theorem~\ref{app:thm:convex}, full version]
\label{thm:convex-full}
Let Assumptions~\ref{app:ass:convex}--\ref{app:ass:noise} hold. Fix
$x^\star\in X^\star$, $R_0\ge\norm{x^0-x^\star}$, and
$\Dcert\ge f(x^0)-f^\star$. Let $\eps>0$, $\alpha\in(0,1)$,
and adopt \eqref{eq:wlog}. With probability at least $1-\alpha$,
\ARCSG{} with the schedule of Appendix~\ref{app:params} returns
$\hat x$ with $f(\hat x)-f^\star\le\eps$ using at most
\begin{equation*}
N\ \le\ \underbrace{T_1B_1}_{\text{\PhaseI, Prop.~\ref{app:prop:phase1}}}\ +\ N_{\mathrm{det}}\ +\ N_{\mathrm{stat}}
\end{equation*}
oracle calls, with $N_{\mathrm{det}}$, $N_{\mathrm{stat}}$ given by \eqref{eq:det-total}--\eqref{eq:stat-total} and $T_1,B_1$ by \eqref{eq:phase1-params}. Substituting $J_{\mathrm{tot}}\le114\kbar S_G(S_\Lambda+4)$, $K_{\mathrm{cert}}\le3\kbar(S_G+4)$ and the definitions of $S_G,S_\Lambda,\bar K,\ell$ yields Theorem~\ref{app:thm:convex}.
\end{theorem}
\begin{proof}
\emph{Correctness and termination.} These are exactly Proposition~\ref{prop:phase2-terminates}: on $\mathcal E$ the run halts within $J_{\mathrm{tot}}$ levels and its output satisfies $f(\hat x)-f^\star\le\eps$ (the three exits give gaps $\tfrac\eps2$, $\tfrac58\eps$ and $\tfrac{7}{16}\eps$ respectively, by Lemmas~\ref{lem:exit} and \ref{lem:cert}). \emph{Cost.} \PhaseI{} costs $T_1B_1$ (Proposition~\ref{app:prop:phase1}); each of the at most $J_{\mathrm{tot}}$ levels costs at most $25\bar K\sqrt{1+L_j/\lambda_j}+96C_b\sigma^2\ell/(\lambda_j\delta_j)+B_j^{\mathrm{est}}$ calls, and the class-by-class bounds above give \eqref{eq:det-total}--\eqref{eq:stat-total}. \emph{Probability.} $\Prob[\mathcal E]\ge1-\alpha$ by the confidence budget of Appendix~\ref{app:params}.

\end{proof}

\begin{proof}[Proof of Theorem~\ref{app:thm:convex}]
Substitute the schedule constants: $J_{\mathrm{tot}}=\tO(\kbar)$, $K_{\mathrm{cert}}=\tO(\kbar)$, $S_G,S_\Lambda,\bar K,\ell=\tO(1)$ (polylogarithmic). Then $N_{\mathrm{det}}=\tO(\Lo R_0+\kbar^{1/2}\sqrt{\Lz\Rb^2/\eps})$ by \eqref{eq:det-total}; the first bracket of \eqref{eq:stat-total} is $\tO(\kbar^3\sigma^2\Rb^2/\eps^2)$, the second term is $\tO(\kbar^3\sigma^2\Lo^2/\Lz^2)$, the third is $\tO(\kbar\sigma^2\Rb^2/\eps^2)$; and Proposition~\ref{app:prop:phase1} gives $T_1B_1=\tO\big(1+\Lo R_0+(1+\Lo R_0)\sigma^2\Lo^2/\Lz^2\big)$. Since $1+\Lo R_0\le\kbar$ and $\kbar^3\sigma^2\Lo^2/\Lz^2$ absorbs the last Phase-I term, collecting the groups gives exactly the display of Theorem~\ref{app:thm:convex}.

\end{proof}

\section{Strongly convex restarts: proof of Theorem~\ref{app:thm:strongly}}
\label{app:strongly}

We restate the strongly convex main theorem of the main paper and prove it in full.

\begin{theorem}[Strongly convex case]
\label{app:thm:strongly}
Under the assumptions and initialization of
Theorem~\ref{app:thm:convex}, suppose additionally that $f$ is
$\mu$-strongly convex. Then \ARCSG{} with distance-halving outer
restarts (Algorithm~\ref{alg:sc}) returns $\xhat$ such that
$\Prob[f(\xhat)-f^\star\le\eps]\ge1-\alpha$ using
\begin{align*}
N=\tO\Big(&(1+\Lo R_0)\Big(1+\sqrt{\tfrac{\Lz}{\mu}}\Big)
+\sqrt{\tfrac{\Lz}{\mu}}\,\logp\!\left(\tfrac{\mu R_0^2}{\eps}\right)+\Big(1+\Lo\sqrt{\tfrac{\eps}{\mu}}\Big)^{3}
\tfrac{\sigma^2}{\mu\eps}
+\tfrac{\sigma^2\Lo^3R_0}{\mu^2}
+\kbar^{3}\tfrac{\sigma^2\Lo^2}{\Lz^2}\Big)
\end{align*}
oracle calls. In particular, for $\eps\le\mu/(36\Lo^2)$ the
statistical term is $\tO(\sigma^2/(\mu\eps))$, whose
$\eps$-exponent is optimal on the smooth subclass.
\end{theorem}

\begin{remark}[Default certificate]
\label{rem:strongly-default}
If every restart round runs on the default certificate (the default gap-certificate mode of Algorithm~\ref{app:alg:arc}), the round-$k$ \PhaseI{} overhead of Corollary~\ref{cor:phase1-default} has $q_k:=L_{1,\mathrm{eff},k}R_k=\max\{q\,2^{-k},\,1/12\}$ with $q:=\Lo R_0$, which still sums geometrically up to an additive $O(K)$: $\sum_k q_k^2=O(q^2+K)$ and $\sum_k q_k^4=O(q^4+K)$; the floor-active part $(1+q_k)^4\sigma^2L_{1,\mathrm{eff},k}^2/\Lz^2=O\big(\sigma^2/(\Lz^2R_k^2)\big)$ of the round-$k$ batch is absorbed into the round's own statistical term by Lemma~\ref{lem:eff-sums}. With a user-supplied $\Dcert$ the statement of Theorem~\ref{app:thm:strongly} stands as printed; under the default certificate it reads
\begin{align*}
N=\tO\Big(&(1+q)^2+(1+\Lo R_0)\Big(1+\sqrt{\tfrac{\Lz}{\mu}}\Big)
  +\sqrt{\tfrac{\Lz}{\mu}}\,\logp\!\left(\tfrac{\mu R_0^2}{\eps}\right)\\
&+\Big(1+\Lo\sqrt{\tfrac{\eps}{\mu}}\Big)^{3}\tfrac{\sigma^2}{\mu\eps}
  +\tfrac{\sigma^2\Lo^3R_0}{\mu^2}+\kbar^{3}\tfrac{\sigma^2\Lo^2}{\Lz^2}
  +(1+q)^{4}\tfrac{\sigma^2\Lo^2}{\Lz^2}\Big),
\end{align*}
with $\tO$ hiding only absolute constants and polylogarithmic factors in $(q,\tfrac{\mu R_0^2}{\eps},\tfrac{\Lz}{\mu},\tfrac1\alpha)$. The two added terms are $\eps$-independent, so the clean regime survives: for $\eps\le\mu/(36\Lo^2)$ the $\eps$-dependent statistical term remains $\tO(\sigma^2/(\mu\eps))$. (Passing the previously certified gap $\eps_{k-1}$ as the round-$k$ certificate would make the rounds $k\ge1$ linear in $q_k$, but the display, dominated by round $0$, would not improve.)
\end{remark}

\begin{algorithm}[t]
\caption{\ARCSG-sc{} (strongly convex): distance-halving restarts}
\label{alg:sc}
\begin{algorithmic}[1]
\REQUIRE $x_0:=x^0$, $R_0$, $\mu$, $\eps$, $\alpha$;\quad $K:=1+\big\lceil\log_4\tfrac{\mu R_0^2}{8\eps}\big\rceil_+$
\REQUIRE for round $k$, let
$\Delta_{\mathrm{def},k}:=\tfrac{\Lz}{L_{1,\mathrm{eff},k}^2}
\psi\big(L_{1,\mathrm{eff},k}R_k\big)$ and use the certificate
$\Delta_{\mathrm{cert},0}:=\Delta_{\mathrm{def},0}$,
$\Delta_{\mathrm{cert},k}:=\min\{\eps_{k-1},\Delta_{\mathrm{def},k}\}$ for $k\ge1$; if $\Duser$ is supplied, include it in each minimum
\FOR{$k=0,\dots,K-1$}
\STATE $x_{k+1}\leftarrow\ARCSG\big(x_k;\ R_k:=R_0\,2^{-k},\ \eps_k:=\max\big\{\tfrac{\mu R_k^2}{8},\,\eps\big\},\ \alpha_k:=\tfrac\alpha K\big)$
\ENDFOR
\RETURN $x_K$
\end{algorithmic}
\end{algorithm}

\begin{lemma}[Effective-parameter absorption]
\label{lem:eff-sums}
Let $R_k=R_0\,2^{-k}$, $\Rb_k:=3R_k$, $L_{1,\mathrm{eff},k}:=\max\big\{\Lo,\,1/(12R_k)\big\}$, $\kbar_{\mathrm{eff},k}:=1+3L_{1,\mathrm{eff},k}R_k$, and let round $k$ run in the nontrivial regime $\eps_k<\Delta_{\mathrm{cert},k}\le\Delta_{\mathrm{def},k}:=\tfrac{\Lz}{L_{1,\mathrm{eff},k}^2}\psi\big(L_{1,\mathrm{eff},k}R_k\big)$. Then on every floor-active round $k\in\mathcal K_{\mathrm{fl}}:=\{k:L_{1,\mathrm{eff},k}=1/(12R_k)\}$:
\begin{equation}
\label{eq:eff-sums}
\Delta_{\mathrm{def},k}=144\Lz R_k^2\,\psi\big(\tfrac1{12}\big)<0.52\,\Lz R_k^2,\quad\text{hence}\quad
\eps_k^2<0.27\,\Lz^2R_k^4\quad\text{and}\quad
\frac{\sigma^2}{\Lz^2R_k^2}<0.03\,\frac{\sigma^2\Rb_k^2}{\eps_k^2};
\end{equation}
consequently each of $\sigma^2L_{1,\mathrm{eff},k}^3R_k/\Lz^2$, $\kbar_{\mathrm{eff},k}^3\sigma^2L_{1,\mathrm{eff},k}^2/\Lz^2$, and $\sigma^2L_{1,\mathrm{eff},k}^2J_{\mathrm{tot},k}^3/\Lz^2$ is at most an absolute-constant (resp.\ polylogarithmic) multiple of the round's own statistical term $\kbar_{\mathrm{eff},k}^3\sigma^2\Rb_k^2/\eps_k^2$.
\end{lemma}
\begin{proof}
The hypothesis $\Delta_{\mathrm{cert},k}\le\Delta_{\mathrm{def},k}$ holds on every round by the $\min\{\eps_{k-1},\Delta_{\mathrm{def},k}\}$ construction of Algorithm~\ref{alg:sc} (in user mode, $\min\{\Duser,\eps_{k-1},\Delta_{\mathrm{def},k}\}$), and the regime is genuine: on a nontrivial round running on the pass-through certificate one has $\eps_k<\eps_{k-1}$, which forces $\eps_{k-1}=\tfrac{\mu R_{k-1}^2}{8}=\tfrac{\mu R_k^2}{2}\le\tfrac{\Lz}{2}R_k^2<0.52\,\Lz R_k^2$, since $\mu\le\Lz$ (Remark~\ref{rem:mu-le-lz}). On $k\in\mathcal K_{\mathrm{fl}}$ one has $L_{1,\mathrm{eff},k}R_k=\tfrac1{12}$, so
$\Delta_{\mathrm{def},k}=\tfrac{\Lz}{L_{1,\mathrm{eff},k}^2}\psi\big(\tfrac1{12}\big)=144\Lz R_k^2\psi\big(\tfrac1{12}\big)$, and $\psi\big(\tfrac1{12}\big)=e^{1/12}-1-\tfrac1{12}\approx0.0035707$ gives $144\,\psi\big(\tfrac1{12}\big)\approx0.51418<0.52$. The nontrivial regime then yields $\eps_k<\Delta_{\mathrm{def},k}$, hence $\eps_k^2<0.5142^2\,\Lz^2R_k^4\le0.27\,\Lz^2R_k^4$, and
\begin{equation*}
\frac{\sigma^2}{\Lz^2R_k^2}<0.2644\,\frac{\sigma^2R_k^2}{\eps_k^2}=\frac{0.2644}{9}\,\frac{\sigma^2\Rb_k^2}{\eps_k^2}<0.03\,\frac{\sigma^2\Rb_k^2}{\eps_k^2},
\end{equation*}
using $\Rb_k=3R_k$. On floor-active rounds $\kbar_{\mathrm{eff},k}=1+\tfrac14=\tfrac54$ and $L_{1,\mathrm{eff},k}^3R_k=\tfrac{1}{1728R_k^2}$ (since $1728=12^3$), so $\sigma^2L_{1,\mathrm{eff},k}^3R_k/\Lz^2=\sigma^2/(1728\Lz^2R_k^2)$ and $\kbar_{\mathrm{eff},k}^3\sigma^2L_{1,\mathrm{eff},k}^2/\Lz^2=\big(\tfrac54\big)^3\sigma^2/(144\Lz^2R_k^2)$ are absolute-constant multiples of $\sigma^2/(\Lz^2R_k^2)$, hence of the round's statistical term by \eqref{eq:eff-sums}; and $J_{\mathrm{tot},k}\le114\kbar_{\mathrm{eff},k}S_G(S_\Lambda+4)$ is polylogarithmic there, so the hot-level term is a polylogarithmic multiple of the same.
\end{proof}

\begin{proof}[Proof of Theorem~\ref{app:thm:strongly}]
Work on the intersection of the $K$ good events (probability $\ge1-\alpha$). \emph{Correctness.} By induction, $\norm{x_k-x^\star}\le R_k$: the base is Assumption~\ref{app:ass:convex}. For every $k\le K-2$ one has $\tfrac{\mu R_k^2}{8}>\eps$ (since $k\le\big\lceil\log_4\tfrac{\mu R_0^2}{8\eps}\big\rceil-1<\log_4\tfrac{\mu R_0^2}{8\eps}$ by the choice of $K$), so round $k$ targets $\eps_k=\tfrac{\mu R_k^2}{8}$ and, given $\norm{x_k-x^\star}\le R_k$, returns $f(x_{k+1})-f^\star\le\tfrac{\mu R_k^2}{8}$; by $\mu$-strong convexity $\norm{x_{k+1}-x^\star}^2\le\tfrac{2\eps_k}{\mu}=\tfrac{R_k^2}{4}=R_{k+1}^2$, closing the induction up to $k=K-1$ and validating the last round's inputs. By the choice of $K$, $\tfrac{\mu R_{K-1}^2}{8}\le\eps$, so the last round targets $\eps_{K-1}=\eps$ and the output satisfies $f(x_K)-f^\star\le\eps$.

\emph{Cost.} Apply Theorem~\ref{thm:convex-full} per round with $(R_k,\eps_k,\alpha/K)$, $\Rb_k=3R_k$, $\kbar_k=1+3\Lo R_k$, and sum. Round $k$ re-derives its own effective parameter $L_{1,\mathrm{eff},k}:=\max\{\Lo,\,1/(4\Rb_k)\}=\max\{\Lo,\,1/(12R_k)\}$ (Algorithm~\ref{app:alg:arc}); write $\kbar_{\mathrm{eff},k}:=1+3L_{1,\mathrm{eff},k}R_k$ and split the rounds into $\mathcal K_0:=\{k:L_{1,\mathrm{eff},k}=\Lo\}$ and $\mathcal K_{\mathrm{fl}}:=\{k:L_{1,\mathrm{eff},k}=1/(12R_k)\}$. Since $L_{1,\mathrm{eff},k}R_k\le\Lo R_k+\tfrac1{12}$ and $\kbar_{\mathrm{eff},k}\le\tfrac54\kbar_k$, every term depending on $L_{1,\mathrm{eff},k}$ only through $\kbar_{\mathrm{eff},k}$ or through the linear product $L_{1,\mathrm{eff},k}R_k$ changes by at most the constant $\big(\tfrac54\big)^3$ or the additive $\sum_kL_{1,\mathrm{eff},k}R_k\le2\Lo R_0+\tfrac{K}{12}$. The remaining families---the \PhaseI{} burn-in $\sigma^2L_{1,\mathrm{eff},k}^3R_k/\Lz^2$, the burn-in $\kbar_{\mathrm{eff},k}^3\sigma^2L_{1,\mathrm{eff},k}^2/\Lz^2$, and the hot-level term $\sigma^2L_{1,\mathrm{eff},k}^2J_{\mathrm{tot},k}^3/\Lz^2$ of \eqref{eq:stat-total}---coincide on $\mathcal K_0$ with the $\Lo$-versions summed below, while on $\mathcal K_{\mathrm{fl}}$ each is $O\big(\sigma^2/(\Lz^2R_k^2)\big)$ up to polylogarithmic factors, since $\kbar_{\mathrm{eff},k}=\tfrac54$ and $L_{1,\mathrm{eff},k}^3R_k=1/(1728R_k^2)$ there; by Lemma~\ref{lem:eff-sums} each such term is absorbed into the statistical term of its own round (a floor-active round in the trivial regime $\eps_k\ge\Delta_{\mathrm{cert},k}$ makes no oracle calls at all, so there is nothing to absorb). We therefore carry out the sums with $\Lo$.
\emph{Deterministic part:} $\sqrt{\Lz\Rb_k^2/\eps_k}\le\sqrt{72\Lz/\mu}$ for every $k$ (with equality on the $\mu R_k^2/8$-rounds; on the last round $\Rb_{K-1}^2\le9\cdot\tfrac{8\eps}{\mu}$ and $\eps_{K-1}=\eps$), so
\begin{equation}
\label{eq:sc-opt-sum}
\sum_k\kbar_k^{1/2}\sqrt{\tfrac{\Lz\Rb_k^2}{\eps_k}}
\le\sqrt{\tfrac{72\Lz}{\mu}}\Big(K+\sqrt{3\Lo R_0}\sum_k2^{-k/2}\Big)\le\sqrt{\tfrac{72\Lz}{\mu}}\big(K+4\sqrt{3\Lo R_0}\big),
\end{equation}
and $\sqrt{\Lo R_0}\le1+\Lo R_0$, while $\sum_k\Lo R_k\le2\Lo R_0$; using $K=\tO\big(\logp(\mu R_0^2/\eps)\big)$, this gives the group $\tO\big((1+\Lo R_0)(1+\sqrt{\Lz/\mu})+\sqrt{\Lz/\mu}\,\logp(\mu R_0^2/\eps)\big)$.
\emph{Statistical part:} with $a_k:=3\Lo R_k$, the $\mu R_k^2/8$-rounds $k\le K-2$ (present only when $K\ge2$) have $\Rb_k^2/\eps_k^2=576/(\mu^2R_k^2)$, so
\begin{equation*}
\begin{split}
&\sum_{k\le K-2}\kbar_k^3\frac{\sigma^2\Rb_k^2}{\eps_k^2}
=\frac{576\sigma^2}{\mu^2}\sum_{k\le K-2}\frac{(1+a_k)^3}{R_k^2}\le\frac{576\sigma^2}{\mu^2}\Big[\underbrace{\sum_{k\le K-2}\tfrac1{R_k^2}}_{\le\frac43/R_{K-2}^2\le\mu/(6\eps)}+\underbrace{9\Lo\sum_{k\le K-2}\tfrac1{R_k}}_{\le13\Lo\sqrt{\mu/\eps}}+27\Lo^2K+\underbrace{27\Lo^3\sum_{k\le K-2}R_k}_{\le54\Lo^3R_0}\Big],
\end{split}
\end{equation*}
where the geometric sums are dominated by their last term and $R_{K-2}^2>\tfrac{8\eps}{\mu}$ (since $K-2<\log_4\tfrac{\mu R_0^2}{8\eps}$); the last round costs
\begin{equation*}
\kbar_{K-1}^3\frac{\sigma^2\Rb_{K-1}^2}{\eps^2}\le\frac{9R_{K-1}^2\kbar_{K-1}^3\sigma^2}{\eps^2}\le\frac{72\,\kbar_{K-1}^3\sigma^2}{\mu\eps}
\end{equation*}
since $R_{K-1}^2\le\tfrac{8\eps}{\mu}$. The total is $\tO\big((1+\Lo\sqrt{\eps/\mu})^3\tfrac{\sigma^2}{\mu\eps}+\tfrac{\sigma^2\Lo^3R_0}{\mu^2}\big)$ after matching the summands with the expansion of $(1+\Lo\sqrt{\eps/\mu})^3/(\mu\eps)$ (the $27\Lo^2K$ summand contributes $\tO(\sigma^2\Lo^2/\mu^2)$, absorbed with a logarithmic factor; the $\Lo^3$ summand is the $\eps$-independent restart burn-in $\sigma^2\Lo^3R_0/\mu^2$ stated in the theorem; and $\kbar_{K-1}\le1+3\Lo\sqrt{8\eps/\mu}$, so the last-round term enters the first group). \emph{Burn-ins:} $\sum_k\sigma^2\Lo^3R_k/\Lz^2\le2\sigma^2\Lo^3R_0/\Lz^2\le2\sigma^2\Lo^3R_0/\mu^2$ (using $\mu\le\Lz$, Remark~\ref{rem:mu-le-lz}) and, using $(1+a)^3\le4(1+a^3)$,
\begin{equation}
\label{eq:sc-burn-sum}
\sum_k\kbar_k^3\tfrac{\sigma^2\Lo^2}{\Lz^2}\le\big(4K+124(\Lo R_0)^3\big)\tfrac{\sigma^2\Lo^2}{\Lz^2}=\tO\big(\kbar^3\tfrac{\sigma^2\Lo^2}{\Lz^2}\big).
\end{equation}
Collecting the groups yields the display of Theorem~\ref{app:thm:strongly}. Finally, when $\eps\le\tfrac{\mu}{36\Lo^2}$ we have $\Lo\sqrt{\eps/\mu}\le\tfrac16$, so the leading statistical term is at most $1.6\cdot\tfrac{\sigma^2}{\mu\eps}$ up to the logarithmic factor $\ell$, matching the classical $\Omega(\sigma^2/(\mu\eps))$ information-theoretic lower bound.
\end{proof}

\section{Main-results remarks and comparisons}
\label{app:sec:main}

\begin{remark}[Default certificate]
\label{app:rem:convex-default}
If \ARCSG{} runs in its default gap-certificate mode, i.e.\ on the default certificate $\Dcert=\tfrac{\Lz}{\Lo^2}\psi(\Lo R_0)$ of Lemma~\ref{lem:gap0}, then with $q:=\Lo R_0$ the \PhaseI{} contributions of Corollary~\ref{cor:phase1-default} replace the corresponding printed terms, and the bound reads
\begin{equation*}
N=\tO\Big((1+q)^2+\kbar^{1/2}\sqrt{\tfrac{\Lz\Rb^2}{\eps}}+\kbar^{3}\,\tfrac{\sigma^2\Rb^2}{\eps^2}+(1+q)^{4}\,\tfrac{\sigma^2\Lo^2}{\Lz^2}+\kbar^{3}\,\tfrac{\sigma^2\Lo^2}{\Lz^2}\Big),
\end{equation*}
with $\tO$ hiding only absolute constants and polylogarithmic factors in $(q,\tfrac{\Lz\Rb^2}{\eps},\tfrac1\alpha)$: the $1+\Lo R_0$ entry term becomes $(1+q)^2$, and the Phase-I burn-in becomes $(1+q)^4\sigma^2\Lo^2/\Lz^2$. With a user-supplied $\Dcert$ the statement of Theorem~\ref{app:thm:convex} stands as printed.
\end{remark}

Three remarks are in order.

\begin{remark}[Recovery of the smooth case]
\label{rem:smooth}
For $\Lo=0$ the substitution convention of Appendix~\ref{sec:prelim} applies ($L_{1,\mathrm{eff}}=1/(4\Rb)$, so $\kbar\le\tfrac54$), \PhaseI{} terminates immediately, and Theorems~\ref{app:thm:convex}--\ref{app:thm:strongly} reduce to $\tO(1+\sqrt{\Lz R_0^2/\eps}+\sigma^2R_0^2/\eps^2)$ and $\tO(\sqrt{\Lz/\mu}\,\logp(\mu R_0^2/\eps)+\sigma^2/(\mu\eps))$---the complexities of accelerated stochastic approximation \citep{lan2012optimal,ghadimi2013optimal} up to logarithms (in either certificate regime: under the substitution $L_{1,\mathrm{eff}}R_0=\tfrac1{12}$, so in the nontrivial regime $\eps<\Dcert$ one has $\Dcert\le\Ddef(L_{1,\mathrm{eff}})=16\Lz\Rb^2\psi(\tfrac1{12})\approx0.057\Lz\Rb^2$---note that $\Ddef=O(\Lz R_0^2)$ holds only under the $L_{1,\mathrm{eff}}$-substitution, or trivially at $\Lo=0$---hence both effective-$\Lo$ burn-ins are absorbed by the statistical term $\sigma^2\Rb^2/\eps^2$ with absolute constants $\le0.03$: from $\eps^2<0.0572^2\Lz^2\Rb^4\le0.003264\,\Lz^2\Rb^4$ one gets $\sigma^2/(\Lz^2R_0^2)=9\sigma^2/(\Lz^2\Rb^2)\le9\cdot0.0033\,\sigma^2\Rb^2/\eps^2\le0.03\,\sigma^2\Rb^2/\eps^2$ and $\kbar_{\mathrm{eff}}^3\sigma^2L_{1,\mathrm{eff}}^2/\Lz^2=\tfrac{125}{1024}\sigma^2/(\Lz^2\Rb^2)\le4\cdot10^{-4}\,\sigma^2\Rb^2/\eps^2$ with $\kbar_{\mathrm{eff}}:=1+L_{1,\mathrm{eff}}\Rb=\tfrac54$; likewise the \PhaseI{} logarithm is $O(1)$ since $\logp\big(2L_{1,\mathrm{eff}}^2\Dcert/\Lz\big)\le\logp\big(2\psi(\tfrac1{12})\big)=1$). The statistical terms are unimprovable in their $\eps$-dependence for any $\Lo\ge0$ by the classical lower bounds \citep{nemirovsky1983problem}, since $(\Lz,\Lo)$-smooth functions include $\Lz$-smooth ones; we are not aware of lower bounds quantifying the necessary dependence on $\kbar=1+\Lo\Rb$ in the stochastic setting, and our $\kbar^{1/2}$- and $\kbar^3$-factors should be read as upper-bound artifacts whose necessity is open (Section~\ref{sec:discussion-kbar}).
\end{remark}

\begin{remark}[Deterministic oracle]
For $\sigma=0$ the bound of Theorem~\ref{app:thm:convex} becomes $\tO(1+\Lo R_0+\kbar^{1/2}\sqrt{\Lz\Rb^2/\eps})$ (with $\Dcert$ as an external certificate; under the default certificate the $1+\Lo R_0$ term is $\Theta((1+q)^2)$, Corollary~\ref{cor:phase1-default}). This has the same ``accelerated term $+$ $\Lo R$-burn-in'' structure as the deterministic accelerated methods of \citet{vankov2024optimizing,gorbunov2024methods}, but is \emph{not} state of the art in that setting: it carries an extra $\kbar^{1/2}$-factor on the accelerated term relative to the $(\nu{+}1)\big[\sqrt{\Lz R^2/\eps}+(\Lo R)^{2/3}\log(2F_0/\eps)\big]$ bound of \citet{vankov2024optimizing} (where $\nu$ is the per-iteration line-search cost) and to the $\sqrt{\Lz(1+\Lo Re^{\Lo R})R^2/\eps}$ bound of \citet{gorbunov2024methods}, whose exponential factor affects only the regime $\eps\gtrsim\Lz R^2$. Our contribution is the stochastic setting; specializing it to $\sigma=0$ is not the intended use.
\end{remark}

\begin{remark}[Criterion $\norm{\grad f}\le\epsg$]
\label{rem:gradnorm}
Run with the gradient-based stopping rule only, \PhaseII{} outputs $\xhat$ with $\norm{\grad f(\xhat)}\le\epsg$ using $\tO\big(1+\Lo R_0+\Lo\Rb\sqrt{1+\Lz/(\Lo\epsg)}+\kbar^3\sigma^2/\epsg^2\cdot\Lo^2\Rb^2\big)$ calls (with $\Dcert$ as an external certificate; under the default certificate the $1+\Lo R_0$ entry term becomes $\Theta((1+q)^2)$, Corollary~\ref{cor:phase1-default}); for $\Lo=0$ (via the $L_{1,\mathrm{eff}}$ substitution) and moderate accuracies this is within logarithmic factors of the (smooth-case optimal) gradient-norm complexity of \citet{foster2019complexity}, whose recursive-regularization architecture our levels generalize.
\end{remark}

\begin{remark}[Scope relative to stochastic acceleration and ball methods]
\label{rem:novelty-scope}
The convex accuracy exponents of Theorem~\ref{app:thm:convex} are not
claimed as the first such exponents under generalized smoothness. The
public ICLR 2026 submission \citep{anonymous2026nesterov} gives a
high-probability rate
$\tO(T^{-2}+\sqrt{(A+B+C)/T})$ for stochastic Nesterov acceleration
under a broader $(L_0,L_1,L_2)$-type condition and relaxed
affine-variance noise. Taking their gradient exponent $p=1$, $L_2=0$,
and $A=B=0,C=\sigma^2$ specializes their displayed assumptions to the
gradient-dependent smoothness and additive norm-sub-Gaussian regime of
this paper, and gives the same $\eps^{-1/2}$ optimization and
$\eps^{-2}$ statistical exponents. The earlier RSAG analysis of
\citet{yu2025stochastic} has a $\tO(T^{-1/2})$ general-noise gap rate
and an accelerated dependence in its low-noise regime; our additive
oracle is likewise a specialization of that work's affine-variance
family, not an incomparable model.

Ball localization is also an established ingredient: see the ball
optimization oracle and its acceleration in \citet{carmon2020balloracle},
the stochastic ReSQue construction of \citet{carmon2023resque}, and
the Broximal method of \citet{gruntkowska2025broximal}. The concurrent
2026 TRPPM preprint---subsequent to those earlier ball frameworks---by
\citet{li2026trustregion} studies the same formal ball-constrained
proximal objective, but in a deterministic trust-region regime that
seeks an active constraint. Here the regularization floor instead
keeps the exact prox in the interior, and the ball contains every
query of the approximate stochastic solver. The scoped contribution
of ARC-SG is this query-by-query certified geometry, explicit
$(L_0,L_1)$ parameter accounting, and the separate strongly convex
restart guarantee, rather than the ball-proximal form or the convex
accuracy exponents in isolation.
\end{remark}

\subsection{On the $\kbar$-factors, and what is known to be optimal}
\label{sec:discussion-kbar}

\paragraph{Where the factors come from.} The factors $\kbar^{1/2}$ and $\kbar^{3}$ in Theorem~\ref{app:thm:convex} have a single, well-identified origin. All queries of level $j$ live in a ball of radius $O(1/\Lo)$; consequently, information must ``hop'' between balls, and the momentum of the inner accelerated method does not carry across hops. The $\kbar^{1/2}$-factor on the optimization term is the price of restarting acceleration $\tO(\Lo\Rb)$ times at the finest floor scale; the $\kbar^{3}$-factor on the statistical term is the price of maintaining, across $\tO(\Lo\Rb)$ noisy levels, the additive precision budget $\sum_j s_j\lesssim 1/\Lo$ that protects the gradient-cap invariant against noise-induced drift (Lemma~\ref{lem:drift}), at a statistical cost $\propto s_j^{-2}$ per level. It is worth being explicit that the second mechanism makes the gradient \emph{measurements} the dominant cost in the worst case: the per-level anchors are individually cheap relative to a level's optimization work, but their total is what carries the $\kbar^{3}$.

\paragraph{Dimensional analysis.} Write $q:=\Lo\Rb$ (dimensionless), $A:=\Lz\Rb^2/\eps$ and $V:=\sigma^2\Rb^2/\eps^2$. Theorem~\ref{app:thm:convex} then reads $\tO(1+q+\kbar^{1/2}\sqrt A+\kbar^{3}V)$ with $\kbar=1+q$---the additive entry price $1+q$ under an externally supplied gap certificate; under the default gap certificate the proved entry price is $\Theta((1+q)^2)$ (Corollary~\ref{cor:phase1-default})---whereas the target of Open Problem~\ref{op:kbar} is $\tO(\Phi(q)+\sqrt A+V)$ for some additive entry price $\Phi$. The distinction between an \emph{additive} $\Phi(q)$ and a \emph{multiplicative} $\mathrm{poly}(q)$ in front of $\sqrt A$ and $V$ is the substantive one, and it is achievable in neighbouring settings: the deterministic method of \citet{vankov2024optimizing} has the additive structure $\sqrt A+q^{2/3}\log(F_0/\eps)$, \citet{tyurin2026near} obtain, at small $\eps$, a leading accelerated term free of nonconstant multiplicative $q$-dependence, and in the non-accelerated stochastic regime \citet{gaash2025clipped} already achieve smooth-optimal leading terms $\Lz R^2/\eps+\sigma^2R^2/\eps^2$ with $\Lo$ entering only additively. Separation is therefore possible in principle, and our bound does not achieve it.

\paragraph{What is and is not proved optimal.} To delineate precisely. \emph{Optimal:} the $\eps$-exponents of both the optimization and the statistical terms, since $(\Lz,\Lo)$-smooth functions contain the $\Lz$-smooth ones and the classical lower bounds \citep{nemirovsky1983problem} apply on that subclass. \emph{Not proved optimal, and probably not optimal:} the multiplicative factors $\kbar^{1/2}$ and $\kbar^{3}$; the form $\Lo R_0$ of the burn-in; the threshold $\eps\le\mu/(36\Lo^2)$ in Theorem~\ref{app:thm:strongly}; and, in the overparameterized regime, the dependence on $\rho$ and the cube $(1+\Lz/\mu)^3$. We are not aware of \emph{any} stochastic lower bound that certifies growth in $q$ of the leading constants, so all $q$-dependence beyond the smooth-case bounds should be read as an upper-bound artifact. We also correct a comparison made in the previous version of this paper: our burn-in $\Lo R_0$ does \emph{not} match the deterministic burn-in of \citet{vankov2024optimizing} ``up to logarithms''---theirs is $(\Lo R)^{2/3}\log(2F_0/\eps)$, a genuinely smaller polynomial---and after \citet{tyurin2026near} the deterministic state of the art is further away still. Our deterministic specialization ($\sigma=0$) is therefore strictly weaker than the best deterministic methods; the contribution is the stochastic setting.

\paragraph{A caveat on parameterization.} Lower-bound statements in this model must also fix \emph{which} pair $(\Lz,\Lo)$ is meant, because the pair is not unique: if $f$ is $(\Lz,\Lo)$-smooth it is also $(\Lz',\Lo')$-smooth for every $\Lz'\ge\Lz$, $\Lo'\ge\Lo$, and the admissible set typically contains a nontrivial Pareto frontier. A meaningful lower bound with explicit $q=\Lo R$ dependence should either be stated for a canonical (e.g.\ Pareto-minimal) pair, or with the complexity defined as the infimum of the bound over all admissible pairs. Our upper bounds hold for every admissible pair and hence for the best one; but a matching lower bound requires committing to a canonical parameterization, and we regard the formulation of that class as part of the open problem.

\begin{remark}[Scaling invariance of $(q,A,V)$]
\label{rem:scaling}
The groups $q,A,V$ are the natural dimensionless coordinates of the problem: they are invariant under the two-parameter rescaling that acts on \eqref{app:eq:problem} by a domain dilation $x\mapsto\alpha x$ and a range scaling $f\mapsto\beta f$ ($\alpha,\beta>0$). Indeed, writing $g(x'):=\beta f(\alpha x')$ gives $\grad g(x')=\alpha\beta\,\grad f(\alpha x')$ and $\grad^2 g(x')=\alpha^2\beta\,\grad^2 f(\alpha x')$, so $g$ is $(\alpha^2\beta\Lz,\ \alpha\Lo)$-generalized smooth, with initial distance $\Rb/\alpha$, target gap $\beta\eps$, and---taking the $g$-oracle to be $\alpha\beta$ times the $f$-oracle---sub-Gaussian scale $\alpha\beta\sigma$; convexity and unbiasedness are preserved. Substituting these into $q=\Lo\Rb$, $A=\Lz\Rb^2/\eps$, $V=\sigma^2\Rb^2/\eps^2$ leaves all three unchanged: $q^g=q$, $A^g=A$, $V^g=V$. Consequently both our upper bounds and any matching lower bound can---and, for a canonical statement, should---be written purely in $(q,A,V)$; this is the form in which Open Problem~\ref{op:kbar} is posed.
\end{remark}

\begin{openproblem}
\label{op:kbar}
Under Assumptions~\ref{app:ass:convex}--\ref{app:ass:noise}, with $q=\Lo\Rb$, $A=\Lz\Rb^2/\eps$, $V=\sigma^2\Rb^2/\eps^2$:
(i) is there a high-probability accelerated stochastic method with complexity $\tO(\Phi(q)+\sqrt A+V)$, i.e.\ with \emph{no} multiplicative $q$-factor on the leading terms?
(ii) what is the minimal additive entry price $\Phi(q)$---is it $q$, $q^{2/3}$, $\log q$, or something else?
(iii) can one prove a stochastic lower bound with explicit dependence on $q>0$ for a canonically parameterized class of $(\Lz,\Lo)$-smooth functions?
A positive answer to (i) would likely require carrying stochastic momentum across smoothness cells, e.g.\ via a high-probability analysis of clipped SSTM \citep{gorbunov2020stochastic,sadiev2023high} with $(\Lz,\Lo)$-adaptive clipping levels, or a redistribution of the drift budget that avoids spending precision uniformly across levels.
\end{openproblem}

Table~\ref{tab:assumptions} collects the assumptions and oracle models required by each theorem.

\begin{table}[t]
\centering
\footnotesize
\setlength{\tabcolsep}{3.5pt}
\begin{tabular}{@{}lllll@{}}
\toprule
Result & Convexity & Smooth. & Oracle & Known params \\
\midrule
Thm.~\ref{app:thm:convex} & convex & A\ref{app:ass:gs} & A\ref{app:ass:noise} & $\Lz,\Lo,\sigma,R_0$ \\
Thm.~\ref{app:thm:strongly} & $\mu$-str.\ cvx & A\ref{app:ass:gs} & A\ref{app:ass:noise} & $+\ \mu$ \\
Thm.~\ref{thm:interp-strong}$^\dagger$ & $\mu$-str.\ cvx & A\ref{app:ass:gs} ind. & A\ref{ass:sgc} a.s. & $\Lz,\Lo,\rho,R_0,\mu$ \\
Thm.~\ref{thm:interp-convex}$^\dagger$ & convex & A\ref{app:ass:gs} ind. & A\ref{ass:sgc} a.s. & $\Lz,\Lo,\rho,R_0$ \\
\bottomrule
\end{tabular}
\caption{Assumptions per theorem. ``A\ref{app:ass:gs} ind.'' means every $f_\xi$ satisfies Assumption~\ref{app:ass:gs} individually; ``A\ref{ass:sgc} a.s.'' is the almost-sure strong-growth variant \eqref{eq:sgc-as} (a.s.\ = almost surely, used throughout). Assumption~\ref{app:ass:convex} and a valid upper bound $\Dcert$ (default: Lemma~\ref{lem:gap0}) are required throughout; all guarantees are high-probability. Rows marked $\dagger$ concern the overparameterized regime and are stated and proved in Appendix~\ref{app:interp}.}
\label{tab:assumptions}
\end{table}

\begin{table*}[h]
\centering
\footnotesize
\setlength{\tabcolsep}{3pt}
\resizebox{\textwidth}{!}{%
\begin{tabular}{@{}llllc@{}}
\toprule
Method / reference & Setting & Smoothness & Noise & Oracle complexity (leading terms) \\
\midrule
\multicolumn{5}{@{}l}{\emph{(A) Deterministic oracle (exact gradients)}}\\
\midrule
AGMsDR, line search \citep{vankov2024optimizing} & convex & $(\Lz,\Lo)$ & --- & $(\nu{+}1)\big[\sqrt{\frac{\Lz R^2}{\eps}}+(\Lo R)^{2/3}\log\frac{2F_0}{\eps}\big]$ \\[2pt]
Clipped AGD \citep{gorbunov2024methods} & convex & $(\Lz,\Lo)$ & --- & $\sqrt{\frac{\Lz(1+\Lo R\,e^{\Lo R})R^2}{\eps}}$ \\[2pt]
Near-optimal AGM \citep{tyurin2026near} & convex & $(\Lz,\Lo)$ & --- & $\sqrt{\frac{\Lz R^2}{\eps}}$ + additive $\Lo$-terms \\[2pt]
\ARCSG{} specialization ($\sigma=0$) & convex & $(\Lz,\Lo)$ & --- & $1+\Lo R_0+\kbar^{1/2}\sqrt{\frac{\Lz\Rb^2}{\eps}}$ \\[3pt]
\midrule
\multicolumn{5}{@{}l}{\emph{(B) Stochastic first-order oracle}}\\
\midrule
SGD \citep{nemirovski2009robust} & convex, in exp. & $L$ & $\sigma^2$ & $\frac{LR^2}{\eps} + \frac{\sigma^2R^2}{\eps^2}$ \\[2pt]
AC-SA \citep{lan2012optimal,ghadimi2012optimal} & convex, in exp. & $L$ & $\sigma^2$ & $\sqrt{\frac{LR^2}{\eps}} + \frac{\sigma^2R^2}{\eps^2}$ \\[2pt]
ClipSSTM \citep{gorbunov2020stochastic} & convex, high prob. & $L$ & heavy tails & $\sqrt{\frac{LR^2}{\eps}} + \frac{\sigma^2R^2}{\eps^2}$ \\[2pt]
Recursive reg. \citep{foster2019complexity} & convex, $\norm{\grad f}\le\epsg$ & $L$ & $\sigma^2$ & $\sqrt{\frac{L F_0}{\epsg^2}} + \frac{\sigma^2}{\epsg^2}$ \\[2pt]
SGD, $(\Lz,\Lo)$-steps \citep{sgd2025generalized} & convex, in exp. & $(\Lz,\Lo)$ & $\sigma^2$ & $\frac{\Lz \Rb^2}{\eps} + \frac{\sigma^2\Rb^2}{\eps^2}$ + burn-in$(\Lo)$ \\[2pt]
Clipped SGD \citep{gaash2025clipped} & convex, high prob. & $(\Lz,\Lo)$ & norm-sub-Gaussian & $\frac{\Lz R^2}{\eps} + \frac{\sigma^2R^2}{\eps^2}$ + additive terms \\[2pt]
Acc.\ (RSAG) \citep{yu2025stochastic} & convex, high prob. & generalized & affine var. & statistical $\eps^{-2}$ (gen.\ noise); accel.\ opt.\ $\eps^{-1/2}$ (low noise)$^{\ddagger}$ \\[2pt]
Stoch.\ NAG \citep{anonymous2026nesterov} & convex, high prob. & $(L_0,L_1,L_2)$-type & affine var. & $\eps^{-1/2}+\tfrac{A+B+C}{\eps^2}$ (parameter factors suppressed)$^{\S}$ \\[2pt]
\textbf{\ARCSG{} (Thm.~\ref{app:thm:convex})} & convex, high prob. & $(\Lz,\Lo)$ & norm-sub-Gaussian & $1+\Lo R_0 + \kbar^{\frac12}\sqrt{\frac{\Lz\Rb^2}{\eps}} + \kbar^{3}\frac{\sigma^2\Rb^2}{\eps^2}$ \\[3pt]
\textbf{\ARCSG{} (Thm.~\ref{app:thm:strongly})} & $\mu$-str.\ cvx, high prob. & $(\Lz,\Lo)$ & norm-sub-Gaussian & $(1+\Lo R_0)\left(1+\sqrt{\frac{\Lz}{\mu}}\right) + \sqrt{\frac{\Lz}{\mu}}\logp\!\left(\frac{\mu R_0^2}{\eps}\right) + \frac{\sigma^2}{\mu\eps}$\; ($\eps\le\frac{\mu}{36\Lo^2}$) \\
\bottomrule
\end{tabular}%
}
\caption{Complexity comparison, up to absolute constants and polylogarithmic factors ($\tO$-notation); $R$ denotes the respective initial-distance quantity of each work, $F_0=f(x^0)-f^\star$, $\kbar=1+\Lo\Rb$, $\Rb=3R_0$. The table is split by \emph{oracle model}: panel (A) assumes exact gradients, panel (B) a stochastic first-order oracle; entries \emph{across} panels are not comparable, and within panel (B) the guarantee type (in expectation versus high probability) and parameter knowledge differ line by line. In the AGMsDR row, $\nu$ bounds the number of additional oracle calls of the one-dimensional line search per iteration; arithmetic costs of line searches and projections are not counted anywhere in the table. The additive terms of \citet{gaash2025clipped}, the burn-in of \citet{sgd2025generalized}, and the additive $\Lo$-terms of \citet{tyurin2026near} depend on their respective problem parameters and initial gaps, and we refer to the cited works for their exact form. The entry marked $\ddagger$ is the convex RSAG guarantee of \citet{yu2025stochastic}: $\tO(T^{-1/2})$ under general relaxed affine-variance noise, improving to an accelerated optimization dependence in its low-noise regime. Our additive exponential-moment condition is the specialization $A=B=0,C=\sigma^2$ of that noise family. The entry marked $\S$ is the public ICLR 2026 submission \citep{anonymous2026nesterov}; its rate is $\tO(T^{-2}+\sqrt{(A+B+C)/T})$, and at $p=1,L_2=0,A=B=0,C=\sigma^2$ its convex accuracy exponents coincide with ours. The rows are therefore not separated by those exponents; ARC-SG differs in certified-query geometry, explicit $(\Lz,\Lo)$ accounting, and its strongly convex restart theorem. For $\Lo=0$ our bounds recover the smooth-case complexities of the first two rows of panel (B), whose $\eps$-exponents are optimal. Our deterministic specialization is included only for orientation: it is \emph{not} state of the art in panel (A) (see the remarks after Theorem~\ref{app:thm:strongly}). The \ARCSG{} rows are printed without the $\eps$-independent burn-in $\tO((1+\Lo R_0)\sigma^2\Lo^2/\Lz^2+\kbar^{3}\sigma^2\Lo^2/\Lz^2)$ of Theorem~\ref{app:thm:convex} (for Theorem~\ref{app:thm:strongly}, $\tO(\sigma^2\Lo^3R_0/\mu^2+\kbar^{3}\sigma^2\Lo^2/\Lz^2)$), which is not dominated by the statistical term over the full accuracy range; and their $\tO$-forms treat the initial-gap bound $\Dcert$ as an externally supplied certificate---under the default gap certificate of Lemma~\ref{lem:gap0}, with $q:=\Lo R_0$, the \PhaseI{} entry terms scale as $(1+q)^2$ in the deterministic part and as $(1+q)^4\sigma^2\Lo^2/\Lz^2$ in the statistical burn-in (Corollary~\ref{cor:phase1-default} and Remarks~\ref{app:rem:convex-default}--\ref{rem:strongly-default}). Method abbreviations: AGMsDR is the Accelerated Gradient Method with Small-Dimensional Relaxation of \citet{vankov2024optimizing}; AGD is accelerated gradient descent; AGM is accelerated gradient method; AC-SA is the Accelerated Stochastic Approximation; ClipSSTM is the clipped stochastic similar-triangles method (SSTM) of \citet{gorbunov2020stochastic}. Cell shorthands: ``in exp.''\ = in expectation; ``high prob.''\ = high probability; ``$\mu$-str.\ cvx''\ = $\mu$-strongly convex; ``sub-Gauss.''\ = sub-Gaussian; ``gen.\ noise''\ = general noise; ``accel.\ opt.''\ = accelerated optimization; ``affine var.''\ = affine variance; ``Recursive reg.''\ = recursive regularization.}
\label{app:tab:comparison}
\end{table*}

\section{Strong-growth and interpolation extensions}
\label{app:interp}

This appendix contains the complete treatment of the overparameterized regime: the setting and its relation to exact interpolation (Appendix~\ref{app:interp-setting}), the main guarantees (Appendix~\ref{app:interp-results}), the hop-and-solve algorithm and its analysis (Appendices~\ref{app:interp-tools}--\ref{app:interp-proofs}), the obstruction behind Open Problem~\ref{op:interp} (Appendix~\ref{app:interp-obstruction}), and the corresponding numerical illustration (Appendix~\ref{app:interp-experiments}). The material is self-contained given Sections~\ref{sec:prelim}--\ref{app:sec:main}. This appendix assumes $\Lo>0$ (equivalently, it applies the $L_{1,\mathrm{eff}}$ convention of Algorithm~\ref{app:alg:arc}).

\subsection{Setting: strong growth under generalized smoothness}
\label{app:interp-setting}

Modern overparameterized models often \emph{interpolate}: a single $x^{\mathrm{int}}$ minimizes every $f_\xi$ simultaneously. In the smooth case this makes SGD, and with momentum also accelerated methods, converge at deterministic-like rates \citep{schmidt2013fast,ma2018power,vaswani2019fast,liu2020accelerating}. The generalized-smooth analogue that our analysis actually uses is the following strong growth condition.

\begin{assumption}[SGC-GS (strong growth condition under generalized smoothness)]
\label{ass:sgc}
Each $f_\xi$ is convex, differentiable and $(\Lz,\Lo)$-smooth in the sense of \eqref{eq:gs}; the oracle returns $g(x,\xi)=\grad f_\xi(x)$; and there exists $\rho\ge1$ such that for all $x\in\R^d$,
\begin{equation}
\label{eq:sgc}
\E_\xi\norm{\grad f_\xi(x)}^2 \;\le\; \rho\,\norm{\grad f(x)}^2 .
\end{equation}
For the theorems below we use the almost-sure variant
\begin{equation}
\label{eq:sgc-as}
\norm{\grad f_\xi(x)}\le\sqrt{\rho}\,\norm{\grad f(x)}\qquad\text{for a.e.\ }\xi\text{ and all }x\in\R^d,
\end{equation}
which implies \eqref{eq:sgc}; Remark~\ref{rem:mom} discusses the obstructions to working under \eqref{eq:sgc} alone.
\end{assumption}

Exact interpolation is a \emph{motivating special case} of Assumption~\ref{ass:sgc}, not an equivalent condition; the next lemma makes the logical relation precise. To avoid the ambiguity of writing ``$f_\xi(x^{\mathrm{int}})=\min f_\xi$'', we state the hypothesis explicitly: the minimum is over the decision variable $x\in\R^d$, for a fixed realization $\xi$, and the minimizing point is required to be the \emph{same} $x^{\mathrm{int}}$ for almost every $\xi$.

\begin{definition}[Exact interpolation]
\label{def:interp}
The family $\{f_\xi\}$ \emph{interpolates} at $x^{\mathrm{int}}\in\R^d$ if
\begin{equation}
\label{eq:interp-def}
f_\xi(x^{\mathrm{int}})=\inf_{x\in\R^d}f_\xi(x)\qquad\text{for a.e.\ }\xi ,
\end{equation}
i.e., the common point $x^{\mathrm{int}}$ minimizes every individual loss $f_\xi$ over the decision space. Equivalently, since each $f_\xi$ is convex and differentiable, $\grad f_\xi(x^{\mathrm{int}})=0$ for a.e.\ $\xi$. Note that \eqref{eq:interp-def} forces $x^{\mathrm{int}}\in X^\star$ and $f^\star=\E_\xi f_\xi(x^{\mathrm{int}})$.
\end{definition}

\begin{lemma}[Interpolation implies generalized weak growth]
\label{lem:interp-implies-sgc}
Let each $f_\xi$ be convex, differentiable and $(\Lz,\Lo)$-smooth, and let the family interpolate at $x^{\mathrm{int}}$ in the sense of Definition~\ref{def:interp}. Write
\begin{equation}
\label{eq:gaps}
\Delta_x:=f(x)-f^\star,\qquad \Delta_\xi(x):=f_\xi(x)-f_\xi(x^{\mathrm{int}})\ \ge0 .
\end{equation}
Then for every $x\in\R^d$,
\begin{equation}
\label{eq:wgc}
\E_\xi\norm{\grad f_\xi(x)}^2\;\le\;4\Lz\,\Delta_x+8\Lo^2\,\E_\xi\big[\Delta_\xi(x)^2\big].
\end{equation}
Assume in addition the \emph{bounded dissimilarity} condition
\begin{equation}
\label{eq:dissim}
\Delta_\xi(x)\le\beta\,\Delta_x\qquad\text{for a.e.\ }\xi\text{ and all }x\in\R^d,
\end{equation}
for some $\beta\ge1$. Then, for every $\Delta>0$, the strong growth condition \eqref{eq:sgc} holds on the sublevel set $\mathcal S_\Delta:=\{x:\Delta_x\le\Delta\}$ with
\begin{equation}
\label{eq:rho-delta}
\begin{split}
\rho(\Delta)&\;\le\;\frac{4\Lz+8\Lo^2\beta^2\Delta}{\gamma(\Delta)},\\
\gamma(\Delta)&:=\inf\Big\{\tfrac{\norm{\grad f(x)}^2}{\Delta_x}\ :\ x\in\mathcal S_\Delta,\ \Delta_x>0\Big\} .
\end{split}
\end{equation}
At points with $\Delta_x=0$ one has $\grad f_\xi(x)=0$ for a.e.\ $\xi$, so \eqref{eq:sgc} holds there trivially for any $\rho\ge1$. If $f$ is $\mu$-strongly convex then $\gamma(\Delta)\ge2\mu$ and
\begin{equation}
\label{eq:rho-delta-sc}
\rho(\Delta)\le\frac{2\Lz+4\Lo^2\beta^2\Delta}{\mu}\ \le\ 2+\frac{2\beta\big(\Lz+2\Lo^2\beta\Delta\big)}{\mu} .
\end{equation}
\end{lemma}

\begin{proof}
Fix $x$ and a realization $\xi$ for which \eqref{eq:interp-def} holds, and abbreviate $G_\xi:=\norm{\grad f_\xi(x)}$, $\Delta_\xi:=\Delta_\xi(x)$. Since $f_\xi$ is convex, $(\Lz,\Lo)$-smooth and attains its infimum at $x^{\mathrm{int}}$, Lemma~\ref{app:lem:localization} applies to $f_\xi$ and gives
\begin{equation}
\label{eq:loc-xi}
G_\xi^2\;\le\;2\big(\Lz+\Lo G_\xi\big)\,\Delta_\xi .
\end{equation}
Viewing \eqref{eq:loc-xi} as the quadratic inequality $G_\xi^2-2\Lo\Delta_\xi G_\xi-2\Lz\Delta_\xi\le0$ in $G_\xi\ge0$ and taking the positive root,
\begin{equation}
\label{eq:loc-root}
G_\xi\;\le\;\Lo\Delta_\xi+\sqrt{\Lo^2\Delta_\xi^2+2\Lz\Delta_\xi} .
\end{equation}
Squaring \eqref{eq:loc-root} and using $(a+b)^2\le2a^2+2b^2$,
\begin{equation}
\label{eq:loc-sq}
G_\xi^2\;\le\;2\Lo^2\Delta_\xi^2+2\big(\Lo^2\Delta_\xi^2+2\Lz\Delta_\xi\big)\;=\;4\Lo^2\Delta_\xi^2+4\Lz\Delta_\xi .
\end{equation}
Taking expectations in \eqref{eq:loc-sq} and using
\begin{equation}
\label{eq:mean-gap}
\E_\xi\big[\Delta_\xi(x)\big]=f(x)-\E_\xi f_\xi(x^{\mathrm{int}})=f(x)-f^\star=\Delta_x
\end{equation}
(the middle equality is Definition~\ref{def:interp}) yields \eqref{eq:wgc}, since $4\Lo^2\le8\Lo^2$.

Now assume \eqref{eq:dissim} and let $x\in\mathcal S_\Delta$ with $\Delta_x>0$. Then $\E_\xi[\Delta_\xi^2]\le\beta^2\Delta_x^2\le\beta^2\Delta\,\Delta_x$, so \eqref{eq:wgc} gives
\begin{equation}
\label{eq:wgc-sublevel}
\E_\xi\norm{\grad f_\xi(x)}^2\;\le\;\big(4\Lz+8\Lo^2\beta^2\Delta\big)\,\Delta_x .
\end{equation}
Dividing \eqref{eq:wgc-sublevel} by $\norm{\grad f(x)}^2>0$ (positive because $\Delta_x>0$) and bounding $\Delta_x/\norm{\grad f(x)}^2\le1/\gamma(\Delta)$ by the definition of $\gamma(\Delta)$ in \eqref{eq:rho-delta} proves \eqref{eq:sgc} with the constant \eqref{eq:rho-delta}. If $\Delta_x=0$ then $x\in X^\star$, hence $f_\xi(x)=f_\xi(x^{\mathrm{int}})$ for a.e.\ $\xi$ by \eqref{eq:interp-def} and \eqref{eq:mean-gap}, and minimality gives $\grad f_\xi(x)=0$ a.s., so both sides of \eqref{eq:sgc} vanish. Finally, if $f$ is $\mu$-strongly convex, the standard Polyak--\L{}ojasiewicz inequality $\norm{\grad f(x)}^2\ge2\mu\Delta_x$ gives $\gamma(\Delta)\ge2\mu$, and substituting into \eqref{eq:rho-delta} yields the first bound in \eqref{eq:rho-delta-sc}; the second follows from $\beta\ge1$ and $\Lz\ge\mu$.
\end{proof}

Lemma~\ref{lem:interp-implies-sgc} shows that interpolation alone yields only the weak-growth inequality \eqref{eq:wgc}, whose second term is \emph{quadratic} in the individual gaps; upgrading it to strong growth genuinely requires extra structure, namely bounded dissimilarity \eqref{eq:dissim} together with a quadratic-growth-type lower bound $\gamma(\Delta)>0$. Assumption~\ref{ass:sgc} is therefore the primitive hypothesis of this appendix, with interpolation as a motivating special case rather than an equivalent one.

\paragraph{Why multiplicative noise is harmless here.} The essential consequence of \eqref{eq:sgc} for algorithm design is that the oracle noise is \emph{multiplicative}: at any query $x$,
\begin{equation}
\label{eq:mult-noise}
\E\norm{g(x,\xi)-\grad f(x)}^2\le\E_\xi\norm{\grad f_\xi(x)}^2\le\rho\,\norm{\grad f(x)}^2 .
\end{equation}
Two structural facts of our framework then render this noise harmless. First, gradient scales along the outer loop are certified: Lemma~\ref{app:lem:prox-grad} and the cap bookkeeping keep $\norm{\grad f}$ bounded by the running cap, so mini-batches of size $\tO(\rho)$ make every certificate (cap updates, terminations) reliable. Second, and less obviously, the noise \emph{floor} of a proximal subproblem with parameter $\lambda$ is proportional to
\begin{equation}
\label{eq:floor-scaling}
\rho\,\norm{\grad f(\xhat)}^2=\rho\,(\lambda+\eta)^2\norm{\xhat-c}^2 ,
\end{equation}
where $\eta\ge0$ is the multiplier of the ball constraint (see \eqref{eq:kkt}); our $\lambda$-schedules tie this quantity to the current \emph{certified} optimality gap rather than to the target accuracy, so it contracts together with the iterates.

These two facts make a three-stage \emph{hop-and-solve} variant of \ARCSG{} (Algorithm~\ref{alg:interp}) converge linearly. After \PhaseI{}, well-conditioned proximal levels with $\lambda=4\Lo\Gest$ drain large gradients at cost $\tO(\rho)$ per level (the \emph{hot} stage); then ball-constrained proximal steps of radius $\tfrac{1}{2\Lo}$ with a gap-proportional $\lambda$-schedule ``hop'' towards $x^\star$, strong convexity converting the certified gap into the distance bound $\norm{c-x^\star}\le\sqrt{2\,\mathrm{gap}/\mu}$ that accelerates the hops as the gap shrinks (the \emph{cold} stage); once $x^\star$ is provably inside the current ball, a single run of projected batched SGD with the $(\Lz,\Lo)$-stepsize, which tolerates multiplicative noise natively, finishes at a linear rate (the \emph{interior} stage). The price of walking the wide, flat part of the valley is the $\eps$-independent burn-in $\rho\kbar+\rho(1+\Lz/\mu)^3$; we do not believe the cube is tight.

\subsection{Main results in the overparameterized regime}
\label{app:interp-results}

\begin{theorem}[Strong growth, strongly convex]
\label{thm:interp-strong}
Let Assumptions~\ref{app:ass:convex}, \ref{app:ass:gs} and \ref{ass:sgc} hold with the a.s.\ variant \eqref{eq:sgc-as}, and let $f$ be $\mu$-strongly convex. Then the hop-and-solve variant of \ARCSG{} (Algorithm~\ref{alg:interp}) with mini-batches of size $\tO(\rho)$ at all measurement and certificate steps (Stage-2 cold-stage inner iterations use batches of size $\tO\big(\rho(1+\Lz/\mu)^{3/2}\big)$ per Lemma~\ref{lem:cold}(c)) returns $\xhat$ with $f(\xhat)-f^\star\le\eps$ with probability at least $1-\alpha$ using
\begin{equation}
\label{eq:thm-interp-strong}
N=\tO\Big(\rho\Big(1+\tfrac{\Lz}{\mu}\Big)\logp\!\left(\tfrac{\Dcert}{\eps}\right)\;+\;\rho\,\kbar\;+\;\rho\Big(1+\tfrac{\Lz}{\mu}\Big)^{3}\Big)
\end{equation}
oracle calls; the last two terms are independent of $\eps$, and no $\sigma^2$-driven term or noise floor is present. Here $\tO$ treats the certificate as an externally supplied input and is monotone in it; under the default $\Dcert=\Ddef$ of Lemma~\ref{lem:gap0} the Stage-0 term is $\tO\big(\rho(1+q)^{4}\big)$ with $q:=\Lo R_0$ (Proposition~\ref{prop:phase1-interp}).
\end{theorem}

\begin{theorem}[Strong growth, convex --- drift-aware variant]
\label{thm:interp-convex}
Let Assumptions~\ref{app:ass:convex}, \ref{app:ass:gs} and \ref{ass:sgc} hold with the a.s.\ variant \eqref{eq:sgc-as}. Consider the following \emph{drift-aware convex variant} of the hop-and-solve method: Stages~1--2 of Algorithm~\ref{alg:interp} are run on the regularized objective $f_{\mathrm{reg}}=f+\tfrac{\mu_{\mathrm{reg}}}{2}\norm{\cdot-c^{\mathrm{I}}}^2$, with $\mu_{\mathrm{reg}}=\eps/\Rb^2$, centered at the Stage-0 output $c^{\mathrm{I}}$ (a choice essential both for the entry certificate and for the moderate size of the certified smoothness constant below); the oracle is drift-corrected to $\gest_{\mathrm{reg}}(x)=\gest(x)+\mu_{\mathrm{reg}}(x-c^{\mathrm{I}})$; the stages use the strong-convexity constant $\mu_{\mathrm{reg}}$ and the critical scale $\bar G'':=\Lz'/\Lo$ in place of $\Gbar$, where $\Lz':=\Lz+\mu_{\mathrm{reg}}+\tfrac{25}{8}\Lo\Gbar=\tfrac{33}{8}\Lz+\mu_{\mathrm{reg}}$: by the convexity comparison $\norm{\grad f(x)}^2\le\norm{\grad f_{\mathrm{reg}}(x)}^2+\norm{\grad f(c^{\mathrm{I}})}^2$ and the entry bound $\norm{\grad f(c^{\mathrm{I}})}\le\tfrac{25}{8}\Gbar$, the objective $f_{\mathrm{reg}}$ is \emph{globally} $(\Lz',\Lo)$-smooth in the sense of Assumption~\ref{app:ass:gs} (Lemma~\ref{lem:cert-reg}); the Stage-1 transfer step and level target are those of Lemma~\ref{lem:hot}, $s:=\tfrac{1}{400\Lo}$ and $\delta_{\mathrm{lvl}}:=\lambda s^2/2$; the cap anchors use the batch $B''_{\mathrm{rel}}:=\lceil1.6\cdot10^5c_g^2\rho\ell_\star\rceil$ and the update $\Gest^{+}:=\tfrac87\norm{\gest}+\tfrac{\bar G''}{28}$ of \eqref{eq:anchor-reg}; the Stage-2 exit threshold is $\Delta\le\min\{\mu_{\mathrm{reg}}/(32\Lo^2),\eps/2\}$; and Stage~3 is skipped (it is unnecessary, since the fixed target $\eps$ is already met at the Stage-2 exit). This variant returns an $\eps$-solution of \eqref{app:eq:problem} with probability at least $1-\alpha$ using
\begin{equation}
\label{eq:thm-interp-convex}
N=\tO\Big(\rho\,\kbar\;+\;\rho\Big(1+\tfrac{\Lz\Rb^2}{\eps}\Big)^{3}\Big)
\end{equation}
oracle calls. This is \emph{not} a black-box application of Theorem~\ref{thm:interp-strong}: the almost-sure strong growth of $f$ need not transfer to $f_{\mathrm{reg}}$, because $\grad f(x)$ and the regularization drift $\mu_{\mathrm{reg}}(x-c^{\mathrm{I}})$ may partly cancel; what survives regularization is the convexity-based comparison $\norm{\grad f(x)}^2\le\norm{\grad f_{\mathrm{reg}}(x)}^2+\norm{\grad f(c^{\mathrm{I}})}^2$, which yields a \emph{global} smoothness certificate whose first constant is inflated by the center-dependent term $\Lo\norm{\grad f(c^{\mathrm{I}})}$, and the centering at $c^{\mathrm{I}}$ rather than $x^0$ is what keeps that term moderate, $\norm{\grad f(c^{\mathrm{I}})}\le\tfrac{25}{8}\Gbar$ (Appendix~\ref{app:interp-proofs}). With $A:=1+\Lz\Rb^2/\eps$ the bound reads $\tO(\rho\kbar+\rho A^3)$: the only surviving $\kbar$-factor is the Stage-0 entry burn-in, which the variant shares with Theorem~\ref{thm:interp-strong}.
\end{theorem}

Table~\ref{tab:interp} positions these bounds against the smooth-case strong-growth literature and the generalized-smooth analysis of \citet{interp2026gs}.

\begin{table*}[t]
\centering
\footnotesize
\setlength{\tabcolsep}{4pt}
\begin{tabular}{@{}lll@{}}
\toprule
Reference & Smoothness & Rate ($\mu$-strongly convex) \\
\midrule
\citet{vaswani2019fast} & $L$, SGC($\rho$) & $\rho\frac{L}{\mu}$;\quad accelerated: $\sqrt{\rho\frac{L}{\mu}}$ \\[2pt]
\citet{liu2020accelerating} & $L$, SGC($\rho$) & $\sqrt{\rho\frac{L}{\mu}}$ \\[2pt]
Clipped SGD \citep{interp2026gs} & $(\Lz,\Lo)$, interpolation & $\rho\frac{\Lz}{\mu}$ + burn-in \\[2pt]
\textbf{Theorem~\ref{thm:interp-strong}} & $(\Lz,\Lo)$, SGC-GS & $\rho\tfrac{\Lz}{\mu}+\rho\kbar+\rho\tfrac{\Lz^3}{\mu^3}$ \\
\bottomrule
\end{tabular}
\caption{Strong-growth/interpolation regime: oracle calls in $\tO$-notation, with the common $\logp(\Dcert/\eps)$ factor of the leading terms omitted from all entries. SGC($\rho$) is the classical strong growth condition; SGC-GS is its generalized-smooth analogue (Assumption~\ref{ass:sgc}; our theorems use the a.s.\ variant \eqref{eq:sgc-as}); $\kbar=1+3\Lo R_0$, and the $\rho\kbar+\rho\Lz^3/\mu^3$ terms are an $\eps$-independent burn-in. Batched implementations of our method use batch size $\tO(\rho)$ at measurement and certificate steps (Stage-2 cold-stage inner iterations use $\tO\big(\rho(1+\Lz/\mu)^{3/2}\big)$, Lemma~\ref{lem:cold}(c)); single-sample variants of smooth-case accelerated methods achieve the $\sqrt{\rho}$-type dependence shown. Entries reflect the assumptions of the cited works (in-expectation versus high-probability, additive versus multiplicative noise) and are not directly comparable line by line; the printed $\tO$-forms treat the certificate $\Dcert$ as an externally supplied input (Proposition~\ref{prop:phase1-interp}).}
\label{tab:interp}
\end{table*}

Theorem~\ref{thm:interp-strong} transfers the interpolation phenomenon of \citet{schmidt2013fast,ma2018power,vaswani2019fast} to generalized smoothness: the $\eps$-dependent term matches the classical smooth-case strong-growth rate $\rho(L/\mu)\logp(\Dcert/\eps)$ with $L$ replaced by $\Lz$; all $\sigma$-driven terms vanish; the guarantee is high-probability rather than in-expectation; and arbitrarily large initial gradients cost only an $\eps$-independent burn-in, relative to the generalized-smooth analysis of \citet{interp2026gs}. What our framework does \emph{not} deliver is the accelerated $\sqrt{\Lz/\mu}$-dependence available in the smooth setting \citep{vaswani2019fast,liu2020accelerating}: as shown in Appendix~\ref{app:interp-obstruction}, inside any epoch-restarted momentum scheme the strong-growth noise at the wandering extrapolation sequence produces a noise floor $\Theta(\kappa^{3/2})$ times the epoch target, forcing batch sizes that erase the $\sqrt{\kappa}$-gain. We believe this obstruction is an artifact of proof technology rather than of the problem, and record:

\begin{openproblem}
\label{op:interp}
Under Assumptions~\ref{app:ass:convex}, \ref{app:ass:gs} and \ref{ass:sgc}, is there an algorithm with oracle complexity $\tO\big(\mathrm{poly}(\rho)\sqrt{\Lz/\mu}\,\logp(\Dcert/\eps)\big)$ plus an $\eps$-independent burn-in for $\mu$-strongly convex $f$? Equivalently: can the accelerated stochastic strong-growth analyses of \citet{vaswani2019fast,liu2020accelerating} be localized to balls of radius $O(1/\Lo)$?
\end{openproblem}

\subsection{Standing conventions, relative noise, and \PhaseI}
\label{app:interp-tools}

Throughout the rest of this appendix Assumptions~\ref{app:ass:convex}, \ref{app:ass:gs}, \ref{ass:sgc} hold with the a.s.\ variant \eqref{eq:sgc-as}, the convention \eqref{eq:wlog} is in force, and $f$ is $\mu$-strongly convex where stated. Note that $\mu$-strong convexity and \eqref{eq:gs} force $\mu\le\Lz$ (Remark~\ref{rem:mu-le-lz}). We write
\begin{equation}
\label{eq:nu-ellstar}
\begin{split}
\nu(x)&:=2\sqrt\rho\,\norm{\grad f(x)},\\
\ell_\star&:=\log\Big(\tfrac{2^{30}(\bar K_{\mathrm{hot}}+\bar K_{\mathrm{cold}}+m)^3(2+\kbar)^4(2+\Lz/\mu)^4}{\alpha}\Big),
\end{split}
\end{equation}
where $\ell_\star$ is an umbrella that dominates (a routine check) every logarithm requested from Proposition~\ref{app:prop:rstm}, Lemma~\ref{lem:batch} and the measurements below at the confidence levels used; the counters $\bar K_{\mathrm{hot}},\bar K_{\mathrm{cold}},m$ are defined in Algorithm~\ref{alg:interp}.

\subsubsection{Relative noise and measurements}
\begin{lemma}[Relative noise]
\label{lem:relnoise}
Under the a.s.\ variant, $\zeta(x,\xi):=g(x,\xi)-\grad f(x)$ satisfies $\norm{\zeta}\le(1+\sqrt\rho)\norm{\grad f(x)}\le\nu(x)$ a.s.; in particular $\zeta$ is norm-sub-Gaussian with parameter $\nu(x)$, and the average $\bar\zeta$ of a batch of size $b$ satisfies $\Prob[\norm{\bar\zeta}\ge c_g\nu(x)\sqrt{t/b}]\le2e^{-t}$ for all $t\ge1$ (Lemma~\ref{lem:batch}). Consequently a batch of size $B_{\mathrm{rel}}:=\lceil256c_g^2\rho\,\ell_\star\rceil$ returns $\gest$ with $\norm{\gest-\grad f(x)}\le\tfrac18\norm{\grad f(x)}$, so $\norm{\grad f(x)}\le\tfrac87\norm{\gest}$ and $\norm{\gest}\le\tfrac98\norm{\grad f(x)}$, with probability $1-2e^{-\ell_\star}$.
\end{lemma}

\subsubsection{\PhaseI{} under interpolation}
\begin{proposition}
\label{prop:phase1-interp}
Run \PhaseI{} (Algorithm~\ref{app:alg:phase1}) with $B_1:=\lceil2500c_g^2\rho\Lambda_{\mathrm{lg}}^2\ell_1\rceil$, $\ell_1:=\log(16T_1^{\max}/\alpha)$, all other parameters as in \eqref{eq:phase1-params}. Under the a.s.\ variant all conclusions of Proposition~\ref{app:prop:phase1} hold with probability at least $1-\alpha/4$, at total cost $T_1B_1=\tO\big(\rho(1+\Lo R_0)\big)$. Here $\tO$ treats the certificate as an externally supplied input and is monotone in it; under the default $\Dcert=\Ddef$ of Lemma~\ref{lem:gap0} one has $\Lambda_{\mathrm{lg}}=\Theta(1+q)$ with $q:=\Lo R_0$, hence $B_1=\tO\big(\rho(1+q)^2\big)$ and the true cost is $\tO\big(\rho(1+q)^4\big)$ with only polylogarithmic factors in $(q,1/\alpha)$ hidden---the same two-regime distinction as in Corollary~\ref{cor:phase1-default}, and the $\tO$-forms of Theorems~\ref{thm:interp-strong}--\ref{thm:interp-convex} inherit it.
\end{proposition}
\begin{proof}
With $\eta:=1/(25\Lambda_{\mathrm{lg}})$, Lemma~\ref{lem:relnoise} and a union bound over the $\le T_1^{\max}+1$ steps give $\norm{\bar\zeta_t}\le\eta G_t$ at every step. The proof of Proposition~\ref{app:prop:phase1} interacts with the estimation accuracy only through two inequalities: at every non-terminal step it uses $\norm{\bar\zeta_t}\le G_t/(23\Lambda_{\mathrm{lg}})$, valid there since $\eta\le1/(23\Lambda_{\mathrm{lg}})$; at the terminal step it uses the output bound, which now reads $G\le\norm{\gest}/(1-\eta)\le3\Gbar\cdot\tfrac{25}{24}\le\tfrac{25}8\Gbar$ (and, conversely, the run cannot stop while $G>\tfrac{25}8\Gbar$, since then $\norm{\gest}\ge(1-\eta)G>3\Gbar$; while it runs, $\norm{\gest_t}>3\Gbar$ gives $G_t\ge3\Gbar/(1+\eta)\ge\tfrac{23}8\Gbar$, so the supercritical chain is intact). The degenerate branch ($\Lo R_0\le\tfrac1{24}$) uses only the $t=0$ measurement and is unaffected.

\end{proof}

\subsubsection{\RSTM{} with state-dependent noise}
\begin{remark}[Noise interface]
\label{rem:rstm-noise}
The proof of Proposition~\ref{app:prop:rstm} interacts with randomness only through the per-query sub-Gaussian parameter $\varsigma_k=c_b\sigma/\sqrt{b_k}$ of the batched estimates at points of $Q_k$. Hence, if the single-sample noise is norm-sub-Gaussian with parameter $\nu_k$ \emph{uniformly over $x\in Q_k$}, replacing \eqref{eq:bk} by
\begin{equation}
\label{eq:bk-state}
b_k:=\max\Big\{1,\ \Big\lceil\frac{C_b\nu_k^2\log(8KN_{\mathrm{ep}}/\alpha')}{\sqrt{\lambda L_F}\,\bar\Delta_k}\Big\rceil\Big\}
\end{equation}
(the lower bound $b_k\ge1$ matters when the state-dependent scale $\nu_k$ vanishes, e.g.\ near the optimum under strong growth) leaves every step of the proof unchanged, and the total statistical cost becomes $\sum_kN_{\mathrm{ep}}b_k\le KN_{\mathrm{ep}}+24C_b\ell'\sum_k\nu_k^2/(\lambda\bar\Delta_k)$ with $\ell':=\log(8KN_{\mathrm{ep}}/\alpha')$.
\end{remark}

\begin{lemma}[Gradient scale on epoch balls]
\label{lem:varsigma}
Let a level have center $c$, cap $\Gest\ge\norm{\grad f(c)}$, parameter $\lambda>0$ and ball radius $r_c\le\tfrac1{2\Lo}$, and let $L_F:=L_c+\lambda$ with $L_c:=4(\Lz+\Lo\Gest)$ (Lemma~\ref{app:lem:ball}). During epoch $k$ of \RSTM{} on $F:=f+\tfrac\lambda2\norm{\cdot-c}^2$ over $\ball{c}{r_c}$, every $x\in Q_k$ satisfies $\norm{\grad f(x)}\le3L_F\sqrt{2\bar\Delta_k/\lambda_F}+\Gest+2\lambda r_c$, where $\lambda_F\ge\lambda$ is the strong convexity modulus used; hence the a.s.\ noise scale on $Q_k$ obeys
$\nu_k^2\le12\rho\big(18L_F^2\bar\Delta_k/\lambda_F+\Gest^2+4\lambda^2r_c^2\big)$.
\end{lemma}
\begin{proof}
For $x\in Q_k$: $\norm{\grad f(x)}\le\norm{\grad F(x)}+\lambda r_c\le\norm{\grad F(x)-\grad F(x_Q)}+\norm{\grad F(x_Q)}+\lambda r_c$. By Lemma~\ref{app:lem:prox-grad} (which covers the ball-constrained case), $\norm{\grad f(x_Q)}\le\norm{\grad f(c)}\le\Gest$, so $\norm{\grad F(x_Q)}\le\Gest+\lambda r_c$; and $\norm{\grad F(x)-\grad F(x_Q)}\le L_F\norm{x-x_Q}\le L_FD_k=3L_F\sqrt{2\bar\Delta_k/\lambda_F}$ along segments of $\ball{c}{r_c}$. Square with $(a+b+c)^2\le3(a^2+b^2+c^2)$ and multiply by $(2\sqrt\rho)^2$.
\end{proof}

\subsubsection{The hop-and-solve algorithm}
\begin{algorithm}[t]
\caption{\ARCSG-interp{} (interpolation/strong growth; hop-and-solve). Constants: $\rsafe=\tfrac1{2\Lo}$, $s=\tfrac1{400\Lo}$, $B_{\mathrm{rel}}=\lceil256c_g^2\rho\ell_\star\rceil$, $\bar K_{\mathrm{hot}}:=\lceil5900\Lz/\mu\rceil+14$, $\bar K_{\mathrm{cold}}:=\lceil108\Lz/\mu\rceil+\lceil2\log_2(16(1+\Lz/\mu))\rceil+24$, $m:=\lceil\log_4(8/\alpha)\rceil$.}
\label{alg:interp}
\begin{algorithmic}[1]
\REQUIRE $x^0$, $R_0$, $\mu$, $\eps$, $\alpha$, $\rho$
\STATE \textbf{Stage 0:} $c\leftarrow\PhaseI(x^0)$ with $B_1$ of Proposition~\ref{prop:phase1-interp}; $\Gest\leftarrow4\Gbar$
\WHILE{$\Gest>2\Gbar$ \textbf{and} fewer than $\bar K_{\mathrm{hot}}$ hot levels used}
\STATE \textbf{Stage 1 (hot):} $\lambda\leftarrow4\Lo\Gest$;\STATE $c\leftarrow\RSTM\big(f+\tfrac{\lambda}2\norm{\cdot-c}^2\ \text{over}\ \ball{c}{\rsafe};\ H_0=\tfrac{\Gest}{4\Lo},\ \delta_{\mathrm{lvl}}=\tfrac{\Gest}{80000\Lo},\ \tfrac{\alpha}{16\bar K_{\mathrm{hot}}}\big)$ \STATE$\Gest\leftarrow\tfrac87\norm{\GradEst(c,B_{\mathrm{rel}})}$
\ENDWHILE
\STATE \textbf{Stage 2 (cold):} $\Delta\leftarrow2\Gest^2/\mu$
\FOR{$j=1,\dots,\bar K_{\mathrm{cold}}$ \textbf{while} $\Delta>\tfrac{\mu}{32\Lo^2}$}
\STATE $\theta\leftarrow\min\{\tfrac12,\tfrac1{2\Lo}\sqrt{\tfrac{\mu}{2\Delta}}\}$;
\STATE $\lambda\leftarrow2\Lo^2\theta\Delta$;
\STATE $c\leftarrow\RSTM\big(f+\tfrac\lambda2\norm{\cdot-c}^2\ \text{over}\ \ball{c}{\rsafe};\ H_0=\Gest\rsafe,\ \delta_{\mathrm{lvl}}=\tfrac{\theta\Delta}{256},\ \tfrac{\alpha}{16\bar K_{\mathrm{cold}}}\big)$
\STATE $\Gest\leftarrow\tfrac87\norm{\GradEst(c,B_{\mathrm{rel}})}$;\STATE $\Delta\leftarrow(1-\tfrac\theta4)\Delta$
\ENDFOR
\STATE \textbf{Stage 3 (interior):} conditional on the Stage-0--2 history, for $i=1,\dots,m$ run mutually independent copies (disjoint fresh sample streams, all started from the fixed common point $x_0^{(i)}:=c$): $T$ steps of projected batched SGD on $f$ over $\ball{c}{\rsafe}$, $x_{t+1}=\Pi(x_t-\tfrac1{4L'}\gest_t)$, batch $\lceil24\rho\rceil$, $L':=12(\Lz+\Lo\Gest)$, $T:=\max\big\{0,\,\lceil\tfrac{4L'}\mu\ln\tfrac{2L'^2}{\Lo^2\eps\mu}\rceil\big\}$
\STATE \emph{Two-step selection:} measure each output with $B_{\mathrm{rel}}$; discard those with $\norm{\gest_i}>2\sqrt{2L'\eps'}$, $\eps':=\tfrac{\eps\mu}{16L'}$; re-measure survivors with $B_{\mathrm{sel}}:=\lceil95c_g^2\rho\ell_\star\rceil$; \RETURN the survivor with smallest re-measured norm
\end{algorithmic}
\end{algorithm}

\subsubsection{The hot stage}
\begin{lemma}[Hot stage]
\label{lem:hot}
On the good event, Stage 1 maintains $\norm{\grad f(c)}\le\Gest\le\tfrac97\norm{\grad f(c)}$ after each anchor (the initial $\Gest=4\Gbar$ is valid by Proposition~\ref{prop:phase1-interp}), and:
(a) each level has $\lambda=4\Lo\Gest$, $L_F\le10\Lo\Gest$, condition number $\le\tfrac52$, unconstrained prox ($\norm{\xhat-c}\le\tfrac1{4\Lo}$), and certified $H_0=\Gest^2/\lambda$;
(b) with $\Psi:=f(c)-f^\star+\tfrac{\Gest}{100\Lo}$, every executed level decreases $\Psi$ by at least $\tfrac{\Gest}{600\Lo}\ge\tfrac{\bar\Delta}{300}$, where $\bar\Delta=\Lz/\Lo^2$;
(c) Stage 1 executes at most $5900\Lz/\mu+14$ levels, exits with $\Gest\le2\Gbar$, and never triggers the $\bar K_{\mathrm{hot}}$ guard;
(d) each level costs at most $10^9C_b\,\rho\,\ell_\star$ oracle calls, so the stage costs $\tO(\rho(1+\Lz/\mu))$.
\end{lemma}
\begin{proof}
\emph{(a)} While $\Gest>2\Gbar$ we have $\Lz\le\Lo\Gest/2$, so $L_c=4(\Lz+\Lo\Gest)\le6\Lo\Gest$ and $L_F=L_c+\lambda\le10\Lo\Gest$, giving $1+L_F/\lambda\le\tfrac72$; a sharper count, $\kappa=L_F/\lambda\le\tfrac{10}4$, is what we use. Containment: $\norm{\xhat-c}\le\norm{\grad f(c)}/\lambda\le\tfrac1{4\Lo}<\rsafe$ by Lemma~\ref{app:lem:prox-grad}, so Remark~\ref{rem:proj} certifies $H_0$, and $\delta_{\mathrm{lvl}}=\lambda s^2/2$ gives $\norm{c^+-\xhat}\le s$.

\emph{(b)} Two cases. \emph{Productive} ($\norm{\xhat-c}\ge\tfrac1{16\Lo}$): $f(\xhat)\le f(c)-\tfrac\lambda2\norm{\xhat-c}^2\le f(c)-\tfrac{\Gest}{128\Lo}$; the transfer along $[\xhat,c^+]\subset\ball{c}{2\rsafe}$ costs $f(c^+)\le f(\xhat)+s\norm{\grad f(\xhat)}+\tfrac{L_c}2s^2\le f(\xhat)+\tfrac{\Gest}{400\Lo}+\tfrac{\Gest}{53000\Lo}$, so the net drop is $\ge\tfrac{\Gest}{200\Lo}$. The anchor can raise the cap: $\norm{\grad f(c^+)}\le\norm{\grad f(\xhat)}+L_cs\le\Gest(1+0.015)$, so $\Gest^+\le\tfrac97\cdot1.015\,\Gest\le1.31\Gest$, and $\Delta\Psi\le-\tfrac{\Gest}{200\Lo}+\tfrac{0.31\Gest}{100\Lo}\le-\tfrac{\Gest}{600\Lo}$. \emph{Collapsing} ($\norm{\xhat-c}<\tfrac1{16\Lo}$): then $\norm{\grad f(\xhat)}=\lambda\norm{\xhat-c}<\tfrac{\Gest}4$, the objective can increase by at most $s\norm{\grad f(\xhat)}+\tfrac{L_c}2s^2\le\tfrac{\Gest}{1550\Lo}\le\tfrac{\Gest}{360\Lo}$, while $\norm{\grad f(c^+)}\le\tfrac{\Gest}4+0.015\Gest$, hence $\Gest^+\le\tfrac97\cdot0.265\,\Gest\le0.35\Gest$ and $\Delta\Psi\le\tfrac{\Gest}{360\Lo}-\tfrac{0.65\Gest}{100\Lo}\le-\tfrac{\Gest}{300\Lo}$. In both cases $\Delta\Psi\le-\tfrac{\Gest}{600\Lo}\le-\tfrac{2\Gbar}{600\Lo}=-\tfrac{\bar\Delta}{300}$.

\emph{(c)} By $\mu$-strong convexity, $f(c)-f^\star\le\norm{\grad f(c)}\norm{c-x^\star}\le\norm{\grad f(c)}\sqrt{2(f(c)-f^\star)/\mu}$, so $f(c)-f^\star\le2\norm{\grad f(c)}^2/\mu$; at entry $\norm{\grad f(c_1)}\le\tfrac{25}8\Gbar$, hence $\Psi_0\le\tfrac{2\cdot625}{64}\tfrac{\Gbar^2}\mu+\tfrac{4\Gbar}{100\Lo}\le19.6\tfrac{\Lz\bar\Delta}\mu+\tfrac{\bar\Delta}{25}$. Since $\Psi\ge0$ and each level removes $\ge\bar\Delta/300$, at most $300\Psi_0/\bar\Delta\le5880\Lz/\mu+12$ levels execute before the loop condition fails, i.e.\ $\Gest\le2\Gbar$. (This also bounds the caps: $\Gest\le100\Lo\Psi_0$, which the umbrella $\ell_\star$ dominates.)

\emph{(d)} By (a), $H_0/\delta_{\mathrm{lvl}}=20000$, so $K\le15$ epochs of $N_{\mathrm{ep}}\le24\sqrt{5/2}+1\le39$ iterations. Remark~\ref{rem:rstm-noise} with Lemma~\ref{lem:varsigma} ($r_c=\rsafe$, $\lambda_F=\lambda$) gives per-level statistical cost at most $24C_b\ell_\star\big[216\rho(\tfrac{L_F}\lambda)^{3/2}KN_{\mathrm{ep}}/N_{\mathrm{ep}}+\dots\big]$; carrying out the three pieces as in Appendix~\ref{app:cost}: the $L_F^2\bar\Delta_k/\lambda$-piece contributes $\le5184C_b\rho\ell_\star\kappa^2K\le5\cdot10^5C_b\rho\ell_\star$; the $\Gest^2$-piece contributes $\le1152C_b\rho\ell_\star\Gest^2/(\lambda\delta_{\mathrm{lvl}})=1152\cdot2\cdot10^4C_b\rho\ell_\star$ (since $\lambda\delta_{\mathrm{lvl}}=\Gest^2/20000$); the $\lambda^2\rsafe^2$-piece contributes $\le4608C_b\rho\ell_\star\lambda\rsafe^2/\delta_{\mathrm{lvl}}=4608\cdot80000\,C_b\rho\ell_\star$ (since $\lambda\rsafe^2/\delta_{\mathrm{lvl}}=2\rsafe^2/s^2=80000$). Adding $25K\sqrt{5/2}$ deterministic-equivalent iterations and $B_{\mathrm{rel}}$ gives the claim.

\end{proof}

\subsubsection{The cold stage}
\begin{lemma}[Cold hops]
\label{lem:cold}
On the good event, Stage 2 satisfies, with $\Delta_j$ the running schedule value and $\theta_j,\lambda_j$ as in Algorithm~\ref{alg:interp}:
(a) $f(c_j)-f^\star\le\Delta_j$ for every hop $j$, and $\norm{c_j-x^\star}\le\sqrt{2\Delta_j/\mu}$;
(b) the stage exits within $108\Lz/\mu+2\log_2(16(1+\Lz/\mu))+22$ hops with $\Delta\le\tfrac\mu{32\Lo^2}$, hence $\norm{c-x^\star}\le\tfrac{\rsafe}2$; the $\bar K_{\mathrm{cold}}$ guard never triggers;
(c) the total cost of the stage is $\tO\big(\rho(1+\Lz/\mu)^3\big)$.
\end{lemma}
\begin{proof}
\emph{(a)} Induction; the base is $f(c_1)-f^\star\le2\norm{\grad f(c_1)}^2/\mu\le2\Gest^2/\mu=\Delta_1$ as in Lemma~\ref{lem:hot}(c). Given the invariant, $D_j:=\norm{c_j-x^\star}\le\sqrt{2\Delta_j/\mu}$. Let $\tilde x$ be the ball-constrained prox. If $D_j\le\rsafe$ then $f(\tilde x)-f^\star\le F(x^\star)-f^\star=\tfrac{\lambda_j}2D_j^2\le\tfrac{\lambda_j\rsafe^2}2=\tfrac{\theta_j\Delta_j}4$ (the last identity is the design $\lambda_j=2\Lo^2\theta_j\Delta_j$). Otherwise set $\theta':=\rsafe/D_j\in[\theta_j,1]$ (using $D_j\le\sqrt{2\Delta_j/\mu}$) and $u:=c_j+\theta'(x^\star-c_j)$: $f(\tilde x)+\tfrac{\lambda_j}2\norm{\tilde x-c_j}^2\le f(u)+\tfrac{\lambda_j}2\rsafe^2$ and convexity give
$f(\tilde x)-f^\star\le(1-\theta')\big(f(c_j)-f^\star\big)+\tfrac{\theta_j\Delta_j}4\le(1-\theta_j)\Delta_j+\tfrac{\theta_j\Delta_j}4$.
\RSTM{} (Remark~\ref{rem:rstm-noise}) returns $c_{j+1}$ with $F(c_{j+1})-F(\tilde x)\le\delta_j=\tfrac{\theta_j\Delta_j}{256}$ and $d:=\norm{c_{j+1}-\tilde x}\le\sqrt{2\delta_j/\lambda_j}=\tfrac1{16\Lo}$. Then, by the two-case argument of \eqref{eq:transfer} with $a:=\norm{\tilde x-c_j}\le\rsafe$,
$f(c_{j+1})\le F(c_{j+1})-\tfrac{\lambda_j}2\norm{c_{j+1}-c_j}^2\le f(\tilde x)+\delta_j+\lambda_j\,a\,d\le f(\tilde x)+\delta_j+\lambda_j\rsafe\,d$,
and $\lambda_j\rsafe\,d=2\Lo^2\theta_j\Delta_j\cdot\tfrac1{2\Lo}\cdot\tfrac1{16\Lo}=\tfrac{\theta_j\Delta_j}{16}$. Collecting, $f(c_{j+1})-f^\star\le(1-\theta_j)\Delta_j+\theta_j\Delta_j(\tfrac14+\tfrac1{16}+\tfrac1{256})\le(1-\tfrac{\theta_j}4)\Delta_j=\Delta_{j+1}$.

\emph{(b)} While $\theta_j=\tfrac12$ (i.e.\ $\Delta_j\le\tfrac\mu{2\Lo^2}$) the schedule contracts by $\tfrac78$ per hop, so at most $\lceil\log_{8/7}16\rceil=21$ such hops occur before $\Delta\le\tfrac\mu{32\Lo^2}$. In the regime $\theta_j<\tfrac12$, on a dyadic band $\Delta\in(\Delta',2\Delta']$ one has $\theta_j\ge\tfrac1{2\Lo}\sqrt{\mu/(4\Delta')}$, so at most $\lceil4\ln2/\theta\rceil\le(16\ln2)\,\Lo\sqrt{\Delta'/\mu}+1$ hops halve the schedule. Summing the geometric series over bands below $\Delta_1\le8\Gbar^2/\mu$ gives at most $38\Lo\sqrt{\Delta_1/\mu}+\log_2\tfrac{\Delta_1}{\Delta_{\mathrm{stop}}}\le108\tfrac\Lz\mu+\log_2\big(256\tfrac{\Lz^2}{\mu^2}\big)$ hops. At exit $\Delta\le\tfrac{\mu\rsafe^2}8$, so $\norm{c-x^\star}\le\sqrt{2\Delta/\mu}\le\tfrac{\rsafe}2$ by (a).

\emph{(c)} By (a) and Lemma~\ref{app:lem:localization}, $\norm{\grad f(c_j)}\le2\Lo\Delta_j+\sqrt{2\Lz\Delta_j}$ regardless of how $c_j$ was produced, so the anchored cap obeys $\Gest_j\le\tfrac97\norm{\grad f(c_j)}\le3\Lo\Delta_j+2\sqrt{\Lz\Delta_j}$ and $L_c\le4\Lz+4\Lo\Gest_j\le8\Lz+16\Lo^2\Delta_j$ (AM--GM), $L_F\le8\Lz+17\Lo^2\Delta_j$. In the regime $\theta_j<\tfrac12$, $\lambda_j=\Lo\sqrt{\mu\Delta_j/2}$ and
$\kappa_j:=L_F/\lambda_j\le\tfrac{12\Lz}{\Lo\sqrt{\mu\Delta_j}}+25\Lo\sqrt{\Delta_j/\mu}$;
over the admissible range $\Delta_j\in(\tfrac\mu{32\Lo^2},\,8\Gbar^2/\mu]$ both terms are at most $25\sqrt8\,\tfrac\Lz\mu$ each, and in the regime $\theta_j=\tfrac12$ ($\lambda_j=\Lo^2\Delta_j$), $\kappa_j\le17+\tfrac{8\Lz}{\Lo^2\Delta_j}\le17+256\tfrac\Lz\mu$; in all cases $\kappa_j\le300(1+\tfrac\Lz\mu)$.
Per hop, Remark~\ref{rem:rstm-noise} with Lemma~\ref{lem:varsigma} gives cost
$25K_j\sqrt{\kappa_j}\;+\;5184C_b\rho\ell_\star\kappa_j^2K_j\;+\;1152C_b\rho\ell_\star\tfrac{\Gest_j^2}{\lambda_j\delta_j}\;+\;4608C_b\rho\ell_\star\tfrac{\lambda_j\rsafe^2}{\delta_j}\;+\;B_{\mathrm{rel}}$,
with $K_j\le\ell_\star$. The last ratio equals $\tfrac{\lambda_j\rsafe^2}{\delta_j}=\tfrac{\theta_j\Delta_j/2}{\theta_j\Delta_j/256}=128$, an absolute constant. For the $\Gest^2$-piece, in the regime $\theta_j<\tfrac12$ one has $\lambda_j\delta_j=\tfrac{\mu\Delta_j}{1024}$ and $\Gest_j^2\le18\Lo^2\Delta_j^2+8\Lz\Delta_j$, so it is at most $1.2\cdot10^6C_b\rho\ell_\star(18\Lo^2\Delta_j+8\Lz)/\mu$; summed over hops with the band counts of (b), the $\Lz$-part gives $\tO(\rho(\Lz/\mu)^2)$ and the $\Lo^2\Delta_j$-part, dominated by the top band $\Delta_1\le8\Gbar^2/\mu$ through $\sum_j\Delta_j\le34\Delta_1^{3/2}\Lo/\sqrt\mu$, gives $\tO(\rho(\Lz/\mu)^3)$; in the regime $\theta_j=\tfrac12$ the piece is $\tO(\rho(1+\Lz/\mu))$ per hop. For the $\kappa_j^2$-piece, the same band accounting yields $\sum_j\kappa_j^2=\tO\big((\Lz/\mu)^3+(1+\Lz/\mu)^2\big)$: the square of the second term of $\kappa_j$, $625\Lo^2\Delta_j/\mu$, summed against band counts $(16\ln2)\,\Lo\sqrt{\Delta_j/\mu}$ is again dominated by $\Delta_1$ and gives the cube; the square of the first term, $144\Lz^2/(\Lo^2\mu\Delta_j)$, summed against the same counts is dominated by $\Delta_{\mathrm{stop}}$ and gives $\tO((\Lz/\mu)^2)$. Multiplying by $5184C_b\rho\ell_\star K_j$ and adding the deterministic part $\sum_j25\ell_\star\sqrt{\kappa_j}=\tO((1+\Lz/\mu)^{3/2})$ proves (c). (The cubes originate in the early, wide-valley hops, where the certified gradient scale can be as large as $\Lo\Delta_j$; we do not believe they are tight, cf.\ Remark~\ref{rem:obstruction}.)
\end{proof}

\subsubsection{The interior stage}
\begin{remark}[Extended ball smoothness and co-coercivity]
\label{rem:cococo}
The proof of Lemma~\ref{app:lem:ball} applies on $\ball{c}{4r}$ for $r\le\tfrac1{2\Lo}$ with the constant $L':=12(\Lz+\Lo\Gest)$, using the sharp form of Lemma~\ref{lem:growth} ($G_z\le e^{\Lo s}G_c+\tfrac{\Lz}{\Lo}(e^{\Lo s}-1)$ with $s=\norm{z-c}\le4r$): the Gr\"onwall factor becomes $e^{4\Lo r}\le e^2$, so $\Lz+\Lo G_z\le e^2(\Lz+\Lo\Gest)$ on the ball, and the two-point bound sums to $\norm{\grad f(x)-\grad f(y)}\le e^2(\Lz+\Lo\Gest)\norm{x-y}\le L'\norm{x-y}$, since $e^2\le12$. Consequently, if $x^\star\in\ball{c}{\rsafe}$, then for every $x\in\ball{c}{\rsafe}$ the virtual point $y:=x-\grad f(x)/L'$ stays in the certified region: $\norm{\grad f(x)}\le\Gest+L'\rsafe$ gives $\norm{y-x}\le\tfrac{\Gest}{L'}+\rsafe\le\tfrac1{12\Lo}+\tfrac1{2\Lo}$, so $y\in\ball{c}{\tfrac{13}{12\Lo}}\subset\ball{c}{4\rsafe}$, and the descent lemma along $[x,y]$ yields the co-coercivity bound $\norm{\grad f(x)}^2\le2L'\big(f(x)-f^\star\big)$ on $\ball{c}{\rsafe}$.
\end{remark}

\begin{lemma}[Interior SGD and selection]
\label{lem:interior}
Suppose Stage 3 starts with $\norm{c-x^\star}\le\tfrac{\rsafe}2$ and a valid cap $\Gest\ge\norm{\grad f(c)}$. Then, conditionally on the Stage-0--2 history: (a) each SGD run satisfies $\E[\norm{x_T-x^\star}^2\mid\mathcal F_2]\le(1-\tfrac\mu{4L'})^T\norm{c-x^\star}^2$ and, with the stated $T$, $\E[f(x_T)-f^\star\mid\mathcal F_2]\le\eps'/4$; (b) with conditional probability at least $1-\alpha/4$, the two-step selection returns $\hat x$ with $f(\hat x)-f^\star\le\tfrac\eps4$. Consequently the same probability bound holds unconditionally. (c) The stage costs $\tO\big(\rho(1+\tfrac\Lz\mu)\logp(\tfrac{\Lz^2}{\Lo^2\eps\mu})\cdot\logp(1/\alpha)\big)$ oracle calls. Moreover at entry $\Gest\le\Gbar$, so $L'\le24\Lz$, and $\mathcal F_2$ denotes the history through Stage 2.
\end{lemma}
\begin{proof}
Fix the Stage-0--2 history throughout; all expectations and probabilities below are conditional on $\mathcal F_2$ until the final unconditioning step.
At entry, Lemma~\ref{app:lem:localization} with $f(c)-f^\star\le\tfrac\mu{32\Lo^2}$ and $\mu\le\Lz$ gives $\norm{\grad f(c)}\le\tfrac\mu{16\Lo}+\sqrt{\tfrac{\Lz\mu}{16\Lo^2}}\le\tfrac{\Gbar}{16}+\tfrac{\Gbar}4$, so the anchored cap is $\le\tfrac97\cdot\tfrac5{16}\Gbar\le\Gbar$ and $L'\le24\Lz$.
\emph{(a)} Projection onto $\ball{c}{\rsafe}$ is nonexpansive and $x^\star$ is feasible, so with $r_t:=\norm{x_t-x^\star}$, $\E r_{t+1}^2\le r_t^2-2\gamma\inner{\grad f(x_t)}{x_t-x^\star}+\gamma^2\E\norm{\gest_t}^2$. Strong convexity gives $\inner{\grad f}{x-x^\star}\ge f(x)-f^\star+\tfrac\mu2r^2$; the batch of size $\ge24\rho$ gives $\E\norm{\bar\zeta}^2\le\tfrac{4\rho}{24\rho}\norm{\grad f}^2\le\tfrac16\norm{\grad f}^2$, so $\E\norm{\gest}^2\le\tfrac76\norm{\grad f}^2\le\tfrac73L'(f-f^\star)$ by Remark~\ref{rem:cococo}. With $\gamma=\tfrac1{4L'}$, $\E r_{t+1}^2\le(1-\gamma\mu)r_t^2-2\gamma(1-\tfrac7{24})(f-f^\star)\le(1-\tfrac\mu{4L'})r_t^2$. Since $[x^\star,x_T]\subset\ball{c}{\rsafe}$ and $f$ is $L'$-smooth there, $\E[f(x_T)-f^\star]\le\tfrac{L'}2\E r_T^2\le\tfrac{L'}2e^{-T\mu/(4L')}\tfrac1{16\Lo^2}\le\tfrac{\eps'}4$ by the choice of $T$ (if the logarithm is non-positive then $T=0$, and the entry certificate $f(c)-f^\star\le\tfrac{\mu}{32\Lo^2}\le\tfrac{\eps'}4$ already delivers the target).
\emph{(b)} By Markov and (a), each run fails ($f(x_T^{(i)})-f^\star>\eps'$) with probability $\le\tfrac14$.
The $m$ runs use disjoint, mutually independent sample batches and share only the history-measurable starting point $c$ and the parameters $(L',T)$ computed before the stage. Conditionally on the Stage-0--2 history these quantities are fixed and the failure events are independent, so the probability that all runs fail simultaneously is $\le4^{-m}\le\tfrac\alpha8$ by the choice $m=\lceil\log_4(8/\alpha)\rceil$.
A good run has $G_i:=\norm{\grad f(x_T^{(i)})}\le\sqrt{2L'\eps'}$ (Remark~\ref{rem:cococo}), so its relative measurement is $\le\tfrac98\sqrt{2L'\eps'}<2\sqrt{2L'\eps'}$ and it survives; every survivor has $G_i\le\tfrac87\cdot2\sqrt{2L'\eps'}=\tfrac{16}7\sqrt{2L'\eps'}$, hence noise scale $\nu_i\le\tfrac{32}7\sqrt{2\rho L'\eps'}$, and the batch $B_{\mathrm{sel}}=\lceil95c_g^2\rho\ell_\star\rceil$ yields absolute accuracy $c_g\nu_i\sqrt{\ell_\star/B_{\mathrm{sel}}}\le\tfrac{\sqrt{\mu\eps}}6$ (using $L'\eps'=\tfrac{\eps\mu}{16}$).
If $\hat x$ minimizes the re-measured norm among survivors, then $\norm{\grad f(\hat x)}\le\sqrt{2L'\eps'}+2\cdot\tfrac{\sqrt{\mu\eps}}6=\sqrt{\tfrac{\mu\eps}8}+\tfrac{\sqrt{\mu\eps}}3$, and $f(\hat x)-f^\star\le\tfrac{\norm{\grad f(\hat x)}^2}{2\mu}\le\tfrac{(0.354+0.334)^2}{2}\eps\le\tfrac\eps4$.
A union bound over the $\le2m$ measurements, each failing with probability at most $2e^{-\ell_\star}$, adds at most $4me^{-\ell_\star}\le\alpha/2^{28}$ (by the definition \eqref{eq:nu-ellstar} of $\ell_\star$). Hence the total Stage-3 failure probability is at most $\alpha/8+\alpha/2^{28}<\alpha/4$, exactly as stated.
\emph{(c)} Immediate: $m\big(T\lceil24\rho\rceil+B_{\mathrm{rel}}+B_{\mathrm{sel}}\big)$.
\end{proof}

\subsection{Proofs of Theorems~\ref{thm:interp-strong} and \ref{thm:interp-convex}}
\label{app:interp-proofs}
\subsubsection{Proof of Theorem~\ref{thm:interp-strong}}
On the intersection of the good events of Proposition~\ref{prop:phase1-interp} and Lemmas~\ref{lem:hot}--\ref{lem:interior}, the output satisfies $f(\hat x)-f^\star\le\tfrac\eps4$. Here is the explicit confidence accounting: Stage 0 consumes at most $\alpha/4$; the hot and cold \RSTM{} calls consume at most $\alpha/16$ each; all hot/cold cap-anchor measurements together consume less than $\alpha/16$ by \eqref{eq:nu-ellstar}; and Lemma~\ref{lem:interior} consumes less than $\alpha/4$. Thus the total failure probability is at most $11\alpha/16<\alpha$. The cost is the sum of: Stage 0, $\tO(\rho(1+\Lo R_0))=\tO(\rho\kbar)$; Stage 1, $\tO(\rho(1+\Lz/\mu))$; Stage 2, $\tO(\rho(1+\Lz/\mu)^3)$; Stage 3, $\tO(\rho(1+\Lz/\mu)\logp(\Dcert/\eps))$ after absorbing $\logp(\Lz/\mu)$ and $\logp(1/\alpha)$ factors into $\tO$.

\hfill$\square$

\subsubsection{Proof of Theorem~\ref{thm:interp-convex}}

\paragraph{The regularization must be centered at the \PhaseI{} output.} A remark on the construction is in order, because the naive choice is not merely suboptimal but leaves a genuine gap. Regularizing around the \emph{initial} point, $f+\tfrac{\mu_{\mathrm{reg}}}{2}\norm{\cdot-x^0}^2$, breaks the entry certificate of the hot stage: Stage~0 is run on $f$ and Proposition~\ref{prop:phase1-interp} certifies a cap for $\norm{\grad f(c)}$ only, whereas the hot stage of Algorithm~\ref{alg:interp} must be started from a valid cap for
\begin{equation}
\label{eq:grad-reg}
\norm{\grad f_{\mathrm{reg}}(c)}=\norm{\grad f(c)+\mu_{\mathrm{reg}}(c-x^0)},
\end{equation}
and the drift term $\mu_{\mathrm{reg}}\norm{c-x^0}$ in \eqref{eq:grad-reg} is not controlled by any \emph{a priori} bound of order $\bar G''$; bounding it through the hot-stage potential would be circular, since that potential is itself established only after the entry cap is known to be valid. The same choice would also defeat the smoothness certificate of Lemma~\ref{lem:cert-reg} below, whose first constant is inflated by $\Lo\norm{\grad f(c)}$ at the center $c$: at $c=x^0$ this is the arbitrarily large $\Lo\norm{\grad f(x^0)}$, while at $c=c^{\mathrm{I}}$ it is at most $\tfrac{25}{8}\Lo\Gbar=\tfrac{25}{8}\Lz$. We therefore center the regularizer at the \PhaseI{} output.

\paragraph{Construction.} Let $c^{\mathrm{I}}$ be the point returned by Stage~0 (i.e., \PhaseI{} run on $f$; Proposition~\ref{prop:phase1-interp}), and set
\begin{equation}
\label{eq:freg}
\begin{split}
f_{\mathrm{reg}}&:=f+\tfrac{\mu_{\mathrm{reg}}}{2}\norm{\cdot-c^{\mathrm{I}}}^2,
\qquad
\mu_{\mathrm{reg}}:=\frac{\eps}{\Rb^2},\\
\Lz'&:=\Lz+\mu_{\mathrm{reg}}+\tfrac{25}{8}\Lo\Gbar
\;=\;\tfrac{33}{8}\Lz+\mu_{\mathrm{reg}},
\qquad
\bar G'':=\frac{\Lz'}{\Lo},
\end{split}
\end{equation}
and write $x^\star_{\mathrm{reg}}:=\argmin f_{\mathrm{reg}}$, $\Deltabar':=\Lz'/\Lo^2$, $A:=1+\Lz\Rb^2/\eps$ and $A':=\Lz'/\mu_{\mathrm{reg}}=1+\tfrac{33}{8}\Lz\Rb^2/\eps$, so that $A\le A'\le\tfrac{33}{8}A$. Then $f_{\mathrm{reg}}$ is $\mu_{\mathrm{reg}}$-strongly convex. The \emph{naive} global pair $(\Lz+\mu_{\mathrm{reg}},\Lo)$ for $f_{\mathrm{reg}}$, however, need not satisfy Assumption~\ref{app:ass:gs}: $\grad f(x)$ and the drift $\mu_{\mathrm{reg}}(x-c^{\mathrm{I}})$ may partly cancel, in which case $\norm{\grad f_{\mathrm{reg}}(x)}$ is small while the local Lipschitz constant of $\grad f_{\mathrm{reg}}$ retains the scale $\Lo\norm{\grad f(x)}+\mu_{\mathrm{reg}}$, of order $\Lo\mu_{\mathrm{reg}}\norm{x-c^{\mathrm{I}}}$, which the naive first constant does not dominate. Concretely, $f(x)=\cosh x$ satisfies $|f''|\le1+|f'|$ (hence Assumption~\ref{app:ass:gs} with the pair $(e,e)$ via Remark~\ref{rem:hessian}), but $f_{\mathrm{reg}}(x)=\cosh x+\tfrac12(x-100)^2$ has, at its own minimizer $\bar x$, $\grad f_{\mathrm{reg}}(\bar x)=0$ and $f_{\mathrm{reg}}''(\bar x)=\cosh\bar x+1\approx95.76>e+1$; \eqref{eq:gs} for $f_{\mathrm{reg}}$ with the naive pair $(e+1,e)$, applied at $x=\bar x$, would force $|f_{\mathrm{reg}}'(\bar x+t)-f_{\mathrm{reg}}'(\bar x)|\le(e+1)\,|t|$ for $|t|\le1/e$, i.e.\ $f_{\mathrm{reg}}''(\bar x)\le e+1$---a contradiction. What fails is only the naive first constant: Lemma~\ref{lem:cert-reg} below shows that $f_{\mathrm{reg}}$ always admits a \emph{global} pair with the same second constant $\Lo$, whose first constant is inflated by $\Lo\norm{\grad f(c^{\mathrm{I}})}$ (in the example the center is $c=100$ and the inflation is the enormous $\Lo\sinh 100$, which illustrates precisely why the center's gradient must be certified small); at the actual center, Proposition~\ref{prop:phase1-interp} gives $\norm{\grad f(c^{\mathrm{I}})}\le\tfrac{25}{8}\Gbar$, producing $\Lz'$ of \eqref{eq:freg}. Concretely, Stages~1--2 of Algorithm~\ref{alg:interp} are run on $f_{\mathrm{reg}}$, using the corrected oracle
\begin{equation}
\label{eq:oracle-reg}
\gest_{\mathrm{reg}}(x):=\gest(x)+\mu_{\mathrm{reg}}(x-c^{\mathrm{I}})
\end{equation}
(the correction is deterministic and free of oracle calls), with $\Gbar$ replaced by $\bar G''$ in all stage thresholds, with the Stage-1 transfer step and level target as in Lemma~\ref{lem:hot}, $s:=\tfrac{1}{400\Lo}$ and $\delta_{\mathrm{lvl}}:=\lambda s^2/2$, with cap anchors of batch $B''_{\mathrm{rel}}:=\lceil1.6\cdot10^5c_g^2\rho\ell_\star\rceil$ and update \eqref{eq:anchor-reg}, with the guards $\bar K_{\mathrm{hot}}:=\lceil2950A'\rceil+8$ and $\bar K_{\mathrm{cold}}:=\lceil108A'\rceil+\lceil2\log_2(16(1+A'))\rceil+24$ in place of the corresponding counters of Algorithm~\ref{alg:interp} (the cold counter is that of Algorithm~\ref{alg:interp} with $\mu\mapsto\mu_{\mathrm{reg}}$, $\Lz/\mu\mapsto A'$; the hot one reflects the sharper anchor tracking of Step~3), with the logarithmic umbrella $\ell_\star$ of \eqref{eq:nu-ellstar} evaluated at the variant's counters ($\mu\mapsto\mu_{\mathrm{reg}}$, $\Lz\mapsto\Lz'$), and with the Stage-2 exit threshold replaced by $\Delta\le\min\{\tfrac{\mu_{\mathrm{reg}}}{32\Lo^2},\tfrac\eps2\}$; Stage~3 is skipped.

\begin{lemma}[Global regularized certificate]
\label{lem:cert-reg}
Let $f$ be convex, differentiable and $(\Lz,\Lo)$-smooth in the sense of Assumption~\ref{app:ass:gs}, let $c\in\R^d$ and $\mu>0$, and set $F:=f+\tfrac\mu2\norm{\cdot-c}^2$ and $G_c:=\norm{\grad f(c)}$. Then, for every $x\in\R^d$,
\begin{equation}
\label{eq:cert-reg}
\norm{\grad f(x)}^2\;\le\;\norm{\grad F(x)}^2+G_c^2,
\qquad
\norm{\grad f(x)}\;\le\;\norm{\grad F(x)}+G_c ,
\end{equation}
and $F$ satisfies Assumption~\ref{app:ass:gs} \emph{globally} with the pair $(\Lz+\mu+\Lo\,G_c,\ \Lo)$, on the same admissible range $\norm{y-x}\le1/\Lo$. The squared comparison is sharp: equality holds for affine $f$ at the minimizer of $F$.
\end{lemma}
\begin{proof}
Write $r:=\norm{x-c}$. By gradient monotonicity, $\inner{\grad f(x)}{x-c}\ge\inner{\grad f(c)}{x-c}\ge-G_c\,r$, so expanding $\norm{\grad F(x)}^2=\norm{\grad f(x)}^2+2\mu\inner{\grad f(x)}{x-c}+\mu^2r^2$ gives
\begin{equation*}
\norm{\grad F(x)}^2\;\ge\;\norm{\grad f(x)}^2+\mu r\,(\mu r-2G_c) .
\end{equation*}
If $\mu r\ge2G_c$ the last term is nonnegative and $\norm{\grad f(x)}\le\norm{\grad F(x)}$; if $\mu r<2G_c$ then $\mu r\,(2G_c-\mu r)\le G_c^2$ (the left side is maximized at $\mu r=G_c$), so $\norm{\grad f(x)}^2\le\norm{\grad F(x)}^2+G_c^2$. This proves the first inequality of \eqref{eq:cert-reg}, and the second follows from $\sqrt{a^2+b^2}\le a+b$; for affine $f(x)=\inner{a}{x}$ the minimizer $x=c-a/\mu$ of $F$ satisfies $\grad F(x)=0$ and $\norm{\grad f(x)}^2=G_c^2$, so the squared comparison holds with equality. (The plain triangle inequality $\norm{\grad f(x)}\le\norm{\grad F(x)}+\mu r$, with the worst uncompensated case $\mu r\approx2G_c$ of the case split, yields only $\norm{\grad f(x)}\le\norm{\grad F(x)}+2G_c$ and hence the looser first constant $\Lz+\mu+2\Lo\,G_c$; instantiated at $c=c^{\mathrm{I}}$ this reads $\tfrac{29}{4}\Lz+\mu_{\mathrm{reg}}$ in place of $\tfrac{33}{8}\Lz+\mu_{\mathrm{reg}}$. We use the sharp comparison.) Finally, for $\norm{y-x}\le1/\Lo$, \eqref{eq:gs} applied to $f$ together with \eqref{eq:cert-reg} gives
\begin{equation*}
\begin{split}
\norm{\grad F(y)-\grad F(x)}
&\le\big(\Lz+\Lo\norm{\grad f(x)}\big)\norm{y-x}+\mu\norm{y-x}\le\big(\Lz+\mu+\Lo\,G_c+\Lo\norm{\grad F(x)}\big)\norm{y-x},
\end{split}
\end{equation*}
which is \eqref{eq:gs} for $F$ with the stated pair.
\end{proof}

\paragraph{Step 1: the entry cap is exact.} Since $\grad f_{\mathrm{reg}}(c^{\mathrm{I}})=\grad f(c^{\mathrm{I}})$, the cap produced by Stage~0 transfers verbatim:
\begin{equation}
\label{eq:entry-cap}
\norm{\grad f_{\mathrm{reg}}(c^{\mathrm{I}})}=\norm{\grad f(c^{\mathrm{I}})}\le\tfrac{25}{8}\Gbar\le\tfrac{25}{8}\bar G''\le4\bar G'',
\end{equation}
so $\Gest_1:=4\bar G''$ is valid, exactly as required by Lemma~\ref{lem:hot}, with no measurement and no appeal to any potential. This is the step that fails for the $x^0$-centered regularizer.

\paragraph{Step 2: global smoothness and the regularized noise.} Instantiating Lemma~\ref{lem:cert-reg} at $(c,\mu)=(c^{\mathrm{I}},\mu_{\mathrm{reg}})$, where $G_{c^{\mathrm{I}}}=\norm{\grad f(c^{\mathrm{I}})}\le\tfrac{25}{8}\Gbar$ by Proposition~\ref{prop:phase1-interp}, certifies $f_{\mathrm{reg}}$ \emph{globally} with the pair $(\Lz',\Lo)$ of \eqref{eq:freg}. Two consequences are the only smoothness inputs that Stages~1--2 use: (i) Lemma~\ref{app:lem:ball} applies to $f_{\mathrm{reg}}$ on every ball $\ball{c}{2r}$ with $r\le\tfrac{1}{2\Lo}$, giving two-point $L_c$-smoothness with $L_c:=4(\Lz'+\Lo\Gest)$ whenever $\Gest\ge\norm{\grad f_{\mathrm{reg}}(c)}$, and the descent inequality with modulus $L_c$ along every segment of such a ball---exactly the hypothesis of Proposition~\ref{app:prop:rstm} and Lemma~\ref{lem:varsigma}; and (ii) Lemma~\ref{app:lem:localization} applies to $f_{\mathrm{reg}}$ at every point, with $\Lz\mapsto\Lz'$. Hence Lemmas~\ref{lem:hot} and \ref{lem:cold} apply to $f_{\mathrm{reg}}$ with the substitution $(\Lz,\mu,\Gbar)\mapsto(\Lz',\mu_{\mathrm{reg}},\bar G'')$, and no level-dependent constant, displacement bound, or induction over levels is needed.

The oracle noise of $f_{\mathrm{reg}}$ is that of $f$: the drift correction in \eqref{eq:oracle-reg} is deterministic and cancels in $\zeta_{\mathrm{reg}}(x,\xi):=\gest_{\mathrm{reg}}(x,\xi)-\grad f_{\mathrm{reg}}(x)=\gest(x,\xi)-\grad f(x)$, so, by Lemma~\ref{lem:relnoise} and \eqref{eq:cert-reg},
\begin{equation}
\label{eq:noise-reg}
\norm{\zeta_{\mathrm{reg}}}\le2\sqrt\rho\,\norm{\grad f(x)}
\le2\sqrt\rho\,\norm{\grad f_{\mathrm{reg}}(x)}+2\sqrt\rho\,G_{c^{\mathrm{I}}}
\le2\sqrt\rho\,\norm{\grad f_{\mathrm{reg}}(x)}+\sigma_{\mathrm{add}},
\qquad
\sigma_{\mathrm{add}}:=\tfrac{25}{4}\sqrt\rho\,\Gbar\le\tfrac{25}{4}\sqrt\rho\,\bar G'' .
\end{equation}
Every query of Stages~1--2 therefore sees norm-sub-Gaussian noise that is \emph{multiplicative} in $\norm{\grad f_{\mathrm{reg}}}$---the structure that Lemma~\ref{lem:varsigma} and Remark~\ref{rem:rstm-noise} already accommodate---plus one additive scale $\sigma_{\mathrm{add}}$ that is constant across levels; its cost is accounted for in Steps~3 and~5.

\paragraph{Step 3: consequences for the analysis.}

\emph{Cap anchors.} By \eqref{eq:noise-reg}, the noise at the anchor point $c_{j+1}$ has scale $\nu(c_{j+1})\le2\sqrt\rho\,\norm{\grad f_{\mathrm{reg}}(c_{j+1})}+\sigma_{\mathrm{add}}$ at every level, hot or cold, regardless of how large $\Gest_j$ is relative to $\bar G''$. With batch $B''_{\mathrm{rel}}=\lceil1.6\cdot10^5c_g^2\rho\ell_\star\rceil$, for which $c_g\sqrt{\rho\ell_\star/B''_{\mathrm{rel}}}\le1/400$, the estimate therefore satisfies
\begin{equation*}
\begin{split}
\norm{\gest-\grad f_{\mathrm{reg}}(c_{j+1})}
&\le c_g\,\nu(c_{j+1})\sqrt{\tfrac{\ell_\star}{B''_{\mathrm{rel}}}}
\le\frac{2\,\norm{\grad f_{\mathrm{reg}}(c_{j+1})}}{400}+\frac{25\,\bar G''}{4\cdot400}\\
&=\frac{\norm{\grad f_{\mathrm{reg}}(c_{j+1})}}{200}+\frac{\bar G''}{64}
\le\tfrac18\norm{\grad f_{\mathrm{reg}}(c_{j+1})}+\tfrac{1}{32}\bar G'',
\end{split}
\end{equation*}
so the update
\begin{equation}
\label{eq:anchor-reg}
\Gest^{+}:=\tfrac87\norm{\gest}+\tfrac{\bar G''}{28}
\end{equation}
is valid and tracks the true gradient: on the one hand
$\Gest^+\ge\tfrac87\big(\tfrac{199}{200}\norm{\grad f_{\mathrm{reg}}}-\tfrac{\bar G''}{64}\big)+\tfrac{\bar G''}{28}\ge\norm{\grad f_{\mathrm{reg}}}$,
since $\tfrac87\cdot\tfrac{199}{200}\ge1$ and $\tfrac{\bar G''}{28}\ge\tfrac87\cdot\tfrac{\bar G''}{64}=\tfrac{\bar G''}{56}$; on the other hand
$\Gest^{+}\le\tfrac87\big(\tfrac{201}{200}\norm{\grad f_{\mathrm{reg}}}+\tfrac{\bar G''}{64}\big)+\tfrac{\bar G''}{28}\le\tfrac97\norm{\grad f_{\mathrm{reg}}}+\tfrac{\bar G''}{14}$.

\emph{Hot stage.} Lemma~\ref{lem:hot} applies to $f_{\mathrm{reg}}$ with $(\Lz,\mu,\Gbar)\mapsto(\Lz',\mu_{\mathrm{reg}},\bar G'')$: while $\Gest>2\bar G''$ one has $\Lz'=\Lo\bar G''\le\Lo\Gest/2$, so $L_c=4(\Lz'+\Lo\Gest)\le6\Lo\Gest$, $L_F=L_c+\lambda\le10\Lo\Gest$ and $\kappa=L_F/\lambda\le\tfrac52$---exactly Lemma~\ref{lem:hot}(a). The prox remains unconstrained: $\norm{\xhat-c}=\norm{\grad f_{\mathrm{reg}}(\xhat)}/\lambda\le\Gest/\lambda=\tfrac{1}{4\Lo}<\rsafe$ by Lemma~\ref{app:lem:prox-grad}, so $H_0=\Gest^2/\lambda$ is certified as in Remark~\ref{rem:proj}, and $\delta_{\mathrm{lvl}}=\lambda s^2/2$ gives $\norm{c^+-\xhat}\le s$. The naive cap control of Lemma~\ref{lem:hot}(b) suffices as stated: $\norm{\grad f_{\mathrm{reg}}(c^+)}\le\norm{\grad f_{\mathrm{reg}}(\xhat)}+L_cs$ with $L_cs\le6\Lo\Gest\cdot\tfrac{1}{400\Lo}=0.015\,\Gest$. The only change to the arithmetic of Lemma~\ref{lem:hot}(b) is the anchor slack $\tfrac{3}{56}\bar G''\le\tfrac{3}{112}\Gest$ in the tracking bound of \eqref{eq:anchor-reg}: with $\Psi:=f_{\mathrm{reg}}(c)-f^\star_{\mathrm{reg}}+\tfrac{\Gest}{100\Lo}$, a productive level ($\norm{\xhat-c}\ge\tfrac{1}{16\Lo}$) drops the objective by $\tfrac{\Gest}{128\Lo}-\tfrac{\Gest}{400\Lo}-\tfrac{\Gest}{53000\Lo}\ge\tfrac{\Gest}{189\Lo}$ (prox decrease minus transfer along $[\xhat,c^+]$, as in Lemma~\ref{lem:hot}(b)) and raises the cap to $\Gest^+\le\tfrac87\cdot\tfrac{201}{200}\cdot1.015\,\Gest+\tfrac{3}{112}\Gest\le1.193\,\Gest$, so $\Delta\Psi\le-\tfrac{\Gest}{189\Lo}+\tfrac{0.193\,\Gest}{100\Lo}\le-\tfrac{\Gest}{300\Lo}$; a collapsing level raises the objective by at most $\tfrac{\Gest}{1600\Lo}+\tfrac{\Gest}{53000\Lo}\le\tfrac{\Gest}{1550\Lo}$ while $\Gest^+\le\tfrac87\cdot\tfrac{201}{200}\cdot0.265\,\Gest+\tfrac{3}{112}\Gest\le0.332\,\Gest$, so $\Delta\Psi\le\tfrac{\Gest}{1550\Lo}-\tfrac{0.668\,\Gest}{100\Lo}\le-\tfrac{\Gest}{166\Lo}$. In both cases
\begin{equation}
\label{eq:psi-drop-reg}
\begin{split}
\Delta\Psi&\;\le\;-\frac{\Gest}{300\Lo}\;\le\;-\frac{2\bar G''}{300\Lo}\;=\;-\frac{\Deltabar'}{150},
\qquad\text{while}\\
\Psi_0&\le\frac{2\,\norm{\grad f_{\mathrm{reg}}(c_1)}^2}{\mu_{\mathrm{reg}}}+\frac{4\bar G''}{100\Lo}
\le\frac{625}{32}\,\frac{\bar G''^2}{\mu_{\mathrm{reg}}}+\frac{\Deltabar'}{25}\le19.6\,\Deltabar'A'+\frac{\Deltabar'}{25}
\end{split}
\end{equation}
(using $\bar G''^2/\mu_{\mathrm{reg}}=\Deltabar'\cdot\Lz'/\mu_{\mathrm{reg}}=\Deltabar'A'$ and \eqref{eq:entry-cap}), so Stage~1 executes at most $150\,\Psi_0/\Deltabar'\le2940A'+6$ levels, exits with $\Gest\le2\bar G''$, and the guard $\bar K_{\mathrm{hot}}$ never triggers. Per level, the cost is that of Lemma~\ref{lem:hot}(d), $\tO(\rho)$, plus the $\sigma_{\mathrm{add}}$-piece $O\big(C_b\ell_\star\sigma_{\mathrm{add}}^2/(\lambda\delta_{\mathrm{lvl}})\big)=O(\rho)$, since $\lambda\delta_{\mathrm{lvl}}=\lambda^2s^2/2=\Gest^2/20000$ and $\sigma_{\mathrm{add}}^2=O(\rho\bar G''^2)$ with $\Gest>2\bar G''$; the anchor batch $B''_{\mathrm{rel}}=O(\rho\ell_\star)$ is of the same order. Hence each hot level costs $\tO(\rho)$, and the hot stage costs $\tO(\rho A')$ in total.

\emph{Cold stage.} Lemma~\ref{lem:cold}(a)--(b) goes through for $f_{\mathrm{reg}}$ verbatim with $(\Lz,\mu,\Gbar)\mapsto(\Lz',\mu_{\mathrm{reg}},\bar G'')$: the argument uses only convexity, $\mu_{\mathrm{reg}}$-strong convexity, the (smoothness-free) transfer inequality \eqref{eq:transfer}, the \RSTM{} guarantee, and the schedule, which is unchanged; the hop count of Lemma~\ref{lem:cold}(b), a schedule fact, becomes $108A'+2\log_2(16(1+A'))+22$ via $\Delta_1\le8\bar G''^2/\mu_{\mathrm{reg}}$ and $\Delta_{\mathrm{stop}}=\mu_{\mathrm{reg}}/(32\Lo^2)$, and the guard $\bar K_{\mathrm{cold}}$ never triggers. The gradient-dependent constants follow from the now \emph{global} localization of Step~2, $\norm{\grad f_{\mathrm{reg}}(c_j)}\le2\Lo\Delta_j+\sqrt{2\Lz'\Delta_j}$, and the anchor tracking of \eqref{eq:anchor-reg}:
\begin{equation}
\label{eq:cold-cap-reg}
\Gest_j\;\le\;\tfrac97\,\norm{\grad f_{\mathrm{reg}}(c_j)}+\tfrac{1}{14}\bar G''
\;\le\;\tfrac97\big(2\Lo\Delta_j+\sqrt{2\Lz'\Delta_j}\big)+\tfrac{1}{14}\bar G''
\;=\;O\big(\Lo\Delta_j+\sqrt{\Lz'\Delta_j}+\bar G''\big),
\end{equation}
hence $L_{c,j}\le4\Lz'+4\Lo\Gest_j=O\big(\Lz'+\Lo^2\Delta_j+\Lo\sqrt{\Lz'\Delta_j}\,\big)$ (using $\Lo\bar G''=\Lz'$), and the condition-number computation of Lemma~\ref{lem:cold}(c) with $(\Lz,\mu)\mapsto(\Lz',\mu_{\mathrm{reg}})$ gives $\kappa_j:=L_{F,j}/\lambda_j=O(A')$ in both cold regimes.

\paragraph{Step 4: exit threshold and output accuracy.} The cold loop exits with $\Delta\le\min\{\tfrac{\mu_{\mathrm{reg}}}{32\Lo^2},\tfrac\eps2\}$, and by \eqref{eq:wlog}, $32\Lo^2\Rb^2\ge2$, hence $\tfrac{\mu_{\mathrm{reg}}}{32\Lo^2}=\tfrac{\eps}{32\Lo^2\Rb^2}\le\tfrac\eps2$: the two thresholds coincide, the loop is exactly that of Algorithm~\ref{alg:interp} with the substitution of Step~2, and no boundary accounting of the hop count is needed. At the exit the certified gap is $f_{\mathrm{reg}}(\hat x)-f^\star_{\mathrm{reg}}\le\tfrac\eps2$. To transfer it to $f$, note that optimality of $x^\star_{\mathrm{reg}}$ for $f_{\mathrm{reg}}$ against the feasible point $x^\star$ gives $\tfrac{\mu_{\mathrm{reg}}}{2}\norm{x^\star_{\mathrm{reg}}-c^{\mathrm{I}}}^2\le f(x^\star)-f(x^\star_{\mathrm{reg}})+\tfrac{\mu_{\mathrm{reg}}}{2}\norm{x^\star-c^{\mathrm{I}}}^2\le\tfrac{\mu_{\mathrm{reg}}}{2}\norm{x^\star-c^{\mathrm{I}}}^2$, so, with Proposition~\ref{app:prop:phase1}(iii),
\begin{equation}
\label{eq:xstar-reg}
\norm{x^\star_{\mathrm{reg}}-c^{\mathrm{I}}}\;\le\;\norm{x^\star-c^{\mathrm{I}}}\;\le\;\tfrac32R_0=\tfrac{\Rb}{2} .
\end{equation}
Therefore
\begin{equation}
\label{eq:reg-transfer}
\begin{split}
f(\hat x)-f^\star&\;\le\;\big(f_{\mathrm{reg}}(\hat x)-f_{\mathrm{reg}}(x^\star_{\mathrm{reg}})\big)+\tfrac{\mu_{\mathrm{reg}}}{2}\norm{x^\star-c^{\mathrm{I}}}^2\le\;\tfrac\eps2+\tfrac{\eps}{2\Rb^2}\cdot\tfrac{\Rb^2}{4}\;=\;\tfrac{5\eps}{8}\;\le\;\eps .
\end{split}
\end{equation}

\paragraph{Step 5: cost.} Stage~0 costs $\tO(\rho(1+\Lo R_0))=\tO(\rho\kbar)$ (Proposition~\ref{prop:phase1-interp}), and the hot stage costs $\tO(\rho A')$ (Step~3). For the cold stage, the multiplicative pieces are those of Lemma~\ref{lem:cold}(c) with $(\Lz,\mu)\mapsto(\Lz',\mu_{\mathrm{reg}})$ and sum to $\tO(\rho A'^3)$, with two modifications. First, the cap \eqref{eq:cold-cap-reg} carries the anchor slack: $\Gest_j^2=O(\Lo^2\Delta_j^2+\Lz'\Delta_j+\bar G''^2)$ against $\lambda_j\delta_j=\tfrac{\mu_{\mathrm{reg}}\Delta_j}{1024}$ (regime $\theta_j<\tfrac12$); the $\Lo^2\Delta_j$-part is dominated by the top band $\Delta_1\le8\bar G''^2/\mu_{\mathrm{reg}}$ and gives $\tO(\rho A'^3)$, the $\Lz'$-part is $O(\rho A')$ per hop and totals $\tO(\rho A'^2)$ over the $O(A')$ hops, and the $\bar G''^2/\Delta_j$-part is dominated by the bottom band. The bottom band is $\Delta_{\mathrm{stop}}=\tfrac{\mu_{\mathrm{reg}}}{32\Lo^2}$---not $\tfrac\eps2$: under \eqref{eq:wlog} one has $\Delta_{\mathrm{stop}}\le\tfrac\eps2$ (Step~4), and the distinction affects only absolute constants, since $\Lo\bar G''=\Lz'$ collapses it to the factor $\sqrt{32}$:
\begin{equation}
\label{eq:cold-sum}
\begin{split}
&\tO\Big(\rho\,\bar G''^2\,\frac{\Lo}{\mu_{\mathrm{reg}}^{3/2}\sqrt{\Delta_{\mathrm{stop}}}}\Big)
\;=\;\tO\Big(\rho\,\bar G''^2\,\frac{\sqrt{32}\,\Lo^2}{\mu_{\mathrm{reg}}^{2}}\Big)=\;\tO\Big(\rho\,\Big(\frac{\Lo\bar G''}{\mu_{\mathrm{reg}}}\Big)^{2}\Big)
\;=\;\tO\big(\rho A'^{2}\big),
\end{split}
\end{equation}
where the first expression is the per-hop cost $O\big(\rho\bar G''^2\ell_\star/(\mu_{\mathrm{reg}}\Delta_j)\big)$ multiplied by the per-band hop count $(16\ln2)\,\Lo\sqrt{\Delta'/\mu_{\mathrm{reg}}}+1$ of Lemma~\ref{lem:cold}(b) and summed over the dyadic bands (the sum is dominated by the bottom band; the $+1$ of the band count contributes a geometric series over the bands, dominated by the bottom band---at most twice the largest per-hop term---and is equally absorbed by $\tO$). Second, the additive-noise piece: per hop it costs $O\big(C_b\ell_\star\sigma_{\mathrm{add}}^2/(\lambda_j\delta_j)\big)$, which has exactly the $\bar G''^2$-form just summed, since $\sigma_{\mathrm{add}}^2=O(\rho\bar G''^2)$ is constant across levels; it therefore contributes the same $\tO(\rho A'^2)$ in the regime $\theta_j<\tfrac12$, and in the regime $\theta_j=\tfrac12$, where $\lambda_j\delta_j=\Lo^2\Delta_j^2/512$, the geometric sum of $1/\Delta_j^2$ over $\Delta_j\ge\Delta_{\mathrm{stop}}$ is dominated by $\Delta_j=\Delta_{\mathrm{stop}}$, and $\bar G''^2/(\Lo^2\Delta_{\mathrm{stop}}^2)=1024\,(\Lo\bar G''/\mu_{\mathrm{reg}})^2=O(A'^2)$ over $O(1)$ such hops. The deterministic part $\sum_j25K_j\sqrt{\kappa_j}=\tO(A'^{3/2})$ and the anchor batches $O(A')\cdot B''_{\mathrm{rel}}=\tO(\rho A')$ are dominated by the terms above. The cold stage therefore costs $\tO\big(\rho A'^3+\rho A'^2\big)$. Adding the three groups and using $A\le A'\le\tfrac{33}{8}A$ gives
\begin{equation}
\label{eq:interp-convex-total}
\begin{split}
N&=\tO\big(\rho\,\kbar+\rho\,A'+\rho\,A'^{2}+\rho\,A'^{3}\big)
\;=\;\tO\big(\rho\,\kbar+\rho\,A'^{3}\big)=\tO\Big(\rho\,\kbar+\rho\Big(1+\tfrac{\Lz\Rb^2}{\eps}\Big)^{3}\Big),
\end{split}
\end{equation}
which is the statement of Theorem~\ref{thm:interp-convex}: the only surviving $\kbar$-factor is the Stage-0 entry burn-in, shared with Theorem~\ref{thm:interp-strong}, and under the default certificate of Lemma~\ref{lem:gap0} that term is $\tO\big(\rho(1+q)^{4}\big)$ with $q:=\Lo R_0$, exactly as in Theorem~\ref{thm:interp-strong} (Proposition~\ref{prop:phase1-interp}).

\hfill
$\square$

\subsubsection{Towards in-expectation strong growth}
\begin{remark}[Towards in-expectation strong growth]
\label{rem:mom}
If only \eqref{eq:sgc} holds (no a.s.\ bound), single-sample noise need not have sub-Gaussian tails, and our proofs do \emph{not} extend by a drop-in replacement of batch means. Two distinct uses of estimates must be separated. (i) \emph{Measurements} (cap anchors, the two-step selection): these only require a high-probability tail bound on a point estimate of a fixed vector, and a geometric median-of-means (MoM) estimator \citep{minsker2015geometric} over $O(\log(1/\alpha'))$ groups delivers sub-Gaussian-type confidence from the second-moment bound $\E\norm{g-\grad f}^2\le\rho\norm{\grad f}^2$; all measurement steps of Algorithm~\ref{alg:interp} therefore survive under \eqref{eq:sgc} alone, at the cost of absolute constants. (ii) \emph{Inner-solver gradients}: here the obstruction is real. Median-of-means is a \emph{biased} estimator of the gradient, while the analysis of Proposition~\ref{prop:stm} decomposes the error as a martingale-difference term (Lemma~\ref{lem:azuma}) plus a squared-norm term, and the martingale part relies on conditional unbiasedness of $\gest_t$ given the past; a per-iteration MoM estimate breaks this centering, and controlling the resulting bias along an accelerated trajectory would require exactly the kind of high-relative-accuracy estimates whose avoidance is the point of the design. Plain batch means, on the other hand, have only the second-moment bound under \eqref{eq:sgc}, which yields in-expectation but not high-probability control of \eqref{eq:stm-master}. We therefore state our results under the a.s.\ variant and leave the in-expectation extension of Theorems~\ref{thm:interp-strong}--\ref{thm:interp-convex} under \eqref{eq:sgc} alone as an open technical question; we expect a restart-and-select outer wrapper around in-expectation inner runs to work, but a complete argument must handle the bias--variance interplay at the extrapolation sequence and we do not claim it here.
\end{remark}

\subsection{The $\kappa^{3/2}$ obstruction behind Open Problem~\ref{op:interp}}
\label{app:interp-obstruction}
\begin{remark}[Noise floor inside momentum epochs]
\label{rem:obstruction}
Suppose one replaces the interior SGD by an accelerated inner solver, e.g.\ \RSTM{} on $f$ itself over $\ball{c}{\rsafe}$ with $\lambda\asymp\mu$ and $\kappa:=L'/\mu$. By Lemma~\ref{lem:varsigma} with $\Gest\asymp0$ at the optimum, the noise scale on epoch $k$ is $\nu_k^2\asymp\rho L'^2\bar\Delta_k/\mu$, so the batch rule of Remark~\ref{rem:rstm-noise} forces $b_k\asymp C_b\rho\ell\,\tfrac{L'^2/\mu}{\sqrt{\mu L'}}=C_b\rho\ell\,\kappa^{3/2}$ \emph{per iteration}: inside a momentum epoch the extrapolation sequence wanders a distance $\Theta(\sqrt{\bar\Delta_k/\mu})$ from the optimum, where the strong-growth noise is $\Theta(\sqrt\rho\,L'\sqrt{\bar\Delta_k/\mu})$---a factor $\Theta(\sqrt\kappa)$ larger than the noise SGD sees at comparable suboptimality---and killing it inside condition \eqref{eq:bk} costs the $\kappa^{3/2}$ batch. The resulting complexity, $\sqrt\kappa\,\logp(\Dcert/\eps)\cdot\rho\kappa^{3/2}=\rho\kappa^2\logp(\Dcert/\eps)$, is worse than the $\rho\kappa\logp(\Dcert/\eps)$ of plain SGD: the $\sqrt\kappa$ gain of momentum is exactly erased. Any resolution must either certify the extrapolation sequence at $o(\sqrt\kappa)$ relative noise or avoid epoch-uniform batching; we leave this as Open Problem~\ref{op:interp}.
\end{remark}

\section{Complete experimental details}
\label{app:experiments}
\label{app:sec:experiments}

This section collects the complete experimental protocol of the main paper: the certified instances and their constants, the oracle and noise calibration, the baselines and tuning protocol, the practical implementation with its component-by-component deviations from the analyzed algorithms, all diagnostic figures, the matched ablations, the sensitivity table, the projection and wall-clock accounting, and the exact-theory cost estimate. It closes with a schedule-faithful sanity experiment on certified small instances; additive-Gaussian batch means are exact in distribution, while the strong-growth check explicitly discloses its Gaussian approximation for very large categorical batches. Figure~\ref{app:fig:experiments} of the main paper is reproduced here for self-containedness.

\begin{figure}[t]
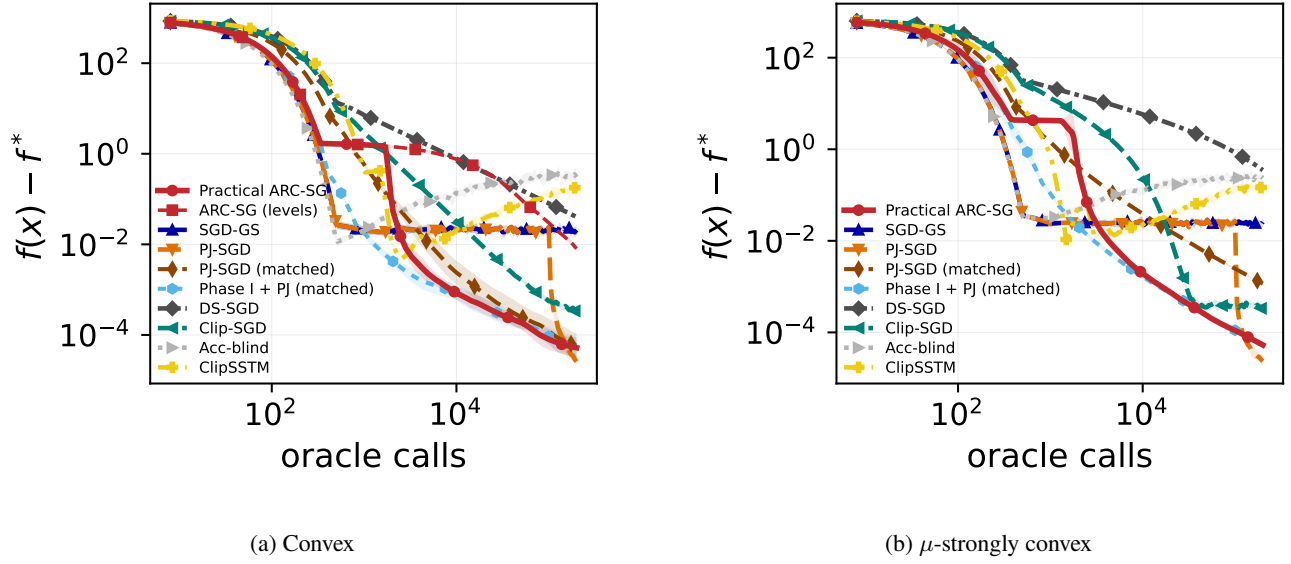

\centering
\begin{subfigure}[t]{0.49\columnwidth}
\includegraphics[width=\textwidth]{figures/exp_convex.pdf}
\caption{Convex}
\end{subfigure}\hfill
\begin{subfigure}[t]{0.49\columnwidth}
\includegraphics[width=\textwidth]{figures/exp_strongly.pdf}
\caption{$\mu$-strongly convex}
\end{subfigure}
\caption{Objective gap versus stochastic oracle calls (log--log; geometric mean over 10 independently regenerated instance--noise runs; the shaded multiplicative $\pm$ one standard deviation band depicts \emph{variability across the 10 runs}, not uncertainty of the mean; gaps floored at $10^{-16}$). (a)~Convex family with quartic flat directions: \textsf{Practical ARC-SG} reaches $5.2\times10^{-5}$, about $6.5\times$ below fully tuned \textsf{Clip-SGD} ($3.4\times10^{-4}$) and $400\times$ below \textsf{SGD-GS} ($2.1\times10^{-2}$), and is still improving at the budget end; \textsf{PJ-SGD} (the same $(\Lz,\Lo)$-stepsize with Polyak--Juditsky averaging) reaches $2.6\times10^{-5}$, the best endpoint accuracy; the \textsf{ARC-SG (levels)} ablation shows the accelerated levels alone reach $8.0\times10^{-3}$, so most of the late-stage gain comes from the interior averaging stage, which Theorems~\ref{app:thm:convex}--\ref{app:thm:strongly} do not cover; the matched ablations \textsf{PJ-SGD (matched)} and \textsf{Phase I + PJ (matched)} decompose the method further (Section~\ref{app:sec:experiments}). (b)~Strongly convex ($\mu=0.245$ exact): the same ordering with almost identical numbers ($5.3\times10^{-5}$ versus $3.5\times10^{-4}$ and $2.6\times10^{-2}$; \textsf{PJ-SGD} $2.4\times10^{-5}$).}
\label{app:fig:experiments}
\end{figure}

\paragraph{Instances.} To obtain \emph{global} $(\Lz,\Lo)$-smoothness with known constants we use objectives that are separable over rotated coordinates $y=Q^\top x$ for an orthogonal $Q=[q_1,\dots,q_d]$, with a random root vector $r$; then the Hessian is diagonal in the rotated basis and it suffices to certify each coordinate term $g_k$ against its own derivative, since $|g_k'(y_k)|\le\norm{\grad f(x)}$. Two term types are used. \emph{Cosh terms} $\sum_s\cosh(C_{k,s}(y_k-r_k))$: within a coordinate all residuals share a sign, so, using $\cosh t\le1+|\sinh t|$,
\[
\begin{split}
g_k''(y_k)
&=\sum_sC_{k,s}^2\cosh(\cdot) \le\ \sum_sC_{k,s}^2+\big(\max_sC_{k,s}\big)\Big|\sum_sC_{k,s}\sinh(\cdot)\Big| =\ \sum_sC_{k,s}^2+\big(\max_sC_{k,s}\big)\,|g_k'(y_k)| .
\end{split}
\]
\emph{Quartic terms} $g_k(u)=(c\,u)^4$: from $g_k''=12c^4u^2$ and $g_k'=4c^4u^3$, splitting at $|cu|=1$ gives $g_k''\le12c^2+3c\,|g_k'|$, i.e.\ each quartic coordinate is exactly $(12c^2,3c)$-generalized smooth; its Hessian \emph{vanishes} at the root. These are \emph{Hessian-form} constants only. Assumption~\ref{app:ass:gs}, however, is the finite-difference condition \eqref{eq:gs} on the range $\norm{y-x}\le1/\Lo$, and the Hessian-form pairs do \emph{not} satisfy it: on the $\cosh$ champion coordinate ($C_s\in\{0.6,0.8,1.0\}$, pair $(2,1)$, admissible range $|y-x|\le1$),
\[
\begin{split}
|g'(1)-g'(0)|&=0.6\sinh(0.6)+0.8\sinh(0.8)+\sinh(1)=2.2677>2=\big(\Lz+\Lo|g'(0)|\big)\cdot|1-0|,
\end{split}
\]
and on an individual term $20\cosh(u)$ of the overparameterized family of Appendix~\ref{app:interp-experiments} (pair $(20,1)$), $|g'(1)-g'(0)|=20\sinh(1)=23.504>20$. The exact finite-difference frontier is as follows.

\begin{lemma}[Certified pairs for the experimental families]
\label{lem:exp-cert}
(i) Let the finite, nonempty family of scales satisfy $C_s\ge0$ and set $B:=\max_sC_s$. If $B=0$, the $1$-D sum $g(u)=\sum_s\cosh(C_su)$ is constant and satisfies \eqref{eq:gs} for every $\Lz>0$, $\Lo\ge0$. If $B>0$, no pair with $\Lo=0$ is admissible; for $\Lo>0$, the sum satisfies \eqref{eq:gs} if and only if $\Lo\ge B/\ln2$ and $\Lz\ge\Lo\sum_sC_s\sinh(C_s/\Lo)$. (Equivalently, arbitrary real scales may be replaced by $|C_s|$, since $\cosh(C_su)=\cosh(|C_s|u)$.) (ii) $g(u)=(cu)^4$ satisfies \eqref{eq:gs} if and only if $\Lz\Lo^2\ge(24+16\sqrt2)\,c^4$. (iii) If every coordinate term of $f(x)=\sum_kg_k((Q^\top x)_k)$ with orthogonal $Q$ satisfies \eqref{eq:gs} with the same pair $(\Lz,\Lo)$, then so does $f$.
\end{lemma}
\begin{proof}
(i) If $B=0$, then $g$ is constant and the claim is immediate, so assume $B>0$. Any admissible pair must then have $\Lo>0$, because $\Lo=0$ would require the nonconstant cosh sum to have a globally Lipschitz derivative, whereas its second derivative is unbounded. \emph{Necessity.} Taking $x=0$, $y=1/\Lo$ in \eqref{eq:gs} and using $g'(0)=0$ gives $|g'(y)-g'(x)|=g'(1/\Lo)=\sum_sC_s\sinh(C_s/\Lo)$, while $(\Lz+\Lo|g'(x)|)\,|y-x|=\Lz/\Lo$, so $\Lz\ge\Lo\sum_sC_s\sinh(C_s/\Lo)$. For the second condition take $y=x+1/\Lo$ and let $x\to+\infty$: by the addition formula, $\sinh(C_sy)-\sinh(C_sx)=\tfrac12e^{C_sx}\big(e^{C_s/\Lo}-1\big)+O(e^{-C_sx})$. Writing $M\ge1$ for the multiplicity of the maximal scale $C_s=B$, the maximal-scale terms dominate both displays and give
\[
|g'(y)-g'(x)|=\tfrac12MBe^{Bx}\big(e^{B/\Lo}-1\big)+o(e^{Bx}),\qquad
\Lo|g'(x)|=\tfrac12\Lo MBe^{Bx}+o(e^{Bx}),
\]
while $|y-x|=1/\Lo$ and $\Lz$ is negligible against $e^{Bx}$. The ratio $|g'(y)-g'(x)|/[(\Lz+\Lo|g'(x)|)\,|y-x|]$ therefore tends to $e^{B/\Lo}-1$, forcing $e^{B/\Lo}-1\le1$, i.e.\ $\Lo\ge B/\ln2$.

\emph{Sufficiency.} \emph{Step 1 (sign sharing).} Every term $C_s\sinh(C_su)$ has the sign of $u$, so $|g'(u)|=\sum_sC_s|\sinh(C_su)|$.
\emph{Step 2 (a scalar inequality).} For all real $X,\tau$,
\begin{equation}
\label{eq:sinh-increment}
|\sinh(X+\tau)-\sinh X|\ \le\ \sinh|\tau|+\big(e^{|\tau|}-1\big)\,|\sinh X|.
\end{equation}
Both sides are invariant under $(X,\tau)\mapsto(-X,-\tau)$, so assume $X\ge0$. \emph{Case $\tau\ge0$:} expanding $\sinh(X+\tau)=\sinh X\cosh\tau+\cosh X\sinh\tau$ and using $e^\tau-\cosh\tau=\sinh\tau$ and $\cosh X-\sinh X=e^{-X}$,
\[
\mathrm{RHS}-\mathrm{LHS}\ =\ \sinh\tau\,(1-\cosh X)+\sinh X\,(e^\tau-\cosh\tau)\ =\ \sinh\tau\,\big(1-e^{-X}\big)\ \ge\ 0 .
\]
\emph{Case $\tau=-t$, $0\le t\le X$:} the left side is $\sinh X-\sinh(X-t)=2\cosh(X-\tfrac t2)\sinh\tfrac t2\le\cosh X\sinh t$ (since $0\le X-\tfrac t2\le X$ gives $\cosh(X-\tfrac t2)\le\cosh X$, and $2\sinh\tfrac t2\le2\sinh\tfrac t2\cosh\tfrac t2=\sinh t$), and $\cosh X\sinh t\le\sinh t+(e^t-1)\sinh X$ because $\cosh X-1\le\sinh X$ (equivalently $e^{-X}\le1$) and $\sinh t\le e^t-1$ (equivalently $\cosh t\ge1$). \emph{Case $\tau=-t$, $t>X$:} then $|\sinh(X-t)-\sinh X|=\sinh(t-X)+\sinh X\le\sinh\big((t-X)+X\big)=\sinh t$ by the superadditivity of $\sinh$ on $[0,\infty)$ ($\sinh$ is convex with $\sinh0=0$, so $\sinh(a+b)\ge\sinh a+\sinh b$ for $a,b\ge0$), and $\sinh t$ is absorbed by the right side of \eqref{eq:sinh-increment}.
\emph{Step 3 (sub-bounds for $|h|\le1/\Lo$, $C\le B$).} Since $s\mapsto\sinh s/s$ is increasing on $s>0$ and $C|h|\le C/\Lo$,
$\sinh(C|h|)\le\Lo|h|\sinh(C/\Lo)$;
and since $t\mapsto e^t-1$ is convex with value $0$ at the origin, $e^{C|h|}-1\le\Lo|h|\,\big(e^{C/\Lo}-1\big)\le\Lo|h|$, using $e^{C/\Lo}\le e^{B/\Lo}\le2$.
\emph{Step 4 (assembly).} Applying \eqref{eq:sinh-increment} termwise with $X=C_su$, $\tau=C_sh$, summing, and using Steps 1 and 3,
\[
\begin{split}
|g'(u+h)-g'(u)|
&\ \le\ \sum_sC_s\Big[\sinh(C_s|h|)+\big(e^{C_s|h|}-1\big)\,|\sinh(C_su)|\Big]\\
&\ \le\ \Lo|h|\sum_sC_s\Big[\sinh(C_s/\Lo)+|\sinh(C_su)|\Big]\\
&\ =\ \Lo|h|\Big[\sum_sC_s\sinh(C_s/\Lo)+|g'(u)|\Big]
\ \le\ \big(\Lz+\Lo|g'(u)|\big)\,|h|,
\end{split}
\]
by the choice of $\Lz$; this is \eqref{eq:gs} at the pair $(x,y)=(u,u+h)$, with $|h|\le1/\Lo$ arbitrary.

(ii) If $c=0$, then $g$ is constant and the claim is immediate. Assume $c\ne0$. At $\Lo=0$ the derivative of this quartic is not globally Lipschitz, and the displayed algebraic condition also fails, so it remains to consider $\Lo>0$. Write $y=x+h$. Since $(x+h)^3-x^3=h\,q(h)$ with $q(h):=3x^2+3xh+h^2$, and $q(h)\ge0$ for all real $h$ (its discriminant is $9x^2-12x^2=-3x^2\le0$), one has $|g'(y)-g'(x)|=4c^4|h|\,q(h)$, while $|g'(x)|=4c^4|x|^3$. The quadratic $q$ is convex in $h$, so its maximum over $|h|\le1/\Lo$ is attained at an endpoint; for $x\ge0$ the maximizer is $h=+1/\Lo$, and by the symmetry $(x,h)\mapsto(-x,-h)$ the binding requirement is
\[
\Lz\ \ge\ 4c^4\max_{x\ge0}\Phi(x),\qquad
\Phi(x):=3x^2+\tfrac{3x}{\Lo}+\tfrac{1}{\Lo^2}-\Lo x^3 .
\]
Now $\Phi'(x)=6x+\tfrac3\Lo-3\Lo x^2$ vanishes at $x^\star=\tfrac{1+\sqrt2}{\Lo}$ (the positive root of $\Lo x^2-2x-\tfrac1\Lo=0$), which is the global maximizer on $[0,\infty)$: $\Phi''(x^\star)=-6\sqrt2<0$, the value $\Phi(0)=\tfrac{1}{\Lo^2}$ is smaller, and $\Phi(x)\to-\infty$ as $x\to\infty$. With $t:=\Lo x^\star=1+\sqrt2$ (so $t^2=3+2\sqrt2$ and $t^3=7+5\sqrt2$),
\[
\Lo^2\Phi(x^\star)\ =\ 3t^2+3t+1-t^3\ =\ (9+6\sqrt2)+(3+3\sqrt2)+1-(7+5\sqrt2)\ =\ 6+4\sqrt2,
\]
giving the requirement $\Lz\Lo^2\ge4(6+4\sqrt2)\,c^4=(24+16\sqrt2)\,c^4$. Necessity is attained by the witness $(x,y)=(x^\star,\,x^\star+1/\Lo)$, for which equality holds in every step above.

(iii) Let $u=Q^\top x$, $v=Q^\top y$; since $Q$ is orthogonal, $|v_k-u_k|\le\norm{v-u}=\norm{y-x}\le1/\Lo$, so the per-coordinate assumption applies to each pair $(u_k,v_k)$, and
\[
\begin{split}
\norm{\grad f(y)-\grad f(x)}^2
&\ =\ \sum_k|g_k'(v_k)-g_k'(u_k)|^2
\ \le\ \sum_k\big(\Lz+\Lo|g_k'(u_k)|\big)^2|v_k-u_k|^2\\
&\ \le\ \max_k\big(\Lz+\Lo|g_k'(u_k)|\big)^2\norm{v-u}^2
\ \le\ \big(\Lz+\Lo\norm{\grad f(x)}\big)^2\norm{y-x}^2,
\end{split}
\]
using $|g_k'(u_k)|\le\norm{\grad f(x)}$ (each coordinate of $g'(u)$ is, up to the orthogonal rotation, a coordinate of $\grad f(x)$) and the monotonicity of $t\mapsto(\Lz+\Lo t)^2$ on $t\ge0$.
\end{proof}

Applied to the champion scales $\{0.6,0.8,1.0\}$ ($B=1$), part (i) gives the tight pair
\[
\begin{split}
(\Lz,\Lo)=\Big(&\tfrac{1}{\ln2}\big[0.6\sinh(0.6\ln2)+0.8\sinh(0.8\ln2)+\sinh(\ln2)\big],\ \tfrac{1}{\ln2}\Big)=(2.1258,\ 1.4427),
\end{split}
\]
with admissible radius $1/\Lo=\ln2\approx0.693$; the required $\Lo$ scales as $c_k/\ln2$ and the required $\Lz$ is increasing in the base scale $c_k\le1$, so the same pair certifies every $\cosh$ coordinate, and for the quartic coordinates $\Lz\Lo^2\approx4.42\ge(24+16\sqrt2)/81\approx0.5756$, so part (ii) is non-binding (even their Hessian-form pair $(\tfrac43,1)$ clears the frontier). By part (iii) the pair lifts to the full objectives. The generic conversion of Remark~\ref{rem:hessian} would certify only the looser pairs $(2e,e)$ or $(2/\ln2,1/\ln2)$; the tight pairs above are used in all reported runs. (a) \emph{Convex:} $d=40$; coordinates $1$--$30$ are cosh with base scales $c_k$ geometrically spaced in $[0.35,1]$ times inner scales $\{0.6,0.8,1.0\}$ ($90$ components, the frontier driven by the $c_k=1$ coordinate), coordinates $31$--$40$ are quartic with $c=\tfrac13$. Hence $\Lz=2.1258$, $\Lo=1/\ln2\approx1.4427$, both tight for the champion coordinate; $x^\star=Qr$ and $f^\star=90$ (the number of cosh components) exactly; $f$ is convex with a unique minimizer but not strongly convex, globally or locally at $x^\star$. The start is $x^0=x^\star+t\,q_{k_\star}$ along a quartic coordinate, with $t$ found by bisection so that $\norm{\grad f(x^0)}=150\,\Gbar$ ($R_0\approx16.5$). (b) \emph{Strongly convex:} $d=40$, all coordinates cosh with three scales ($120$ components); since $g_k''\ge\sum_sC_{k,s}^2$ pointwise, $f$ is globally $\mu$-strongly convex with the exact constant $\mu=\min_k\sum_sC_{k,s}^2=2\cdot0.35^2=0.245$, with the same certified pair $(\Lz,\Lo)=(2.1258,1/\ln2)$ as in (a); no quadratic term is added, so $\norm{\grad^2f(x^\star)}\le\Lz$ and all constants supplied to the methods are simultaneously valid; $f^\star=120$ exactly. The interpolating family used in Appendix~\ref{app:interp-experiments} is built the same way and described there.

\paragraph{Oracle, noise calibration, budgets, statistics.} In (a)--(b) the oracle returns the full gradient corrupted by additive Gaussian noise, $\grad f(x)+\tfrac{s_g}{\sqrt b}\,\xi$ with $\xi\sim\mathcal N(0,I_d)$ for a batch of size $b$ and generation scale $s_g=0.25$. Assumption~\ref{app:ass:noise} requires $\E\exp(\norm\zeta^2/\sigma^2)\le e$; for $\zeta\sim\mathcal N(0,s_g^2I_d)$ this holds iff $\sigma^2\ge2s_g^2/(1-e^{-2/d})$, so the certified parameter is $\sigma=s_g\sqrt{2/(1-e^{-2/d})}\approx1.601$ at $d=40$, and \emph{this calibrated value}---not the generation scale---is what the batch rule of \textsf{Practical ARC-SG} consumes. Every method is charged one call per sampled component per iteration; the nominal budget is $2\cdot10^5$ calls in both families (the overparameterized family of Appendix~\ref{app:interp-experiments} uses $4\cdot10^5$). The historical driver starts a whole mini-batch whenever its counter is below the nominal budget and, for a proximal level, completes the subsequent adaptive cap measurement; hence its final counter can exceed the nominal budget. All reported $25\%/50\%/100\%$ endpoints and plotted grid values are obtained by linear interpolation of $\log_{10}$ gap at the nominal call count, rather than by substituting the post-budget point.

All curves report the geometric mean over evaluation seeds $0,\ldots,9$. For each seed the code independently regenerates the orthogonal rotation, root, and oracle-noise stream; methods at the same seed use the same regenerated instance, so the final top-two comparison is paired by seed (although different algorithms naturally consume their random streams differently). Baseline selection uses separate validation seeds $101,102,103$. The band is mean $\pm$ one standard deviation of $\log_{10}$ gap across the 10 runs; gaps are floored at $10^{-16}$ before averaging, so ``reaching the floor'' means reaching this numerical limit, not an exact zero. In addition to the mean curves, the final-gap standard deviation of the $\log_{10}$ gap over the 10 seeds is reported (it lies between $\approx0.07$ and $\approx0.23$ for the main methods), endpoint gaps at $25\%$, $50\%$ and $100\%$ of the budget are recorded, and for the two best methods (\textsf{PJ-SGD} and \textsf{Practical ARC-SG}) the paired per-seed difference of the final $\log_{10}$ gaps is computed: $-0.31\pm0.23$ in (a) and $-0.35\pm0.14$ in (b) (\textsf{PJ-SGD} minus \textsf{Practical ARC-SG}, negative on all 10 seeds in both families). A guard clips $\cosh$-arguments at $500$ in all oracles and function evaluations; an activation counter confirmed it never fired in any reported run (it exists to keep diverging configurations, encountered only during baseline tuning, finite).

\paragraph{Baselines, grids, and the tuning protocol.} \textsf{SGD-GS}: $\gamma_t=\tfrac12(\Lz+\Lo\norm{\gest_t})^{-1}$, batch 4 (batch $\lceil\rho\rceil=400$ in the overparameterized family, where $\rho=N^2$ is the almost-sure strong-growth constant of Appendix~\ref{app:interp-experiments}). \textsf{PJ-SGD}: the same recursion with Polyak--Juditsky averaging over the second half of the iterates, no additional free parameters. \textsf{DS-SGD}: $x_{t+1}=x_t-\gamma_0\gest_t/\sqrt{t+1}$ with $\gamma_0$ tuned, batch 4 in the additive-noise families and batch $\lceil\rho\rceil=400$ in the overparameterized family. \textsf{Clip-SGD}: $x_{t+1}=x_t-\gamma\,\mathrm{clip}(\gest_t,c)$ with $\gamma\in\{\tfrac{10^{-3}}4,\tfrac{10^{-2}}4,\tfrac{10^{-1}}4,\tfrac14\}$ and $c\in\{2,20,200\}$, batch 4. \textsf{Acc-blind}: the Nesterov-type recursion $y_t=x_t+\tfrac{t}{t+3}(x_t-x_{t-1})$, $x_{t+1}=y_t-\mathrm{scale}\cdot\gest_t/L_{\mathrm{blind}}$ with $L_{\mathrm{blind}}=\Lz+\Lo\norm{\grad f(x^0)}$ and $\mathrm{scale}\in\{1,4,16\}$ tuned, batch 4 (we deliberately do not call this AC-SA: it is a generic accelerated template, which is the point of the comparison---a generalized-smoothness-blind accelerated method would naturally use a heuristic constant of order $L_{\mathrm{blind}}$, which is \emph{not} a certified sublevel bound; see Section~\ref{app:sec:experiments}). \textsf{ClipSSTM}: the accelerated similar-triangles template with clipped gradients, $L\in\{4,40,400,4000\}$ and clipping level in $\{2,20,200\}$, no restarts, batch 4. \textsf{M-SGD-SGC} (momentum SGD under the strong growth condition; used only in Appendix~\ref{app:interp-experiments}): heavy-ball SGD with $\gamma\in\{10^{-2},10^{-1},1\}/(4\rho)$ and $\beta\in\{0.5,0.9,0.95\}$ tuned, batch $\lceil\rho\rceil$. Selection uses 3 validation seeds and the anytime objective (mean $\log_{10}$ gap over calls $\ge5000$), evaluated \emph{over the full budget of the corresponding problem}: tuning on a shorter horizon systematically selected configurations of \textsf{ClipSSTM} and \textsf{M-SGD-SGC} that looked best early and destabilized or diverged later, which would have made the comparison unfairly bad for the baselines. We note the protocol asymmetry explicitly: the baselines receive an oracle-tuned grid search, while \textsf{Practical ARC-SG} is run once with fixed constants; this favors the baselines.

\paragraph{\textsf{Practical ARC-SG} and its deviations from the theory.} We run the two-phase architecture of Algorithms~\ref{app:alg:phase1}--\ref{app:alg:arc} with moderate practical constants and no per-instance tuning; Algorithm~\ref{alg:practical} is an exact summary of the implementation and Table~\ref{tab:deviations} lists its departures from the analyzed schedules. (1) In the additive-noise families \PhaseI{} uses $B_1=4$; in the overparameterized family it uses $B_1=\lceil\rho\rceil=400$. It hands off when an adaptive measurement $\gest$ (not the inaccessible exact gradient) satisfies $\norm{\gest}\le\tfrac32\Gbar$, instead of the theoretical $3\Gbar$ test. (2) Additive-family proximal levels use $\lambda_j=\max\{4\Lo\Gest_j,\Lambda_j,\mu_{\mathrm{floor}}\}$ with $\Lambda_j=\Lz2^{1-j}$, halved \emph{every} level (the analyzed schedule starts at $32\Lz$ and halves only at binding floor levels), $\mu_{\mathrm{floor}}=\mu$ in the strongly convex family and $0$ in the convex family; their ball radius is fixed at $1/(2\Lo)$ rather than $2\Gest_j/\lambda_j$. (3) The inner restart runs $K=\min\{12,\max\{5,\lceil\log_2(H_{0,j}/\delta_j)\rceil\}\}$ epochs of $N_{\mathrm{ep}}=\lceil7\sqrt{1+L_j/\lambda_j}\rceil+3$ iterations. Its similar-triangles weight is $a_{t+1}=(t+2)/(2L_F)$, whereas analyzed \STM{} uses $(t+1)/(4L_F)$. (4) The transfer target is $\delta_j=\max\{\lambda_js_j^2/2,H_{0,j}/d_j\}$, with $d_j=1024$ when $\Gest_j\le\tfrac12$ and $\sigma>0$, and $d_j=64$ otherwise; $s_j=\min\{R_0/128,\iota/L_j\}$ and $\iota=\Gbar/256$. Thus $J_{\mathrm{prac}}=32$ is only a scale in the $s$- and $\iota$-formulas, not the number of levels. The $\max$ deliberately makes $\delta_j$ larger than the analyzed $\min\{\lambda_js_j^2/2,\eps/8\}$. (5) The exact implemented batch heuristic is $b_k=\min\{32,1+\lceil4\sigma^2/(\lambda_j\bar\Delta_k)\rceil\}$ with calibrated $\sigma\approx1.601$; it omits the logarithmic and $\sqrt{\lambda_jL_F}$ factors and changes the constant of \eqref{eq:bk}, so it is not merely \eqref{eq:bk} with a cap. (6) The cap update is $\Gest_{j+1}=\max\{10^{-9},\min[\Gest_j+L_j\max\{\norm{c_{j+1}-c_j},s_j\},\tfrac87\norm{\gest_{j+1}}+\tau_j]\}$, with adaptive measurement batches $8,32,128,256$ (a factor-four increase up to 256). (7) After at least one additive-family level, $\Gest_j\le2\Gbar$ or $40\%$ of the nominal budget triggers the interior driver: full-step $(\Lz,\Lo)$-SGD, batch $8+\lfloor t/40\rfloor$, with a running average that includes its starting anchor. This driver is outside Algorithms~\ref{app:alg:phase1}--\ref{alg:finisher} and the theorems; \textsf{ARC-SG (levels)} disables it. (8) Consequently, the finisher never fires in the main runs. (9) The overparameterized branch is different: immediately after \PhaseI{}, with no proximal levels or hop stages, it runs fixed-batch-400 half-step $(\Lz,\Lo)$-SGD and returns the last iterate without averaging. (10) The code has an unreachable-in-these-runs 2000-level safety guard. It also completes a batch started below the budget and the post-level adaptive measurement; the reported endpoints use interpolation at the nominal budget as described above. Runtime for the full study is a few minutes on one CPU core (NumPy, double precision).

\begin{algorithm}[t]
\caption{\textsf{Practical ARC-SG} (complete practical implementation run in Section~\ref{app:sec:experiments} and Appendix~\ref{app:interp-experiments})}
\label{alg:practical}
\begin{algorithmic}[1]
\REQUIRE $x^0$, $R_0$, certified $(\Lz,\Lo)$, calibrated $\sigma$, $\mu_{\mathrm{floor}}$ ($=\mu$ for the strongly convex additive family and $0$ for the convex family), oracle budget; fixed constants, no per-instance tuning
\STATE \textbf{Phase I:} $x\leftarrow x^0$; $B_1\leftarrow4$ for additive noise and $B_1\leftarrow\lceil\rho\rceil=400$ for the overparameterized family
\WHILE{the call counter is below the nominal budget}
    \STATE $\gest\leftarrow\GradEst(x,B_1)$
    \IF{$\norm{\gest}\le\tfrac32\Gbar$} \STATE $(\gest_m,\tau)\leftarrow$ adaptive measurement at $x$ (batches $8,32,128,256$)
        \IF{$\norm{\gest_m}\le\tfrac32\Gbar$} \STATE \textbf{break} \ENDIF
        \STATE $\gest\leftarrow\gest_m$
    \ENDIF
    \STATE $x\leftarrow x-\gest/[2(\Lz+\Lo\norm{\gest})]$
\ENDWHILE
\STATE If no measurement was made before budget exhaustion, make one now; $c_1\leftarrow x$, $\Gest_1\leftarrow\max\{\tfrac87\norm{\gest_m}+\tau,\tfrac14\Gbar\}$, $\Lambda_1\leftarrow\Lz$
\IF{the family is overparameterized}
    \WHILE{the call counter is below the nominal budget}
        \STATE $\gest\leftarrow\GradEst(c,400)$;\quad $c\leftarrow c-\gest/[2(\Lz+\Lo\norm{\gest})]$
    \ENDWHILE
    \RETURN $c$ \hfill (last iterate; no levels, hops, or averaging)
\ENDIF
\STATE \textbf{Additive-family Phase II:} execute at least one level while the counter is below budget (driver safety guard: at most 2000 levels)
\STATE $\lambda_j\leftarrow\max\{4\Lo\Gest_j,\Lambda_j,\mu_{\mathrm{floor}}\}$;\quad $L_j\leftarrow4(\Lz+\Lo\Gest_j)$;\quad $\Lambda_{j+1}\leftarrow\Lambda_j/2$
\STATE $s_j\leftarrow\min\{R_0/128,\iota/L_j\}$, $\iota=\Gbar/256$;\quad $H_{0,j}\leftarrow\Gest_j^2/\lambda_j$;\quad $\delta_j\leftarrow\max\{\lambda_js_j^2/2,H_{0,j}/d_j\}$, where $d_j=1024$ if $\Gest_j\le\tfrac12,\sigma>0$, and $64$ otherwise
\STATE Run $K=\min\{12,\max\{5,\lceil\log_2(H_{0,j}/\delta_j)\rceil\}\}$ epochs, each with $N_{\mathrm{ep}}=\lceil7\sqrt{1+L_j/\lambda_j}\rceil+3$ iterations on $\ball{c_j}{1/(2\Lo)}$
\STATE In epoch $k$, $\bar\Delta_k=H_{0,j}2^{-k}$ and $b_k=\min\{32,1+\lceil4\sigma^2/(\lambda_j\bar\Delta_k)\rceil\}$; inside \STM{}, $a_{t+1}=(t+2)/(2L_F)$ and every $z$-update is projected onto the single ball
\STATE Set $c_{j+1}$ to the final epoch output; make the adaptive measurement $(\gest_{j+1},\tau_j)$
\STATE $\Gest_{j+1}\leftarrow\max\{10^{-9},\min[\Gest_j+L_j\max\{\norm{c_{j+1}-c_j},s_j\},\tfrac87\norm{\gest_{j+1}}+\tau_j]\}$
\STATE Repeat levels unless $\Gest_{j+1}\le2\Gbar$ or at least $40\%$ of the nominal budget has been spent
\STATE \textbf{Additive-family interior stage (only after that trigger):} $y_0\leftarrow c_{j+1}$, $\bar y_0\leftarrow y_0$
\FOR{$t=0,1,\dots$ while the call counter is below the nominal budget}
    \STATE $b_t\leftarrow8+\lfloor t/40\rfloor$;\quad $\gest_t\leftarrow\GradEst(y_t,b_t)$;\quad $y_{t+1}\leftarrow y_t-\dfrac{\gest_t}{\Lz+\Lo\norm{\gest_t}}$
    \STATE $\bar y_{t+1}\leftarrow\big((t+1)\bar y_t+y_{t+1}\big)/(t+2)$
\ENDFOR
\RETURN $\bar y_t$ \hfill (the average includes the starting anchor)
\end{algorithmic}
\end{algorithm}

\begin{table*}[t]
\centering
\footnotesize
\setlength{\tabcolsep}{4pt}
\begin{tabularx}{\textwidth}{@{}>{\raggedright\arraybackslash}p{0.17\textwidth}>{\raggedright\arraybackslash}p{0.28\textwidth}>{\raggedright\arraybackslash}X@{}}
\toprule
Component & Theoretical \ARCSG{} & \textsf{Practical ARC-SG} \\
\midrule
\PhaseI{} batch & $B_1$ of \eqref{eq:phase1-params} & $4$ (additive); $400$ (overparameterized) \\
\PhaseI{} threshold & $3\Gbar$ & $1.5\Gbar$ \\
Floor schedule $\Lambda_j$ & $32\Lz$, halved only at binding floor levels & $\Lz\cdot2^{1-j}$, halved every level \\
Ball radius & $2\Gest_j/\lambda_j$ & $1/(2\Lo)$, fixed \\
Epochs $K$ & $\lceil\log_2(H_0/\delta_{\mathrm{lvl}})\rceil_{+}$ & $\min\{12,\max\{5,\lceil\log_2(H_{0,j}/\delta_j)\rceil\}\}$ \\
Iterations per epoch & $N_{\mathrm{ep}}=\lceil24\sqrt{L_F/\lambda}\rceil$ & $N_{\mathrm{ep}}=\lceil7\sqrt{\kappa_j}\,\rceil+3$ \\
\STM{} weight & $a_{t+1}=(t+1)/(4L_F)$ & $a_{t+1}=(t+2)/(2L_F)$ \\
Transfer target $\delta_j$ & $\min\{\lambda_js_j^2/2,\ \eps/8\}$ & $\max\{\lambda_js_j^2/2,\ H_{0,j}/d_j\}$ \\
Batches & $b_k$ of \eqref{eq:bk} & $\min\{32,1+\lceil4\sigma^2/(\lambda_j\bar\Delta_k)\rceil\}$ \\
Cap update & certified bookkeeping (Lemma~\ref{lem:inv}) & heuristic observed displacement/measurement, floor $10^{-9}$ \\
Additive interior stage & absent & full-step SGD, $b_t=8+\lfloor t/40\rfloor$, average includes anchor \\
Overparameterized branch & hop stages of Algorithm~\ref{alg:interp} & post-Phase-I half-step SGD, fixed batch $400$, last iterate \\
Budget boundary & schedules stop within declared call bound & started batch and post-level measurement complete; interpolate at nominal budget \\
Driver level guard & schedule cap $J^{\mathrm{sch}}$ then certified finisher & safety guard at 2000 levels (not reached) \\
Finisher & Algorithm~\ref{alg:finisher} & never fires \\
Projections & exact two-ball (Remark~\ref{rem:proj}) & single-ball (computational-model deviation) \\
\bottomrule
\end{tabularx}
\caption{Component-by-component deviations of \textsf{Practical ARC-SG} from the analyzed Algorithms~\ref{app:alg:phase1}--\ref{alg:finisher}. None of the practical choices in the third column is covered by Theorems~\ref{app:thm:convex}--\ref{app:thm:strongly}; the single-ball projection is a computational-model deviation rather than a change to a quantity entering a theorem (projection calls are not part of the oracle complexity, Remark~\ref{rem:proj}).}
\label{tab:deviations}
\end{table*}

\begin{figure}[t]
\centering
\includegraphics[width=0.98\columnwidth]{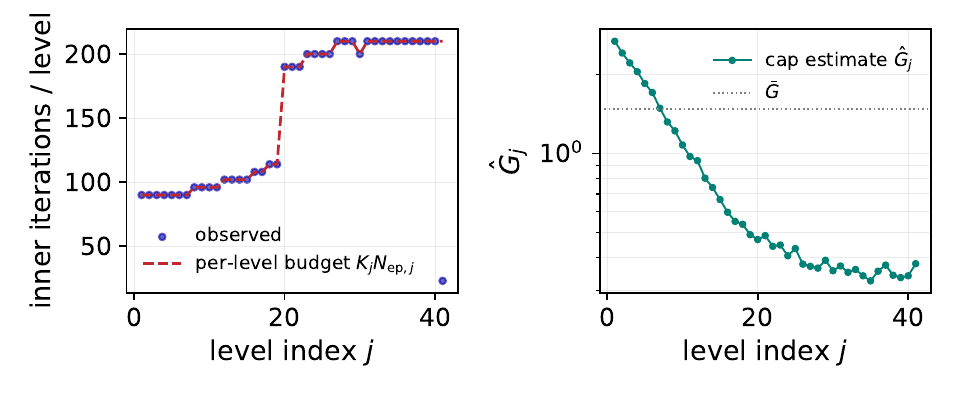}
\caption{Diagnostics of the proximal levels (convex setting, pure accelerated levels with the interior stage disabled; the diagnostic instance starts at $\norm{\grad f(x^0)}=4050\,\Gbar=5967.6$, the value logged by the run). Left: inner iterations per executed level (dots) together with the scheduled per-level budget $K_jN_{\mathrm{ep},j}$ (dashed). The implementation runs each level for exactly its scheduled budget---there is no early stopping inside a level---so the two curves agree by construction, except at the last level, truncated by the global oracle budget; the plot therefore displays the \emph{schedule} (in particular the jump caused by the switch of the transfer-target divisor $d_j$ from $64$ to $1024$ in the deep-endgame regime), not adaptive savings. Right: the cap estimate $\Gest_j$ contracts geometrically across levels, crossing $\Gbar$ (dotted); the observed level count ($\approx41$) is far below the schedule cap $J^{\mathrm{sch}}$, which is evidence about these well-conditioned instances and not about the necessity of the worst-case $\kbar$-factors.}
\label{fig:levels}
\end{figure}

\begin{figure}[t]
\centering
\includegraphics[width=0.7\columnwidth]{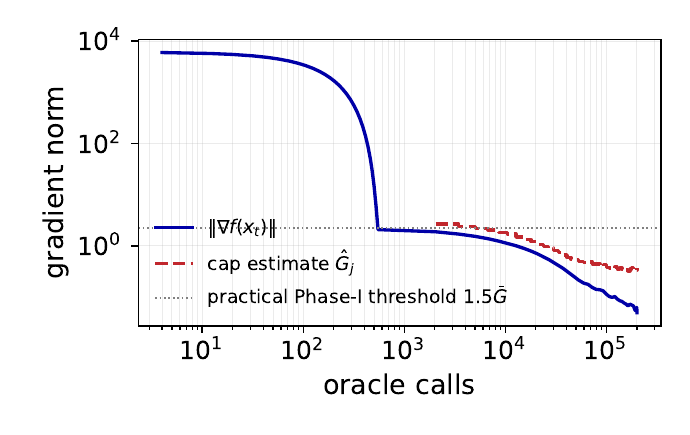}
\caption{Gradient norm along the same pure-levels diagnostic run. The diagnostic instance scales the initial displacement along a quartic coordinate by a factor of three; since the quartic gradient is cubic in the displacement, the run starts at $\norm{\grad f(x^0)}=5967.6=4050\,\Gbar$ (logged by the run, not estimated). The two-phase structure is visible: super-critical gradients are tamed at a linear rate by \PhaseI{} (practical threshold $1.5\Gbar$, dotted), after which proximal levels take over. The dashed curve is the \emph{practical} cap estimate $\Gest_j$: on this run it lay above a high-accuracy estimate of $\norm{\grad f(c_j)}$ after every anchor and within a small constant factor of it. This is an a-posteriori observation about the implementation, not an instance of the proved invariant of Lemma~\ref{lem:inv}: the practical cap update (item (6) above) replaces the certified distance-to-prox bound by the observed center displacement and is a heuristic.}
\label{fig:gradnorm}
\end{figure}

\paragraph{Additional observations.} On these instances \PhaseI{} already brings the gradient near the hand-off threshold before more than one or two proximal levels fire in the full \textsf{Practical ARC-SG}, so the late-stage gains visible in Figure~\ref{app:fig:experiments} are produced by the interior averaging stage; the pure-levels ablation (final gap $8.0\times10^{-3}$ in the convex family, versus $5.2\times10^{-5}$ for the full method) quantifies this attribution, and the \textsf{PJ-SGD} baseline isolates it directly: averaging alone reaches $2.6\times10^{-5}$, so the levels contribute speed rather than endpoint accuracy on these instances. In the diagnostic ablation runs of Figures~\ref{fig:levels}--\ref{fig:gradnorm}, \PhaseI{} tames the super-critical gradient ($\norm{\grad f(x^0)}=5967.6=4050\,\Gbar$ with the certified $\Gbar=\Lz/\Lo\approx1.4735$) at a visibly linear rate within a few hundred calls, and the subsequent $\approx41$ levels contract the cap estimate geometrically. We do not claim a parametric statistical floor in the convex family: $\grad^2f(x^\star)$ is singular along the quartic directions, $\sigma^2\operatorname{tr}(\grad^2f(x^\star)^{-1})$ is undefined, and the interior stage is simply still descending at the budget end; in the strongly convex family the curve is likewise still decreasing, so the budget, not a floor, determines the final numbers. Finally, the centralized guard counter---incremented inside the stochastic oracle \emph{and} inside every deterministic function/gradient evaluation used for logging---reported zero $\cosh$-argument clips across all reported runs, so no reported trajectory was affected by the numerical safeguard.

\paragraph{Matched ablations.} Two further curves decompose \textsf{Practical ARC-SG} into its components under the same 10 independently regenerated instance--noise seeds (geometric-mean gaps at $25\%/50\%/100\%$ of the budget): \textsf{PJ-SGD (matched)}, the interior PJ driver run from $x^0$ with the same stepsize $1/(\Lz+\Lo\norm{\gest})$, growing batch $8+\lfloor t/40\rfloor$, an average that includes the starting anchor, and the same budget; and \textsf{Phase I + PJ (matched)}, the same driver started from the \PhaseI{} output. Convex family: \textsf{PJ-SGD (matched)} reaches $2.5\times10^{-4}/1.2\times10^{-4}/5.8\times10^{-5}$, \textsf{Phase I + PJ (matched)} reaches $1.8\times10^{-4}/9.3\times10^{-5}/5.5\times10^{-5}$, and \textsf{Practical ARC-SG} (reference) reaches $1.9\times10^{-4}/8.0\times10^{-5}/5.2\times10^{-5}$. Strongly convex family: $5.4\times10^{-3}/2.5\times10^{-3}/1.2\times10^{-3}$ for \textsf{PJ-SGD (matched)}, $2.4\times10^{-4}/1.1\times10^{-4}/5.5\times10^{-5}$ for \textsf{Phase I + PJ (matched)}, and $2.3\times10^{-4}/1.1\times10^{-4}/5.3\times10^{-5}$ for \textsf{Practical ARC-SG}. The reading is two-sided: on the convex family the raw PJ driver already captures nearly all of the practical method's accuracy, and \PhaseI{} and the levels buy only small early-budget gains; on the strongly convex family the raw driver stalls at $1.2\times10^{-3}$, \PhaseI{} is essential ($2.4\times10^{-4}$ already at $25\%$ of the budget), and the proximal levels add little over \PhaseI{} plus the matched driver ($5.3\times10^{-5}$ versus $5.5\times10^{-5}$ at the budget end).

\paragraph{Sensitivity.} Table~\ref{tab:sensitivity} reports final geometric-mean gaps (evaluation seeds $0,1,2$, convex family) under perturbations of the Gaussian generation scale $s_g$ and of the initial-distance scale about the default configuration. At $d=40$ the corresponding certified Assumption~\ref{app:ass:noise} parameter is $\sigma=6.4037749\,s_g$: the three noise settings below therefore use calibrated $\sigma\in\{0.4002359,1.6009437,6.4037749\}$, not $\{0.0625,0.25,1\}$. \textsf{Clip-SGD} ranks third throughout, but the top two methods swap at high noise: \textsf{PJ-SGD} is best at the default and low-noise settings, whereas \textsf{Practical ARC-SG} is best at $s_g=1$. The distance-scale rows agree only at the displayed two-significant-digit precision: for scales $0.3$ and $3.0$, the unrounded \textsf{Practical ARC-SG} gaps are respectively $3.8273\times10^{-5}$ and $3.7853\times10^{-5}$, and the \textsf{Clip-SGD} gaps are $3.8834\times10^{-4}$ and $3.8892\times10^{-4}$; the underlying objective at a fixed seed is unchanged, but the initial points and trajectories are distinct.

\begin{table}[t]
\centering
\footnotesize
\setlength{\tabcolsep}{4pt}
\begin{tabular}{@{}lccc@{}}
\toprule
Configuration & \textsf{Practical ARC-SG} & \textsf{PJ-SGD} & \textsf{Clip-SGD} \\
\midrule
$s_g=0.25$ ($\sigma=1.6009$; default) & $3.8\times10^{-5}$ & $2.5\times10^{-5}$ & $3.9\times10^{-4}$ \\
$s_g=0.0625$ ($\sigma=0.4002$) & $1.2\times10^{-5}$ & $4.9\times10^{-6}$ & $9.6\times10^{-5}$ \\
$s_g=1.0$ ($\sigma=6.4038$) & $2.3\times10^{-4}$ & $3.0\times10^{-4}$ & $5.7\times10^{-3}$ \\
distance scale $0.3$ & $3.8\times10^{-5}$ & $2.5\times10^{-5}$ & $3.9\times10^{-4}$ \\
distance scale $3.0$ & $3.8\times10^{-5}$ & $2.5\times10^{-5}$ & $3.9\times10^{-4}$ \\
\bottomrule
\end{tabular}
\caption{Sensitivity on the convex family: final geometric-mean gaps over 3 independently regenerated instance--noise seeds. Noise rows give both the Gaussian generation scale $s_g$ and its calibrated Assumption~\ref{app:ass:noise} parameter $\sigma$.}
\label{tab:sensitivity}
\end{table}

\paragraph{Projection, wall-clock, and retained artifacts.} \textsf{Practical ARC-SG} uses single-ball projections only: the two-ball intersection of Remark~\ref{rem:proj} never arises because the practical ball radius is fixed at $1/(2\Lo)$. The run counters report $90$ projection calls in the convex and strongly convex runs and $0$ in the overparameterized run (no levels fire there); the \textsf{ARC-SG (levels)} ablation reports $6191$. Coarse single-seed wall-clock timings (NumPy, a single CPU (central processing unit) core): \textsf{Practical ARC-SG} $0.2$\,s versus $1.3$--$2.3$\,s for the baselines in the convex and strongly convex families; in the overparameterized family \textsf{Clip-SGD} takes $4.1$\,s and \textsf{ClipSSTM} $5.0$\,s. The historical timing stdout was not retained, so these machine-dependent timings are not independently recoverable from the bundled artifacts; the supplied \texttt{timing} stage now writes \texttt{experiments/timing.json} on rerun. Likewise, the bundled \texttt{res\_*.pkl} files retain aggregate log-gap curves, endpoint statistics, and all ten final per-seed log-gaps, but not the full raw per-seed trajectories. Fresh seeded runs reconstruct the trajectories, while the exact historical intermediate arrays and the maximum historical budget overshoot cannot be recovered from those aggregate files.

\textbf{What these experiments do not test.} We state this explicitly rather than leaving it to the appendix. (1) The late-stage advantage in Figure~\ref{app:fig:experiments} is produced mainly by the interior Polyak--Juditsky averaging stage, which is \emph{not} part of Algorithms~\ref{app:alg:phase1}--\ref{alg:finisher}; the theory-shaped component is the levels-only ablation, which ends between \textsf{SGD-GS} and tuned \textsf{Clip-SGD}. (2) The finisher (Algorithm~\ref{alg:finisher}) never fires on these instances, although it is exactly what makes the worst-case guarantee unconditional; targeted small instances where the finisher and the strong-growth hop stages execute are reported separately in the analyzed-schedule sanity experiment (Appendix~\ref{app:exact-sanity}). (3) The runs use the practical schedule of Appendix~\ref{app:experiments}, not the analyzed constants, so the observed level counts and batch sizes carry no information about the tightness of the $\kbar$-factors. An exact-constants run of Algorithms~\ref{app:alg:phase1}--\ref{alg:finisher} is computationally prohibitive: evaluated analytically from the paper's own schedule (Appendix~\ref{app:params}) on the convex instance at the run's accuracy, the analyzed schedule admits $J_{\mathrm{tot}}\approx4.0\times10^{6}$ levels, and already the representative first level ($\lambda_1\approx68$, $L_F\approx110$, $K=24$ epochs, $N_{\mathrm{ep}}=31$) has a last-epoch batch of $\approx6.2\times10^{9}$ per iteration even with the best-case constant $C_b=263$ of \eqref{eq:bk}, for a total of $\approx3.2\times10^{18}$ oracle calls ($\approx1.6\times10^{13}\times$ the $2\times10^{5}$ budget; $\approx1.2\times10^{20}$ with a realistic $C_b\sim10^{4}$)---this is why the theory-shaped curve in Figure~\ref{app:fig:experiments} is the levels-only ablation with moderate practical constants rather than the analyzed algorithm itself. (4) The empirical statement that the cap tracks the true gradient from above is a \emph{posteriori} observation against a high-accuracy gradient estimate on the shown runs---the practical cap update is a heuristic, and the proved bookkeeping invariant (Lemma~\ref{lem:inv}) applies to the analyzed method, not to this implementation. (5) Theorems~\ref{thm:interp-strong}--\ref{thm:interp-convex} are not tested here; see Appendix~\ref{app:interp-experiments} for what the overparameterized experiment does and does not exercise. (6) Under the certified constants the overparameterized family no longer shows \textsf{Practical ARC-SG} competitive with tuned \textsf{Clip-SGD}: with the batch $\lceil\rho\rceil=400$ that the almost-sure strong-growth condition \eqref{eq:sgc-as} imposes, the $(\Lz,\Lo)$-stepsize methods reach only $\approx10^{-2}$ at the budget end, while tuned \textsf{Clip-SGD} (batch 4) reaches the numerical floor within the first quarter of the budget; the experiment now instantiates the strong-growth assumption the theorems actually use, but it still does not exercise the hop stages of Algorithm~\ref{alg:interp}.

\subsection{Analyzed-schedule sanity experiment}
\label{app:exact-sanity}

We report a control run of the analyzed update rules---\PhaseI{} (Algorithm~\ref{app:alg:phase1}), \PhaseII{} (Algorithm~\ref{app:alg:phase2}) with the full schedule of Appendix~\ref{app:params}, and the finisher (Algorithm~\ref{alg:finisher})---with the labelled best-case concentration constants $c_g=c_{\mathrm{Az}}=c_b=1$, $c_1=2c_{\mathrm{Az}}$, $C_b=263$ of \eqref{eq:bk}, and $\alpha=0.05$. The additive-Gaussian runs below simulate every batch mean exactly in distribution; the separately identified strong-growth run uses a disclosed covariance-matched approximation for very large categorical batches. A self-check verified the two-ball projection of Remark~\ref{rem:proj} against Dykstra's algorithm on 50 random cases. The test instance is the certified 1-D $\cosh$-sum of Lemma~\ref{lem:exp-cert}(i) with scales $\{0.6,0.8,1.0\}$: certified pair $(\Lz,\Lo)=(2.1258,1/\ln2)$, $x^\star=0$, $f^\star=3$, started at $x^0=2.5$ with $R_0=2.5$; the finite-difference condition \eqref{eq:gs} was verified numerically on a grid, with maximal ratio $1.00000000$ (the frontier of Lemma~\ref{lem:exp-cert} is tight).

\paragraph{Deterministic oracle.} With $\sigma=0$ (a deterministic oracle, a valid degenerate case of Assumption~\ref{app:ass:noise}) and $\eps=0.3$, the run executes 15 levels and exits via the gradient termination at the head of level 16 with true gap $f(\xhat)-f^\star=4.45\times10^{-6}\le\eps$ at 22{,}237 oracle calls (the exit-certified bound is $0.15$). On \emph{every} executed level the certified invariants of Lemma~\ref{lem:inv} hold: $\norm{\grad f(c_j)}\le\Gest_j$, $\norm{c_j-x^\star}\le\Rb$, the certified prox lies inside the ball ($\norm{\xhat_j-c_j}\le r_j/2$), and the \RSTM{} output satisfies $\norm{c_{j+1}-\xhat_j}\le\sqrt{2\delta_j/\lambda_j}\le s_j$ (each inequality is flagged individually in the log). A strongly convex companion instance (the same $\cosh$-sum plus $0.25\,u^2$) has exact modulus $\mu=0.6^2+0.8^2+1^2+0.5=2.5$ and certified pair $(2.6258,1/\ln2)$; the quadratic contribution $0.5$ is merely a conservative lower bound. It executes 16 levels and exits at the head of level 17, at 23{,}297 calls with true gap $2.6\times10^{-5}$.

\paragraph{Noisy oracle.} With $\sigma=0.01$ the identical trajectory (same levels, same exit (a), true gap $4.7\times10^{-6}$) requires $2.65\times10^{15}$ oracle calls: the batch rule \eqref{eq:bk} climbs to $\approx1.3\times10^{12}$ samples per iteration on the gradient-driven levels. This quantifies, on a certified instance on which every invariant is checked level by level, why the printed worst-case constants are not what is run at scale (cf.\ the practical implementation of Algorithm~\ref{alg:practical} and the deviations of Table~\ref{tab:deviations}).

\paragraph{Targeted finisher instances.} The finisher (Algorithm~\ref{alg:finisher})---the component that never fires on the families of Figure~\ref{app:fig:experiments}---is exercised in two regimes on the same certified instance with $\sigma=0$. (i) Under an artificial schedule cap $J^{\mathrm{sch}}=4$ ($\eps=0.3$), control passes to the finisher, which exits correctly via exit (a) inside the finisher at its fifth level, at $5.7\times10^{4}$ oracle calls with true gap $2.4\times10^{-6}$. (ii) In the organic trigger attempt at $\eps=10^{-6}$, the cap falls below the drift resolution ($\Gest_j=5.66\times10^{-6}<32\,\iota_{\mathrm{abs}}$) and one literal finisher level is executed. The driver then aborts at its one-level iteration guard. This is \emph{not} an exit of the paper's algorithm and carries no exit certificate; it only records that the true gap at the aborted point is $0$ to machine precision.

\paragraph{Strong-growth algorithm.} On a 3-component orthogonal interpolating $\cosh$ family (scales $\{0.08,0.13,0.2\}$, shared root $x^\star=0$, $\mu=\min_kC_k^2=0.0064$, $\rho=N^2=9$, certified individual pair $(0.1298,0.2885)$, valley half-width $\approx70\gg1/\Lo$), the hop-and-solve update rules and literal batch \emph{counts} of Algorithm~\ref{alg:interp} (Stages 0--2) fire the hot stage for 4 levels and the cold stage for 1{,}218 hops, exiting with certified $\Delta\le\mu/(32\Lo^2)$ and true gap $2.8\times10^{-11}$; the hop certificate $f(c_j)-f^\star\le\Delta_j$ of Lemma~\ref{lem:cold}(a) held at every hop (worst ratio $0.008$). To make these enormous batches executable, batches up to $5\cdot10^6$ use exact categorical index draws, whereas 9{,}521 larger batch means are drawn from a Gaussian with the same mean and exact covariance divided by the literal batch size. The call counter is still charged the literal number of samples. Thus this is a covariance-matched Gaussian approximation and a computational certificate sanity check, not an exact finite-sample reproduction of the categorical oracle.

Per-level tables of all runs are in the supplementary log file \texttt{experiments/exact\_sanity.log}. The experiment code and logs accompany the submission.

\section{Additional experiments and limitations}
\label{app:additional}

\subsection{Experiments in the overparameterized regime}
\label{app:interp-experiments}

This subsection reports the experiment corresponding to Appendix~\ref{app:interp}; it is placed here, with the theory it illustrates, rather than in Section~\ref{app:sec:experiments}.

\paragraph{Instance and constants.} $d=20$, one $\cosh$ component per rotated coordinate with scales $C_k$ geometrically spaced in $[0.35,1]$ ($N=20$ components), a shared root $x^{\mathrm{int}}=Qr$ so that the family interpolates in the sense of Definition~\ref{def:interp}, $f^\star=20$ exactly, and local strong convexity $\mu_{\mathrm{loc}}=\min_kC_k^2=0.1225$ exactly (the Hessian is diagonal in the rotated basis). The oracle draws one index uniformly and returns $N\grad f_i(x)$, so the noise is multiplicative and vanishes at $x^{\mathrm{int}}$. Each scaled sample $Nf_i$ has the \emph{Hessian-form} pair $(NC_i^2,C_i)$, but Assumption~\ref{ass:sgc} requires the finite-difference condition \eqref{eq:gs} for each individual loss, and the Hessian-form pair fails it: for $C_i=1$, at $x=0$, $y=1$ (within the admissible range $|y-x|\le1$), $|g_i'(1)-g_i'(0)|=20\sinh(1)=23.504>20$. By the frontier of Lemma~\ref{lem:exp-cert}(i) applied to $N\cosh(C_iu)$ ($B=C_i$), the valid pairs for the $i$-th individual loss are $\Lo\ge C_i/\ln2$, $\Lz\ge\Lo NC_i\sinh(C_i/\Lo)$, both right-hand sides being maximized at $C_i=1$, so the tight \emph{uniform individual} constants are
\begin{equation}
\label{eq:ind-constants}
\begin{split}
(\Lz^{\mathrm{ind}},\Lo^{\mathrm{ind}})&=\Big(\frac{N\sinh(\ln2)}{\ln2},\ \frac{1}{\ln2}\Big)=\big(\tfrac{15}{\ln2},\ \tfrac{1}{\ln2}\big)=(21.6404,\ 1.4427),
\end{split}
\end{equation}
using $\sinh(\ln2)=\tfrac34$; necessity is attained at $C_i=1$ (at $x=0$, $y=\ln2$ for $\Lz$, and by the $x\to+\infty$ asymptotics for $\Lo$). These---not the Hessian-form constants $(1,1)$ of the average $f$---are what every $(\Lz,\Lo)$-aware method receives, exactly as Assumption~\ref{ass:sgc} requires. The strong-growth constants are exact, not estimated. Because the component supports are orthogonal, $\norm{\grad f(x)}^2=\sum_i\norm{\grad f_i(x)}^2$, hence $\E_i\norm{N\grad f_i(x)}^2=\tfrac1N\sum_iN^2\norm{\grad f_i(x)}^2=N\norm{\grad f(x)}^2$: the smallest constant in the expected variant \eqref{eq:sgc} is exactly $\rho=N=20$ (an identity, so no pilot-run measurement of it is involved). The theorems, however, use the almost-sure variant \eqref{eq:sgc-as}, whose smallest uniform constant is exactly $\rho=N^2=400$: at single-active-coordinate points $x=x^{\mathrm{int}}+t\,q_i$ one has $\norm{N\grad f_i(x)}=N\norm{\grad f(x)}$, forcing $\rho\ge N^2$, and $\norm{N\grad f_i(x)}=N\norm{\grad f_i(x)}\le N\norm{\grad f(x)}$ for all $x$ gives the matching upper bound. Everywhere the theory calls for a batch of size $\tO(\rho)$ the runs therefore use $\lceil\rho\rceil=400$. The budget is $4\cdot10^5$ oracle calls; baselines are tuned on 3 validation seeds over that full budget, and \textsf{M-SGD-SGC} \citep{vaswani2019fast}---momentum SGD tuned for strong growth, with batch $\lceil\rho\rceil=400$---is added to the comparison.

\begin{figure}[h]
\centering
\includegraphics[width=0.6\columnwidth]{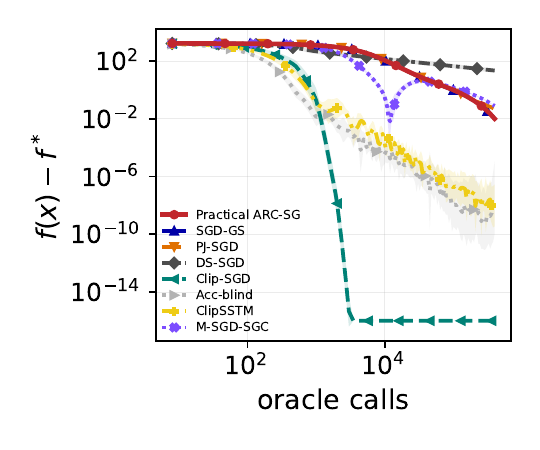}
\caption{Overparameterized regime (interpolating $\cosh$-family, single-sample oracle, uniform individual constants \eqref{eq:ind-constants}; \textsf{Practical ARC-SG}, \textsf{SGD-GS}, \textsf{PJ-SGD}, \textsf{DS-SGD}, and \textsf{M-SGD-SGC} use the strong-growth batch $\lceil\rho\rceil=400$). Geometric mean over 10 independently regenerated instance seeds; gaps floored at $10^{-16}$. No method exhibits a noise floor, as Theorem~\ref{thm:interp-strong} predicts qualitatively: tuned \textsf{Clip-SGD} (batch 4) reaches the numerical floor before $25\%$ of the budget, and \textsf{Acc-blind} and \textsf{ClipSSTM} also converge linearly, only far more slowly ($3.5\times10^{-8}$ and $1.0\times10^{-8}$ at the budget end); the untuned $(\Lz,\Lo)$-stepsize methods \textsf{Practical ARC-SG} and \textsf{SGD-GS} (overlapping curves) descend linearly to $\approx1.0\times10^{-2}$ (from $\approx1.8\times10^{-1}$ at half the budget), \textsf{M-SGD-SGC} reaches $7.9\times10^{-2}$, \textsf{PJ-SGD} $5.0\times10^{-2}$, and \textsf{DS-SGD} ends at $2.2\times10^{1}$.}
\label{fig:interp}
\end{figure}

\paragraph{What the run shows, and what it does not.} With the certified individual constants \eqref{eq:ind-constants} and the almost-sure strong-growth batch $\lceil\rho\rceil=400$, the untuned $(\Lz,\Lo)$-stepsize methods do not reach high accuracy on this family within the budget: \textsf{Practical ARC-SG} and \textsf{SGD-GS} stand at $\approx1.0\times10^{-2}$ at the budget end ($\approx1.8\times10^{-1}$ at half the budget, so the linear descent is still in progress), tuned \textsf{M-SGD-SGC} at $7.9\times10^{-2}$, and \textsf{PJ-SGD} at $5.0\times10^{-2}$; tuned \textsf{Clip-SGD} (batch 4) reaches the $10^{-16}$ floor within the first quarter of the budget, while \textsf{Acc-blind} ($3.5\times10^{-8}$) and \textsf{ClipSSTM} ($1.0\times10^{-8}$) converge linearly but far more slowly, and \textsf{DS-SGD} ends at $2.2\times10^{1}$. Three honest qualifications are needed.

First, \emph{the hop stages of Algorithm~\ref{alg:interp} never fire on these instances}: \PhaseI{} already lands inside the locally strongly convex basin, after which \textsf{Practical ARC-SG} proceeds directly to its interior driver---plain $(\Lz,\Lo)$-SGD with batch $\lceil\rho\rceil=400$ and $\sigma=0$---which is linearly convergent under strong growth. Consequently Figure~\ref{fig:interp} evaluates that driver, which coincides with \textsf{SGD-GS} up to the \PhaseI{} threshold logic (hence the two curves overlap), and \emph{not} the hop-and-solve method of Theorems~\ref{thm:interp-strong}--\ref{thm:interp-convex}. Small targeted instances on which Stages~1--2 are genuinely exercised are reported in the analyzed-schedule sanity experiment (Appendix~\ref{app:exact-sanity}).

Second, tuned \textsf{Clip-SGD} reaches the numerical floor within the first quarter of the budget, while the untuned $(\Lz,\Lo)$-stepsize methods reach only $\approx10^{-2}$ at the budget end. This is not evidence against the theory but a direct consequence of \eqref{eq:ind-constants}: the uniform individual constants that the strong-growth theory must use are $N$-fold pessimistic relative to the smoothness of the average on this instance, and tuning buys back exactly that factor---and the tuned baseline (batch 4) additionally avoids the almost-sure batch $\lceil\rho\rceil=400$ that the theory-shaped runs pay. The honest reading is that $(\Lz,\Lo)$-aware stepsizes purchase tuning-freeness and worst-case validity, not instance-optimal constants.

Third, the qualitative prediction that \emph{is} tested---the disappearance of the noise floor under multiplicative noise---is confirmed for every method that uses a stepsize compatible with the individual constants, and fails for none; it is now illustrated under the certified individual constants and the almost-sure strong-growth batch that Theorems~\ref{thm:interp-strong}--\ref{thm:interp-convex} actually use.

\subsection{Limitations (full list)}
\label{sec:limitations}

We collect the limitations of this work in one place; each item cross-references the statement where the issue is discussed rather than re-arguing it.

\emph{(i) $\kbar$-factors on the leading terms.} The leading terms of Theorems~\ref{app:thm:convex}--\ref{app:thm:strongly} carry the multiplicative factors $\kbar^{1/2}$ and $\kbar^{3}$ with $\kbar=1+\Lo\Rb$; these are not proved optimal and are probably not optimal (Section~\ref{sec:discussion-kbar}). Whether they can be removed, and what the minimal additive entry price is, is Open Problem~\ref{op:kbar}.

\emph{(ii) External versus default initial-gap certificate.} The printed $\tO$-forms treat $\Dcert$ as an externally supplied certificate. Under the default gap certificate of Lemma~\ref{lem:gap0} the \PhaseI{} entry terms scale as $(1+q)^2$ in the deterministic part and as $(1+q)^4\sigma^2\Lo^2/\Lz^2$ in the statistical burn-in, where $q=\Lo R_0$ (Corollary~\ref{cor:phase1-default}; Remarks~\ref{app:rem:convex-default}--\ref{rem:strongly-default}); whether the quadratic large-$q$ dependence is sharp is open.

\emph{(iii) Exact-projection model.} All theorems assume an \emph{exact} Euclidean projection oracle onto balls and, in \RSTM{}, onto intersections of two balls; projection calls are counted separately from stochastic-gradient oracle calls, and no projection-error (inexact-projection) analysis is provided (Remark~\ref{rem:proj}, and the projection-model paragraph after Proposition~\ref{prop:stm}).

\emph{(iv) Parameter knowledge.} The analyzed method requires knowledge of $(\Lz,\Lo,\sigma,R_0)$, of $\mu$ in the strongly convex case, of $\rho$ in the strong-growth regime, and of a gap certificate $\Dcert$ for the printed $\tO$-forms (the default certificate of Lemma~\ref{lem:gap0} is available at the price of item~(ii)); the full schedule is in Appendix~\ref{app:params}.

\emph{(v) Exact versus practical algorithms.} The analyzed Algorithms~\ref{app:alg:phase1}--\ref{alg:finisher} are \emph{not} what the experiments run: an exact-constants run is computationally prohibitive because of the worst-case schedule constants (on the convex experimental instance, $J_{\mathrm{tot}}\approx4\times10^{6}$ levels and $\gtrsim3\times10^{18}$ oracle calls at the run's accuracy; Appendix~\ref{app:experiments}). The practical implementation is a distinct, clearly labelled heuristic (Appendix~\ref{app:experiments}) whose late-stage gain comes from a Polyak--Juditsky averaging stage that lies outside Theorems~\ref{app:thm:convex}--\ref{app:thm:strongly} (Section~\ref{app:sec:experiments}).

\emph{(vi) Components not exercised by the main experiments.} The finisher (Algorithm~\ref{alg:finisher}) and the hot/cold hop stages of Algorithm~\ref{alg:interp} never fire on the families of Section~\ref{app:sec:experiments} and Appendix~\ref{app:interp-experiments}. In the targeted sanity runs (Appendix~\ref{app:exact-sanity}), an artificially capped schedule enters the finisher and terminates correctly; a separate organic-trigger attempt enters it but is aborted by a driver guard after one level and is not certified. Both hop stages fire in the strong-growth check, whose 9{,}521 very large categorical batch means use the disclosed covariance-matched Gaussian approximation.

\emph{(vii) Almost-sure strong growth.} The strong-growth theory assumes the almost-sure variant \eqref{eq:sgc-as} of Assumption~\ref{ass:sgc}, which is strong; the obstruction to an in-expectation extension is discussed in Remark~\ref{rem:mom}.

\emph{(viii) Theorem~\ref{thm:interp-convex} is polynomial in $1/\eps$.} The convex strong-growth rate $\tO\big(\rho\kbar+\rho(1+\Lz\Rb^2/\eps)^{3}\big)$ has no linear (i.e.\ $\log(1/\eps)$) dependence, in contrast to the strongly convex Theorem~\ref{thm:interp-strong}. The intermediate $\kbar^{3}A+\kbar^{2}A^{2}$ form of the previous version is superseded by the global regularized certificate (Lemma~\ref{lem:cert-reg}); the price is the $\tfrac{33}{8}$ constant in $\Lz'$ and the entry-only $\kbar$-factor.

\emph{(ix) No matching stochastic lower bounds.} The $\eps$-exponents are optimal on the smooth subclass, but we are not aware of any stochastic lower bound certifying the necessary $q$-dependence of the leading constants for $\Lo>0$; all $q$-dependence beyond the smooth-case bounds should be read as an upper-bound artifact (Section~\ref{sec:discussion-kbar}, Open Problem~\ref{op:kbar}).

\emph{(x) Synthetic experiments.} The experiments use synthetic $\cosh$/quartic families with constants certified for Assumption~\ref{app:ass:gs}; they are illustrative of the \emph{mechanism} of the method, not of its worst-case regimes or constants (Section~\ref{app:sec:experiments}).

\clearpage
\ifdefined\ARCIncludedSupplement
\let\ARCFinishSupplement\relax
\else
\bibliography{refs}
\def\ARCFinishSupplement{\end{document}}
\fi
\ARCFinishSupplement

\bibliography{refs}

@inproceedings{zhang2020gradient,
  author    = {Zhang, Jingzhao and He, Tianxing and Sra, Suvrit and Jadbabaie, Ali},
  title     = {Why Gradient Clipping Accelerates Training: A Theoretical Justification for Adaptivity},
  booktitle = {International Conference on Learning Representations (ICLR)},
  year      = {2020}
}

@inproceedings{zhang2020improved,
  author    = {Zhang, Bohang and Jin, Jikai and Fang, Cong and Wang, Liwei},
  title     = {Improved Analysis of Clipping Algorithms for Non-convex Optimization},
  booktitle = {Advances in Neural Information Processing Systems (NeurIPS)},
  volume    = {33},
  year      = {2020}
}

@inproceedings{koloskova2023revisiting,
  author    = {Koloskova, Anastasia and Hendrikx, Hadrien and Stich, Sebastian U.},
  title     = {Revisiting Gradient Clipping: Stochastic Bias and Tight Convergence Guarantees},
  booktitle = {Proceedings of the 40th International Conference on Machine Learning (ICML)},
  series    = {Proceedings of Machine Learning Research},
  volume    = {202},
  pages     = {17343--17363},
  year      = {2023}
}

@inproceedings{li2023convex,
  author    = {Li, Haochuan and Qian, Jian and Tian, Yi and Rakhlin, Alexander and Jadbabaie, Ali},
  title     = {Convex and Non-convex Optimization Under Generalized Smoothness},
  booktitle = {Advances in Neural Information Processing Systems (NeurIPS)},
  volume    = {36},
  year      = {2023}
}

@inproceedings{chen2023generalized,
  author    = {Chen, Ziyi and Zhou, Yi and Liang, Yingbin and Lu, Zhaosong},
  title     = {Generalized-Smooth Nonconvex Optimization is As Efficient As Smooth Nonconvex Optimization},
  booktitle = {Proceedings of the 40th International Conference on Machine Learning (ICML)},
  series    = {Proceedings of Machine Learning Research},
  volume    = {202},
  pages     = {5396--5427},
  year      = {2023}
}

@inproceedings{crawshaw2022robustness,
  author    = {Crawshaw, Michael and Liu, Mingrui and Orabona, Francesco and Zhang, Wei and Zhuang, Zhenxun},
  title     = {Robustness to Unbounded Smoothness of Generalized Sign{SGD}},
  booktitle = {Advances in Neural Information Processing Systems (NeurIPS)},
  volume    = {35},
  year      = {2022}
}

@inproceedings{huebler2024parameter,
  author    = {H{\"u}bler, Florian and Yang, Junchi and Li, Xiang and He, Niao},
  title     = {Parameter-Agnostic Optimization under Relaxed Smoothness},
  booktitle = {Proceedings of the 27th International Conference on Artificial Intelligence and Statistics (AISTATS)},
  series    = {Proceedings of Machine Learning Research},
  volume    = {238},
  pages     = {4861--4869},
  year      = {2024}
}

@inproceedings{wang2023convergence,
  author    = {Wang, Bohan and Zhang, Huishuai and Ma, Zhiming and Chen, Wei},
  title     = {Convergence of {AdaGrad} for Non-convex Objectives: Simple Proofs and Relaxed Assumptions},
  booktitle = {Proceedings of the 36th Annual Conference on Learning Theory (COLT)},
  series    = {Proceedings of Machine Learning Research},
  volume    = {195},
  pages     = {161--190},
  year      = {2023}
}

@inproceedings{vankov2024optimizing,
  author    = {Vankov, Daniil and Rodomanov, Anton and Nedi{\'c}, Angelia and Sankar, Lalitha and Stich, Sebastian U.},
  title     = {Optimizing $(L_0,L_1)$-Smooth Functions by Gradient Methods},
  booktitle = {International Conference on Learning Representations (ICLR)},
  year      = {2025},
  url       = {https://proceedings.iclr.cc/paper_files/paper/2025/file/28a9255e2a0e7062e8e0af715f2e5b86-Paper-Conference.pdf}
}

@inproceedings{gorbunov2024methods,
  author    = {Gorbunov, Eduard and Tupitsa, Nazarii and Choudhury, Sayantan and Aliev, Alen and Richt{\'a}rik, Peter and Horv{\'a}th, Samuel and Tak{\'a}{\v{c}}, Martin},
  title     = {Methods for Convex $(L_0,L_1)$-Smooth Optimization: Clipping, Acceleration, and Adaptivity},
  booktitle = {International Conference on Learning Representations (ICLR)},
  year      = {2025},
  eprint    = {2409.14989},
  archivePrefix = {arXiv},
  url       = {https://arxiv.org/abs/2409.14989}
}

@article{lobanov2024linear,
  author    = {Lobanov, Aleksandr and Gasnikov, Alexander and Gorbunov, Eduard and Tak{\'a}{\v{c}}, Martin},
  title     = {Linear Convergence Rate in Convex Setup is Possible! {G}radient Descent Method Variants under $(L_0,L_1)$-Smoothness},
  journal   = {arXiv preprint arXiv:2412.17050},
  year      = {2024},
  eprint    = {2412.17050},
  archivePrefix = {arXiv},
  url       = {https://arxiv.org/abs/2412.17050}
}

@article{sgd2025generalized,
  author    = {Lobanov, Aleksandr and Gasnikov, Alexander},
  title     = {Power of Generalized Smoothness in Stochastic Convex Optimization: First- and Zero-Order Algorithms},
  journal   = {arXiv preprint arXiv:2501.18198},
  year      = {2025},
  eprint    = {2501.18198},
  archivePrefix = {arXiv},
  url       = {https://arxiv.org/abs/2501.18198}
}

@article{nemirovski2009robust,
  author    = {Nemirovski, Arkadi and Juditsky, Anatoli and Lan, Guanghui and Shapiro, Alexander},
  title     = {Robust Stochastic Approximation Approach to Stochastic Programming},
  journal   = {SIAM Journal on Optimization},
  volume    = {19},
  number    = {4},
  pages     = {1574--1609},
  year      = {2009}
}

@article{lan2012optimal,
  author    = {Lan, Guanghui},
  title     = {An Optimal Method for Stochastic Composite Optimization},
  journal   = {Mathematical Programming},
  volume    = {133},
  number    = {1},
  pages     = {365--397},
  year      = {2012}
}

@article{ghadimi2012optimal,
  author    = {Ghadimi, Saeed and Lan, Guanghui},
  title     = {Optimal Stochastic Approximation Algorithms for Strongly Convex Stochastic Composite Optimization {I}: A Generic Algorithmic Framework},
  journal   = {SIAM Journal on Optimization},
  volume    = {22},
  number    = {4},
  pages     = {1469--1492},
  year      = {2012}
}

@article{ghadimi2013optimal,
  author    = {Ghadimi, Saeed and Lan, Guanghui},
  title     = {Optimal Stochastic Approximation Algorithms for Strongly Convex Stochastic Composite Optimization {II}: Shrinking Procedures and Optimal Algorithms},
  journal   = {SIAM Journal on Optimization},
  volume    = {23},
  number    = {4},
  pages     = {2061--2089},
  year      = {2013}
}

@book{nesterov2018lectures,
  author    = {Nesterov, Yurii},
  title     = {Lectures on Convex Optimization},
  edition   = {2nd},
  publisher = {Springer},
  year      = {2018}
}

@article{nesterov1983method,
  author    = {Nesterov, Yurii},
  title     = {A Method for Solving the Convex Programming Problem with Convergence Rate {$O(1/k^2)$}},
  journal   = {Doklady Akademii Nauk SSSR},
  volume    = {269},
  number    = {3},
  pages     = {543--547},
  year      = {1983}
}

@inproceedings{foster2019complexity,
  author    = {Foster, Dylan J. and Sekhari, Ayush and Shamir, Ohad and Srebro, Nathan and Sridharan, Karthik and Woodworth, Blake},
  title     = {The Complexity of Making the Gradient Small in Stochastic Convex Optimization},
  booktitle = {Proceedings of the 32nd Annual Conference on Learning Theory (COLT)},
  series    = {Proceedings of Machine Learning Research},
  volume    = {99},
  pages     = {1319--1345},
  year      = {2019}
}

@article{nesterov2012make,
  author    = {Nesterov, Yurii},
  title     = {How to Make the Gradients Small},
  journal   = {Optima. Mathematical Optimization Society Newsletter},
  volume    = {88},
  pages     = {10--11},
  year      = {2012}
}

@inproceedings{allenzhu2018make,
  author    = {Allen-Zhu, Zeyuan},
  title     = {How To Make the Gradients Small Stochastically: Even Faster Convex and Nonconvex {SGD}},
  booktitle = {Advances in Neural Information Processing Systems (NeurIPS)},
  volume    = {31},
  year      = {2018}
}

@inproceedings{gorbunov2020stochastic,
  author    = {Gorbunov, Eduard and Danilova, Marina and Gasnikov, Alexander},
  title     = {Stochastic Optimization with Heavy-Tailed Noise via Accelerated Gradient Clipping},
  booktitle = {Advances in Neural Information Processing Systems (NeurIPS)},
  volume    = {33},
  year      = {2020}
}

@inproceedings{sadiev2023high,
  author    = {Sadiev, Abdurakhmon and Danilova, Marina and Gorbunov, Eduard and Horv{\'a}th, Samuel and Gidel, Gauthier and Dvurechensky, Pavel and Gasnikov, Alexander and Richt{\'a}rik, Peter},
  title     = {High-Probability Bounds for Stochastic Optimization and Variational Inequalities: the Case of Unbounded Variance},
  booktitle = {Proceedings of the 40th International Conference on Machine Learning (ICML)},
  series    = {Proceedings of Machine Learning Research},
  volume    = {202},
  pages     = {29563--29648},
  year      = {2023}
}

@article{jin2019short,
  author    = {Jin, Chi and Netrapalli, Praneeth and Ge, Rong and Kakade, Sham M. and Jordan, Michael I.},
  title     = {A Short Note on Concentration Inequalities for Random Vectors with Sub{G}aussian Norm},
  journal   = {arXiv preprint arXiv:1902.03736},
  year      = {2019}
}

@inproceedings{vaswani2019fast,
  author    = {Vaswani, Sharan and Bach, Francis and Schmidt, Mark},
  title     = {Fast and Faster Convergence of {SGD} for Over-Parameterized Models and an Accelerated Perceptron},
  booktitle = {Proceedings of the 22nd International Conference on Artificial Intelligence and Statistics (AISTATS)},
  series    = {Proceedings of Machine Learning Research},
  volume    = {89},
  pages     = {1195--1204},
  year      = {2019}
}

@inproceedings{liu2020accelerating,
  author    = {Liu, Chaoyue and Belkin, Mikhail},
  title     = {Accelerating {SGD} with Momentum for Over-Parameterized Learning},
  booktitle = {International Conference on Learning Representations (ICLR)},
  year      = {2020}
}

@article{schmidt2013fast,
  author    = {Schmidt, Mark and Le Roux, Nicolas},
  title     = {Fast Convergence of Stochastic Gradient Descent under a Strong Growth Condition},
  journal   = {arXiv preprint arXiv:1308.6370},
  year      = {2013}
}

@inproceedings{ma2018power,
  author    = {Ma, Siyuan and Bassily, Raef and Belkin, Mikhail},
  title     = {The Power of Interpolation: Understanding the Effectiveness of {SGD} in Modern Over-Parametrized Learning},
  booktitle = {Proceedings of the 35th International Conference on Machine Learning (ICML)},
  series    = {Proceedings of Machine Learning Research},
  volume    = {80},
  pages     = {3325--3334},
  year      = {2018}
}

@article{interp2026gs,
  author    = {Lobanov, Aleksandr and Koloskova, Anastasia},
  title     = {Avoiding Bias in Clipped {SGD} for Overparameterized Models under Generalized Smoothness},
  journal   = {arXiv preprint arXiv:2605.14800},
  year      = {2026},
  eprint    = {2605.14800},
  archivePrefix = {arXiv},
  url       = {https://arxiv.org/abs/2605.14800}
}

@inproceedings{gaash2025clipped,
  author    = {Gaash, Ofir and Levy, Kfir Yehuda and Carmon, Yair},
  title     = {Convergence of Clipped {SGD} on Convex $(L_0,L_1)$-Smooth Functions},
  booktitle = {Advances in Neural Information Processing Systems (NeurIPS)},
  volume    = {38},
  year      = {2025},
  eprint    = {2502.16492},
  archivePrefix = {arXiv},
  url       = {https://papers.nips.cc/paper_files/paper/2025/hash/ae6c3f14de6f59ff9a44425a792e5bd1-Abstract-Conference.html}
}

@article{tovmasyan2025proximal,
  author    = {Tovmasyan, Zhirayr and Malinovsky, Grigory and Condat, Laurent and Richt{\'a}rik, Peter},
  title     = {Revisiting Stochastic Proximal Point Methods: Generalized Smoothness and Similarity},
  journal   = {Journal of Nonlinear and Variational Analysis},
  volume    = {10},
  number    = {3},
  pages     = {471--505},
  year      = {2026},
  doi       = {10.23952/JNVA.10.2026.3.01},
  url       = {https://doi.org/10.23952/JNVA.10.2026.3.01},
  eprint    = {2502.03401},
  archivePrefix = {arXiv}
}

@inproceedings{malitsky2020adaptive,
  author    = {Malitsky, Yura and Mishchenko, Konstantin},
  title     = {Adaptive Gradient Descent without Descent},
  booktitle = {Proceedings of the 37th International Conference on Machine Learning (ICML)},
  series    = {Proceedings of Machine Learning Research},
  volume    = {119},
  pages     = {6702--6712},
  year      = {2020}
}

@article{malitsky2020golden,
  author    = {Malitsky, Yura},
  title     = {Golden Ratio Algorithms for Variational Inequalities},
  journal   = {Mathematical Programming},
  volume    = {184},
  number    = {1--2},
  pages     = {383--425},
  year      = {2020}
}

@book{nemirovsky1983problem,
  author    = {Nemirovsky, Arkadi S. and Yudin, David B.},
  title     = {Problem Complexity and Method Efficiency in Optimization},
  publisher = {Wiley-Interscience},
  year      = {1983}
}

@article{minsker2015geometric,
  author  = {Minsker, Stanislav},
  title   = {Geometric median and robust estimation in {B}anach spaces},
  journal = {Bernoulli},
  volume  = {21},
  number  = {4},
  pages   = {2308--2335},
  year    = {2015}
}

@inproceedings{tyurin2026near,
  author    = {Tyurin, Alexander},
  title     = {Near-Optimal Convergence of Accelerated Gradient Methods under Generalized and $(L_0,L_1)$-Smoothness},
  booktitle = {Proceedings of the 43rd International Conference on Machine Learning (ICML)},
  year      = {2026},
  eprint    = {2508.06884},
  archivePrefix = {arXiv},
  url       = {https://arxiv.org/abs/2508.06884}
}

@article{yu2025stochastic,
  author    = {Yu, Chenhao and Hong, Yusu and Lin, Junhong},
  title     = {Convergence Analysis of Stochastic Accelerated Gradient Methods for Generalized Smooth Optimizations},
  journal   = {arXiv preprint arXiv:2502.11125},
  year      = {2025},
  eprint    = {2502.11125},
  archivePrefix = {arXiv},
  url       = {https://arxiv.org/abs/2502.11125}
}

@misc{anonymous2026nesterov,
  author       = {{Anonymous}},
  title        = {Convergence Analysis of {N}esterov's Accelerated Gradient Descent under Relaxed Assumptions},
  howpublished = {OpenReview submission to the International Conference on Learning Representations},
  year         = {2026},
  url          = {https://openreview.net/forum?id=AlYT0ZD51A},
  note         = {Public ICLR 2026 submission}
}

@inproceedings{carmon2020balloracle,
  author    = {Carmon, Yair and Jambulapati, Arun and Jiang, Qijia and Jin, Yujia and Lee, Yin Tat and Sidford, Aaron and Tian, Kevin},
  title     = {Acceleration with a Ball Optimization Oracle},
  booktitle = {Advances in Neural Information Processing Systems (NeurIPS)},
  volume    = {33},
  year      = {2020},
  eprint    = {2003.08078},
  archivePrefix = {arXiv},
  url       = {https://proceedings.neurips.cc/paper/2020/hash/dba4c1a117472f6aca95211285d0587e-Abstract.html}
}

@article{carmon2023resque,
  author    = {Carmon, Yair and Jambulapati, Arun and Jin, Yujia and Lee, Yin Tat and Liu, Daogao and Sidford, Aaron and Tian, Kevin},
  title     = {{ReSQue}ing Parallel and Private Stochastic Convex Optimization},
  journal   = {arXiv preprint arXiv:2301.00457},
  year      = {2023},
  eprint    = {2301.00457},
  archivePrefix = {arXiv},
  url       = {https://arxiv.org/abs/2301.00457}
}

@article{gruntkowska2025broximal,
  author    = {Gruntkowska, Kaja and Li, Hanmin and Rane, Aadi and Richt{\'a}rik, Peter},
  title     = {The Ball-Proximal (``Broximal'') Point Method: A New Algorithm, Convergence Theory, and Applications},
  journal   = {arXiv preprint arXiv:2502.02002},
  year      = {2025},
  eprint    = {2502.02002},
  archivePrefix = {arXiv},
  url       = {https://arxiv.org/abs/2502.02002}
}

@article{li2026trustregion,
  author    = {Li, Hanmin and Gruntkowska, Kaja and Richt{\'a}rik, Peter},
  title     = {Stabilized Proximal Point Method via Trust Region Control},
  journal   = {arXiv preprint arXiv:2604.02943},
  year      = {2026},
  eprint    = {2604.02943},
  archivePrefix = {arXiv},
  url       = {https://arxiv.org/abs/2604.02943}
}

\end{document}